\documentclass[12pt]{amsart}
\usepackage{latexsym}
\usepackage{indentfirst}
\usepackage{amssymb,amsfonts,amsmath,amsthm}
\usepackage{graphicx}
\usepackage{mathrsfs}
\usepackage{array}
\usepackage{color}
\usepackage{epstopdf}
\usepackage[all]{xy}
\usepackage{verbatim}

\usepackage{url}

\newtheorem{thm}{Theorem}[section]
\newtheorem{prop}[thm]{Proposition}
\newtheorem{lem}[thm]{Lemma}

\newtheorem{cor}[thm]{Corollary}

\numberwithin{equation}{section}
\usepackage{relsize}
\usepackage[bbgreekl]{mathbbol}
\DeclareSymbolFontAlphabet{\mathbb}{AMSb} 
\DeclareSymbolFontAlphabet{\mathbbl}{bbold}

 \usepackage[normalem]{ulem}

\usepackage{tikz-cd}

\theoremstyle{definition}

\newtheorem{construction}[thm]{Construction}

\newtheorem{notation}[thm]{Notation}

\newtheorem{convention}[thm]{Convention}

\newtheorem{question}[thm]{Question}

\newtheorem{dfn}[thm]{Definition}
\newtheorem{exam}[thm]{Example}

\newtheorem{rmk}[thm]{Remark}

\usepackage[a4paper,bindingoffset=0.2in,           left=0.5in,right=0.51in,top=0.6in,bottom=0.5in,            footskip=.25in]{geometry}

\usepackage{color}
\definecolor{ghcolor}{RGB}{0, 150, 200} 
\definecolor{winestain}{rgb}{0.5,0,0}
\usepackage[linktocpage=true, hidelinks, colorlinks, linkcolor=blue, citecolor=red, bookmarksopen=true, urlcolor=brown]{hyperref}

\begin{document}

\newcommand{\fraka}{{\mathfrak a}}
\newcommand{\frakb}{{\mathfrak b}}
\newcommand{\frakc}{{\mathfrak c}}
\newcommand{\frakd}{{\mathfrak d}}
\newcommand{\frake}{{\mathfrak e}}
\newcommand{\frakf}{{\mathfrak f}}
\newcommand{\frakg}{{\mathfrak g}}
\newcommand{\frakh}{{\mathfrak h}}
\newcommand{\fraki}{{\mathfrak i}}
\newcommand{\frakj}{{\mathfrak j}}
\newcommand{\frakk}{{\mathfrak k}}
\newcommand{\frakl}{{\mathfrak l}}
\newcommand{\frakm}{{\mathfrak m}}
\newcommand{\frakn}{{\mathfrak n}}
\newcommand{\frako}{{\mathfrak o}}
\newcommand{\frakp}{{\mathfrak p}}
\newcommand{\frakq}{{\mathfrak q}}
\newcommand{\frakr}{{\mathfrak r}}
\newcommand{\fraks}{{\mathfrak s}}
\newcommand{\frakt}{{\mathfrak t}}
\newcommand{\fraku}{{\mathfrak u}}
\newcommand{\frakv}{{\mathfrak v}}
\newcommand{\frakw}{{\mathfrak w}}
\newcommand{\frakx}{{\mathfrak x}}
\newcommand{\fraky}{{\mathfrak y}}
\newcommand{\frakz}{{\mathfrak z}}

\newcommand{\frakA}{{\mathfrak A}}
\newcommand{\frakB}{{\mathfrak B}}
\newcommand{\frakC}{{\mathfrak C}}
\newcommand{\frakD}{{\mathfrak D}}
\newcommand{\frakE}{{\mathfrak E}}
\newcommand{\frakF}{{\mathfrak F}}
\newcommand{\frakG}{{\mathfrak G}}
\newcommand{\frakH}{{\mathfrak H}}
\newcommand{\frakI}{{\mathfrak I}}
\newcommand{\frakJ}{{\mathfrak J}}
\newcommand{\frakK}{{\mathfrak K}}
\newcommand{\frakL}{{\mathfrak L}}
\newcommand{\frakM}{{\mathfrak M}}
\newcommand{\frakN}{{\mathfrak N}}
\newcommand{\frakO}{{\mathfrak O}}
\newcommand{\frakP}{{\mathfrak P}}
\newcommand{\frakQ}{{\mathfrak Q}}
\newcommand{\frakR}{{\mathfrak R}}
\newcommand{\frakS}{{\mathfrak S}}
\newcommand{\frakT}{{\mathfrak T}}
\newcommand{\frakU}{{\mathfrak U}}
\newcommand{\frakV}{{\mathfrak V}}
\newcommand{\frakW}{{\mathfrak W}}
\newcommand{\frakX}{{\mathfrak X}}
\newcommand{\frakY}{{\mathfrak Y}}
\newcommand{\frakZ}{{\mathfrak Z}}

\newcommand{\bA}{{\mathbb A}}
\newcommand{\bB}{{\mathbb B}}
\newcommand{\bC}{{\mathbb C}}
\newcommand{\bD}{{\mathbb D}}
\newcommand{\bE}{{\mathbb E}}
\newcommand{\bF}{{\mathbb F}}
\newcommand{\bG}{{\mathbb G}}
\newcommand{\bH}{{\mathbb H}}
\newcommand{\bI}{{\mathbb I}}
\newcommand{\bJ}{{\mathbb J}}
\newcommand{\bK}{{\mathbb K}}
\newcommand{\bL}{{\mathbb L}}
\newcommand{\bM}{{\mathbb M}}
\newcommand{\bN}{{\mathbb N}}
\newcommand{\bO}{{\mathbb O}}
\newcommand{\bP}{{\mathbb P}}
\newcommand{\bQ}{{\mathbb Q}}
\newcommand{\bR}{{\mathbb R}}
\newcommand{\bS}{{\mathbb S}}
\newcommand{\bT}{{\mathbb T}}
\newcommand{\bU}{{\mathbb U}}
\newcommand{\bV}{{\mathbb V}}
\newcommand{\bW}{{\mathbb W}}
\newcommand{\bX}{{\mathbb X}}
\newcommand{\bY}{{\mathbb Y}}
\newcommand{\bZ}{{\mathbb Z}}

\newcommand{\calA}{{\mathcal A}}
\newcommand{\calB}{{\mathcal B}}
\newcommand{\calC}{{\mathcal C}}
\newcommand{\calD}{{\mathcal D}}
\newcommand{\calE}{{\mathcal E}}
\newcommand{\calF}{{\mathcal F}}
\newcommand{\calG}{{\mathcal G}}
\newcommand{\calH}{{\mathcal H}}
\newcommand{\calI}{{\mathcal I}}
\newcommand{\calJ}{{\mathcal J}}
\newcommand{\calK}{{\mathcal K}}
\newcommand{\calL}{{\mathcal L}}
\newcommand{\calM}{{\mathcal M}}
\newcommand{\calN}{{\mathcal N}}
\newcommand{\calO}{{\mathcal O}}
\newcommand{\calP}{{\mathcal P}}
\newcommand{\calQ}{{\mathcal Q}}
\newcommand{\calR}{{\mathcal R}}
\newcommand{\calS}{{\mathcal S}}
\newcommand{\calT}{{\mathcal T}}
\newcommand{\calU}{{\mathcal U}}
\newcommand{\calV}{{\mathcal V}}
\newcommand{\calW}{{\mathcal W}}
\newcommand{\calX}{{\mathcal X}}
\newcommand{\calY}{{\mathcal Y}}
\newcommand{\calZ}{{\mathcal Z}}

\newcommand{\rA}{{\mathrm A}}
\newcommand{\rB}{{\mathrm B}}
\newcommand{\rC}{{\mathrm C}}
\newcommand{\rD}{{\mathrm D}}
\newcommand{\rd}{{\mathrm d}}
\newcommand{\rE}{{\mathrm E}}
\newcommand{\rF}{{\mathrm F}}
\newcommand{\rG}{{\mathrm G}}
\newcommand{\rH}{{\mathrm H}}
\newcommand{\rI}{{\mathrm I}}
\newcommand{\rJ}{{\mathrm J}}
\newcommand{\rK}{{\mathrm K}}
\newcommand{\rL}{{\mathrm L}}
\newcommand{\rM}{{\mathrm M}}
\newcommand{\rN}{{\mathrm N}}
\newcommand{\rO}{{\mathrm O}}
\newcommand{\rP}{{\mathrm P}}
\newcommand{\rQ}{{\mathrm Q}}
\newcommand{\rR}{{\mathrm R}}
\newcommand{\rS}{{\mathrm S}}
\newcommand{\rT}{{\mathrm T}}
\newcommand{\rU}{{\mathrm U}}
\newcommand{\rV}{{\mathrm V}}
\newcommand{\rW}{{\mathrm W}}
\newcommand{\rX}{{\mathrm X}}
\newcommand{\rY}{{\mathrm Y}}
\newcommand{\rZ}{{\mathrm Z}}

\newcommand{\rmd}{{\mathrm d}}
\newcommand{\Zp}{{\bZ_p}}
\newcommand{\Zl}{{\bZ_l}}
\newcommand{\Qp}{{\bQ_p}}
\newcommand{\Ql}{{\bQ_l}}
\newcommand{\Cp}{{\bC_p}}
\newcommand{\Fp}{{\bF_p}}
\newcommand{\Fl}{{\bF_l}}

\newcommand{\Binf}{{\mathrm{B_{inf}}}}
\newcommand{\Acris}{{\mathrm{A_{cris}}}}
\newcommand{\Bcrisp}{{\mathrm{B_cris}^+}}
\newcommand{\Bcris}{{\mathrm{B_{cris}}}}
\newcommand{\Bstp}{{\mathrm{B_{st}^+}}}
\newcommand{\Bst}{{\mathrm{B_{st}}}}

\newcommand{\Ainf}{{\mathbf{A}_{\mathrm{inf}}}}
\newcommand{\BdRp}{{\mathbf{B}_{\mathrm{dR}}^+}}
\newcommand{\BdR}{{\mathbf{B}_{\mathrm{dR}}}}

\newcommand{\Rinf}{{\widehat R_{\infty}}}
\newcommand{\Rinfp}{{\widehat R_{\infty}^+}}

\newcommand{\OXp}{{\widehat \calO_X^+}}
\newcommand{\OX}{{\widehat \calO_X}}
\newcommand{\AAinf}{{\bA_{\mathrm{inf}}}}            
\newcommand{\BBinf}{{\bB_{\mathrm{inf}}}}            
\newcommand{\BBbd}{{\bB^{\mathrm{bd}}}}              
\newcommand{\BBdRp}{{\bB_{\mathrm{dR}}^+}}           
\newcommand{\BBdR}{{\bB_{\mathrm{dR}}}}              
\newcommand{\OAinf}{{\calO\bA_{\mathrm{inf}}}}         
\newcommand{\OBinf}{{\calO\bB_{\mathrm{inf}}}}         
\newcommand{\OBdRp}{{\calO\bB_{\mathrm{dR}}^+}}        
\newcommand{\OBdR}{{\calO\bB_{\mathrm{dR}}}}           
\newcommand{\OC}{{\calO\bC}}                           


\newcommand{\Bun}{{\mathrm{Bun}}}           
\newcommand{\Cech}{{\check{H}}}             
\newcommand{\Coker}{{\mathrm{Coker}}}       
\newcommand{\colim}{{\mathrm{colim}}}       
\newcommand{\dlog}{{\mathrm{dlog}}}         
\newcommand{\End}{{\mathrm{End}}}           
\newcommand{\Exp}{{\mathrm{Exp}}}           
\newcommand{\Ext}{{\mathrm{Ext}}}           
\newcommand{\Fil}{{\mathrm{Fil}}}           
\newcommand{\Fitt}{{\mathrm{Fitt}}}         
\newcommand{\Frac}{{\mathrm{Frac}}}         
\newcommand{\Frob}{{\mathrm{Frob}}}         
\newcommand{\Gal}{{\mathrm{Gal}}}           
\newcommand{\Gr}{{\mathrm{Gr}}}             
\newcommand{\Hom}{{\mathrm{Hom}}}           
\newcommand{\HIG}{{\mathrm{HIG}}}           
\newcommand{\id}{{\mathrm{id}}}             
\newcommand{\Ima}{{\mathrm{Im}}}            
\newcommand{\Isom}{{\mathrm{Isom}}}         
\newcommand{\Ker}{{\mathrm{Ker}}}           
\newcommand{\Lie}{{\mathrm{Lie}}}           
\newcommand{\Log}{{\mathrm{Log}}}           
\newcommand{\MIC}{{\mathrm{MIC}}}
\newcommand{\Mod}{{\mathrm{Mod}}}           
\newcommand{\Perf}{{\mathrm{Perf}}}         
\newcommand{\pr}{{\mathrm{pr}}}             
\newcommand{\Rep}{{\mathrm{Rep}}}           
\newcommand{\Res}{{\mathrm{Res}}}           
\newcommand{\RGamma}{{\mathrm{R\Gamma}}}    
\newcommand{\rk}{{\mathrm{rk}}}             
\newcommand{\Rlim}{{\mathrm{R}\underleftarrow{\lim}}} 
\newcommand{\sgn}{{\mathrm{sgn}}}           
\newcommand{\Sh}{{\mathrm{Shv}}}             
\newcommand{\Sht}{{\mathrm{Sht}}}           
\newcommand{\Spa}{{\mathrm{Spa}}}           
\newcommand{\Spf}{{\mathrm{Spf}}}           
\newcommand{\Spec}{{\mathrm{Spec}}}         
\newcommand{\Strat}{{\mathrm{Strat}}}       
\newcommand{\Sym}{{\mathrm{Sym}}}           
\newcommand{\Tor}{{\mathrm{{Tor}}}}         
\newcommand{\Tot}{{\mathrm{Tot}}}           
\newcommand{\Vect}{{\mathrm{Vect}}}         


\newcommand{\GL}{{\mathrm{GL}}}             
\newcommand{\SL}{{\mathrm{SL}}}             


\newcommand{\aff}{{\mathrm{aff}}}          
\newcommand{\an}{{\mathrm{an}}}             
\newcommand{\can}{{\mathrm{can}}}           
\newcommand{\cl}{{\mathrm{cl}}}             
\newcommand{\cofib}{{\mathrm{cofib}}}
\newcommand{\cts}{{\mathrm{cts}}}           
\newcommand{\cris}{{\mathrm{cris}}}         
\newcommand{\cyc}{{\mathrm{cyc}}}           
\newcommand{\dR}{{\mathrm{dR}}}             
\newcommand{\et}{{\mathrm{\acute{e}t}}}    
\newcommand{\fib}{{\mathrm{fib}}}
\newcommand{\fl}{{\mathrm{fl}}}             
\newcommand{\fppf}{{\mathrm{fppf}}}         
\newcommand{\geo}{{\mathrm{geo}}}           
\newcommand{\gp}{{\mathrm{gp}}}             
\newcommand{\la}{{\mathrm{la}}}
\newcommand{\nil}{{\mathrm{nil}}}
\newcommand{\pd}{{\mathrm{pd}}}
\newcommand{\perf}{\mathrm{perf}}           
\newcommand{\proet}{{\mathrm{pro\acute{e}t}}}
\newcommand{\sm}{{\mathrm{sm}}}             
\newcommand{\st}{{\mathrm{st}}}             
\newcommand{\tor}{{\mathrm{tor}}}           
\newcommand{\Zar}{{\mathrm{Zar}}}           


\newcommand{\Prism}{{\mathlarger{\mathbbl{\Delta}}}} 

\newcommand{\OPrism}{{\calO_{\Prism}}}
\newcommand{\IPrism}{{\calI_{\Prism}}}

\newcommand{\ya}{{\rangle}}
\newcommand{\za}{{\langle}}


\renewcommand{\log}{{\mathrm{log}}}
\newcommand{\mic}{\mathrm{MIC}}

 \newcommand{\gh}[1]{{\color{blue}#1}}
\newcommand{\my}[1]{{\color{red}#1}}
\newcommand{\wyp}[1]{{\color{magenta}#1}}

\newcommand{\brig}{{\B_{\rig, \kinfty}^+}}

\newcommand{\BP}{{\mathrm{BP}}}

\newcommand{\smat}[1]{\left( \begin{smallmatrix} #1 \end{smallmatrix} \right)}
\newcommand{\dacc}[1]{\{\!\{ #1 \}\!\}}
 \newcommand{\bm}{{\mathbb M}}

\newcommand{\bdrpluskpinfty}{{\B_{\dR, \kpinfty}^+}}

 \newcommand{\obnht}{\mathcal{O}\mathbf{B_{\mathrm{nHT}}}}
\newcommand{\nht}{{\mathrm{nHT}}}
\newcommand{\ndR}{{\mathrm{ndR}}}
\newcommand{\Conn}{\mathrm{Conn}}
  \newcommand{\kpinftyt}{{\kpinfty[[t]]}}
 \newcommand{\kinftylambda}{{\kinfty[[\lambda]]}}
  \newcommand{\bdrplusm}{{\mathbf{B}^+_{\dR, m}}}
\newcommand{\bdrplusml}{{\mathbf{B}^+_{\dR, m, L}}}
\newcommand{\bdrplusl}{{\mathbf{B}^+_{\dR,  L}}}
 \newcommand{\bdrpluslhatgpa}{{(\B_{\dR, L}^+)^{\hat{G}\dpa}}}

 \newcommand{\fg}{{\mathrm{fg}}}
 \newcommand{\Endo}{\mathrm{Endo}}

\newcommand{\gal}{\mathrm{Gal}}

\def \inj {\hookrightarrow }
\def \into {\hookrightarrow }
\def \to {\rightarrow}
\def \onto {\twoheadrightarrow}
\renewcommand{\projlim}{\varprojlim}
\renewcommand{\injlim}{\varinjlim}


\newcommand{\tr}{\operatorname*{tr}}
\newcommand{\Tr}{\operatorname*{Tr}}
\newcommand{\vect}{\mathrm{Vect}}

\newcommand{\gld}{\mathrm{GL}_d}
\newcommand{\gln}{\mathrm{GL}_n}
\newcommand{\Aut}{\mathrm{Aut}}

\def\Mat{\mathrm{Mat}}
\def\mat{\mathrm{Mat}}
\def\det{\mathrm{det}}
\def\Der{\mathrm{Der}}

\def \Md {{\mathrm M}_d}
\newcommand{\col}{\mathrm{col}}
\newcommand{\row}{\mathrm{row}}

\newcommand{\loc}{{\mathrm{loc}}}
\newcommand{\new}{{\mathrm{new}}}
\newcommand{\isoto}{\stackrel{\sim}{\rightarrow}}
\def\val{\mathrm val}

\newcommand{\fr}{\mathrm{fr}}
\newcommand{\nc}{{\mathrm{nc}}}

\newcommand{\rep}{\mathrm{Rep}}

\renewcommand{\hom}{\mathrm{Hom}}

\def\cont{\mathrm{cont}}
\def\inf{{\mathrm{inf}}}
\def\sup{\mathrm{sup}}

\renewcommand{\Im}{\textnormal{Im}}

\renewcommand{\sp}{Sp}

\newcommand{\rgamma}{\mathrm{R}\Gamma}

\def\ad{\mathrm{ad}}
\def\an{\mathrm{an}}

\def\can{\mathrm{can}}
\def\et{{\mathrm{\acute{e}t}}}
\def\eh{{\mathrm{\acute{e}h}}}
\def\proet{{\mathrm{pro\acute{e}t}}}
\def\profet{{\mathrm{prof\acute{e}t}}}
\def\proket{{\mathrm{prok\acute{e}t}}}
\def\alg{{\mathrm alg}}

\newcommand{\dotimes}{{\otimes}^{\mathbb L}}
\newcommand{\hotimes}{\widehat{\otimes}}
\newcommand{\hatotimes}{\widehat{\otimes}}
\def\cycl{\mathrm{cycl}}
\def\ord{\mathrm{ord}}

\newcommand{\nr}{\mathrm{nr}}
\newcommand{\unr}{\mathrm{unr}}
\newcommand{\un}{\mathrm{un}}
\newcommand{\ur}{\mathrm{ur}}
\newcommand{\sep}{\mathrm{sep}}
\newcommand{\Kpinfty}{{K_{p^\infty}}}
\newcommand{\kpinfty}{{K_{p^\infty}}}
\newcommand{\hatkpinfty}{{\widehat{K_{p^\infty}}}}
\newcommand{\Kinfty}{{K_{\infty}}}
\newcommand{\kinfty}{{K_{\infty}}}
\newcommand{\hatkinfty}{{\widehat{K_{\infty}}}}
\newcommand{\gammak}{{\Gamma_K}}
\newcommand{\gammabbk}{{\Gamma_{\mathbb{K}}}}
\newcommand{\gk}{{G_K}}
\newcommand{\gkinfty}{{G_{K_\infty}}}
\newcommand{\gkpinfty}{{G_{K_{p^\infty}}}}

\newcommand{\pa}{{\mathrm{pa}}}
\newcommand{\dan}{\text{$\mbox{-}\mathrm{an}$}}
\newcommand{\dla}{\text{$\mbox{-}\mathrm{la}$}}
\newcommand{\dfin}{\text{$\mbox{-}\mathrm{fin}$}}
\newcommand{\dpa}{\text{$\mbox{-}\mathrm{pa}$}}

\def\cris{{\mathrm{cris}}}
\def\crys{{\mathrm{crys}}}
\def\st{{\mathrm{st}}}
\def\pst{{\mathrm{pst}}}
\def\dR{{\mathrm{dR}}}
\def\dr{{\mathrm{dR}}}
\def\HT{{\mathrm{HT}}}
\def\Sen{{\mathrm{Sen}}}
\def\Bri{{\mathrm{Bri}}}
\def\sen{{\mathrm{Sen}}}
\def\bri{{\mathrm{Bri}}}
\def\dif{{\mathrm{dif}}}
\def\rig{{\mathrm{rig}}}

\newcommand{\ainf}{{\mathbf{A}_{\mathrm{inf}}}}
\newcommand{\acris}{{\mathbf{A}_{\mathrm{cris}}}}
\newcommand{\Bcrisplus}{{\mathbf{B}^+_{\mathrm{cris}}}}
\newcommand{\bcris}{{\mathbf{B}_{\mathrm{cris}}}}
\newcommand{\bcrisplus}{{\mathbf{B}^+_{\mathrm{cris}}}}
\newcommand{\bst}{{\mathbf{B}_{\mathrm{st}}}}
\newcommand{\bdrplus}{{\mathbf{B}^+_{\mathrm{dR}}}}
\newcommand{\bdrplusnabla}{{\mathbf{B}^{+, \nabla}_{\mathrm{dR}}}}

\newcommand{\bdr}{{\mathbf{B}_{\mathrm{dR}}}}
\newcommand{\bdrnabla}{{\mathbf{B}^{\nabla}_{\mathrm{dR}}}}

\newcommand{\bdrL}{{\mathbf{B}_{\mathrm{dR}, L}}}
\newcommand{\bht}{{\mathbf{B}_{\mathrm{HT}}}}

\newcommand{\bbcris}{{\mathbb{B}_{\mathrm{cris}}}}
\newcommand{\bbst}{{\mathbb{B}_{\mathrm{st}}}}
\newcommand{\bbdr}{{\mathbb{B}_{\mathrm{dR}}}}
\newcommand{\bbdrplus}{{\mathbb{B}_{\mathrm{dR}}^{+}}}
\newcommand{\bbdrplusm}{{\mathbb{B}_{\mathrm{dR}, m}^{+}}}

\newcommand{\bbht}{{\mathbb{B}_{\mathrm{HT}}}}
\newcommand{\obbdrplus}{{\mathcal{O}\bbdrplus}}
\newcommand{\obbdr}{{\mathcal{O}\bbdr}}
\newcommand{\obbht}{{\mathcal{O}\bbht}}
\newcommand{\obbc}{{\mathcal{O}\mathbb C}}

\newcommand{\genrep}{\mathrm{GenRep}}
\newcommand{\locsys}{\mathrm{LocSys}}

\newcommand{\ox}{{\mathcal{O}_X}}
\newcommand{\hatox}{{\hat{\mathcal O}_X}}
\newcommand{\oxhat}{{\hat{\mathcal O}_X}}

 \newcommand{\AcrisL}{\mathbf{A}_{\mathrm{cris}, L}}
\newcommand{\bdrplusL}{\mathbf{B}^+_{\mathrm{dR}, L}}

 \newcommand{\A}{ {\mathbf{A}}   }

\newcommand{\B}{  {\mathbf{B}}  }

\newcommand{\wtA}{   {\widetilde{{\mathbf{A}}}}  }
\newcommand{\wta}{   {\widetilde{{\mathbf{A}}}}  }
\newcommand{\wtadagger}{   {\widetilde{{\mathbf A}}^\dagger}  }
\newcommand{\wtbdagger}{   {\widetilde{{\mathbf A}}^\dagger}  }
\newcommand{\wtb}{   {\widetilde{{\mathbf{B}}}}  }
\newcommand{\wtbrig}{   {\widetilde{{\mathbf{B}}}_\rig^\dagger}  }
\newcommand{\wtbrigrn}{   {\widetilde{{\mathbf{B}}}_\rig^{\dagger, r_n}  }}
\newcommand{\wtbrigL}{   {\widetilde{{\mathbf{B}}}_{\rig, L}^\dagger}  }
\newcommand{\wtblog}{   {\widetilde{{\mathbf{B}}}_\log^\dagger}  }
\newcommand{\wtB}{   {\widetilde{{\mathbf{B}}}}  }
\newcommand{\wtE}{   {\widetilde{{\mathbf{E}}}}  }
\newcommand{\wte}{   {\widetilde{{\mathbf{E}}}}  }

 \def \E   {{\mathcal E}}
 \def \O {{\mathcal{O}}}

 \renewcommand{\OE}{{\mathcal{O}_{\mathcal E}}}
 \newcommand{\bigOE}{{\mathcal{O}_{\mathcal{E}}}}
\def \ku {k \llbracket u\rrbracket}
\def \uu {\llbracket u\rrbracket}

\def \OK {{\mathcal{O}_K}}
\def \ok {{\mathcal{O}_K}}
\def \ko {{K_{0}}}
\def \oc {{\mathcal{O}_C}}
\def \obbc {{\mathcal{O}_{\mathbb C}}}
\def \OKO {{\mathcal{O}_{K_0}}}
\def \oko {{\mathcal{O}_{K_0}}}
\def \obbko {{\mathcal{O}_{\mathbb{K}_0}}}
\def \OKbar {{\mathcal{O}_{\overline{K}}}}

\newcommand{\llb}{\llbracket}  
\newcommand{\rrb}{\rrbracket}
\newcommand*{\wt}[1]{\widetilde{#1}}
\newcommand*{\wh}[1]{\widehat{#1}}
\newcommand*{\mc}[1]{\mathcal{#1}}
\newcommand*{\mf}[1]{\mathfrak{#1}}
\newcommand*{\mbb}[1]{\mathbb{#1}}
\newcommand*{\mbf}[1]{\mathbf{#1}}
\newcommand{\dtilde}{\widetilde{D}}

\newcommand{\bba}{{\mathbb{A}}}
\newcommand{\bbb}{{\mathbb{B}}}
  \newcommand{\bbc}{{\mathbb{C}}}
    \newcommand{\bbC}{{\mathbb{C}}}
  \newcommand{\bbd}{{\mathbb{D}}}
   \newcommand{\bbD}{{\mathbb{D}}}
 \newcommand{\bbK}{{\mathbb{K}}}

 \newcommand{\Q}{{\mathbb{Q}}}
 \newcommand{\bbq}{{\mathbb{Q}}}
 \newcommand{\bbr}{{\mathbb{R}}}
 \newcommand{\Z}{{\mathbb{Z}}}
  \newcommand{\bbz}{{\mathbb{Z}}}
   \newcommand{\bbn}{\mathbb{N}}
\newcommand{\bbN}{{\mathbb{N}}}
\newcommand{\bbF}{{\mathbb{F}} }
\newcommand{\bbf}{{\mathbb{F}} }
 \newcommand{\bbk}{{\mathbb{K}}}

\newcommand{\zp}{{\mathbb{Z}_p}}
\newcommand{\qp}{{\mathbb{Q}_p}}
\newcommand{\fp}{{\mathbb{F}_p}}
\newcommand{\Qpbar}{{\overline{\mathbb{Q}}_p}}
\newcommand{\qpbar}{{\overline{\mathbb{Q}}_p}}
\newcommand{\barqp}{{\overline{\mathbb{Q}}_p}}

\newcommand{\Fpbar}{{\overline{\mathbb{F}}_p}}
\newcommand{\fpbar}{{\overline{\mathbb{F}}_p}}
\newcommand{\barfp}{{\overline{\mathbb{F}}_p}}

 \newcommand{\Fq}{{\mathbb{F}_q}}
\newcommand{\fq}{{\mathbb{F}_q}}

\newcommand{\repc}{\mathrm{Rep}_{\gk}(C)}
\newcommand{\repqp}{\mathrm{Rep}_{\gk}(\qp)}
\newcommand{\repbdrplus}{\mathrm{Rep}_{\gk}(\bdrplus)}
\newcommand{\repbdr}{\mathrm{Rep}_{\gk}(\bdr)}
\newcommand{\barK}{{\overline{K}}}

 \renewcommand*{\u}[1]{\underline{#1}}
\def\ueps{\underline \varepsilon}
\def\upi{\underline \pi}

\def \cp {{\mathcal P}}
\def \cw {{\mathcal W}}
\def \ct {{\mathcal T} }

\newcommand{\cala}{{\mathcal A}}
\newcommand{\cA}{{\mathcal A}}
\newcommand{\cB}{{\mathcal B}}
\newcommand{\cC}{{\mathcal C}}
\newcommand{\calc}{{\mathcal C}}
\newcommand{\cD}{{\mathcal D}}
  \newcommand{\cd}{{\mathcal{D}}}
    \newcommand{\cald}{{\mathcal{D}}}
\newcommand{\bigD}{{\mathcal{D}}}
\newcommand{\cE}{{\mathcal E}}
\newcommand{\ce}{{\mathcal E}}
\newcommand{\cale}{{\mathcal{E}}}
\newcommand{\cF}{{\mathcal F}}
\newcommand{\cG}{{\mathcal G}}
\newcommand{\cH}{{\mathcal H}}
\newcommand{\ch}{{\mathcal H}}
\newcommand{\cI}{{\mathcal I}}
\newcommand{\cJ}{{\mathcal J}}
\newcommand{\cK}{{\mathcal K}}
\newcommand{\cL}{{\mathcal L}}
\newcommand{\CL}{{\mathcal{L}}}
\def \cl {{\mathcal L} }
\newcommand{\cM}{{\mathcal M}}
\newcommand{\cm}{{\mathcal{M}} }
\newcommand{\calm}{\mathcal{M}}
\newcommand{\bigM}{\mathcal{M}}
\newcommand{\cN}{{\mathcal N}}
\newcommand{\cn}{{\mathcal N}}
\newcommand{\bigN}{\mathcal{N}}
\newcommand{\cO}{{\mathcal O}}
\newcommand{\co}{{\mathcal O}}
\newcommand{\bigO}{{\mathcal{O}}}
  \def \calo {{\mathcal{O}}}
   \def \calO {{\mathcal{O}}}
\newcommand{\cP}{{\mathcal P}}
\newcommand{\cQ}{{\mathcal Q}}
\newcommand{\cR}{{\mathcal R}}
 \def \calr {{\mathcal R}}
 \def \calR {{\mathcal R}}
\newcommand{\cS}{{\mathcal S}}
\newcommand{\cT}{{\mathcal T}}
\newcommand{\cU}{{\mathcal U}}
\newcommand{\cV}{{\mathcal V}}
\newcommand{\cW}{{\mathcal W}}
\newcommand{\cX}{{\mathcal X}}
\newcommand{\cx}{{\mathcal X}}
\newcommand{\calx}{{\mathcal X}}
\newcommand{\cY}{{\mathcal Y}}
\newcommand{\caly}{{\mathcal Y}}
\newcommand{\cZ}{{\mathcal Z}}
\newcommand{\calf}{\mathcal{F}}

\newcommand{\gA}{{\mathfrak{A}}}
\newcommand{\gB}{{\mathfrak{B}}}

\newcommand{\gc}{{\mathfrak{c}}}
\newcommand{\gs}{{\mathfrak{S}}}
\newcommand{\gS}{{\mathfrak{S}}}
\newcommand{\gt}{{\mathfrak{t}}}
\newcommand{\gu}{{\mathfrak{u}}}

\newcommand{\gL}{{\mathfrak{L}}}
\newcommand{\gm}{{\mathfrak{M}}}
\newcommand{\m}{{\mathfrak{M}}}
\newcommand{\M}{{\mathfrak{M}}}
\newcommand{\gM}{{\mathfrak{M}}}
\newcommand{\gn}{{\mathfrak{N}}}
\newcommand{\fkx}{{\mathfrak{X}}}
\newcommand{\fky}{{\mathfrak{Y}}}

\newcommand{\minf}{{\mathfrak{M}^{\mathrm{inf}}}}

\newcommand{\fkc}{{\mathfrak{c}}}
\newcommand{\fkb}{{\mathfrak{b}}}
\newcommand{\fke}{{\mathfrak{e}}}
\newcommand{\fkf}{{\mathfrak{f}}}
\newcommand{\fkg}{{\mathfrak{g}}}
\newcommand{\fkh}{{\mathfrak{h}}}
\newcommand{\fki}{{\mathfrak{i}}}
\newcommand{\fkj}{{\mathfrak{j}}}
\newcommand{\fkk}{{\mathfrak{k}}}
\newcommand{\fkm}{{\mathfrak{m}}}
\newcommand{\fkn}{{\mathfrak{m}}}
\newcommand{\fkp}{{\mathfrak{p}}}
\newcommand{\fkq}{{\mathfrak{q}}}
\newcommand{\fkt}{{\mathfrak{t}}}

\renewcommand{\oe}{{\mathcal{O}_{\mathcal{E}}}}
\newcommand{\oen}{{\mathcal{O}_{\mathcal{E}, n}}}

 \renewcommand{\wr}{{W(R)}}

\def \hR {{\widehat{\mathcal R}} }
\def \hM {{\widehat{\mathfrak{M}}} }
\def \hm {{\widehat{\mathfrak{M}}} }
\newcommand{\hatm}{\hat{\mathfrak{M}}}
\newcommand{\hatM}{\hat{\mathfrak{M}}}

\newcommand{\hatl}{\hat{\mathfrak{L}}}
\newcommand{\hatL}{\hat{\mathfrak{L}}}
\newcommand{\hatN}{\hat{\mathfrak{N}}}
\newcommand{\hatn}{\hat{\mathfrak{n}}}
\newcommand{\hattimes}{\widehat{\otimes}}
\newcommand{\barchi}{\overline{\chi}}
\newcommand{\barrho}{\overline{\rho}}
\newcommand{\chibar}{\overline{\chi}}
\newcommand{\rhobar}{\overline{\rho}}
\newcommand{\taubar}{\overline{\tau}}
\newcommand{\bolde}{\boldsymbol{e}}
\newcommand{\boldd}{\boldsymbol{d}}

\newcommand{\wtO}{   {\widetilde{{\mathcal{O}}}}  }

\newcommand{\bfD}{ {\mathbf{D}}}
\newcommand{\bftD}{ {\widetilde{\mathbf{D}}}}
\newcommand{\bftd}{ {\widetilde{\mathbf{D}}}}
\newcommand{\bfd}{ {\mathbf{D}}}
\newcommand{\bfm}{ {\mathbf{M}}}
\newcommand{\cbf}{\mathbf{c}}
\newcommand{\ibf}{\mathbf{i}}
\newcommand{\jbf}{\mathbf{j}}
\newcommand{\kbf}{\mathbf{k}}
\newcommand{\bfk}{\mathbf{k}}
\newcommand{\ybf}{\mathbf{y}}
\newcommand{\vece}{\vec{e}}
\newcommand{\bfa}{\mathbf{A}}
\newcommand{\bfb}{\mathbf{B}}
\newcommand{\bfB}{\mathbf{B}}
\newcommand{\wtbfd}{\wt{\bfd}}
\newcommand{\wtd}{\wt{\bfd}}

\def \ku {k \llbracket u\rrbracket}
\newcommand{\otimesvarphi}{\otimes_{\varphi}}

  \newcommand{\kpf}{{K^{\mathrm{pf}}}}
\newcommand{\pf}{{\mathrm{pf}}}
\newcommand{\bbko}{{\mathbb{K}_0}}
\newcommand{\bbkinfty}{{\bbk_\infty}}
\newcommand{\hatbbkinfty}{{\widehat{\bbk_\infty}}}
\newcommand{\bbkpinfty}{{\bbk_{p^\infty}}}
\newcommand{\gbbkinfty}{G_{\mathbb{K}_\infty}}
 \newcommand{\bbl}{{\mathbb{L}}}
  \newcommand{\bbL}{{\mathbb{L}}}

\newcommand{\ra}{\rightarrow}

\newcommand{\pibar}{\overline{\pi}}
\newcommand{\logpi}{\log[\overline{\pi}]}

\newcommand{\btst}[2]{\widetilde{\mathbf{B}}^{\dagger #1}_{\mathrm{log} #2}}
\newcommand{\btrig}[2]{\widetilde{\mathbf{B}}^{\dagger  #1}_{\mathrm{rig}  #2}}
\newcommand{\btlog}[2]{\widetilde{\mathbf{B}}^{\dagger #1}_{\mathrm{log} #2}}
\newcommand{\bbrig}[2]{ {\mathbb{B}}^{\dagger #1}_{\mathrm{rig} #2}}
\newcommand{\bfrig}[2]{ {\mathbf{B}}^{\dagger #1}_{\mathrm{rig} #2}}

\newcommand{\bnst}[2]{\mathbf{B}^{\dagger #1}_{\mathrm{log} #2}}
\newcommand{\bnrig}[2]{\mathbf{B}^{\dagger #1}_{\mathrm{rig} #2}}
\newcommand{\btstplus}[1]{\widetilde{\mathbf{B}}^{+}_{\mathrm{log} #1}}
\newcommand{\btrigplus}[1]{\widetilde{\mathbf{B}}^{+}_{\mathrm{rig} #1}}
\newcommand{\btlogplus}[1]{\widetilde{\mathbf{B}}^{+}_{\mathrm{log} #1}}
\newcommand{\blogplusnew}{ {\mathbf{B}}^+_{\mathrm{log}}}

\newcommand{\atst}[2]{\widetilde{\mathbf{A}}^{\dagger #1}_{\mathrm{log} #2}}
\newcommand{\atrig}[2]{\widetilde{\mathbf{A}}^{\dagger #1}_{\mathrm{rig} #2}}
\newcommand{\anst}[2]{\mathbf{A}^{\dagger #1}_{\mathrm{log} #2}}
\newcommand{\anrig}[2]{\mathbf{A}^{\dagger #1}_{\mathrm{rig} #2}}
\newcommand{\dnst}[1]{\mathbf{D}^{\dagger #1}_{\mathrm{log}}}
\newcommand{\dnrig}[1]{\mathbf{D}^{\dagger #1}_{\mathrm{rig}}}
\newcommand{\bmax}{\mathbf{B}_{\mathrm{max}}}
\newcommand{\bcont}{\mathbf{B}_{\mathrm{cont}}}

\newcommand{\ddR}{\mathbf{D}_{\mathrm{dR}}}
\newcommand{\dpst}{\mathbf{D}_{\mathrm{pst}}}
\newcommand{\ddif}{\mathbf{D}_{\mathrm{dif}}}
\newcommand{\dsen}{\mathbf{D}_{\mathrm{Sen}}}

\renewcommand{\ddag}[1]{\mathbf{D}^{\dagger #1}}

\newcommand{\vale}{v_\mathbf{E}}

\newcommand{\FF}{{\mathbb{F}}}
\newcommand{\NN}{{\mathbb{N}}}
\newcommand{\ZZ}{{\mathbb{Z}}}
\newcommand{\OO}{{\mathcal{O}}}
\renewcommand{\div}{{\mathrm{div}} }
\newcommand{\Nm}{{\mathrm{N}} }

\newcommand{\diam}{{\Diamond}}


\newcommand{\prism}{{\mathlarger{\mathbbl{\Delta}}}}

\newcommand{\xprism}{{X_\prism}}
\newcommand{\okprism}{{(\ok)_\prism}}
\newcommand{\okpris}{{(\mathcal{O}_K)_\prism}}

\newcommand{\okprislog}{{(\mathcal{O}_K)_{\prism, \mathrm{log}}}}
\newcommand{\okprisast}{{(\mathcal{O}_K)_{\prism, \ast}}}
\newcommand{\oklogpris}{{(\mathcal{O}_K)_{\prism, \mathrm{log}}}}
\newcommand{\okprisperf}{(\calO_K)^{\perf}_{\Prism}}

\newcommand{\okprislogperf}{(\calO_K)^{\perf}_{\Prism, \log}}
\newcommand{\oprism}{{\mathcal{O}_\prism}}
\newcommand{\opris}{{\mathcal{O}_\prism}}
\newcommand{\oprismplus}{\O_{\prism, \rig}^+}
\newcommand{\baroprism}{{\overline{\O}_\prism}}

\newcommand{\oprisdrplus}{ \calO_{\prism, \dR}^+  }

\newcommand{\oprisdr}{ \calO_{\prism, \dR}  }

 \newcommand{\strat}{\mathrm{Strat}}
\newcommand{\TorStrat}{\mathrm{TorStrat}}

 \newcommand{\DM}{\mathrm{DM}}
\newcommand{\DF}{\mathrm{DF}}

\newcommand{\pris}{{\mathrm{pris}}}
\newcommand{\qsp}{\mathrm{QRSPerfd}}
\newcommand{\qrsp}{{\mathrm{qrsp}}}

 \newcommand{\qs}{{\mathrm{QSyn}}}

 \newcommand{\gstwo}{{\mathfrak{S}^{(2)}}}


\newcommand{\gmast}{\gm^{\ast}}

 \newcommand{\oepi}{{\mathcal{O}_{\mathcal{E}, \vec{\pi}}}}
\newcommand{\gspi}{{\gs}_{\vec{\pi}}}
\newcommand{\wtl}{{\wt{\mathcal{L}}}}
\renewcommand{\wtA}{{\widetilde{A}}}

\newcommand{\cflat}{{C^\flat}}
\newcommand{\ocflat}{{\mathcal{O}_C^\flat}}
\newcommand{\obbcflat}{{\mathcal{O}_{\mathbb C}^\flat}}

\title[]{$p$-adic Simpson correspondence via prismatic crystals}

\author[]{Yu Min}
\address{Y. Min, Department of Mathematics, Imperial College London, London SW7 2RH, UK.}
\email{y.min@imperial.ac.uk}

\author[]{Yupeng Wang}
\address{Y. Wang, Beijing International Center of Mathematical Research, Peking University, YiHeYuan Road 55, Beijing, 100190, China.}
\email{2306393435@pku.edu.cn}



\begin{abstract}
Let $\frakX$ be a smooth $p$-adic formal scheme over $\calO_K$ with adic generic fiber $X$. We obtain a global equivalence between the category $\Vect((\frakX)_{\Prism},\overline\calO_{\Prism}[\frac{1}{p}])$ of rational Hodge--Tate crystals on the absolute prismatic site $(\frakX)_{\Prism}$ and the category $\HIG^{\nil}_*(X)$ of enhanced Higgs bundles on $X$. Along the way, we construct an inverse Simpson functor from $\HIG^{\nil}_*(X)$ to the category $\Vect(X_{\proet},\widehat\calO_X)$ of generalised representations on $X$, which turns out to be fully faithful.
\end{abstract}


\maketitle
\setcounter{tocdepth}{1}
\tableofcontents

\section{Introduction}
\subsection{Overview}

  In their groundbreaking work \cite{BS22}, Bhatt and Scholze introduced the prismatic cohomology,  which is ``motivic" in the sense that it can recover most existing $p$-adic cohomology theories (e.g. crystalline cohomology, de Rham cohomology and \'etale cohomology). Recently, the coefficient theory of the prismatic cohomology has caught a lot of attention. For example, one can recover \'etale $\Zp$-local systems on the generic fibers of bounded $p$-adic formal schemes $\frakX$ by studying the \emph{Laurent $F$-crystals} on the absolute prismatic site $(\frakX)_{\Prism}$ of $\frakX$ (cf. \cite{Wu21}, \cite{BS23}, \cite{MW21a}). When $\frakX$ is smooth over $\calO_K$, the ring of integers of a $p$-adic field $K$, there exists a nice equivalence between the category of crystalline $\Zp$-local systems on its generic fiber and that of \emph{(analytic) prismatic $F$-crystals} on $(\frakX)_{\Prism}$  (cf. \cite{ALB19}, \cite{BS23}, \cite{DL23}, \cite{DLMS24}, \cite{GR}). If  $\frakX$ is smooth over $\calO_{\widehat{\overline K}}$, then the theory of prismatic crystals provides a $q$-deformation of local $p$-adic Simpson correspondence (cf. \cite{GLSQ19}, \cite{MT}, \cite{GLSQ23}, etc.). 
  
  In this paper, we will focus on the theory of \emph{Hodge--Tate prismatic crystal}. Note that in the geometric setting, Hodge--Tate crystals on the relative prismatic site can be {\bf locally}  understood as certain Higgs bundles (cf. \cite{Tia23}). This phenomenon also appears in characteristic $p$ and one can understand Hodge--Tate crystals as nilpotent Higgs bundles for ``good'' smooth scheme over $\Fp$ (cf. \cite{Og22}). In the arithmetic setting, Hodge--Tate crystals on the absolute prismatic site of $\calO_K$ can be understood as certain semi-linear $\widehat{\overline{K}}$-representations of $\Gal(\overline{K}/K)$ and closely related with classical Sen theory (cf. \cite{GMW-HT}, \cite{BL22a}, \cite{GMW}, \cite{AHLB22}, etc.).

  Continuing the work \cite{GMW-HT} (joint with Hui Gao), we are going to combine the two aspects, i.e. the geometric side and the arithmetic side, of Hodge--Tate crystals mentioned above. More precisely, for a smooth $p$-adic formal scheme $\frakX$ over $\calO_K$ with generic fiber $X$, we will study the category $\Vect((\frakX)_{\Prism},\overline\calO_{\Prism}[\frac{1}{p}])$ of rational Hodge--Tate crystal on $(\frakX)_{\Prism}$ and investigate how the arithmetic part and the geometric part interact with each other. In fact, we will establish a \emph{global} equivalence between the category of rational Hodge--Tate crystals and the category of enhanced Higgs bundles, which can be understood as a hybrid of Higgs bundles and
  (arithmatic) Sen operators. Along the way, we will obtain a fully faithful functor from the category of enhanced Higgs bundles on $X$ to the category of generalised representations on $X_{\proet}$, which we call the \emph{inverse Simpson functor}. 

\subsection{Main result}
  We freely use notations introduced in \S\ref{Intro-Notations}. In particular, $K$ is a complete discretely valued field of mixed characteristic $(0,p)$ with ring of integers $\calO_K$ and perfect residue field $\kappa$ and $C = \widehat{\overline K}$ is a fixed $p$-complete algebraic closure of $K$. 
\subsubsection{Statement of main result}
 We first introduce some notions that will be considered in this paper. Let $\Vect((\frakX)_{\Prism},\overline \calO_{\Prism}[\frac{1}{p}])$ (resp. $\Vect((\frakX)^{\perf}_{\Prism},\overline \calO_{\Prism}[\frac{1}{p}])$) be the category of rational Hodge--Tate crystals on the {\bf absolute} prismatic site $(\frakX)_{\Prism}$ (resp, the absolute prismatic site of perfect prisms) of $\frakX$ (cf. Definition \ref{Dfn-HT Crystal}). Let $\Vect(X_{\proet},\OX)$ denote the category of generalised representations (i.e. locally finite free $\widehat \calO_X$-modules) on $X_{\proet}$ (cf. Example \ref{Exam-GRep}) and $\HIG^{\nil}_*(X)$ be the category of enhanced Higgs bundles on $X$; that is, Higgs bundles defined on $X$ with a Sen operator satisfying certain conditions (cf. Definition \ref{Dfn-EnhancedHiggsBundle}). Let $\HIG_{G_K}(X_C)$ denote the category of $G_K$-Higgs bundles on $X_C$; that is, Higgs bundles on $X_C$ together with compatible $G_K$-actions (cf. Definition \ref{Dfn-G-HiggsBundle}).
  
  Then our main theorem is as follows.
  \begin{thm}[Theorem \ref{Thm-HTasHiggs-Global}]\label{Intro-MainTheorem}
   Let $\frakX$ be a smooth $p$-adic formal scheme over $\calO_K$ with rigid analytic generic fiber $X$.
   Then there exists an equivalence of categories
   \[\rho:\Vect((\frakX)_{\Prism},\overline \calO_{\Prism}[\frac{1}{p}])\simeq \HIG^{\nil}_*(X),\]
   which preserves ranks, tensor products and dualities. Moreover, this equivalence fits into the following commutative diagram of fully faithful functors:
   \begin{equation}\label{Intro-Diag-CommutativeDiagram}
       \xymatrix@C=0.45cm{
         \Vect((\frakX)_{\Prism},\overline \calO_{\Prism}[\frac{1}{p}])\ar[rrrr]^{\rho}_{\simeq}\ar[d]^{\Res}&&&&\HIG^{\nil}_*(X)\ar[d]^{\rF}\ar[lld]_{\rF_{S}}\\
         \Vect((\frakX)_{\Prism}^{\perf},\overline \calO_{\Prism}[\frac{1}{p}])\ar[rr]^{\simeq}
         &&\Vect(X_{\proet},\OX)\ar[rr]^{\simeq}&&\HIG_{G_K}(X_C),
       }
   \end{equation}
   where $\Res$ is induced by inclusion $(\frakX)_{\Prism}^{\perf}\subset(\frakX)_{\Prism}$ of sites, $\rF$ will be defined in Construction \ref{Construction-F} and all arrows with ``$\simeq$'' are equivalences of categories. All functors above are natural in the sense that they do not depend on other choices and are functorial in $\frakX$. In particular, we obtain a fully faithful inverse Simpson functor 
   \[\rF_S:\HIG^{\nil}_*(X)\to \Vect(X_{\proet},\OX).\]

   As we will see, the equivalences appearing in the bottom line of Diagram (\ref{Intro-Diag-CommutativeDiagram}) preserve cohomologies (cf. Theorem \ref{Intro-EtaleRealisation} and Theorem \ref{Intro-Simpson}(3)). When $\frakX=\Spf(R^+)$ is small affine (see paragraphs above Convention \ref{Convention-Crystal}), the equivalence $\rho$ also preserves cohomologies even on the integral level (cf. Theorem \ref{Intro-HTasHiggs-Local}(1)).
 \end{thm}

 \begin{rmk} \label{rem: history}

     We comment on the history and independent works related with Theorem \ref{Intro-MainTheorem}. We refer the readers to the cited papers for more details and comparison of approaches.
     \begin{enumerate}
         \item When $\frakX=\Spf(\calO_K)$ (the point case), these results (including the log-prismatic case) are proved in our joint work with Gao \cite{GMW-HT}. Independently,   these results  are also proved by Ansch\"utz--Heuer--Le Bras \cite{AHLB22} using a stacky approach, building on works of Drinfeld \cite{Dri20} and Bhatt--Lurie \cite{BL22a, BL22b}. 
         \item For smooth $\frakX$, again Ansch\"utz--Heuer--Le Bras independently prove relevant results in \cite{AHLB-b,AHLB23c}, using the stacky approach.
We note that they can also describe the essential image of $\rF_S$ (cf. Remark \ref{rmk:essential image}). 

\item One can define a notion of $\BBdRp$- prismatic crystals, which are infinitesimal thickenings of Hodge--Tate crystals. In the point case, they are studied in \cite{GMW}, generalizing the work of \cite{GMW-HT}. Higher dimensional cases of these crystals will be discussed in the forthcoming joint work with Hui Gao \cite{GMWdRrel}.
     \end{enumerate}
 \end{rmk}
 
 \begin{rmk}\label{Intro-Rmk-Main-I}
     The relation between Hodge--Tate crystals and Higgs bundles is not new: indeed, if one fixes a transversal prism $(A,I)$ and considers a small (cf. \cite[Def. 8.5]{BMS18}) affine formal scheme $\frakX = \Spf(R^+)$ over $A/I$, then there exists an equivalence between the category of topologically nilpotent Higgs bundles over $\frakX$ and the category of Hodge--Tate crystals on the {\bf relative} prismatic site $(\frakX/(A,I))_{\Prism}$ of $\frakX$ (cf. \cite[Thm. 5.12]{Tia23}), which is non-canonical and actually depends on the choice of charts. Our improvement is that if we consider the {\bf absolute} prismatic site of $\frakX$ and rational Hodge--Tate crystals, then the Higgs bundles coming from rational Hodge--Tate crystals can be equipped with an additional  Sen operator so that we can establish a global theory. We emphasize that both inverting $p$ and considering the absolute prismatic site are essential in our construction. 
 \end{rmk}

 \begin{rmk}\label{Intro-Rmk-Main-II}
     In the classical theory of Simpson correspondence over the field of complex numbers $\bC$, a nilpotent Higgs bundle $(\calH,\theta_{\calH})$ over a smooth projective variety $X$ may not induce a representation of the fundamental group $\pi_1(X(\bC))$ or a $\bG_m$-action on $(\calH,\theta_{\calH})$. However, in the setting of Theorem \ref{Intro-MainTheorem}, if a nilpotent Higgs bundle $(\calH,\theta_{\calH})$ is enhanced, then it will induce a generalised representation and a $G_K$-action on its base-change to $C$, the field of $p$-adic complex numbers. We point out that the desired $G_K$ on the base-change of $(\calH,\Theta_{\calH})$ is {\bf not} the naive base-change ``$(\calH,\theta_{\calH})\otimes_{K}C$'' (in which $G_K$ only acts on $C$). See Construction \ref{Construction-F} for details.
 \end{rmk}

 \begin{rmk}\label{Intro-Rmk-Main-III}
    Theorem \ref{Intro-MainTheorem} sheds light on studying $p$-adic Simpson correspondence via prismatic crystals. Indeed, our approach is also compatible Faltings' $p$-adic Simpson correspondence (cf. \cite{Fal,AGT,Wan23}) when $\frakX$ is smooth over $\calO_C$ with a lifting to $A_{\inf}/\xi^2$.
    More precisely, in this case, one can firstly check that when $\frakX = \Spf(R^+)$ is small affine, after identifying $\Vect((\frakX/(\Ainf,\xi))_{\Prism},\overline \calO_{\Prism}[\frac{1}{p}])$ with the category of topologically nilpotent Higgs bundles on $\frakX$ (cf. \cite[Thm. 5.12]{Tia23}), the composed functor 
    \[\Vect((\frakX/(\Ainf,\xi))_{\Prism},\overline \calO_{\Prism}[\frac{1}{p}])\to\Vect((\frakX/(\Ainf,\xi))^{\perf}_{\Prism},\overline \calO_{\Prism}[\frac{1}{p}])\to\Vect(X,\OX)\]
    in Theorem \ref{Intro-MainTheorem} is compatible with Faltings' local Simpson functor (e.g. \cite[Thm. 4.3]{Wan23}). Then using Faltings' $p$-adic Simpson correspondence (e.g. \cite[Thm. 1.1]{Wan23}) instead of $\rF_S$ in (\ref{Intro-Diag-CommutativeDiagram}), our approach can show that the local construction in \cite{Tia23} glues if one restricts to the categories of small Higgs bundles and ``small'' Hodge--Tate crystals (which induces small generalised representations via the above composed functor). Eventually, one can establish the equivalences among the categories of small Higgs bundles, small Hodge--Tate crystals, and small generalised representations. We leave the details to the interested readers. 
 \end{rmk}
 
  Next, we are going to explain our strategy of proving Theorem \ref{Intro-MainTheorem}. We will first introduce the key tools used in the globalisation process and then talk about the local constructions and complete the strategy of the proof.
  
\subsubsection{$p$-adic Simpson correspondence and Hodge--Tate crystals on perfect site}

Similar to \cite{Tia23}, we can get a local correspondence between Hodge--Tate crystals and Higgs bundles in the absolute case.
The main difficulty of getting a global one then lies in comparing local constructions. Our strategy is to compare local data in a bigger category through a fully faithful functor. The key ingredient is the arithmetic $p$-adic Simpson correspondence (cf. Theorem \ref{Intro-Simpson}).
  
  Based on the previous work of Scholze \cite{Sch-IHES,Sch-Pi}, Liu--Zhu \cite{LZ} assigned to each \'etale $\Qp$-local system on a smooth rigid analytic space over $K$ a nilpotent Higgs field by using decompletion theory in \cite{KL16} and got the rigidity of de Rham local systems. Their method also works in the logarithmic case \cite{DLLZ}. We will work on their arenas and prove the following result, in which $X$ is not required to be a rigid generic fiber of a smooth $p$-adic formal scheme.

 \begin{thm}[Theorem \ref{Thm-Simpson}]\label{Intro-Simpson}
   Let $X$ be a smooth rigid space over $K$ and $\nu:X_{\proet}\to X_{C,\et}$ be the natural map of sites. Then there is a period sheaf $\OC$ together with a natural Higgs field $\Theta$ (cf. Equation (\ref{Equ-Theta})) on $X_{\proet}$, and for any generalised representation $\calL\in \Vect(X_{\proet},\OX)$, if we put $\Theta_{\calL}:=\id_{\calL}\otimes\Theta$ (cf. Equation (\ref{Equ-Theta_L})), then the rule
   \[\calL\mapsto (\calH(\calL),\theta_{\calH(\calL)}):= (\nu_*(\calL\otimes_{\OX}\OC),\nu_*(\Theta_{\calL}))\]
   induces a rank-preserving equivalence from the category $\Vect(X_{\proet},\OX)$ of generalised representations on $X_{\proet}$ to the category $\HIG_{G_K}(X_{C})$ of $G_K$-Higgs bundles on $X_{C,\et}$, which preserves tensor products and dualities. Moreover, the following assertions are true:
   
   \begin{enumerate}
       \item For any $i\geq 1$, the higher direct image $\rR^i\nu_*(\calL\otimes_{\OX}\OC) = 0$.
       
       \item Let $\Theta_{\calH(\calL)}:=\theta_{\calH(\calL)}\otimes\id_{\OC}+\id_{\calH(\calL)}\otimes\Theta$ (cf. Equation (\ref{Equ-Theta_H})). Then there exists a natural isomorphism
       \[(\calH\otimes_{\calO_{X_{C}}}\OC_{\mid X_{C}},\Theta_{\calH(\calL)}) \xrightarrow{\cong} (\calL\otimes_{\OX}\OC,\Theta_{\calL})_{\mid X_{C}}\]
       of Higgs fields.
       
       \item Let $(\HIG(\calH(\calL),\theta_{\calH(\calL)})$ denote the Higgs complex induced by $(\calH(\calL),\theta_{\calH(\calL)})$. Then there exists a natural quasi-isomorphism
       \[\rR\Gamma(X_{C,\proet},\calL)\cong \rR\Gamma(X_{C,\et},(\HIG(\calH(\calL),\theta_{\calH(\calL)}))\]
       which is compatible with $G_K$-actions. As a consequence, we get a quasi-isomorphism
       \[\rR\Gamma(X_{\proet},\calL)\cong \rR\Gamma(G_K,\rR\Gamma(X_{C,\et},(\HIG(\calH(\calL),\theta_{\calH(\calL)}))).\]
       
       \item Let $X'\to X$ be a smooth morphism of rigid spaces over $K$. Then the equivalence in (3) is compatible with pull-back along $f$. In other words, for any $\calL\in\Vect(X_{\proet},\OX)$ with corresponding $(\calH,\theta_{\calH})\in\HIG_{G_K}(X_C)$, we have 
       \[(\calH(f^*\calL),\theta_{\calH(f^*\calL)})\cong (f^*\calH,f^*\theta_{\calH}).\]
   \end{enumerate}
 \end{thm}

 An immediate corollary is the following.
 
 \begin{cor}[Corollary \ref{Cor-Simpson}]\label{Intro-Cor-Simpson}
    Let $d$ be the dimension of $X$ over $K$. 
    
    \begin{enumerate}
       \item Assume $X$ is quasi-compact. Then $\rR\Gamma(X_{C,\proet},\calL)$ is concentrated in degree $[0,2d]$.
       
       \item If moreover $X$ is proper, then $\rR\Gamma(X_{C,\proet},\calL)$ is a perfect complex of $C$-representations of $G_K$ and $\rR\Gamma(X_{\proet},\calL)$ is a perfect complex of $K$-vector spaces concentrated in degree $[0,2d+1]$. 
    \end{enumerate}
  \end{cor}

  \begin{rmk}\label{Intro-Rmk-Simpson-I}
     By almost \'etale descent, one can replace $C$ in Theorem \ref{Intro-Simpson} and Corollary \ref{Intro-Cor-Simpson} by its any perfectoid sub-field containing $\zeta_{p^\infty}$ (cf. Remark \ref{Rmk-Simpson}). By analytic--\'etale comparison, one can use the analytic site instead of the \'etale site above.
  \end{rmk}

  \begin{rmk}\label{Intro-Rmk-Simpson-II}
      (1) When $X$ is affine, Theorem \ref{Intro-Simpson} was achieved by Tsuji \cite[Thm. 15.2]{Tsu} by choosing a certain integral model (with log structure) of $X$, which is not necessary in our approach.

      (2) For generalised representations $\calL$ coming from $\Qp$-local systems $\bL$ (i.e. $\calL=\OX\otimes_{\Qp}\bL$), Theorem \ref{Intro-Simpson} was proved by Liu--Zhu as \cite[Thm. 2.1]{LZ} by using decompletion theory for certain overconvergent period rings ``$\widetilde{\mathbf B}^{\dagger}$''. Our proof was inspired by theirs and works for any generalised representations by using decompletion theory developed in \cite{DLLZ} (cf.\S \ref{Sec-DLLZ}).

      (3) Note that when $X$ admits a good reduction $\frakX$ over $\rW(\kappa)$, the base-change along $\rW(\kappa) \to \bf A_{\inf}$ induces a lifting $\widetilde \frakX$ of $\frakX_C:=\frakX\times_{\rW({\kappa})}\calO_C$. So for a small generalised representations $\calL$, the work of Abbes--Gros--Tsuji \cite{AGT} also produces a Higgs bundle on $X_C$ with $G_K$-action. Theorem \ref{Intro-Simpson} should be compatible with theirs. Note that when $X$ has semi-stable reduction $\frakX$ over $\rW(\kappa)$ merely, then the base-change of $\frakX$ to $\bf A_{\inf}$ is not a lifting of $\frakX_C$ as log schemes. Indeed, we do not know whether a lifting of $\frakX_C$ to $\bf A_{\inf}$ (as log scheme) exists except the curve case and the affine small case. So it seems not easy to apply results in \cite{AGT} directly in this case.

      (4) When $X$ is either abeloid or $X$ is curve and $\calL$ is a line bundle, Theorem \ref{Intro-Simpson} can be also deduced from \cite{Heu} and \cite{HMW} by noting that $\mathrm{HT}\log$ in loc.cit. is $G_K$-equivariant. We thank Ben Heuer for pointing out this to us. In general, if $X$ is furthermore proper, then Theorem \ref{Intro-Simpson} is a consequence of a very recent work \cite{Heu-Simpson} of Heuer as well.
  \end{rmk}

  Before explaining how to apply Theorem \ref{Intro-Simpson}, let us exhibit the relation between rational Hodge--Tate crystals and generalised representations. We again assume $X$ is the rigid generic fiber of a smooth $p$-adic formal scheme $\frakX$ over $\calO_K$. The key observation is that both $X_{\proet}$ (or $X_v$) and $(\frakX)_{\Prism}^{\perf}$ have a basis consisting of perfectoid algebras. Using this, we can prove the following \'etale comparison theorem:
  
   \begin{thm}[Theorem \ref{Thm-EtaleRealisation}]\label{Intro-EtaleRealisation}
     There exists a natural equivalence of categories 
     \[\rL:\Vect((\frakX)_{\Prism}^{\perf},\overline \calO_{\Prism}[\frac{1}{p}])\to \Vect(X_{\proet},\OX),\]
     which preserves ranks, tensor products and dualities, such that for any rational Hodge--Tate crystal $\bM$ on $(\frakX)_{\Prism}^{\perf}$, we have a quasi-isomorphism
     \[\rR\Gamma((\frakX)_{\Prism}^{\perf},\bM)\cong \rR\Gamma(X_{\proet},\rL(\bM)),\]
   which is functorial in $\bM$.
  \end{thm}

  Due to Theorem \ref{Intro-EtaleRealisation}, the inclusion of sites $(\frakX)^{\perf}_{\Prism}\subset(\frakX)_{\Prism}$ induces a natural functor 
  \[\Res:\Vect((\frakX)_{\Prism},\overline \calO_{\Prism}[\frac{1}{p}])\to \Vect((\frakX)_{\Prism}^{\perf},\overline \calO_{\Prism}[\frac{1}{p}])\simeq \Vect(X_{\proet},\OX).\]
  We will see that this functor is fully faithful (cf. Proposition \ref{Prop-F-FullyFaithful}) and then use it to compare local constructions that we are going to explain right now.
 
\subsubsection{Local constructions}
   Let $\frakX = \Spf(R^+)$ be small affine. By ($p$-completely) faithfully flat descent and \cite[Prop. 2.7]{BS23}, one can study (rational) Hodge--Tate crystals in terms of stratifications by choosing a certain cover $(\frakS(R^+),(E))$ (depending on the choice of charts on $R^+$) of the final object of $\Sh((R^+)_{\Prism})$. In this case, parts of Theorem \ref{Intro-Simpson} upgrade to the integral level and one can give an explicit description of the functor $\Res$. Also, $\Vect(X_{\proet},\OX)$ is equivalent to the category $\Rep_{\Gamma(\overline K/K)}(\widehat R_{C,\infty})$ of semi-linear $\widehat R_{C,\infty}$-representations of $\Gamma(\overline K/K)\cong(\oplus_{i=1}^d\Zp\gamma_i)\rtimes G_K$ (cf. Notation \ref{Notation-LocalChart-II} \& Lemma \ref{Lem-EvaluateGRep}). Let $c:G_K\to\Zp$ be the function determined by $g(\pi^{\flat}) = \epsilon^{c(g)}\pi^{\flat}$ for any $g\in G_K$ and $\lambda$ be the reduction of $\frac{\xi}{E([\pi^{\flat}])}\in \Ainf$ modulo $\xi$ (See the paragraphs above Convention \ref{Convention-LinearIndependence}).
   \begin{thm}\label{Intro-HTasHiggs-Local}
     Keep notations as above and fix a chart $\Box:\frakX\to\Spf(\calO_K\za T_1^{\pm 1},\dots,T_d^{\pm 1}\ya)$.
     \begin{enumerate}
         \item \emph{[Theorem \ref{Thm-HTasHiggs-Local}]}
         The evaluation at $(\frakS(R^+),(E))$ induces an equivalence of categories 
         \[\rho:\Vect((R^+)_{\Prism},\overline \calO_{\Prism})\to\HIG_*^{\nil}(R^+)~(\text{resp.}~\rho:\Vect((R^+)_{\Prism},\overline \calO_{\Prism}[\frac{1}{p}])\to\HIG_*^{\nil}(R)),\]
        which preserves ranks, tensor products and dualities, such that for any Hodge--Tate crystal (resp. rational Hodge--Tate crystal) $\bM$ with associated enhanced Higgs module (cf. Definition \ref{Dfn-EnhancedHiggsMod-Local}) $(H,\theta_H,\phi_H)$ over $R^+$ (resp. $R$), there exists a quasi-isomorphism
        \[\rR\Gamma((R^+)_{\Prism},\bM)\cong \HIG(H,\theta_H,\phi_H),\]
        which is functorial in $\bM$.

         \item \emph{[Theorem \ref{Lem-Restriction}]} For any $(H,\theta_H,\phi_H)\in\HIG^{\nil}_*(R)$ corresponding to the rational Hodge--Tate crystal $\bM$ with induced $\widehat R_{C,\infty}$-representation $\Res(\bM)$ of $\Gamma(\overline K/K)$, we have $\Res(\bM) = H\otimes_R\widehat R_{C,\infty}$ such that for any $1\leq i\leq d$, any $g\in G_K$ and any $x\in H$, 
         \[\gamma_i(x) = \exp(-(\zeta_p-1)\lambda\theta_i)(x)~\text{and}~g(x) = (1+\pi E'(\pi)(\zeta_p-1)\lambda c(g))^{-\frac{\phi_H}{E'(\pi)}}(x).\]
         Here, $\theta_i\in \End_{R}(H)$ are the components of $\theta_H$ with respect to the fixed chart $\Box$; that is, $\theta_H = \sum_{i=1}^d\theta_i\otimes\dlog T_i\{-1\}$.
     \end{enumerate}
 \end{thm}
 \begin{rmk} \label{rem: BL work}
     Note that one can naturally assign to each Hodge--Tate crystal $\bM\in \Vect((R^+)_{\Prism},\overline \calO_{\Prism})$  a Hodge--Tate crystal, still denoted by $\bM$ for short, on the relative prismatic site $(R^+/(\frakS,(E)))_{\Prism}$ by considering the natural functor $(R^+/(\frakS,(E)))_{\Prism}\to(R^+)_{\Prism}$. One can apply \cite[Thm. 5.12]{Tia23} to the resulting $\bM$ and obtain a topologically nilpotent Higgs field over $R^+$, which turns out to be the underlying Higgs field $(H,\theta_H)$ of $\rho(\bM)$.
 \end{rmk}
 \begin{rmk}
     When $\calO_K = \rW(\kappa)$, the first part of Theorem \ref{Intro-HTasHiggs-Local} was also obtained by Bhatt--Lurie \cite{BL22b} by using the local splitting of the Hodge--Tate structure map ${\rm WCart}_{\frakX}^{\rm HT}\to \frakX$. Up to now, it is still a problem to achieve a global theory as their method needs a global “Frobenius endomorphism” (cf. \cite[\S 9]{BL22b}), which is a very restrictive condition. However, our Theorem \ref{Intro-MainTheorem} suggests it should be reasonable to ask whether the ``generic fiber" of the Hodge--Tate structure map ${\rm WCart}_{\frakX}^{\rm HT}[\frac{1}{p}]\to \frakX[\frac{1}{p}]$ admits a global splitting.  
 \end{rmk}

   Inspired by Theorem \ref{Intro-HTasHiggs-Local} (2), one can assign to each enhanced Higgs bundle $(\calH,\theta_{\calH},\phi_{\calH})$ over $X$ (which is necessary to be affine) a $G_K$-Higgs bundle over $X_C$ by letting $g\in G_K$ act via the cocycle $(1+\pi E'(\pi)(\zeta_p-1)\lambda c(g))^{-\frac{\phi_H}{E'(\pi)}}$ and hence get the global functor $\rF$ in Theorem \ref{Intro-MainTheorem}. To prove $\rF$ is fully faithful (cf. Proposition \ref{Prop-F-FullyFaithful}), we need to relate prismatic theory to classical Sen theory, which will be discussed in \S \ref{Sen-Intro}.
   
   Now since everything can be described explicitly, one can check directly that Theorem \ref{Intro-MainTheorem} holds true in the affine small case. In particular, Diagram (\ref{Intro-Diag-CommutativeDiagram}) is commutative in this case. To complete the proof of Theorem \ref{Intro-MainTheorem}, we need to check that for rational Hodge--Tate crystals, the local equivalence $\rho$ in Theorem \ref{Intro-HTasHiggs-Local} (1) is independent of the choice of charts on $R^+$. Since $\rF$ is fully faithful, it is enough to show $\rF\circ\rho$ is independent of the choice of charts. But this is obvious now as all functors except $\rho$ in Diagram (\ref{Intro-Diag-CommutativeDiagram}) are globally defined and Diagram (\ref{Intro-Diag-CommutativeDiagram}) is commutative in affine case.

\subsubsection{Sen operators}\label{Sen-Intro}

Finally, we explain how our work is related with Sen theory. Recall that for any semi-linear $C$-representation $V$ of $G_K$, by considering the cyclotomic extension, Sen constructed an endomorphism $\Theta_V$, which is called \emph{Sen operator}, defined over $K$ satisfying certain properties. His result was then generalised in several different settings, and we would like to list some relevant results in the following:
\begin{itemize}
 \item In \cite{BC16}, Berger and Colmez dealt with the Galois extensions of $K$ whose Galois group $G$ is a $p$-adic Lie group of arbitrary dimension. Using their method, for a modular curve $X$, Pan defined a geometric Sen operator on the ``local analytic'' part of the pro-\'etale structure sheaf $\OX$, and applied it to the Fontaine--Mazur conjecture \cite{Pan}.

 \item Still considering the cyclotomic extension, it was proved by Shimizu \cite{Shi} and Petrov \cite{Pet} that the Sen theory also holds true for the arithmetic families. More precisely, for any quasi-compact smooth rigid space $X$ over $K$ and any $G_K$-equivariant vector bundle $\calE$ over $X_{C,\et}$, they constructed an $\calO_{\calX}$-vector bundle over the ringed space $\calX = (X,\calO_{\calX}=\calO_X\otimes_KK_{\cyc})$ together with a Sen operator satisfying certain properties.
\end{itemize}

The constructions above are both compatible with that of Sen \cite{Sen}. The readers are referred to \S\ref{Sect-Sen operator} for further discussion.

  Now by Theorem \ref{Intro-Simpson} together with the result of Petrov (cf. Proposition \ref{Prop-Petrov}), we can assign to each  generalised representation $\calL$ over $X_{\proet}$ a Higgs bundle with a Sen operator $(\calH(\calL),\theta_{\calH(\calL)},\phi_{\calH(\calL)})$ over $\calX$. We also call $\phi_{\calH(\calL)}$ the \emph{Sen operator of $\calL$} and denote it by $\phi_{\calL}$. After base changing to $X_C$, this will give rise to an \emph{arithmetic Higgs bundle} over $X_C$ (cf. Definitions \ref{Dfn-ArithmeticHiggsBundle} and \ref{Dfn-ArithmeticHiggsBundle-II}). In other words, there is a functor
  \[
\rD:\Vect(X_{\proet},\OX)\to\HIG^{\rm arith}(X_C)
  \]
  which is closely related with the cyclotomic extension (cf. Corollary \ref{Cor-GReptoArith}). Before moving on, we raise the following question.

  \begin{question}\label{Intro-Question}
         Which arithmetic Higgs bundles come from  generalised
         representations under the functor $\rD$?
     \end{question}
In fact, it is not clear to us how to describe the essential image of the functor $\rD$ in general. When $\frakX = \Spf(\calO_K)$, Fontaine gave a complete answer to Question \ref{Intro-Question} after classifying all semi-linear $C$-representations of $G_K$ (cf. \cite[Thm. 2.14]{Fon}) by working with a certain algebraic group. But his method looks very difficult to be generalised to the higher dimensional case. We will give a partial answer to this question provided by the theory of Hodge--Tate crystals at the end of this subsection.

  Now we move back to investigate the relation between prismatic theory and Sen theory. Note that we have a local equivalence between rational Hodge--Tate crystals and enhanced Higgs bundles at the moment. The latter can give rise to arithmetic Higgs bundles as follows.
  
  \begin{exam}[Construction \ref{Construction-TwoFunctor} (1)]\label{Intro-ExamArith}
      Assume $(\calH,\theta_{\calH},\phi_{\calH})$ is an enhanced Higgs bundle over $X$. Then $(\calH\otimes_{\calO_{X}}\calO_{X_C},-(\zeta_p-1)\lambda\theta_{\calH},-\frac{\phi_{\calH}}{E'(\pi)})$ is an arithmetic Higgs bundle over $X_C$, where $\lambda$ is a certain unit in $\calO_C$. We denote the induced functor by $\rF_{\infty}:\HIG^{\nil}_*(X)\to \HIG^{\rm arith}(X_C)$.
  \end{exam}

  Given a rational Hodge--Tate crystal, we now have two ways to construct arithmetic Higgs bundles: The one is induced by the composites of the following functors
  \[\Vect((\frakX)_{\Prism},\overline \calO_{\Prism}[\frac{1}{p}])\xrightarrow{\Res}\Vect((\frakX)_{\Prism}^{\perf},\overline \calO_{\Prism}[\frac{1}{p}])\simeq\Vect(X_{\proet},\OX)\xrightarrow{\rD}\HIG^{\rm arith}(X_C).\]
  We call this the \emph{cyclotomic way} as the functor $\rD$ is defined via the decompletion theory corresponding to the cyclotomic extenion of $K$ (See \S\ref{SSec-cyclotomic Sen}). The other is induced by the composite of the following functors
  \[\Vect((\frakX)_{\Prism},\overline \calO_{\Prism}[\frac{1}{p}])\xrightarrow{\rho}\HIG_*^{\nil}(X)\xrightarrow{\rF_{\infty}}\HIG^{\rm arith}(X_C).\]
  We call this the \emph{Kummer way} as it indeed corresponds to the Kummer extension of $K$ (See \S\ref{SSec-Kummer Sen}). It turns out that the two constructions coincide with each other. Namely, we have the following result:

   \begin{thm}[Theorem \ref{Kummer Sen} and \ref{Prismatic-Sen}]\label{Intro-Sen}
       Let $\bM$ be a rational Hodge--Tate crystal on $(\frakX)_{\Prism}$, where $\frakX=\Spf(R^+)$ is small affine, with induced enhanced Higgs bundle $(\calH,\theta_{\calH},\phi_{\calH})$ and generalised representation $\calL$. Then the Sen operator $\phi_{\calL}$ of $\calL$ is exactly $-\frac{\phi_{\calH}}{E'(\pi)}$.
   \end{thm}
   \begin{rmk}
       The above theorem is due to Hui Gao. When $R^+ = \calO_K$, the result was already proved in \cite{GMW-HT}. When $R^+ = \rW(\kappa)$ is absolutely unramified extension of $\Zp$, the result was also obtained in \cite[\S 3.8]{BL22a}.
   \end{rmk}

   We will prove Theorem \ref{Intro-Sen} by following the strategy in \cite{GMW-HT}; that is, we will apply the theory of locally analytic vectors of Berger--Colmez \cite{BC16} to compare Sen operators obtained from cyclotomic extension and Kummer extension. As a consequence, we also conclude the full faithfulness of $\rF:\HIG^{\nil}_*(X)\to \HIG_{G_K}(X_C) $ which will complete the proof of Theorem \ref{Intro-MainTheorem}.
   
   Let $\rF_{\cyc}$ denote the composition of functors $\HIG^{\nil}_*(X)\xrightarrow{\rF_S} \Vect(X_{\proet},\OX)\xrightarrow{\rD}\HIG^{\rm arith}(X_C)$. Then Theorem \ref{Intro-Sen} implies that there is an equivalence of functors $\rF_{\cyc}\simeq \rF_{\infty}:\HIG^{\nil}_*(X)\to \HIG^{\rm arith}(X_C)$. As an application, we get a partial answer to Question \ref{Intro-Question}.

 \begin{thm}[Corollary \ref{Cor-AnwserQuestion}]\label{Intro-Answer}

      If an arithmetic Higgs bundle is induced by an enhanced Higgs bundle under the functor $\rF_{\infty}$, then it comes from a generalised representation over $X_{\proet}$.
   \end{thm}
   
  In summary, we have the following commutative diagram. (Recall $\rF_{\cyc}= \rD\circ\rF_S$.) 
   \begin{equation}\label{Diag-WholeDiagram}
    \xymatrix@C=0.45cm{
    \Vect((\frakX)_{\Prism},\overline \calO_{\Prism}[\frac{1}{p}])\ar[rrrr]^{\rho}_{\rm{Th.}\ref{Thm-HTasHiggs-Global}}\ar[d]^{\Res}&&&&\HIG^{\nil}_*(X)\ar[d]^{{\rm Const.}\ref{Construction-F}}_{\rF}\ar[lld]_{\rF_{S}}\ar[rrdd]_{\rF_{\infty}}^{\rm{Const.}\ref{Construction-TwoFunctor}}\\
         \Vect((\frakX)_{\Prism}^{\perf},\overline \calO_{\Prism}[\frac{1}{p}])\ar[rr]^{\simeq}_{\rm{Th.}\ref{Thm-EtaleRealisation}}
         &&\Vect(X_{\proet},\OX)\ar[rr]^{\simeq}_{\rm{Th.}\ref{Thm-Simpson}}\ar[rrrrd]_{\rD}^{{\rm Cor.}\ref{Cor-GReptoArith}}&&\HIG_{G_K}(X_C)\\
    &&&&&&\HIG^{\rm arith}(X_C).
    }
   \end{equation}
   
\subsection{Organization}
  We review the decompletion theory of \cite{DLLZ} in \S \ref{Sec-DLLZ} and use it to establish the $p$-adic Simpson correspondence (i.e. Theorem \ref{Intro-Simpson}) in \S \ref{Sec-Simpson}. In \S \ref{Sec-LocalConstruction}, we construct the equivalence $\rho$ in Theorem \ref{Intro-MainTheorem} in the small affine case. In \S \ref{Sec-EtaleComparison}, we prove Theorem \ref{Intro-EtaleRealisation}. The \S \ref{Sec-Global} is devoted to proving Theorem \ref{Intro-MainTheorem} by leaving the full faithfulness of $\rF$ to \S \ref{Sect-Sen operator}. Finally, we show that $\rF$ is fully faithful by using Sen theory in \S \ref{Sect-Sen operator} and prove Theorem \ref{Intro-Answer} as an application.

\subsection{Acknowledgement}
   The authors would like to thank Bhargav Bhatt, Hui Gao, Ben Heuer, Ruochuan Liu and Takeshi Tsuji for their interest and valuable comments on the early draft of this paper. The authors also want to thank Hui Gao for sharing his result (i.e. Theorem \ref{Intro-Sen}) in this paper. The work was finished when the authors stayed in Morningside Center of Mathematics as postdocs, and they thank the institute for providing the opportunity of cooperation. The authors would like to thank the anonymous referee for the careful reading and useful feedbacks.
   Yu Min was supported by China Postdoctoral Science Foundation E1900503. 
   Yupeng Wang is partially supported by CAS Project for Young Scientists in Basic Research, Grant No. YSBR-032.

\subsection{Conventions and Notations}\label{Intro-Notations}
  Let $K$ be a complete discretely valued field of mixed characteristic $(0,p)$ with ring of integers $\calO_K$ and perfect residue field $\kappa$. We fix an algebraic closure $\overline K$ of $K$ and let $C$ denote the $p$-adic completion of $\overline K$. For any Galois extension $L/K$ in $\overline K$, we denote by $\Gal(L/K)$ the corresponding Galois group. We fix a uniformizer $\pi\in \calO_K$ and denote its minimal polynoimal over $\rW(\kappa)$ by $E(u)$. Then $E(u)\in\frakS:=\rW(\kappa)[[u]]$.
  
  Let $\{\zeta_{p^n}\}_{n\geq 0}$ (resp. $\{\pi^{\frac{1}{p^n}}\}_{n\geq 0}$) be a compatible sequence $\{\zeta_{p^n}\}_{n\geq 0}$ of primitive $p^n$-th units (resp. of $p^n$-th roots of $\pi$) in $\overline K$ and let $\epsilon = (1,\zeta_p,\dots)$ (resp. $\pi^{\flat} = (\pi,\pi^{\frac{1}{p}},\dots)$) be the induced element in $C^{\flat}$, the tilting of $C$. Let $K_{\cyc}:=\cup_{n\geq 0}K(\zeta_{p^n})$ (resp. $K_{\infty} = \cup_{n\geq 0}K(\pi^{\frac{1}{p^n}})$, resp. $K_{\cyc,\infty} = K_{\cyc}K_{\infty}$) with the $p$-adic completion $\widehat K_{\cyc}$ (resp. $\widehat K_{\infty}$, resp. $\widehat K_{\cyc,\infty}$) in $C$. Then $\widehat K_{\cyc}$, $\widehat K_{\infty}$ and $\widehat K_{\cyc,\infty}$ are all perfectoid fields. Let $\chi:G_K\to \Gal(K_{\cyc}/K)\to \bZ_p^{\times}$ be the $p$-adic cyclotomic character and let $c:G_K\to \Gal(K_{\cyc,\infty}/K)\to \Zp$ be the function such that for any $g\in G_K$, $g(\pi^{\flat}) = \epsilon^{c(g)}\pi^{\flat}$. Let $t = \log [\epsilon]$ be the $p$-adic analogue of ``$2\pi i$''.

  Put $\Ainf = \rW(\calO_C^{\flat})$ and $\xi = \frac{[\epsilon]-1}{[\epsilon]^{\frac{1}{p}}-1}$. Then $\xi$ is a generator of Fontaine's map $\theta:\Ainf\to \calO_C$ and $(\Ainf,(\xi))$ is a prism. We equip $\frakS$ with a $\delta$-structure such that the induced Frobenius on $\frakS$ carries $u$ to $u^p$. Then $(\frakS,(E))$ is also a prism and the morphism $\frakS\to\Ainf$ sending $u$ to $[\pi^{\flat}]$ defines a morphism $(\frakS,(E))\to(\Ainf,(\xi))$ of prisms in $(\calO_K)_{\Prism}$.
  
  \begin{convention}\label{Convention-LinearIndependence}
     In this paper, we always assume $K_{\cyc}\cap K_{\infty} = K$. This is always the case for $p\geq 3$ and for $p = 2$ when $K = K(\zeta_4)$ (cf. \cite[Lem. 5.1.2]{Liu08}). In general, for $p=2$, the assumption is satisfied for a suitable choice of $\pi$ (cf. \cite[Lem. 2.1]{WangX}). Under this assumption, we have 
     \[\Gal(K_{\cyc,\infty}/K) \cong \Zp\tau\rtimes \Gal(K_{\cyc}/K),\]
     where $\tau\in \Gal(K_{\cyc,\infty}/K_{\cyc})$ such that $\tau(\pi^{\flat}) = \epsilon\pi^{\flat}$ (i.e. $c(\tau) = 1$). Then for any 
     \[\gamma\in \Gal(K_{\cyc,\infty}/K_{\infty})\cong \Gal(K_{\cyc}/K),\]
     we have $\gamma\tau\gamma^{-1} = \tau^{\chi(\gamma)}$. Moreover, one can identify $\Gal(K_{\cyc,\infty}/K)$ with an open subgroup of 
     $\{\smat{a & b \\ c & d}\in \GL_2(\Zp)\mid c= 0, d=1\}$ 
     via the homomorphism 
     \[X:\Gal(K_{\cyc,\infty}/K)\to \GL_2(\Zp)\]
     sending each $ g\in \Gal(K_{\cyc,\infty}/K)$ to $X(g) = \smat{\chi(g) & c(g) \\ 0 & 1}$.
  \end{convention}

  Let $Y,Y_1,\dots,Y_d$ be free variables. For any $n\geq 0$, we define
  \[\binom{Y}{n}:=\frac{Y(Y-1)\cdots (Y-n+1)}{n!}\]
  and
  \[Y^{[n]} := \frac{Y^{n}}{n!}.\]
  For any $\underline n = (n_1,\dots,n_d)\in\bN^d$, we then define 
  \[\binom{\underline Y}{\underline n} = \prod_{i=1}^d\binom{Y_i}{n_i}\]
  and
  \[\underline Y^{\underline n} = \prod_{i=1}^dY_i^{[n_i]}.\]
  Put $\underline 0 = (0,0,\dots,0)\in\bZ^d$ and for any $1\leq i\leq d$, put $\underline 1_i= (0,\dots,1,\dots,0)$ with $1$ appearing exactly in the $i$-th component. For any $\underline m = (m_1,\dots,m_d), \underline n = (n_1,\dots,n_d)\in\bN^d$, define
  \[\binom{\underline m+\underline n}{\underline n} = \prod_{i=1}^d\binom{m_i+n_i}{n_i}\]
  and 
  \[|\underline n| = n_1+n_2+\cdots+n_d.\]

  We will also use the language of stratifications to study certain vector bundles on sites.
  \begin{convention}\label{Convention-Strat}
     Let $S^{\bullet}$ be a cosimplicial ring with face maps $p_i^n$'s and degeneracy maps $\sigma_i^n$'s; that is, $p_i^n:S^n\to S^{n+1}$ is induced by the injection $[n]\to [n+1]\setminus\{i\}$ and that $\sigma_i^n:S^{n+1}\to S^n$ is induced by the surjection $[n+1]\to[n]$ satisfying $(\sigma^{n}_i)^{-1}(i) = \{i,i+1\}$. Let $q^n_i:S^0\to S^n$ be the morphism induced by the inclusion $[0]\xrightarrow{0\mapsto i}[n]$ of simplices.
     We often write $p_i$ (resp. $\sigma_i$, resp. $q_i$) instead of $p_i^n$ (resp. $\sigma_i^n$, resp. $q_i^n$) for short when its source and target are clear. 
     
     By a \emph{stratification} with respect to $S^{\bullet}$, we mean a finite projective $S^0$-module together with an $S^1$-linear isomorphism \[\varepsilon:M\otimes_{S^0,p_0}S^1\xrightarrow{\cong}M\otimes_{S^0,p_1}S^1.\]
     Sometimes, we also say $\varepsilon$ is a stratification on $M$.
     We say a stratification $(M,\varepsilon)$ satisfies the \emph{cocycle condition}, if the following assertions are true:
     \begin{enumerate}
         \item $p_2^*(\varepsilon)\circ p_0^*(\varepsilon) = p_1^*(\varepsilon):M\otimes_{S^0,q_2}S^2\to M\otimes_{S^0,q_0}S^2$.
         
         \item $\sigma_0^*(\varepsilon) = \id_M:M\to M$.
     \end{enumerate}
     We denote by $\Strat(S^{\bullet})$ the category of stratifications with respect to $S^{\bullet}$ satisfying the cocycle condition.
  \end{convention}
  
  We also consider representations of several topological groups over certain topological rings.
  
  \begin{convention}\label{Convention-TopoRep}
    Let $R$ be a topological ring with a continuous action of a topological group $G$. By a \emph{representation of $G$ over $R$ of rank $l$}, we mean a finite projective $R$-module $M$ of rank $l$ together with a semi-linear continuous action of $G$; that is, for any $r\in R$, $m\in M$ and $g\in G$,
    \[g(rm) = g(r)g(m).\]
    Sometimes, we also call $M$ an \emph{$R$-representation of $G$ of rank $l$}.
    We say an $R$-representation $M$ of $G$ is \emph{free}, if $M$ is finite free over $R$.
    We denote by $\Rep_G(R)$ the category of representations of $G$ over $R$ and denote by $\Rep_G^{\rm free}(R)$ the full subcategory of free representations.
  \end{convention}

  Assume $G$ acts on $R$ trivially and that $M$ is an topological $R$ module equipped with a continuous action of $G$. We denote by $\rR\Gamma(G,M)$ the continuous group cohomology of $M$, which is computed by the complex $C(G^{\bullet},M)$. 
  Here and from now on, we always denote by $C(G,M)$ the group of continuous functions from $G$ to $M$.
  For any closed normal subgroup $H<G$ with quotient group $G/H$, in general, there is no Hochschild spectral sequence for $\rR\Gamma(G,M)$. However, 
  if $G\to G/H$ admits a continuous cross-section (which is not necessary an homomorphism), then Hochschild--Serre spectral sequence holds true (See the proof of \cite[Lem 3.3]{Ked16}, \cite[Chap. V (3.2), pp. 193]{Laz} for more details). In this paper, we always deal with the case for $H$ open or $G\cong H\rtimes G/H$ and hence can use Hochschild--Serre spectral sequence freely.
  
  Let $X = \Spa(R,R^+)$ be a smooth affinoid space of dimension $d$ over $K$. By a \emph{toric chart} on $X$, we mean an \'etale morphism 
  \[\Box: X\to\bG_m^d:=\Spa(K\za T_1^{\pm 1},\dots,T_d^{\pm 1}\ya,\calO_K\za T_1^{\pm 1},\dots,T_d^{\pm 1}\ya).\]
  Note that toric charts always exists \'etale locally on $X$.
  Let $\frakX = \Spf(R)$ be a smooth\footnote{In this paper, smooth $p$-adic formal schemes are always assumed to be separated and of topologically finite type.} $p$-adic formal scheme of dimension $d$ over $\calO_K$. We say it is \emph{small}, if there exists a \emph{framing} on $\frakX$; that is, an \'etale morphism
  \[\Box: \frakX\to\Spf(\calO_K\za T_1^{\pm 1},\dots,T_d^{\pm 1}\ya).\]
  Note that in this case, $\Box$ induces a toric chart on the rigid generic fibre of $\frakX$.
  Clearly, framings always exist \'etale locally on $\frakX$. However, it is worth pointing out that framings even exist Zariski locally on $\frakX$ by \cite[Lem. 4.9]{Bha16}.
  By abuse of notations, we still denote by $\Box$ the induced morphism on the rings of coordinates.
  
  We call an \'etale morphism of rigid spaces \emph{standard \'etale}, if it is a composition of rational localisations and finite \'etale morphisms.
  
  \begin{convention}\label{Convention-Crystal}
      Let $\calC$ be a site and $\bA$ be a sheaf of rings on $\calC$. By an \emph{$\bA$-crystal} on $\calC$, we mean a sheaf $\bM$ of $\bA$-modules such that 
      \begin{enumerate}
          \item For any object $C\in\calC$, the $\bM(C)$ is a a finite projective $\bA(C)$-module of rank $r$.

          \item For any arrow $C_1\to C_2$ in $\calC$, it induces a canonical isomorphism 
          \[\bM(C_2)\otimes_{\bA(C_2)}\bA(C_1)\xrightarrow{\cong}\bM(C_1)\]
          of $\bA(C_1)$-modules.
      \end{enumerate}
      We denote by $\Vect(\calC,\bA)$ the category of $\bA$-crystals on $\calC$.
  \end{convention}
  \begin{exam}\label{Exam-GRep}
      For any rigid analytic variety $X$ over $K$, let $X_{\proet}$ be the pro-\'etale site of $X$ in the sense of \cite{Sch-Pi}. In particular, $X_{\proet}$ has affinoid perfectoid spaces as a basis for the topology (cf. \cite[Prop. 4.8]{Sch-Pi}). Let $X_{\proet}^{{\rm aff},\perf}$ be the full subcategory of $X_{\proet}$ consisting of affinoid perfectoid spaces. By \cite[Thm. 3.5.8]{KL16} and $v$-descent (and hence pro-\'etale descent) for vector bundles \cite[Lem. 17.1.8]{SW}, an $\widehat \calO_X$-crystal on $X_{\proet}^{{\rm aff},\perf}$ is exactly a \emph{generalised representation} (i.e. a locally finite free $\widehat \calO_X$-module) on $X_{\proet}$. For this reason, by abuse of notations, we denote by $\Vect(X_{\proet},\widehat \calO_X)$ the category of generalised representation on $X_{\proet}$. When $X = \Spa(K,\calO_K)$, this is exactly the category of finite dimensional continuous $C$-representations of $G_K$.
  \end{exam}
  The main object we will study in this paper is (rational) Hodge--Tate crystals.
  \begin{dfn}\label{Dfn-HT Crystal}
      Let $\frakX$ be a smooth $p$-adic  formal scheme over $\calO_K$. A \emph{Hodge--Tate crystal} (resp. a rational Hodge--Tate crystal) is an $\overline \calO_{\Prism}$-crystal (resp. $\overline \calO_{\Prism}[\frac{1}{p}]$) on $(\frakX)_{\Prism}$. The category of Hodge--Tate crystals (resp. rational Hodge--Tate crystals) on $(\frakX)_{\Prism}$ is denoted by $\Vect((\frakX)_{\Prism},\overline \calO_{\Prism})$ (resp. $\Vect((\frakX)_{\Prism},\overline \calO_{\Prism}[\frac{1}{p}]$). We will also consider the rational Hodge--Tate crystals (i.e. $\overline \calO_{\Prism}[\frac{1}{p}]$-crystals) on the sub-site of perfect prism $(\frakX)_{\Prism}^{\perf}$ and denote the corresponding category by $\Vect((\frakX)_{\Prism}^{\perf},\overline \calO_{\Prism}[\frac{1}{p}])$.

      When $\frakX = \Spf(R^+)$ is small affine, we also denote $\Vect((\frakX)_{\Prism},\overline \calO_{\Prism})$ by $\Vect((R^+)_{\Prism},\overline \calO_{\Prism})$, and similarly define $\Vect((R^+)_{\Prism},\overline \calO_{\Prism}[\frac{1}{p}])$ and $\Vect((R^+)_{\Prism}^{\perf},\overline \calO_{\Prism}[\frac{1}{p}])$.
  \end{dfn}

  \section{Recollection of the decompletion theory of Diao--Lan--Liu--Zhu}\label{Sec-DLLZ}
 In this section, we recall the decompletion theory formulated in \cite{DLLZ} and provide some examples, which will be used in the next sections. 
 
 \subsection{Decompletion systems in \cite[Appendix A]{DLLZ}}
 We give a quick review of decompletion theory developed in \cite[Appendix A]{DLLZ} in this subsection.
 
 Let $\{A_i\}_{i\in I}$ be a direct system of topological rings with a small filtered index category $I$, the $\widehat A_{\infty}$ be a complete topological ring with compatible homomorphisms $A_i\to \widehat A_{\infty}$ such that the induced map $\varinjlim_iA_i\to \widehat A_{\infty}$ has dense image, and $\Gamma$ be a topological group acting continuously and compatibly on $A_i$'s and $\widehat A_{\infty}$.
 
 \begin{dfn}[\emph{\cite[Def. A.1.2]{DLLZ}}]\label{Decompletion system}
   We call the triple $(\{A_i\}_{i\in I},\widehat A_{\infty},\Gamma)$ a \emph{decompletion system} (resp. \emph{weak decompletion system}) if the following two conditions hold:
   
   \begin{enumerate}
       \item For any (resp. free) $\widehat A_{\infty}$-representation $L_{\infty}$ of $\Gamma$, there exists some $i\in I$, some (resp. free) $A_i$-representation of $\Gamma$, and some $\Gamma$-equivariant continuous $A_i$-linear morphism $\iota_i: L_i\to L_{\infty}$ which induces an isomorphism
       $L_i\otimes_{A_i}\widehat A_{\infty}\xrightarrow{\iota_i\otimes\id}L_{\infty}$
       of representations of $\Gamma$ over $\widehat A_{\infty}$. We call such a pair $(L_i,\iota_i)$ a \emph{model} of $L_{\infty}$ over $A_i$.
       
       \item For each model $L_i$ over $A_i$, there exists some $i_0\geq i$ such that for any $i'\geq i_0$, the model $(L_{i'},\iota_{i'}):=(L_i\otimes_{A_i}A_{i'},\iota_i\otimes\id)$ is \emph{good} in the sense that $\iota_{i'}$ induces a quasi-isomorphism $\rR\Gamma(\Gamma,L_{i'})\to\rR\Gamma(\Gamma,L_{\infty})$.
   \end{enumerate}
 \end{dfn}
 \begin{rmk}[\emph{\cite[Rem. A.1.3]{DLLZ}}]\label{Rmk-Decompletion}
    If $(\{A_i\}_{i\in I},\widehat A_{\infty},\Gamma)$ is a (weak) decompletion system, then the natural functor $\varinjlim_i\Rep_{\Gamma}(A_i)\to\Rep_{\Gamma}(\widehat A_{\infty})$ (resp. $\varinjlim_i\Rep^{\rm free}_{\Gamma}(A_i)\to\Rep^{\rm free}_{\Gamma}(\widehat A_{\infty})$) is an equivalence and for any two models $(L_{i,1},\iota_{i,1})$ and $(L_{i,2},\iota_{i,2})$ of $L_{\infty}$ over $A_i$, there exists some $i'\geq i$ such that $(L_{i,1}\otimes_{A_i}A_{i'},\iota_{i,1}\otimes\id)\cong(L_{i,2}\otimes_{A_i}A_{i'},\iota_{i,2}\otimes\id)$.
 \end{rmk}
 
 The main results in \cite[Appendix A]{DLLZ} give some sufficient conditions to clarify whether a triple $(\{A_i\}_{i\in I},\widehat A_{\infty},\Gamma)$ is a (weak) decompletion system. Recall a complex $(C^{\bullet},\rd)$ of Banach modules over a Banach ring $A$ is called \emph{uniformly strict exact} with respect to some $c\geq 0$, if for any cocycle $f\in C^s$, there exists a cochain $g\in C^{s-1}$ satisfying $|g|\leq c|f|$ such that $f = \rd(g)$.
 
 \begin{dfn}[\emph{\cite[Def. A.1.6]{DLLZ}}]\label{WeakDecompletingSystem}
   Let $(\{A_i\}_{i\in I},\widehat A_{\infty},\Gamma)$ be a triple as above. Assume that $A_i\to \widehat A_{\infty}$ are closed embeddings and that $\Gamma$ is profinite. We say $(\{A_i\}_{i\in I},\widehat A_{\infty},\Gamma)$ is \emph{weakly decompleting} if there exists a norm $|\cdot|$ on $\widehat A_{\infty}$ making it a Banach ring (and hence making $A_i$'s Banach subrings) and an inverse system $\{\Gamma_i\}_{i\in I}$ of closed normal subgroup converging to $1$ such that the projection $\Gamma\to\Gamma/\Gamma_i$ admits a continuous cross-section (which is not necessary a homomorphism) for each $i\in I$, satisfying the following conditions:
   
   \begin{enumerate}
       \item The action of $\Gamma$ on $\widehat A_{\infty}$ is isometric.
       
       \item For each $i$, the projection $\widehat A_{\infty}\to \widehat A_{\infty}/A_i$ admits an isometric section of Banach $A_i$-modules.
       
       \item There exists some $c>0$ such that the complex $C(\Gamma_i^{\bullet},\widehat A_{\infty}/A_i)$ is uniformly strict exact with respect to $c$, where we equip $C(\Gamma_i^s,\widehat A_{\infty}/A_i)$ with the supreme norm $|f|:=\sup_{\gamma\in\Gamma_i^s}|f(\gamma)|$ for each $s$.
   \end{enumerate}
 \end{dfn}
 Then the first main result in \cite[Appendix A]{DLLZ} is 
 \begin{thm}[\emph{\cite[Thm. A.1.8]{DLLZ}}]\label{Thm-DLLZ-I}
   A weakly decompleting triple is a weak decompletion system.
 \end{thm}
 Recall a Banach ring $A$ is \emph{stably uniform} if it is either the underlying ring of a stably uniform Huber pair in the sense of \cite[Def. 5.2.4]{SW} or a stably uniform adic Banach ring in the sense of \cite[Rem. 2.8.5]{KL15} over a nonarchimedean field.
 \begin{dfn}[\emph{\cite[Def. A.1.9]{DLLZ}}]\label{StablyDecompletingSystem}
   Let $(\{A_i\}_{i\in I},\widehat A_{\infty},\Gamma)$ be a triple as above. Assume that $A_i\to \widehat A_{\infty}$ are closed embeddings and that $\Gamma$ is profinite. We say $(\{A_i\}_{i\in I},\widehat A_{\infty},\Gamma)$ is \emph{stably decompleting} if the following conditions are true:
   
   \begin{enumerate}
       \item $A_i$'s and $\widehat A_{\infty}$ are stably uniform over a nonarchimedean field.
       
       \item Each rational subset $U$ of $\Spa(A_i,A_i^{\circ})$ is stablized by some open normal subgroup $\Gamma_U$ of $\Gamma$; and the pull-back of $(\{A_j\}_{j\geq i},\widehat A_{\infty},\Gamma_U)$ to each such $U$ is weakly decompleting.
   \end{enumerate}
 \end{dfn}
 Then the second main result in \cite[Appendix]{DLLZ} is 
 \begin{thm}[\emph{\cite[Thm. A.1.10]{DLLZ}}]\label{Thm-DLLZ-II}
   A stably decompleting triple is a decompletion system.
 \end{thm}
 
 We give some examples in the next subsection.
 
\subsection{Examples of decompletion system}

\subsubsection{Generalised arithmetic tower}\label{Sec-GeneralArithtower}
 Assume $d\geq 0$ and define $\Gamma_{\geo} = \oplus_{i=1}^d\Zp\gamma_i$, which is a finite free $\Zp$-module with basis $\gamma_1,\dots,\gamma_d$.
 Let $\Gamma : = \Gamma_{\geo}\rtimes\Gal(K_{\cyc}/K)$ such that for any $g\in \Gal(K_{\cyc}/K)$ and any $1\leq i\leq d$, $g\gamma_ig^{-1} = \gamma_i^{\chi(g)}$.
 
 Let $(R,R^+)$ be an affinoid Tate algebra over $K$. For any $n\geq 1$, let $R_n^+:= R\otimes_{\calO_K}\calO_{K(\zeta_{p^n})}$, $R_n:= R\otimes_KK(\zeta_{p^n})$, and $\Gamma_n: = \Gamma_{\geo}^{p^n}\rtimes\Gal(K_{\cyc}/K(\zeta_{p^n}))$. Put $\widehat R_{\infty}^+: = R^+\widehat \otimes_{\calO_K}\calO_{K_{\cyc}}$ and $\widehat R_{\infty}: = R\widehat \otimes_KK_{\cyc}$. Then $\Gamma$ acts on $R_n$'s and $\widehat R_{\infty}$ compatibly and continuously via the quotient map $\Gamma\to\Gal(K_{\cyc}/K)$.
 
 \begin{lem}\label{Lem-arithdecompletion}
   The triple $(\{R_n\}_{n\geq 0},\widehat R_{\infty},\Gamma)$ is stably decompleting.
 \end{lem}
 \begin{proof}
   By proceeding as in the proof of \cite[Prop. A.2.1.1]{DLLZ}, it suffices to show that for sufficiently large $l$ and for any $n\geq 0$, $\rH^n(\Gamma_l,\widehat R_{\infty}^+/R_l^+)$ is killed by $p^2$. Note that by \cite[Lem. 7.3]{BMS18}, one may compute $\rR\Gamma(\Gamma_{\geo},\widehat R_{\infty}^+/R_l^+)$ by using Koszul complex $\rK(\gamma_1-1,\dots,\gamma_d-1;\widehat R_{\infty}^+/R_l^+)$:
   \[\widehat R_{\infty}^+/R_l^+\xrightarrow{(\gamma_1-1,\dots,\gamma_d-1)}(\widehat R_{\infty}^+/R_l^+)^d\to\cdots\to \widehat R_{\infty}^+/R_l^+.\]
   Since $\Gamma_{\geo}$ acts on $\widehat R_{\infty}^+/R_l^+$ trivially, $\rR\Gamma(\Gamma_{\geo},\widehat R_{\infty}^+/R_l^+)$ is computed by 
   $\rK(0,\dots,0;\widehat R_{\infty}^+/R_l^+)$.
   By using Hochschild--Serre spectral sequence, in order to conclude our result, we have to show that $\rR\Gamma(\Gal(K_{\cyc}/K(\zeta_{p^l})),\widehat R_{\infty}^+/R_l^+)$ is killed by $p^2$, which was confirmed in the proof of \cite[Prop. A.2.1.1]{DLLZ}.
 \end{proof}

 \begin{thm}\label{Thm-ArithDecompletion}
   The triple $(\{R_n\}_{n\geq 0},\widehat R_{\infty},\Gamma)$ is a decompletion system.
 \end{thm}
 \begin{proof}
   This follows from Theorem \ref{Thm-DLLZ-II} together with Lemma \ref{Lem-arithdecompletion}.
 \end{proof}
 \begin{rmk}
   In \cite{Shi}, Shimizu essentially showed that the triple $(\{R_n\}_{n\geq 0},\widehat R_{\infty},\Gamma)$ is a weak decompletion system by using the formalism of Tate--Sen theory developed in \cite{BC}. When $d = 0$, the result was first proved by Sen in \cite{Sen}.
 \end{rmk}
 
 \begin{prop}\label{Prop-Unipotent}
   For any $\widehat R_{\infty}$-representation $M$ of $\Gamma$, the subgroup $\Gamma_{\geo}$ acts on $M$ quasi-unipotently.
 \end{prop}
 \begin{proof}
   By Theorem \ref{Thm-ArithDecompletion}, there exists an $n\geq 1$ and a representation $N$ of $\Gamma$ over $R_n$ such that $M\cong N\otimes_{R_n}\widehat R_{\infty}$. Since $\Gamma':=\Gamma_{\geo}\rtimes\Gal(K_{\cyc}/K(\zeta_{p^n}))$ acts on $R_n$ trivially. We see that $N$ is a linear representation of $\Gamma'$. Then the result follows from the same proof of \cite[Lem. 2.15]{LZ}.
 \end{proof}
 
 \begin{rmk}\label{Rmk-QuasiUnip}   
   Proposition \ref{Prop-Unipotent} is optimal in the sense that one can not expect that $\Gamma_{\geo}$ acts on any $\widehat R_{\infty}$-representation $M$ of $\Gamma$ unipotently in general. For example, we may assume $M = R_{\widehat K_{\cyc}}e$ and $\Gamma = \Zp\gamma\rtimes\Gal(K_{\cyc}/K)$ such that $\gamma$ acts $e$ via the scalar $\zeta_{p^n}$ for some fixed $n\geq 1$ and $\Gal(K_{\cyc}/K)$ acts on $e$ trivially. Then $M\in \Rep_{\Gamma}(R_{\widehat K_{\cyc}})$ on which the $\Gamma_{\geo}$-action is not unipotent.
 \end{rmk}

\subsubsection{Toric tower}\label{Sec-Toric tower}
 Let $X = \Spa(R,R^+)$ be a smooth affinoid space of dimension $d$ over $K$ which admits a toric chart $\Box$. Let $X_{\widehat K_{\cyc}}: = \Spa(R_{\widehat K_{\cyc}},R_{\widehat K_{\cyc}}^+)$ denote the base-change of $X$ along $K\hookrightarrow \widehat K_{\cyc}$.
 \begin{notation}\label{Notation-LocalChart}
   For any $n\geq 0$, let 
   $R_n^+:= R^+\widehat\otimes_{\calO_K\za T^{\pm 1}_1,\dots,T^{\pm 1}_d\ya,\Box}\calO_{K(\zeta_{p_n})}\za T^{\pm \frac{1}{p^n}}_1,\dots,T^{\pm \frac{1}{p^n}}_d\ya$, $R_{\infty}^+ = \varinjlim_nR_n^+$, and $\widehat R_{\infty}^+ = (\varinjlim_nR_n^+)^{\wedge}_p$ be the $p$-adic completion of $R_{\infty}^+$. Put $R_n = R_n^+[\frac{1}{p}]$ and $\widehat R_{\infty}: = \widehat R_{\infty}^+[\frac{1}{p}]$. Then both of them are stably uniform adic Banach rings. Let $X_n$ (resp. $X_{\infty}$) be the base-change (resp. the perfectoid space corresponding to the base-change) of $X$ along 
  \[\Spa(K(\zeta_{p^n})\za T^{\pm \frac{1}{p^n}}_1,\dots,T^{\pm \frac{1}{p^n}}_d\ya,\calO_{K(\zeta_{p^n})}\za T^{\pm \frac{1}{p^n}}_1,\dots,T^{\pm \frac{1}{p^n}}_d\ya)\to \bG_m^d\]
  \[(\text{resp.}~ \Spa(\widehat K_{\cyc}\za T^{\pm \frac{1}{p^{\infty}}}_1,\dots,T^{\pm \frac{1}{p^{\infty}}}_d\ya,\calO_{\widehat K_{\cyc}}\za T^{\pm \frac{1}{p^{\infty}}}_1,\dots,T^{\pm \frac{1}{p^{\infty}}}_d\ya\to \bG_m^d).\]
  Then $R_n$ (resp. $\widehat R_{\infty}$) is the ring of regular functions on $X_n$ (resp. $X_{\infty}$). For any $n\geq 0$, the natural map $X_{\infty}\to X_n$ is a Galois cover with Galois group $\Gamma_n$. In particular, we denote $\Gamma:=\Gamma_0$. The map $X_{\infty}\to X_{\widehat K_{\cyc}}$ is also a Galois cover with Galois group 
  \begin{equation}\label{Equ-GaloisGroup}
      \Gamma_{\geo}\cong \Zp\gamma_1\oplus\cdots\oplus\Zp\gamma_d
  \end{equation}
  where for any $1\leq i,j\leq d$ and any $n\geq 0$, $\gamma_i(T_j^{\frac{1}{p^n}}) = \zeta_{p^n}^{\delta_{ij}}T_j^{\frac{1}{p^n}}$ and $\delta_{ij}$ denotes Kronecker's delta. Then for any $n\geq 0$, we have an exact sequence
  \begin{equation}\label{ExactSeqofGalGroup}
     1\to \Gamma_{\geo}^{p^n}\to \Gamma_n\to \Gal(K_{\cyc}/K(\zeta_{p^n}))\to 1,
  \end{equation}
  which splits and induces an isomrphism $\Gamma_n \cong \Gamma_{\geo}^{p^n}\rtimes\Gal(K_{\cyc}/K(\zeta_{p^n}))$ such that $g\gamma_i^{p^n}g^{-1} = \gamma_i^{p^n\chi(g)}$ for any $g\in \Gal(K_{\cyc}/K(\zeta_{p^n}))$ and any $1\leq i\leq d$.
 \end{notation}
 
 \begin{lem}\label{Lem-toricdecompletion}
   The triple $(\{R_n\}_{n\geq 0},\widehat R_{\infty},\Gamma)$ is stably decompleting.
 \end{lem}
 \begin{proof}
   It suffces to show that any $(\{R_n\}_{n\geq 0},\widehat R_{\infty},\Gamma)$ is weakly decompleting, as its pullbacks to rational localizations of $X_n$ satisfy the same assumptions. We remark that $\Gamma \to \Gamma/\Gamma_n$ admits a continuous cross-section as the target is finite. 
   
   We endow $\widehat R_{\infty}$ with the spectral norm and then the condition (1) of Definition \ref{WeakDecompletingSystem} is satisfied by \cite[Rem. 2.8.3(a)]{KL15}. For condition (2), by noting that for any $n\geq 0$, the $\widehat R_{\infty}$ admits a $\Gamma$-equivariant decomposition
   \begin{equation}\label{Equ-TopoDecomposition}
       \widehat R_{\infty} = \widehat \bigoplus_{(\alpha_1,\dots,\alpha_d)\in(\bN[\frac{1}{p}]\cap[0,1))^d}R_{\widehat K_{\cyc},n}T_1^{\frac{1}{p^n}\alpha_1}\cdots T_d^{\frac{1}{p^n}\alpha_d},
   \end{equation}
   where $R_{\widehat K_{\cyc},n} = R_{\widehat K_{\cyc},n}^+[\frac{1}{p}]$ and $R_{\widehat K_{\cyc},n}^+ = R_n^+\widehat \otimes_{\calO_{K(\zeta_{p^n})}}\calO_{\widehat K_{\cyc}}$, it suffices to check $R_{\widehat K_{\cyc},n}\to R_{\widehat K_{\cyc},n}/R_n$ admits an isometric section as Banach $R_n$-modules, which can be verified as in the proof of \cite[Prop. A.2.1.1]{DLLZ}. 
   
   It remains to check condition (3) of Definition \ref{WeakDecompletingSystem}. In other words, we have to show that there exists some $c>0$ such that for any $n\geq 0$, the complex $C(\Gamma_n^{\bullet},\widehat R_{\infty}/R_n)$ is uniformly strict exact with respect to $c$.

   Similar to decomposition (\ref{Equ-TopoDecomposition}), we see that 
   \[\widehat R_{\infty}^+/R_n^+ = R_{\widehat K_{\cyc},n}^+/R_n^+\oplus \widehat \bigoplus_{\underline 0\neq (\alpha_1,\dots,\alpha_d)\in(\bN[\frac{1}{p}]\cap[0,1))^d}R_{\widehat K_{\cyc},n}^+T_1^{\frac{1}{p^n}\alpha_1}\cdots T_d^{\frac{1}{p^n}\alpha_d},\]
   which induces a decomposition of complex
   \[C(\Gamma_n^{\bullet},\widehat R_{\infty}/R_n) = C(\Gamma_n^{\bullet},R_{\widehat K_{\cyc},n}/R_n)\oplus C(\Gamma_n^{\bullet},\widehat \bigoplus_{\underline 0\neq (\alpha_1,\dots,\alpha_d)\in(\bN[\frac{1}{p}]\cap[0,1))^d}R_{\widehat K_{\cyc},n}T_1^{\frac{1}{p^n}\alpha_1}\cdots T_d^{\frac{1}{p^n}\alpha_d}).\]
   
   By Lemma \ref{Lem-arithdecompletion}, there exists some $c_1>0$ such that for any $m\geq 0$, the complex $C(\Gamma_m^{\bullet},R_{\widehat K_{\cyc},m}/R_m)$ is uniformly strict exact with respect to $c_1$.
   
   To conclude, it suffices to show that $C(\Gamma_n^{\bullet},\widehat \bigoplus_{\underline 0\neq (\alpha_1,\dots,\alpha_d)\in(\bN[\frac{1}{p}]\cap[0,1))^d}R_{\widehat K_{\cyc},n}T_1^{\frac{1}{p^n}\alpha_1}\cdots T_d^{\frac{1}{p^n}\alpha_d})$ is uniformly strict exact with respect to $c_2=p^{\frac{2}{p-1}}$. (So we may finally choose $c = \max(c_1,c_2)$.)
   For this purpose, we first claim that 
   \[\rR\Gamma(\Gamma_n,\widehat \bigoplus_{\underline 0\neq (\alpha_1,\dots,\alpha_d)\in(\bN[\frac{1}{p}]\cap[0,1))^d}R_{\widehat K_{\cyc},n}^+T_1^{\frac{1}{p^n}\alpha_1}\cdots T_d^{\frac{1}{p^n}\alpha_d})\]
   is killed by $\zeta_p-1$.
   
   By \cite[Lem. 7.3]{BMS18}, $\rR\Gamma(\Gamma_{\geo}^{p^n},R_{\widehat K_{\cyc},n}^+T_1^{\frac{1}{p^n}\alpha_1}\cdots T_d^{\frac{1}{p^n}\alpha_d})$ can be computed via the Koszul complex $\rK(\gamma_1^{p^n}-1,\dots,\gamma_d^{p^n}-1;R_{\widehat K_{\cyc},n}^+T_1^{\frac{1}{p^n}\alpha_1}\cdots R_{\widehat K_{\cyc},n}T_d^{\frac{1}{p^n}\alpha_d})$:
   \[R_{\widehat K_{\cyc},n}^+T_1^{\frac{1}{p^n}\alpha_1}\cdots T_d^{\frac{1}{p^n}\alpha_d}\xrightarrow{(\gamma_1^{p^n}-1,\dots,\gamma_d^{p^n}-1)}(R_{\widehat K_{\cyc},n}^+T_1^{\frac{1}{p^n}\alpha_1}\cdots T_d^{\frac{1}{p^n}\alpha_d})^d\to\cdots\to R_{\widehat K_{\cyc},n}^+T_1^{\frac{1}{p^n}\alpha_1}\cdots T_d^{\frac{1}{p^n}\alpha_d}.\]
   Then for $\underline \alpha\neq \underline 0$, we argue as in the proof of \cite[Lem. 5.5]{Sch-Pi} to conclude that \[\rR\Gamma(\Gamma_{\geo}^{p^n},R_{\widehat K_{\cyc},n}^+T_1^{\frac{1}{p^n}\alpha_1}\cdots T_d^{\frac{1}{p^n}\alpha_d})\]
   is concentrated in degree $\geq 1$ and is annihilated by $\zeta_p-1$. Therefore $\rR\Gamma(\Gamma_n,R_{\widehat K_{\cyc},n}^+T_1^{\frac{1}{p^n}\alpha_1}\cdots T_d^{\frac{1}{p^n}\alpha_d})$ is also concentrated in degree $\geq 1$ and is annihilated by $\zeta_p-1$ and hence so is \[\rR\Gamma(\Gamma_n,\widehat \bigoplus_{\underline 0\neq (\alpha_1,\dots,\alpha_d)\in(\bN[\frac{1}{p}]\cap[0,1))^d}R_{\widehat K_{\cyc},n}^+T_1^{\frac{1}{p^n}\alpha_1}\cdots T_d^{\frac{1}{p^n}\alpha_d})\]
   as desired. This completes the proof of claim.

   Now, let us go back to show the uniformly strict exactness of 
   \[C^{\bullet}:=C(\Gamma_n^{\bullet},\widehat \bigoplus_{\underline 0\neq (\alpha_1,\dots,\alpha_d)\in(\bN[\frac{1}{p}]\cap[0,1))^d}R_{\widehat K_{\cyc},n}T_1^{\frac{1}{p^n}\alpha_1}\cdots T_d^{\frac{1}{p^n}\alpha_d})\]
   with respect to $c_2 = p^{\frac{2}{p-1}}$. In other words, we need to prove that for any cocycle $f\in C^s$, there is a cochain $g\in C^{s-1}$ satisfying $|g|\leq p^{\frac{2}{p-1}}|f|$ such that $f = \rd(g)$. 
   Define 
   \[C^{\bullet,+}:= C(\Gamma_n^{\bullet},\widehat \bigoplus_{\underline 0\neq (\alpha_1,\dots,\alpha_d)\in(\bN[\frac{1}{p}]\cap[0,1))^d}R_{\widehat K_{\cyc},n}^+T_1^{\frac{1}{p^n}\alpha_1}\cdots T_d^{\frac{1}{p^n}\alpha_d}).\]
   As we consider the supreme norm on $C^{\bullet}$ (cf. Definition \ref{WeakDecompletingSystem}(3)), for any $h\in C^{s}$, we have
   \[h\in C^{s,+} \text{~if and only if~}|h|\leq 1.\]
   Replacing $f$ by $(\zeta_p-1)^mf$ for some $m\in \bZ$, we may assume $|\zeta_p-1|<|f|\leq 1$ and thus it is a cocycle
   \[f\in C^{s,+}:= C(\Gamma_n^{s},\widehat \bigoplus_{\underline 0\neq (\alpha_1,\dots,\alpha_d)\in(\bN[\frac{1}{p}]\cap[0,1))^d}R_{\widehat K_{\cyc},n}^+T_1^{\frac{1}{p^n}\alpha_1}\cdots T_d^{\frac{1}{p^n}\alpha_d}).\]
   As $C^{\bullet,+}$ represents $\rR\Gamma(\Gamma_n,\widehat \bigoplus_{\underline 0\neq (\alpha_1,\dots,\alpha_d)\in(\bN[\frac{1}{p}]\cap[0,1))^d}R_{\widehat K_{\cyc},n}^+T_1^{\frac{1}{p^n}\alpha_1}\cdots T_d^{\frac{1}{p^n}\alpha_d})$, applying the above claim, we conclude that $(\zeta_p-1)f$ is a coboundary in $C^{s,+}$. So there exists a cochain $h\in C^{s-1,+}$ such that $\rd(h) = (\zeta_p-1)f$. Put $g = \frac{1}{\zeta_p-1}h\in C^{s-1}$. Then we have $\rd(g) = f$ and that 
   \[|g| = |\frac{1}{\zeta_p-1}h|\leq p^{\frac{1}{p-1}}|h| \leq p^{\frac{1}{p-1}} = p^{\frac{2}{p-1}}|\zeta_p-1|<p^{\frac{2}{p-1}}|f|\]
   as desired. The proof is complete.
 \end{proof}
 \begin{thm}\label{Thm-ToricDecompltion}
   The triple $(\{R_n\}_{n\geq 0},\widehat R_{\infty},\Gamma)$ is a decompletion system.
 \end{thm}
 \begin{proof}
   This follows from Theorem \ref{Thm-DLLZ-II} together with Lemma \ref{Lem-toricdecompletion}.
 \end{proof}
 \begin{notation}
   Let $\Rep_{\Gamma}^{\rm uni}(R_{\widehat K_{\cyc}})$ be the full subcategory of $\Rep_{\Gamma}(R_{\widehat K_{\cyc}})$ of representations on which $\Gamma_{\geo}$ acts unipotently.
 \end{notation}
 The following result is a direct consequence of Theorem \ref{Thm-ToricDecompltion}.
 \begin{prop}\label{Prop-UnipotentEquiv}
   There exists an equivalence of categories
   \[\Rep_{\Gamma}^{\rm uni}(R_{\widehat K_{\cyc}}) \to \Rep_{\Gamma}(\widehat R_{\infty})\]
   via base-change such that for any $M\in \Rep_{\Gamma}^{\rm uni}(R_{\widehat K_{\cyc}})$ with associated $M_{\infty}\in \Rep_{\Gamma}(\widehat R_{\infty})$, there exist quasi-isomorphisms
   \[\rR\Gamma(\Gamma_{\geo},M)\to \rR\Gamma(\Gamma_{\geo},M_{\infty})\]
   and 
   \[\rR\Gamma(\Gamma,M)\to \rR\Gamma(\Gamma,M_{\infty}).\]
   The equivalence is compatible with the standard \'etale localisation of $X$.
 \end{prop}
 \begin{proof}
   Let $M\in \Rep_{\Gamma}^{\rm uni}(R_{\widehat K_{\cyc}})$ with $M_{\infty} = M\otimes_{R_{\widehat K_{\cyc}}}\widehat R_{\infty}$.
   Using decomposition (\ref{Equ-TopoDecomposition}) for $n=0$, we obtain a $\Gamma$-equivariant decomposition
   \begin{equation}\label{Equ-TopoDecomposition-II}
       M_{\infty} = \widehat \bigoplus_{(\alpha_1,\dots,\alpha_d)\in(\bN[\frac{1}{p}]\cap[0,1))^d}MT_1^{\alpha_1}\cdots T_d^{\alpha_d}.
   \end{equation}
   Since $\Gamma_{\geo}$ acts on $M$ unipotently, if $\alpha_i \neq 0$ for some $1\leq i\leq d$, we deduce that $\gamma_i-1$ acts on $MT_1^{\alpha_1}\cdots T_d^{\alpha_d}$ invertibly. Therefore, we see that
   \[\rR\Gamma(\Gamma_{\geo},\widehat \bigoplus_{\underline 0\neq (\alpha_1,\dots,\alpha_d)\in(\bN[\frac{1}{p}]\cap[0,1))^d}MT_1^{\alpha_1}\cdots T_d^{\alpha_d}) = 0,\]
   which implies that the natural map 
   \[\rR\Gamma(\Gamma_{\geo},M)\to \rR\Gamma(\Gamma_{\geo},M_{\infty})\]
   is a quasi-isomorphism and hence so is
   \[\rR\Gamma(\Gamma,M)\to \rR\Gamma(\Gamma,M_{\infty}).\]
   In particular, we see the functor $\Rep_{\Gamma}^{\rm uni}(R_{\widehat K_{\cyc}}) \to \Rep_{\Gamma}(\widehat R_{\infty})$ is fully faithful.
   
   We have to check the essential surjectivity of the above functor. Fix an $M_{\infty} \in \Rep_{\Gamma}(\widehat R_{\infty})$. By Theorem \ref{Thm-ToricDecompltion}, it admits a good model $M_n$ over $R_n$ for some $n\gg 0$.
   We claim that $\Gamma_{\geo}$ acts on $M_n$ and hence $M_{\widehat K_{\cyc},n}:=M_n\otimes_{R_n}R_{\widehat K_{\cyc},n}$ quasi-unipotently in the sense of \cite[Lem. 2.15]{LZ}.

   Note that $\Gamma_n\cong\Gamma_{\geo}^{p^n}\rtimes\Gal(K_{\cyc}/K(\zeta_{p^n}))$ acts $R_n$-linearly on $M_n$. 
   To prove the claim, we need to show there exists some $m\gg 0$ and $N \gg 0$ such that for any $1\leq i\leq d$, the $R_n$-linear morphism $((\gamma_i^{p^n})^{p^m}-1)^N$ acts as the zero map on $M_n$. Up to an \'etale localisation of $R_n$, we may assume $M_n$ is finite free of rank $r$ and $R_n$ is an integral domain. Fix an $R_n$-basis of $M_n$, and then we get a homomorphism $\rho:\Gamma_n\to \GL_r(R_n)$. As for any $g\in \Gal(K_{\cyc}/K(\zeta_{p^m}))$ and any $1\leq i\leq d$, $g\gamma_ig^{-1} = \gamma_i^{\chi(g)}$, we have $\rho(g)\rho(\gamma_i^{p^n})\rho(g)^{-1} = \rho(\gamma_i^{p^n})^{\chi(g)}$. So if $a$ is an eigenvalue of $\rho(\gamma_i^{p^n})$ in the algebraic closure $\overline{\Frac(R_n)}$ of the fractional field of $R_n$, then so is $a^{\chi(g)}$. As $\chi(\Gal(K_{\cyc}/K(\zeta_{p^n})))$ is an open subgroup of $\bZ_p^{\times}$, we know that $a^{p^{m_i}} = 1$ for some $m_i\gg 0$. So there exists some $N_i\gg 0$ such that $(\rho(\gamma_i^{p^n})^{p^{m_i}}-1)^{N_i} = 0$. Now, we can conclude by letting $m = \max(m_1,\dots,m_d)$ and $N = \max(N_1,\dots,N_d)$.

   Using the above claim, the arguments in the paragraph below \cite[Lem. 2.15]{LZ} applies. But for the convenience of reader, we repeat the details as follows:
   
   Since $\Gamma_{\geo}$ acts on $M_{\widehat K_{\cyc},n}$ quasi-unipotently, we have a decomposition 
   \[M_{\widehat K_{\cyc},n} = \oplus_{\tau}M_{\widehat K_{\cyc},n,\tau}\]
   where $\tau$'s are finite characters of $\Gamma_{\geo}$ and 
   \[M_{\widehat K_{\cyc},n,\tau} = \{x\in M_{\widehat K_{\cyc},n}\mid (\gamma-\tau(\gamma))^m(x) = 0~\text{for} ~m\gg 0,\forall \gamma\in \Gamma_{\geo}\}\]
   denotes the corresponding generalised eigenspaces. Each $M_{\widehat K_{\cyc},n,\tau}$ is finite projective over $R_{\widehat K_{\cyc}}$ as $M_{\widehat K_{\cyc},n}$ is. 
   After enlarging $n$ if necessary, we may assume the orders of all characters appearing in the above decomposition divide $p^n$. Then for each $\tau$, there exists a $T_1^{\alpha_1}\cdots T_d^{\alpha_d}\in R_{\widehat K_{\cyc},n}$ on which $\Gamma_{\geo}$ acts via $\tau$. Denote by $M:= M_{\widehat K_{\cyc},n,1}$ the generalised eigenspace corresponding the trivial character. Then $M$ is stable by the action of $\Gamma$ and the natural map $M\otimes_{R_{\widehat K_{\cyc}}}R_{\widehat K_{\cyc},n}\to M_{\widehat K_{\cyc},n}$ is surjective and hence an isomorphism. Then the essential surjectivity follows as \[M\otimes_{R_{\widehat K_{\cyc}}}\widehat R_{\infty}\cong M_{\widehat K_{\cyc},n}\otimes_{R_{\widehat K_{\cyc}}}\widehat R_{\infty} \cong M_{\infty}\]
   and $\Gamma_{\geo}$ acts on $M$ unipotently.
   
   Finally, we complete the proof by noting that all constructions above are compatible with standard \'etale localisations of $X$.
 \end{proof}
 
 \begin{rmk}
   Proposition \ref{Prop-Unipotent} shows that for any $M\in \Rep_{\Gamma}(R_{\widehat K_{\cyc}})$, the group $\Gamma_{\geo}$ acts quasi-unipotently and Remark \ref{Rmk-QuasiUnip} shows that the inclusion $\Rep^{\rm uni}_{\Gamma}(R_{\widehat K_{\cyc}})\hookrightarrow\Rep_{\Gamma}(R_{\widehat K_{\cyc}})$ is not an equivalence. For any $M\in \Rep_{\Gamma}(R_{\widehat K_{\cyc}})$, we describe its image via the composition
   \[\Box^{\rm uni}:\Rep_{\Gamma}(R_{\widehat K_{\cyc}}) \to \Rep_{\Gamma}(\widehat R_{\infty})\xrightarrow{\simeq} \Rep_{\Gamma}^{\rm uni}(R_{\widehat K_{\cyc}}),\]
   where the first arrow is induced by base-change, as follows:
   
   Since $\Gamma_{\geo}$ acts on $M$ quasi-unipotently, $M$ admits a decomposition
   \[M = \oplus_{\tau}M_{\tau}\]
   where $\tau$'s are finite characters on $\Gamma_{\geo}$ and $M_{\tau}$ is the corresponding generalised eigenspace as above. Let $\underline T^{\tau}\in \widehat R_{\infty}$ be of the form $T_1^{\alpha_1}\cdots T_d^{\alpha_d}$ on which $\Gamma_{\geo}$ acts via $\tau^{-1}$. Then 
   \[M':= \oplus_{\tau}M_{\tau}\underline T^{\tau}\subset M\otimes_{R_{\widehat K_{\cyc}}}\widehat R_{\infty}\]
   such that $\Gamma_{\geo}$ acts on $M'$ unipotently and that 
   \[M'\otimes_{R_{\widehat K_{\cyc}}}\widehat R_{\infty} = M\otimes_{R_{\widehat K_{\cyc}}}\widehat R_{\infty}.\]
   It is easy to see that $M^{\rm uni}:=M'$ is the image of $M$ under $\Box^{\rm uni}$.
 \end{rmk}
 \begin{rmk}
   Theorem \ref{Thm-ToricDecompltion} and Proposition \ref{Prop-UnipotentEquiv} were proved in \cite[\S 14]{Tsu} by working with a suitable integral model (with some certain log structure) of $X$.
 \end{rmk}
 \begin{notation}\label{Notation-LocalChart-II}
   For any $p$-adically complete field $F$ containing $\widehat K_{\cyc}$, let $R_{F,n}$ and $\widehat R_{F,\infty}$ be the $p$-complete base-changes of $R_{\widehat K_{\cyc},n}$ and $\widehat R_{\infty}$ along the inclusion $\widehat K_{\cyc}\to F$, respectively.
   Let $L/K$ be a Galois extension containing $K_{\cyc}$ in $\overline K$ and let $\Gamma(L/K)$ be the Galois group of the cover $X_{\widehat L,\infty} = \Spa(\widehat R_{\widehat L,\infty},\widehat R_{\widehat L,\infty}) \to X$. Then we have 
   \[\Gamma(L/K)\cong \Gamma_{\geo}\rtimes\Gal(L/K)\]
   such that $g\gamma_ig^{-1} = \gamma_i^{\chi(g)}$ for any $1\leq i\leq d$ and any $g\in \Gal(L/K)$. Clearly, $\Gamma(K_{\cyc}/K) = \Gamma$.
 \end{notation}
 \begin{rmk}\label{Rmk-Ax--Tate--Sen}
     Let $X$ be a rigid analytic space over $K$. For any $p$-adic complete field extension $E$ of $K$, let $X_E$ be the base-change of $X$ along $K\to E$. Let $L$ be a Galois extension of $K$ in $\overline K$ containing $K_{\cyc}$. Then the classical Tate--Sen theory gives rise to an equivalence of categories
     \[\Vect_{\Gal(K_{\cyc}/K)}(X_{\widehat K_{\cyc}})\xrightarrow{\simeq} \Vect_{\Gal(L/K)}(X_{\widehat L}),\]
     which is induced by linear-extension and has quasi-inverse induced by taking $\Gal(L/K_{\cyc})$-invariants.
     Here $\Vect_{\Gal(K_{\cyc}/K)}(X_{\widehat K_{\cyc}})$ denotes the category of $\Gal(K_{\cyc}/K)$-equivariant analytic vector bundles on $X_{\widehat K_{\cyc}}$ and $\Vect_{\Gal(L/K)}(X_{\widehat L})$ is defined similarly. 
     
     To see this, by \'etale descent (cf. Lemma \ref{lem: analytic descent}), it is enough to show that when $X = \Spa(R,R^+)$, for any finite free $V\in \Rep^{\rm free}_{\Gal(L/K)}(R_{\widehat L})$, the $V^{\Gal(L/K_{\cyc})}$ is a finite projective $R_{\widehat K_{\cyc}}$-module such that
     \[V^{\Gal(L/K_{\cyc})}\otimes_{R_{\widehat K_{\cyc}}}R_{\widehat L}\cong V.\]
 
     Let $r$ be the rank of $V$.
     By \cite[Prop. 3.1.4, Prop. 4.1.1 and Cor. 3.2.2]{BC}, there exists a Galois extension $F\subset L$ of $K$, which is finite over $K_{\cyc}$, such that
     $\rH^1(\Gal(L/F),\GL_r(R_{\widehat L})) = 1$. In particular, we see that 
     \[W:=V^{\Gal(L/F)}\in \Rep_{\Gal(F/K)}^{\rm free}(R_{\widehat F})\]
     satisfying $W\otimes_{R_{\widehat F}}R_{\widehat L}\cong V$. As $R_{\widehat K_{\cyc}}\to  R_{\widehat F}$ is a finite \'etale morphism with Galois group $\Gal(F/K_{\cyc})$, by \'etale descent, we see that $W^{\Gal(F/K_{\cyc})}$ is a finite projective $R_{\widehat K_{\cyc}}$-module satisfying 
     \[W^{\Gal(F/K_{\cyc})}\otimes_{R_{\widehat K_{\cyc}}}R_{\widehat F}\cong W.\]
     Now, we can conclude by noting that $W^{\Gal(F/K_{\cyc})} = V^{\Gal(L/K_{\cyc})}$. 
 \end{rmk}

 \begin{lem}\label{lem: analytic descent}
     Let $X = \Spa(R,R^+)$ be affinoid over $K$. Let $L$ be a Galois extension of $K$ containing $K_{\cyc}$. Then for any $V\in \Vect_{\Gal(L/K)}(X_{\widehat L})$, there exists a finite \'etale covering $Y \to X$ such that the pull-back of $V$ along $Y_{\widehat L}\to X_{\widehat L}$ is trivial. 
 \end{lem}
 \begin{proof}

 We first choose an analytic covering $g:Z_{\widehat L} \to X_{\widehat L}$ trivializing $V$ over $Z_{\widehat L}$. Now by \cite[Lem. 2.1.3(i)]{Ber}, there exists a finite Galois extension $K^{\prime}/K$ in $L$ together with a finite \'etale covering $Z \to X_{K^{\prime}}$ such that $Z_{\widehat L} = X_{\widehat L}\times_{X_{K^{\prime}}}Z$.
 Let $Q\subset \Gal(L/K)$ be a set of lifting of elements in $\Gal(K^{\prime}/K)$. For any $\tau\in Q$, define 
 \[Z_{\tau}= Z\times_{X_{K^{\prime}},\tau}X_{K^{\prime}} \text{ and }Z_{\tau,\widehat L}:=X_{\widehat L}\times_{X_{K^{\prime}}}Z_{\tau},\]
 and let $g_{\tau}:Z_{\tau,\widehat L}\to X_{\widehat L}$ be the natural projection.
 As $V$ is $\Gal(L/K)$-equivariant, we see that $g_{\tau}^*V$ is also trivial over $Z_{\tau,\widehat L}$. Now, we define
 \[Z_Q:=\sqcup_{\tau\in Q} Z_{\tau} ,\text{} Z_{Q,\widehat L}:=\sqcup_{\tau\in Q} Z_{\tau,\widehat L},\text{ and } g_{Q}:=\sqcup_{\tau\in Q}g_{\tau}: Z_{Q,\widehat L}\to X_{\widehat L}.\]
 Then $g_Q$ trivializes $V$ over $Z_{Q,\widehat L}$ and is induced by a $\Gal(K^{\prime}/K)$-equivariant covering $Z_{Q}\to X_{K^{\prime}}$. As $X_{K^{\prime}}\to X$ is a finite \'etale Galois covering with Galois group $\Gal(K^{\prime}/K)$, by \'etale descent, there exists a finite \'etale covering $Y\to X$ such that $Z_Q = Y\times_XX_{K^{\prime}}$.
 As
 \[Z_{Q,\widehat L} = Z_Q\times_{X_{K^{\prime}}}X_{\widehat L} = Y\times_XX_{K^{\prime}} \times_{X_{K^{\prime}}}X_{\widehat L}=Y_{\widehat L},\]
 one can conclude by checking that $Y$ is desired.
 \end{proof}
 
 \begin{rmk}\label{Rmk-UnipotentEquiv}
   Note that for any Galois extension $L/K$ containing $K_{\cyc}$ in $\overline K$, the 
   pro-\'etale covering $\Spa(\widehat R_{\widehat L,\infty},\widehat R_{\widehat L,\infty}^+)\to \Spa(\widehat R_{\infty},\widehat R_{\infty}^+)$ is a Galois covering with Galois group $\Gal(L/K_{\cyc})$. By descent of vector bundles \cite[Lem.17.1.8]{SW} (and Lemma \ref{Lem-EvaluateGRep}), the base-change along $\widehat R_{\infty}\to \widehat R_{\widehat L,\infty}$ induces an equivalence of categories
   \[\Rep_{\Gamma}(\widehat R_{\infty})\xrightarrow{\simeq}\Rep_{\Gamma(L/K)}(\widehat R_{\widehat L,\infty}).\]
   As $\Gamma_{\geo}$ commutes with $\Gal(L/K_{\cyc})$, by Remark \ref{Rmk-Ax--Tate--Sen}, the scalar extension induces the following equivalence of categories
   \[\Rep_{\Gamma}^{\rm uni}(\widehat R_{\widehat K_{\cyc}})\xrightarrow{\simeq}\Rep_{\Gamma(L/K)}^{\rm uni}(\widehat R_{\widehat L}).\]
   Therefore, by Proposition \ref{Prop-Unipotent}, we have the following commutative diagram of equivalences of categories:
   \[\xymatrix@C=0.5cm{
   \Rep_{\Gamma}^{\rm uni}(\widehat R_{\widehat K_{\cyc}})\ar[rrrr]^{\simeq}\ar[d]^{\simeq}&&&&\Rep_{\Gamma}(\widehat R_{\infty})\ar[d]^{\simeq}\\
    \Rep_{\Gamma(L/K)}^{\rm uni}(\widehat R_{\widehat L})\ar[rrrr]^{\simeq}&&&&\Rep_{\Gamma(L/K)}(\widehat R_{\widehat L,\infty}),
   }\]
   which is induced by the scalar extensions.
 \end{rmk}
 A similar but much easier argument for the proof of Lemma \ref{Lem-toricdecompletion} also shows the following.
 \begin{lem}\label{Lem-geotoricdecompleting}
   The triple $(\{R_{F,n}\},\widehat R_{F,\infty},\Gamma_{\geo})$ is stably decompleting.
 \end{lem}
 We leave its proof to readers as we will not use this lemma in this paper.
 \begin{thm}\label{Thm-GeoToricDecompletion}
   The triple $(\{R_{F,n}\},\widehat R_{F,\infty},\Gamma_{\geo})$ is a decompletion system.
 \end{thm}
 \begin{proof}
   Just combine Theorem \ref{Thm-DLLZ-II} with Lemma \ref{Lem-geotoricdecompleting}.
 \end{proof}
 \begin{rmk}
   \begin{enumerate}
       \item Theorem \ref{Thm-GeoToricDecompletion} was proved essentially in \cite{AGT} and \cite{Tsu}.
       
       \item Unlike Proposition \ref{Prop-UnipotentEquiv}, an $\widehat R_{L,\infty}$-representation of $\Gamma_{\geo}$ may not admits a model over $R_L$. However, if $M$ is small, which means that the $\Gamma_{\geo}$-action on $M$ is ``close'' to the identity, then it admits a good model over $R_L$. See \cite[Chap. II.14]{AGT}, \cite[\S 12]{Tsu} or \cite[\S 3]{Wan23} for details.
   \end{enumerate}
 \end{rmk}

\section{A $p$-adic Simpson correspondence for smooth rigid analytic varieties over $K$}\label{Sec-Simpson}
 In this section, we construct a $p$-adic Simpson correspondence for generalised representations on $X_{\proet}$ for \emph{smooth} rigid analytic varieties $X$, by working on the arenas of \cite{LZ} and \cite{DLLZ}. The main ingredient is a period sheaf $\OC$ on $X_{\proet}$, which was firstly studied in \cite{Hy} in the local case. It is defined as the graded piece $\Gr^0\calO\BBdR$ of the de Rham period sheaf on $X_{\proet}$ in \cite[\S 6]{Sch-Pi} and was used to study $p$-adic Simpson correspondence in \cite{LZ}. In particular, the period sheaf $\OC$ is equipped with a natural Higgs field
 \begin{equation}\label{Equ-Theta}
     \Theta : \OC\to \OC\otimes_{\calO_X}\Omega^1_X(-1)
 \end{equation}
 such that the induced Higgs complex $\HIG(\OC,\Theta)$ is an resolution of $\OX$.
 As remarked in \cite[Rem. 2.1]{LZ}, let
 \[0\to\OX\to\calE\to\OX\otimes_{\calO_X}\Omega^1_X(-1)\to 0\]
 denote the Faltings' extension (which can be found in \cite[Cor. 6.14]{Sch-Pi} up to a Tate twist). Then 
 \[\OC = \varinjlim_n\Sym^n\calE\]
 where the map $\Sym^n\calE\to\Sym^{n+1}\calE$ sends each local section $x_1\otimes\cdots\otimes x_n\in\Sym^n\calE$ to $1\otimes x_1\otimes\cdots\otimes x_n\in \Sym^{n+1}\calE$ and the natural Higgs field $\Theta$ is induced by sending each local section $x_1\otimes\cdots\otimes x_n$ of $\Sym^n\calE$ to the local section $-\sum_{i=1}^nx_1\otimes\cdots x_{i-1}\otimes x_{i+1}\otimes\cdots\otimes x_n\otimes \overline x_i$ of $\Sym^{n-1}\calE\otimes_{\calO_X}\Omega^1_X(-1)$\footnote{Our Higgs field $\Theta$ differs from that in \cite{Hy} by a sign but is compatible with $\Gr^0\nabla$ in \cite[\S 6]{Sch-Pi}.}, where $\overline x_i$ denotes the image of $x_i$ under the projection $\calE\to \OX\otimes_{\calO_X}\Omega^1_X(-1)$. See \cite{Hy} for more details.
 \begin{dfn}\label{Dfn-G-HiggsBundle}
   Let $L/K$ be an algebraic Galois extension containing $\mu_{p^{\infty}}$ with Galois group $\Gal(L/K)$ and $\widehat L$ be the completion of $L$ in $C$. 
   By a \emph{$\Gal(L/K)$-Higgs bundle of rank $l$} on $X_{\widehat L,\et}$, we mean a Higgs bundle $(\calH,\theta_{\calH})$ with $\theta_{\calH}$ nilpotent of rank $l$ on $X_{\widehat L,\et}$ together with a continuous $\Gal(L/K)$-action on $\calH$ with respect to the $p$-adic topology on $\calH$ such that the morphism
   \[\calH\xrightarrow{\theta_{\calH}}\calH\otimes_{\calO_X}\Omega^1_X(-1)\]
   is $\Gal(L/K)$-equivariant. 
   We denote by $\HIG_{\Gal(L/K)}(X_{\widehat L})$ the category of $\Gal(L/K)$-Higgs bundles on $X_{\widehat L,\et}$. 
   In particular, when $L = \overline K$ (and hence $\widehat L = C$), we denote $\HIG_{\Gal(\overline K/K)}(X_{\widehat{\overline K}})$ by $\HIG_{G_K}(X_C)$.
   
   In particular, when $X = \Spa(R,R^+)$ is smooth affinoid, one can define \emph{$\Gal(L/K)$-Higgs modules $(H,\theta_H:H\to H\otimes_{R}\Omega^1_R(-1))$ of rank $l$} over $R_{\widehat L}$ similarly and denote the category of those $(H,\theta_H)$ by $\HIG_{\Gal(L/K)}(R_{\widehat L})$. We remark that in this case, taking global sections on $X$ induces an equivalence of categories $\HIG_{\Gal(L/K)}(X_{\widehat L})\xrightarrow{\simeq} \HIG_{\Gal(L/K)}(R_{\widehat L})$ sending each $(\calH,\theta_{\calH})$ to $(H,\theta_H)$ with $H = \calH(X)$ and $\theta_H$ being induced by $\theta_{\calH}$.
 \end{dfn}
 \begin{rmk}
   We claim that the nilpotency condition on $\theta_{\calH}$ in Definition \ref{Dfn-G-HiggsBundle} can be deduced from other assumptions. To see this, we may assume $X = \Spa(R,R^+)$ admits a toric chart and assume $(\calH,\theta_{\calH})$ is induced by a $\Gal(K_{\cyc}/K)$-Higgs module $(H,\theta_H)$ over $R_{\widehat K_{\cyc}}$. By Theorem \ref{Thm-ArithDecompletion}, we may further assume $(H,\theta_H)$ is defined over $R_{K(\zeta_{p^n})}$ for some $n\geq 0$. By abuse of notation, we now assume $H$ is itself a finite projective $R_{K(\zeta_{p^n})}$-module. Write $\theta_H = \sum_{i=1}^d\theta_i\otimes\frac{\dlog T_i}{t}$ with $\theta_i$'s belonging to $\End_{R_{K(\zeta_{p^n})}}(H)$. Then for any $x\in H$ and any $g\in \Gal(K_{\cyc}/K)$, we have 
   \[g(\theta_H(x)) = g(\sum_{i=1}^d\theta_i(x)\otimes\frac{\dlog T_i}{t}) = \chi(g)^{-1}\sum_{i=1}^dg(\theta_i(x))\otimes\frac{\dlog T_i}{t}\]
   and that
   \[\theta_H(g(x)) = \sum_{i=1}^d\theta_i(g(x))\otimes\frac{\dlog T_i}{t}.\]
   So we get $g\theta_ig^{-1} = \chi(g)\theta_i$ for any $1\leq i\leq d$. To see $\theta_i$'s are nilpotent, we may apply the proof of \cite[Lem. 2.15]{LZ}, just like the third paragraph in the proof of Proposition \ref{Prop-UnipotentEquiv}. Roughly, up to an \'etale localisation of $R_{K(\zeta_{p^n})}$, we may assume that $H$ is finite free over $R_{K(\zeta_{p^n})}$ and $R_{K(\zeta_{p^n})}$ is an integral domain, and reduce the problem to showing that the eigenvalues of $\theta_i$ are all zero. Note that for any $g\in \Gal(K_{\cyc}/K(\zeta_{p^n}))$, it acts $R_{K(\zeta_{p^n})}$-linearly on $H$. As $g\theta_ig^{-1} = \chi(g)\theta_i$, we see that if $a$ is an eigenvalue of $\theta_i$, then so is $\chi(g)a$ for any $g\in \Gal(K_{\cyc}/K(\zeta_{p^n}))$. This forces that $a = 0$ as $\chi(\Gal(K_{\cyc}/K(\zeta_{p^n})))$ is an open subgroup of $\bZ_p^{\times}$.
 \end{rmk}
 Before we move on, we introduce some notations which will appear several times in this section.
 For a generalised representation $\calL\in\Vect(X_{\proet},\OX)$, put
 \begin{equation}\label{Equ-Theta_L}
     \Theta_{\calL} := \id_{\calL}\otimes\Theta:\calL\otimes_{\OX}\OC\to\calL\otimes_{\OX}\OC\otimes_{\calO_X}\Omega^1_X(-1).
 \end{equation}
 The $\Theta_{\calL}$ satisfies $\Theta_{\calL}\wedge\Theta_{\calL} = 0$ as this holds true for $\Theta$. For a Higgs bundle $(\calH,\theta_{\calH})$ on $X_{C,\et}$, put
 \begin{equation}\label{Equ-Theta_H}
   \Theta_{\calH} := \theta_{\calH}\otimes\id_{\OC}+\id_{\calH}\otimes\Theta:\calH\otimes_{\calO_{X_C}}\OC_{\mid X_C}\to\calH\otimes_{\calO_{X_C}}\OC_{\mid X_C}\otimes_{\calO_X}\Omega_X^1(-1).
 \end{equation}
 The $\Theta_{\calH}$ also satisfies $\Theta_{\calH}\wedge\Theta_{\calH} = 0$ as this holds true for both $\theta_{\calH}$ and $\Theta$.
 
 Then our main result can be stated as follows:
 
 \begin{thm}\label{Thm-Simpson}
   Let $\nu:X_{\proet}\to X_{C,\et}$ be the natural map of sites and let $(\OC,\Theta)$ be as above. 
   For any generalised representation $\calL\in \Vect(X_{\proet},\OX)$, let $\Theta_{\calL}=\id_{\calL}\otimes\Theta$, then the rule
   \[\calL\mapsto (\calH(\calL),\theta_{\calH(\calL)}):= (\nu_*(\calL\otimes_{\OX}\OC),\nu_*(\Theta_{\calL}))\]
   induces a rank-preserving equivalence from the category $\Vect(X_{\proet},\OX)$ of generalised representations on $X_{\proet}$ to the category $\HIG_{G_K}(X_{C})$ of $G_K$-Higgs bundles on $X_{C,\et}$, which preserves tensor products and dualities. Moreover, the following assertions are true:
   
   \begin{enumerate}
       \item For any $i\geq 1$, the higher direct image $\rR^i\nu_*(\calL\otimes_{\OX}\OC) = 0$.
       
       \item Let $\Theta_{\calH(\calL)}=\theta_{\calH(\calL)}\otimes\id_{\OC}+\id_{\calH(\calL)}\otimes\Theta$. Then there exists a natural isomorphism
       \[(\calH(\calL)\otimes_{\calO_{X_{C}}}\OC_{\mid X_{C}},\Theta_{\calH(\calL)}) \xrightarrow{\cong} (\calL\otimes_{\OX}\OC,\Theta_{\calL})_{\mid X_{C}}\]
       of Higgs fields.
       
       \item Let $(\HIG(\calH(\calL),\theta_{\calH(\calL)})$ denote the Higgs complex induced by $(\calH(\calL),\theta_{\calH(\calL)})$. Then there exists a natural $G_K$-equivariant quasi-isomorphism
       \[\rR\Gamma(X_{C,\proet},\calL)\cong \rR\Gamma(X_{C,\et},(\HIG(\calH(\calL),\theta_{\calH(\calL)}))\]
       which is compatible with $G_K$-actions. As a consequence, we get a quasi-isomorphism
       \[\rR\Gamma(X_{\proet},\calL)\cong \rR\Gamma(G_K,\rR\Gamma(X_{C,\et},(\HIG(\calH(\calL),\theta_{\calH(\calL)}))).\]
       
       \item Let $X'\to X$ be a smooth morphism of rigid analytic varieties over $K$. Then the equivalence in (3) is compatible with pull-back along $f$. In other words, for any $\calL\in\Vect(X_{\proet},\OX)$ with corresponding $(\calH,\theta_{\calH})\in\HIG_{G_K}(X_C)$, we have 
       \[(\calH(f^*\calL),\theta_{\calH(f^*\calL)})\cong (f^*\calH,f^*\theta_{\calH}).\]
   \end{enumerate}
 \end{thm}
  \begin{rmk}\label{Rmk-Simpson}
       By Remark \ref{Rmk-Ax--Tate--Sen}, for any Galois extension $L$ of $K$ containing $K_{\cyc}$, there is an equivalence of categories
       \[\HIG_{\Gal(L/K)}(X_{\widehat L})\to\HIG_{G_K}(X_C)\]
       induced by the scalar extension, whose quasi-inverse is induced by taking $\Gal(\overline K/L)$-invariants.
       For this reason, one can replace $C$ and $G_K$ by $\widehat L$ and $\Gal(L/K)$ respectively in the statement of Theorem \ref{Thm-Simpson} and then conclude an equivalence of categories
       \[\Vect(X_{\proet},\OX)\simeq \HIG_{\Gal(L/K)}(X_{\widehat L})\]
       such that the corresponding assertions (1)-(4) hold true, and are compatible with Theorem \ref{Thm-Simpson} via the base-change along $\widehat L\to C$.
 \end{rmk}
 \begin{rmk}\label{Rmk-AnalyticEtaleComparison}
    By analytic-\'etale comparison (cf. \cite[Prop. 8.2.3]{FvdP}), we may replace \'etale site $X_{C,\et}$ by analytic site $X_{C,\an}$ in the statement of Theorem \ref{Thm-Simpson} such that all results still hold true. The same remark also applies to Remark \ref{Rmk-Simpson}.
 \end{rmk}
 \begin{rmk}
   \begin{enumerate}
       \item When $X$ is affine, Theorem \ref{Thm-Simpson} was also achieved by Tsuji \cite[Thm 15.2]{Tsu} by choosing a certain integral model of $X$ (and considering a certain log structure on the chosen model), which is not necessary in our approach.
       
       \item When $\calL = \bL\otimes_{\bZ}\OX$ is induced by a $\Qp$-local system on $X_{\et}$, Theorem \ref{Thm-Simpson} reduces to \cite[Thm 2.1]{LZ} by studying decompletion theory for relative analogue of the overconvergent period ring $\widetilde \bfB^{\dagger}$. 
       Our proof is inspired of the work in \cite{LZ} and improve theirs to any generalised representations by using decompletion theory in the previous section.
       
       \item Theorem \ref{Thm-Simpson} can be also deduced from \cite{Heu} when $X$ is a curve and \cite{HMW} when $X$ is abeloid by noticing that the morphism ``$\rH\rT\log$'' in loc.cit. is $G_K$-equivariant.

       \item[(4)] As the canonical map $K\to\BdRp$ gives rise to a $G_K$-equivariant lifting of $X_C$ to $\BdRp/t^2$, the equivalence and cohomological comparison in Theorem \ref{Thm-Simpson} can be deduced from a very recent work of Heuer \cite{Heu-Simpson} when $X$ is furthermore proper. It is still a problem to compare two constructions in this paper and in loc.cit..
   \end{enumerate}
 \end{rmk}
  Before moving on, let us give a immediate corollary of Theorem \ref{Thm-Simpson}.
  
  \begin{cor}\label{Cor-Simpson}
    Let $d$ be the dimension of $X$ over $K$. 
    
    \begin{enumerate}
       \item Assume $X$ is quasi-compact. Then $\rR\Gamma(X_{C,\proet},\calL)$ is concentrated in degree $[0,2d]$.
       
       \item If moreover $X$ is proper, then $\rR\Gamma(X_{C,\proet},\calL)$ is a perfect complex of $C$-representations of $G_K$ and $\rR\Gamma(X_{\proet},\calL)$ is a perfect complex of $K$-vector spaces concentrated in degree $[0,2d+1]$. 
    \end{enumerate}
  \end{cor}
  \begin{proof}
     We first prove Item (1). Since $X$ is quasi-compact and locally noetherian, it is a noetherian space. Then we obtain that $\rR\Gamma(X_{C,\proet},\calL)$ is concentrated in degree $[0,2d]$ by combining Remark \ref{Rmk-AnalyticEtaleComparison} with Grothendieck's vanishing theorem \cite[Thm 3.6.5]{Gro}.

     Now, we prove Item (2). The first assertion follows from the properness of $X$ and Item (1). It remains to show that $\rR\Gamma(X_{\proet},\calL)$ is a perfect complex of $K$-vector spaces concentrated in degree $[0,2d+1]$.
     Due to Remark \ref{Rmk-Simpson}, one may replace $\widehat K_{\cyc}$ and $\Gal(K_{\cyc}/K)$ instead of $C$ and $G_K$ in Theorem \ref{Thm-Simpson} and deduce that $\rR\Gamma(X_{\widehat K_{\cyc},\proet},\calL)$ is a perfect complex of $\widehat K_{\cyc}$-representations of $\Gal(K_{\cyc}/K)$  concentrated in degree $[0,2d]$. 
     Thanks to Theorem \ref{Thm-Simpson} (3), we can conclude by the following well-known fact:
     For any finite dimensional representations $V$ of $\Gal(K_{\cyc}/K)$ over $K_{\cyc}$, the $\rR\Gamma(\Gal(K_{\cyc}/K),V)$ is concentrated in degree $[0,1]$ (for example see \cite[Thm. 7.17]{GMW-HT}).
  \end{proof}
  
\subsection{A local version of Theorem \ref{Thm-Simpson}}

 This subsection is devoted to a local version of Theorem \ref{Thm-Simpson}. More precisely, we assume $X = \Spa(R,R^+)$ is smooth affinoid of dimension $d$ over $K$, which admits a toric chart, and keep the notations in Notation \ref{Notation-LocalChart} and Notation \ref{Notation-LocalChart-II}.
 
 The following lemma is well-known:
 \begin{lem}\label{Lem-EvaluateGRep}
   The evaluation at $X_{\widehat L,\infty}$ induces an equivalence from the category $\Vect(X_{\proet},\OX)$ to the category $\Rep_{\Gamma(L/K)}(\widehat R_{\widehat L,\infty})$ of representations of $\Gamma(L/K)$ over $\widehat R_{\widehat L,\infty}$, which preserves tensor products and dualities. Moreover, for any $\calL\in\Vect(X_{\proet},\OX)$, there exists a natural quasi-isomorphism
   \[\rR\Gamma(X_{\proet},\calL)\simeq \rR\Gamma(\Gamma(L/K),\calL(X_{\widehat L,\infty})).\]
 \end{lem}
 \begin{proof}
   Note that we may regard generalised representations as $\OX$-crystals on $X_{\proet,\rm{aff},\perf}$. Since $X_{\widehat L,\infty}$ is a cover of $X$ with Galois group $\Gamma(L/K)$, it is a cover of the final object of $\Sh(X_{\proet})$. We denote by $X_{\widehat L,\infty}^{\bullet/X}$ the corresponding \v Cech nerve. In particular, for any $n\geq 0$, the $X_{\widehat L,\infty}^{n/X}$ is the self-product of $n+1$ copies of $X_{\widehat L,\infty}$ over $X$. We claim that there exists an natural isomorphism 
   \[X_{\widehat L,\infty}^{n/X} = \Spa(C(\Gamma(L/K)^n,\widehat R_{\widehat L,\infty}),C(\Gamma(L/K)^n,\widehat R_{\widehat L,\infty}^+)),\]
   where $C(\Gamma(L/K)^n,\widehat R_{\widehat L,\infty}^+)$ denotes the continuous functions from $\Gamma(L/K)^n$ to $\widehat R_{\widehat L,\infty}^+$. 
   
   Indeed, the claim follows from the same argument in the proof \cite[Lem. 5.6]{Sch-Pi}: 
   As $\Gamma(L/K)^n$ is a profinite group, it can be viewed as an object $\underline{\Gamma(L/K)}^n$ in $\Spa(K,\calO_K)_{\text{prof\'et}}$ and hence an object in $\Spa(K,\calO_K)_{\proet}$ such that
   \[\begin{split}
       \underline{\Gamma(L/K)}^n = \Spa(C(\Gamma(L/K)^n,K),C(\Gamma(L/K)^n,\calO_K)).
   \end{split}\]
   Via the pull-back to $X_{\proet}$, it can be viewed as an object in $X_{\proet}$, and then we obtain an object 
   \[X_{\widehat L,\infty}\times \underline{\Gamma(L/K)}^n = \Spa(C(\Gamma(L/K)^n,\widehat R_{\widehat L,\infty}),C(\Gamma(L/K)^n,\widehat R_{\widehat L,\infty}^+))\in X_{\proet},\] which is affinoid perfectoid as $X_{\widehat L,\infty}$ is. On the other hand, as $X_{\widehat L,\infty}\to X$ is a Galois covering whose Galois group is exactly $\Gamma(L/K)$, we have an natural isomorphism 
   $X_{\widehat L,\infty}^{n/X} \cong X_{\widehat L,\infty}\times \underline{\Gamma(L/K)}^n$, which implies the claim as desired.
   
   Thanks to the claim, we have an isomorphism of cosimplicial rings
   \[\OX(X_{\widehat L,\infty}^{\bullet/X})\cong C(\Gamma(L/K)^{\bullet},\widehat R_{\widehat L,\infty}).\]
   By $v$-descent of vector bundles (cf. Example \ref{Exam-GRep} or \cite[Lem. 17.1.8]{SW}), we get an equivalence of categories
   \[\Vect(X_{\proet},\OX)\simeq{\rm Strat}(\OX(X_{\widehat L,\infty}^{\bullet/X}))\simeq {\rm Strat}(C(\Gamma(L/K)^{\bullet},\widehat R_{\widehat L,\infty})).\]

   Before we move on, let us recall that if $\Spa(S,S^+)$ is affinoid perfectoid, then any finite projective $S$-module $E$ is endowed with the canonical topology in the sense of \cite[Lem. 2.2.6(1)]{KL15} such that any morphism $E_1\to E_2$ of finite projective is always bounded by \cite[Lem. 2.2.6(2)]{KL15}. In particular, any automorphism of a finite projective $S$-module is alway a bounded (and hence an automorphism in the category of Banach $S$-modules) by \cite[Thm. 2.2.8]{KL15}. Note that $C(\Gamma(L/K)^{\bullet},\widehat R_{\widehat L,\infty})$ is perfectoid. The above argument implies that for any stratification $(M,\varepsilon_M)\in {\rm Strat}(C(\Gamma(L/K)^{\bullet},\widehat R_{\widehat L,\infty}))$, the $\varepsilon_M$ and its inverse $\varepsilon_M^{-1}$ are both bounded.

   Now, we are going to show that there is an equivalence of categories
   \[\Rep_{\Gamma(L/K)}(\widehat R_{\widehat L,\infty})\simeq {\rm Strat}(C(\Gamma(L/K)^{\bullet},\widehat R_{\widehat L,\infty})).\]
   
   Given a representaion $M\in \Rep_{\Gamma(L/K)}(\widehat R_{\widehat L,\infty})$, we can construct a stratification
   \[\varepsilon_M:M\otimes_{\widehat R_{\widehat L,\infty},p_0} C(\Gamma(L/K),\widehat R_{\widehat L,\infty})\xrightarrow{\cong} M\otimes_{\widehat R_{\widehat L,\infty},p_1} C(\Gamma(L/K),\widehat R_{\widehat L,\infty}) \]
   such that for any $m\in M$, we have $\varepsilon_M(m)={\rm ev}_m$ where ${\rm ev}_m$ is the evaluation map at $m$ via the isomorphism
   $M\otimes_{\widehat R_{\widehat L,\infty},p_1} C(\Gamma(L/K),\widehat R_{\widehat L,\infty})\cong C(\Gamma(L/K),M)$. One can check $\varepsilon_M$ satisfies the cocycle condition by noting that $M$ is a $\Gamma(L/K)$-representation.
   Conversely, given a stratification $(M,\varepsilon_M)\in {\rm Strat}(C(\Gamma(L/K)^{\bullet},\widehat R_{\widehat L,\infty}))$, as it satisfies the cocycle condition, we can get a $\Gamma(L/K)$-action on $M$ by letting $g(m) = \varepsilon_M(m)(g)$ for any $g\in \Gamma(L/K)$ and $m\in M$ after identifying 
   $M\otimes_{\widehat R_{\widehat L,\infty},p_1} C(\Gamma(L/K),\widehat R_{\widehat L,\infty})$ with $C(\Gamma(L/K),M)$. As $\varepsilon_M$ is an isomorphism in the category of Banach $C(\Gamma(L/K),\widehat R_{\widehat L,\infty})$-modules, we see that the above $\Gamma(L/K)$-action on $M$ is continuous. So we finally get an object in $\Rep_{\Gamma(L/K)}(\widehat R_{\widehat L,\infty})$.
   
   One can check from their constructions that the above two functors induce the desired equivalence of categories. It follows from the standard linear algebra that the equivalence above preserves tensor products and dualities.

   It remains to prove the ``moreover'' part. Using Example \ref{Exam-GRep} again, we get isomorphisms
   \[\calL(X_{\widehat L,\infty}^{\bullet/X})\cong\calL(X_{\widehat L,\infty})\otimes_{\widehat R_{\widehat L,\infty}}C(\Gamma(L/K)^{\bullet},\widehat R_{\widehat L,\infty}) \cong C(\Gamma(L/K)^{\bullet},\calL(X_{\widehat L,\infty})),\]
   where $\calL(X_{\widehat L,\infty})$ is endowed with the canonical topology.
   Since $\rR\Gamma(\Gamma(L/K),\calL(X_{\widehat L,\infty}))$ is computed by $C(\Gamma(L/K)^{\bullet},\calL(X_{\widehat L,\infty}))$, it suffices to show that $\rR\Gamma(X_{\proet},\calL)$ can be computed by the \v Cech-Alexander complex $\calL(X_{\widehat L,\infty}^{\bullet/X})$.
   To see this, as we have
   \[\rR\Gamma(X_{\proet},\calL) \simeq \rR\lim(\RGamma(X_{\proet}/X_{\widehat L,\infty}^{\bullet/X},\calL))\simeq \rR\lim(\calL(X_{\widehat L,\infty}^{\bullet/X}))\]
   where the last quasi-isomorphism follows from \cite[Prop. 2.3]{LZ} as $X_{\widehat L,\infty}^{\bullet/X}$ is affinoid perfectoid. So we can conclude as $\rR\lim(\calL(X_{\widehat L,\infty}^{\bullet/X}))$ is exactly the \v Cech-Alexander complex $\calL(X_{\widehat L,\infty}^{\bullet/X})$.
 \end{proof}
 
 According to Lemma \ref{Lem-EvaluateGRep}, we may work with $\widehat R_{\infty}$-representations of $\Gamma$ instead of generalised representations on $X_{\proet}$ in this subsection. We also need the following local description of $\OC$.
 \begin{lem}\label{Lem-OC}
   Let $L/K$ be a Galois extension in $\overline K$ containing $K_{\cyc}$. Then sending $Y_i$ to $t^{-1}\log(\frac{[T_i]^{\flat}}{T_i})$ for any $1\leq i\leq d$ induces an isomorphism of sheaves of $\OX$-algebras on $X_{\proet}/X_{\widehat L,\infty}$
   \[\iota: \OX[Y_1,\dots,Y_d]_{\mid X_{\widehat L,\infty}}\to \OC_{\mid X_{\widehat L,\infty}}.\]
   Moreover, the following assertions are true:
   \begin{itemize}
       \item[(1)] The pull-back of the Higgs field $\Theta$ (cf. Equation (\ref{Equ-Theta})) along $\iota$ on $\OX[Y_1,\dots,Y_d]$ is given by $\Theta=\sum_{i=1}^d-\frac{\partial}{\partial Y_i}\otimes\frac{\dlog T_i}{t}$.

       \item[(2)] The action of $\Gamma(L/K)\cong \Gamma_{\geo}\rtimes\Gal(L/K)$ on $Y_i$'s is determined such that for any $1\leq i,j\leq d$ and any $g\in \Gal(L/K)$, we have $\gamma_i(Y_j) = Y_j+\delta_{ij}$ and $g(Y_j) = \chi(g)^{-1}Y_j$.
   \end{itemize}
 \end{lem}
 \begin{proof}
   Let $X_i:= T_i-[T_i]^{\flat}$ as in \cite[Prop. 6.10]{Sch-Pi}. Then the desired isomorphism follows from \cite[Cor. 6.15]{Sch-Pi} by noting that $\frac{X_i}{\xi}$ and $Y_i$ differ from a unit in $\Gr^0\calO\bB_{\dR}$. It remains to prove the ``moreover part''. We now identify $\OC$ with $\OX[Y_1,\dots, Y_d]$ via $\iota$.

   \begin{itemize}
       \item[(1)] Recall that for any $n\geq 0$, the connection $\nabla$ on $\calO\bB_{\dR}^+/(\Ker(\theta)^n)$ is induced by the $\BBdRp/(\Ker(\theta)^n)$-linear extension of the canonical connection on $\calO_X$ (cf. \cite[the paragraph below Remark 6.9]{Sch-Pi}). See \cite[Def. 6.1(i) and Def. 6.2(ii)]{Sch-Pi} for the definition of $\theta$. In particular, modulo $\Ker(\theta)^n$ for any $n\geq 0$, we have 
       \[\begin{split}
           \nabla(Y_i)  &= \nabla(t^{-1}\log(\frac{[T_i]^{\flat}}{T_i}))\\
           &= t^{-1}\frac{\partial}{\partial T_i}(\sum_{m\geq 1}(-1)^{m-1}\frac{([T_i^{\flat}]/T_i-1)^{m}}{m})\otimes\rd T_i\\
           &= \frac{[T_i^{\flat}]}{T_i}\sum_{m\geq 1}(-1)^m(\frac{[T_i^{\flat}]}{T_i}-1)^{m-1}\otimes\frac{\dlog T_i}{t}\\
           &=-\frac{\dlog T_i}{t}
       \end{split}.\]
       By letting $n$ tend to $+\infty$, we see that $\nabla(Y_i) = -\frac{\dlog T_i}{t}$ in $\OBdRp$.
       So we can conclude as $\Theta$ is induced by $\nabla$ via the identification $\OC = \Gr^0\OBdR$ (cf. \cite[(2.2)]{LZ}).

       \item[(2)] The $\Gamma(L/K)$-action on $Y_i$'s can be deduced similarly as in the proof of \cite[Lem. 6.17]{Sch-Pi}.
   \end{itemize}
 \end{proof}
 
 \begin{notation}\label{Notation-LocalOC}
   Let $C_{L,\infty}: = \OC(X_{\widehat L,\infty})$ with induced Higgs field $\Theta$. Then Lemma \ref{Lem-OC} provides a $\Gamma(L/K)$-equivariant isomorphism of Higgs fields
   \[\iota: (\widehat R_{\widehat L,\infty}[Y_1,\dots,Y_d],\sum_{i=1}^d-\frac{\partial}{\partial Y_i}\otimes\frac{\dlog T_i}{t})\xrightarrow{\cong}(C_{L,\infty},\Theta),\]
   where the $\Gamma$-action on $\widehat R_{\widehat L,\infty}[Y_1,\dots,Y_d]$ is determined as in Lemma \ref{Lem-OC}. Then we obtain a $\Gamma(L/K)$-equivariant inclusion of Higgs fields 
   $(R_{\widehat L}[Y_1,\dots,Y_d],\sum_{i=1}^d-\frac{\partial}{\partial Y_i}\otimes\frac{\dlog T_i}{t})\subset (C_{L,\infty},\Theta)$
   via $\iota$.
 \end{notation}

 Then the main result in this subsection is 
 
 \begin{thm}\label{Thm-SimpsonLocal}
   Keep notations as above.
   
   \begin{enumerate}
       \item For any $M_{\infty}\in \Rep_{\Gamma(L/K)}(\widehat R_{\widehat L,\infty})$, let $H(M_{\infty}):= (M_{\infty}\otimes_{\widehat R_{\widehat L,\infty}}C_{L,\infty})^{\Gamma_{\geo}}$ and $\theta_{H(M_{\infty})}$ be the restriction of $\Theta_{M_{\infty}}:=\id_{M_{\infty}}\otimes\Theta$ (cf. (\ref{Equ-Theta_L})) to $H(M_{\infty})$. Then $(H(M_{\infty}),\theta_{H(M_{\infty})})$ defines a $\Gal(L/K)$-Higgs module over $R_{\widehat L}$. Moreover, for any $i\geq 1$, we have 
       \[\rH^i(\Gamma_{\geo},M_{\infty}\otimes_{\widehat R_{\widehat L,\infty}}C_{L,\infty}) = 0.\]
       
       \item For any $(H,\theta_H)\in\HIG_{\Gal(L/K)}(R_{\widehat L})$, let $\Theta_H:=\theta_H\otimes\id_{C_{L,\infty}}+\id_H\otimes\Theta$ (cf. (\ref{Equ-Theta_H})). Then $M_{\infty}(H,\theta_{H}):=(H\otimes_{R_{\widehat L}}C_{L,\infty})^{\Theta_H=0}$ is an $\widehat R_{\widehat L,\infty}$-representation of $\Gamma(L/K)$.
       
       \item The functor $M_{\infty}\mapsto (H(M_{\infty}),\theta_{H(M_{\infty})})$ induces an equivalence of categories
       \[\Rep_{\Gamma(L/K)}(\widehat R_{\widehat L,\infty})\to\HIG_{\Gal(L/K)}(R_{\widehat L}),\]
       whose quasi-inverse is given by the functor $(H,\theta_H)\mapsto M_{\infty}(H,\theta_H)$. The equivalence preserves ranks, tensor products and dualities.
       
       \item For any $M_{\infty}\in \Rep_{\Gamma(L/K)}(\widehat R_{\widehat L,\infty})$ with the associated $\Gal(L/K)$-Higgs module $(H,\theta_H)\in\HIG_{\Gal(L/K)}(R_{\widehat L})$. Then there exists a $\Gamma(L/K)$-equivariant isomorphism of Higgs fields
       \[(H\otimes_{R_{\widehat L}}C_{L,\infty},\Theta_H)\cong (M_{\infty}\otimes_{\widehat R_{\widehat L,\infty}}C_{L,\infty},\Theta_{M_{\infty}})\]
       which induces an $\Gal(L/K)$-equivariant quasi-isomorphism 
       \[\rR\Gamma(\Gamma_{\geo},M_{\infty}) \cong \HIG(H,\theta_H),\]
       where $\HIG(H,\theta_H)$ denotes the Higgs complex induced by $(H,\theta_H)$.
       As a consequence, we have a quasi-isomorphism 
       \[\rR\Gamma(\Gamma(L/K),M_{\infty}) \cong \rR\Gamma(\Gal(L/K),\HIG(H,\theta_H)).\]
       
       \item All results above are compatible with standard \'etale localisation of $X$.
   \end{enumerate}
 \end{thm}
 \begin{proof}
   (1) Let $M$ be the representation of $\Gamma(L/K)$ over $R_{\widehat L}$ corresponding to $M_{\infty}$ in the sense of Proposition \ref{Prop-UnipotentEquiv} (and Remark \ref{Rmk-UnipotentEquiv}). In other words, it satisfies the following conditions:
   
      (a) $M_{\infty} \cong M\otimes_{R_{\widehat L}}\widehat R_{\widehat L,\infty}$ as representations of $\Gamma(L/K)$ over $\widehat R_{\widehat L,\infty}$.
       
      (b) $\Gamma_{\geo}$ acts on $M$ unipotently and induces quasi-isomorphisms
      \[\rH^i(\Gamma_{\geo},M)\cong \rH^i(\Gamma_{\geo},M_{\infty}).\]
   
   In particular, we have 
   \[M\otimes_{R_{\widehat L}}C_{L,\infty}\cong M_{\infty}\otimes_{\widehat R_{\widehat L,\infty}}C_{L,\infty}.\]
   Using this to replace \cite[Prop. 2.8]{LZ}, we may conclude by the argument in the paragraph below \cite[Prop. 2.10]{LZ}. But for the further use, we provide more details here.
   
   Consider the $\Gamma(L/K)$-equivariant inclusion 
   \[R_{\widehat L}[Y_1,\dots,Y_d]\subset \widehat R_{\widehat L,\infty}[Y_1,\dots,Y_d] = C_{L,\infty}.\]
   Then by condition (b) and the argument in the paragraph after \cite[Lem. 6.17]{Sch-Pi}, for any $i\geq 0$, we get an isomorphism
   \[\rH^i(\Gamma_{\geo},M\otimes_{R_{\widehat L}}R_{\widehat L}[Y_1,\dots,Y_d])\cong \rH^i(\Gamma_{\geo},M_{\infty}\otimes_{\widehat R_{L,\infty}} C_{L,\infty})).\]
   As $\Gamma_{\geo}\cong \Zp\gamma_1\oplus\cdots\oplus\Zp\gamma_d$ acts on $M$ unipotently, we see that $(\gamma_d-1)$ is a nilpotent $R_{\widehat L}[Y_1,\dots,Y_{d-1}]$-linear morphism on $M\otimes_{R_{\widehat L}}R_{\widehat L}[Y_1,\dots,Y_{d-1}]$. As $\gamma_d(Y_d) = Y_d+1$, we deduce from \cite[Lem. 2.10]{LZ} that 
   \[\rH^0(\Zp\gamma_d,M\otimes_{R_{\widehat L}}R_{\widehat L}[Y_1,\dots,Y_d]) \cong M\otimes_{R_{\widehat L}}R_{\widehat L}[Y_1,\dots,Y_{d-1}]\]
   and that
   \[\rH^1(\Zp\gamma_d,M\otimes_{R_{\widehat L}}R_{\widehat L}[Y_1,\dots,Y_d])=0.\]
   In particular, by Hochschild--Serre spectral sequence, we know that for any $i\geq 0$, 
   \[\rH^i(\Gamma_{\geo},M\otimes_{R_{\widehat L}}R_{\widehat L}[Y_1,\dots,Y_d])\cong \rH^i(\Zp\gamma_1\oplus\cdots\oplus\Zp\gamma_{d-1},M\otimes_{R_{\widehat L}}R_{\widehat L}[Y_1,\dots,Y_{d-1}]).\]
   By iteration, we finally conclude that 
   \[\rH^i(\Gamma_{\geo},M\otimes_{R_{\widehat L}}R_{\widehat L}[Y_1,\dots,Y_d])= 0\]
   for any $i\geq 1$, which settles the ``moreover'' part of (1). 
   
   Since $\gamma_i^{-1}-1$ acts on $M$ nilpotently, for any $x\in M$, 
   \[\prod_{i=1}^d\gamma_i^{-Y_i}x: = \sum_{n_1,\dots,n_d\geq 0}(\prod_{i=1}^d(\gamma_i^{-1}-1)^{n_i}(x))\binom{Y_1}{n_1}\cdots\binom{Y_d}{n_d}\]
   is well-defined in $M\otimes_{R_{\widehat L}}R_{\widehat L}[Y_1,\dots,Y_d]$
   . By chasing the proof of \cite[Lem. 2.10]{LZ}, we see that
   \begin{equation}\label{Equ-HiggsModule-I}
       H(M_{\infty}) = \rH^0(\Gamma_{\geo},M\otimes_{R_{\widehat L}}R_{\widehat L}[Y_1,\dots,Y_d]) = \prod_{i=1}^d\gamma_i^{-Y_i}(M) := \{\prod_{i=0}^d\gamma_i^{-Y_i}x\mid x\in M\}
   \end{equation}
   is a finite projective $R_{\widehat L}$-module as desired. By Lemma \ref{Lem-OC}, for any $g\in \Gal(L/K)$ and any $1\leq i\leq d$, we have $g\gamma_ig^{-1} = \gamma_i^{\chi(g)}$ and $g(Y_i) = \chi(g)^{-1}Y_i$. In particular, for any $x\in M$, we have
   \begin{equation}\label{Equ-Product of gamma is Galois equivariant}
       \begin{split}
           g(\prod_{i=0}^d\gamma_i^{-Y_i}x) = & g(\sum_{n_1,\dots,n_d\geq 0}(\prod_{i=1}^d(\gamma_i^{-1}-1)^{n_i}(x))\binom{Y_1}{n_1}\cdots\binom{Y_d}{n_d})\\
           = & \sum_{n_1,\dots,n_d\geq 0}(\prod_{i=1}^d(g\gamma_i^{-1}g^{-1}-1)^{n_i}(g(x)))\binom{g(Y_1)}{n_1}\cdots\binom{g(Y_d)}{n_d}\\
           = & \prod_{i=0}^d(g\gamma_ig^{-1})^{-g(Y_i)}g(x)\\
           = & \prod_{i=0}^d(\gamma_i^{\chi(g)})^{-\chi(g)^{-1}Y_i}g(x)\\
           = & \prod_{i=0}^d\gamma_i^{-Y_i}g(x).
       \end{split}
   \end{equation}
   In other words, the $\prod_{i=1}^d\gamma_i^{-Y_i}$ is $\Gal(L/K)$-equivariant.
   So the $\Gal(L/K)$-action on $H(M_{\infty})$ is determined such that for any $\prod_{i=0}^d\gamma_i^{-Y_i}x\in H(M_{\infty})$,
   \begin{equation}\label{Equ-G-actian on H}
        g(\prod_{i=0}^d\gamma_i^{-Y_i}x) = \prod_{i=0}^d\gamma_i^{-Y_i}g(x)
   \end{equation}
   Using the expression of $\Theta$ in Lemma \ref{Lem-OC}, we deduce from (\ref{Equ-HiggsModule-I}) that 
   \begin{equation}\label{Equ-HiggsModule-II}
       \theta_{H(M_{\infty})} = \sum_{i=1}^{d}\log\gamma_i\otimes\frac{\dlog T_i}{t}.
   \end{equation}
   By $g\gamma_ig^{-1} = \gamma_i^{\chi(g)}$ for any $g\in \Gal(L/K)$ again, we have $g\log\gamma_ig^{-1} = \chi(g)\log\gamma_i$. Therefore, we obtain that $(H(M_{\infty}),\theta_{H(M_{\infty})})\in\HIG_{\Gal(L/K)}(R_{\widehat L})$.
   
   (2) We write $\theta_H = \sum_{i=1}^d\theta_i\otimes\frac{\dlog T_i}{t}$ with nilpotent $R_{\widehat L}$-linear endomorphisms $\theta_i$'s of $H$ which commute with each other such that for any $g\in \Gal(L/K)$, we have $g\theta_ig^{-1} = \chi(g)\theta_i$. Note that for any $x \in H\otimes_{R_{\widehat K_{\cyc}}}\widehat R_{\widehat K_{\cyc},\infty}[Y_1,\dots,Y_d]$, it is uniquely written as 
   \[x = \sum_{\underline n = (n_1,\dots,n_d)\in\bN^{d}}h_{\underline n}Y_1^{[n_1]}\cdots Y_d^{[n_d]}\]
   for finitely many $h_{\underline n}$'s in $H\otimes_{R_{\widehat K_{\cyc}}}\widehat R_{\widehat K_{\cyc},\infty}$, where $Y^{[n]}$ denotes the polynomial $\frac{Y^n}{n!}$ for all $n\geq 0$. Using the description of $\Theta$ in Lemma \ref{Lem-OC}, we see that 
   \begin{equation*}
       \begin{split}
           \Theta_H(x) & = \sum_{i=1}^d\sum_{\underline n = (n_1,\dots,n_d)\in\bN^{d}}(\theta_i(h_{\underline n})Y_1^{[n_1]}\cdots Y_d^{[n_d]}-h_{\underline n}Y_1^{[n_1]}\cdots Y_i^{[n_i-1]}\cdots Y_d^{[n_d]})\otimes\frac{\dlog T_i}{t}\\
           & = \sum_{i=1}^d\sum_{\underline n = (n_1,\dots,n_d)\in\bN^{d}}(\theta_i(h_{\underline n})-h_{\underline n+\underline 1_i})Y_1^{[n_1]}\cdots Y_d^{[n_d]}\otimes\frac{\dlog T_i}{t}.
       \end{split}
   \end{equation*}
   Therefore, we deduce that $\Theta_H(x) = 0$ if and only if for any $\underline n\in\bN^d$ and any $1\leq i\leq d$,
   \[h_{\underline n} = \theta_i(h_{\underline n-\underline 1_i}).\]
   By iteration, we obtain that for any $\underline n = (n_1,\dots,n_d) \in\bN^d$,
   \[h_{\underline n} = \theta_1^{n_1}\cdots\theta_d^{n_d}(h_{\underline 0}).\]
   Since $\theta_i$'s are nilpotent, we see that $h_{\underline n} = 0$ for $|\underline n|\gg 0$ and that 
   \[\prod_{i=1}^d\exp(\theta_iY_i)(h):=\sum_{\underline n = (n_1,\dots,n_d)\in\bN^d}\theta_1^{n_1}\cdots\theta_d^{n_d}(h)Y_1^{[n_1]}\cdots Y_d^{[n_d]}\]
   is well defined for any $h\in H\otimes_{R_{\widehat L}}\widehat R_{\widehat L,\infty}$. 
   So we obtain that 
   \begin{equation}\label{Equ-GaloisRep-I}
       M_{\infty}(H,\theta_H) = \prod_{i=1}^d\exp(\theta_iY_i)(H\otimes_{R_{\widehat L}}\widehat R_{\widehat L,\infty}) = \{\prod_{i=1}^d\exp(\theta_iY_i)(h)\mid h\in H\otimes_{R_{\widehat L}}\widehat R_{\widehat L,\infty}\},
   \end{equation}
   which is finite projective over $\widehat R_{\widehat L,\infty}$. By Lemma \ref{Lem-OC}, the $\Gamma_{\geo}$-action on $M$ is induced such that for any $1\leq i\leq d$, 
   \begin{equation}\label{Equ-GaloisRep-II}
       \gamma_i = \exp(\theta_i)
   \end{equation}
   Then for any $g\in \Gal(L/K)$, we have $g\gamma_ig^{-1} = \gamma_i^{\chi(g)}$ (as $g\theta_ig^{-1} = \chi(g)\theta_i$). In other words, $M_{\infty}(H,\theta_H)$ belongs to $\Rep_{\Gamma(L/K)}(\widehat R_{\widehat L,\infty})$.
   
   Similar to (\ref{Equ-Product of gamma is Galois equivariant}), we can check that $\prod_{i=1}^d\exp(\theta_iY_i)$ is $\Gal(L/K)$-equivariant; that is, for any $g\in \Gal(L/K)$ and any $h\in H\otimes_{R_{\widehat L}}\widehat R_{\widehat L,\infty}$, we have 
   \begin{equation}\label{Equ-Product of exponential is Galois equivariant}
       g(\prod_{i=1}^d\exp(\theta_iY_i)(h)) = \prod_{i=1}^d\exp(\theta_iY_i)(g(h)).
   \end{equation}
   
   (3) Fix an $M_{\infty}\in \Rep_{\Gamma(L/K)}(\widehat R_{\widehat L,\infty})$ and let $M$ be as in the proof of (1). Since $\Gamma_{\geo}$ acts on $M$ unipotently, we see that $\prod_{i=1}^d\gamma_i^{-Y_i}$ is a well-defined isomorphism of $M\otimes_{R_{\widehat L}}R_{\widehat L}[Y_1,\dots,Y_d]$. 
   Noting that $H(M_{\infty}) = \prod_{i=1}^d\gamma_i^{-Y_i}(M)$  cf. (\ref{Equ-HiggsModule-I}), we see that 
   \[M\otimes_{R_{\widehat L}}R_{\widehat L}[Y_1,\dots,Y_d] \cong H(M_{\infty})\otimes_{R_{\widehat L}}R_{\widehat L}[Y_1,\dots,Y_d]\]
   and hence that
   \[M_{\infty}\otimes_{\widehat R_{\widehat L,\infty}}\widehat R_{\widehat L,\infty}[Y_1,\dots,Y_d] \cong H(M_{\infty})\otimes_{R_{\widehat L}}\widehat R_{\widehat L,\infty}[Y_1,\dots,Y_d].\]
   
   We claim the two isomorphisms are compatible with both Higgs fields and $\Gamma(L/K)$-actions. We only prove the first one as the second is just the base-change of the first along $R_{\widehat L}\to\widehat R_{\widehat L,\infty}$. Granting this, we can obtain that $M_{\infty} = M_{\infty}(H(M_{\infty}),\theta_{H(M_{\infty})})$ as desired.

   For Higgs field: Recall the Higgs field on $M_{\infty}\otimes_{\widehat R_{\widehat L,\infty}}\widehat R_{\widehat L,\infty}[Y_1,\dots,Y_d]$ is given by 
   \[\Theta_M=\id_M\otimes(\sum_{i=1}^d-\frac{\partial}{\partial Y_i}\otimes\frac{\dlog T_i}{t})\] 
   while the Higgs field on $H(M_{\infty})\otimes_{R_{\widehat L}}\widehat R_{\widehat L,\infty}[Y_1,\dots,Y_d]$ is given by 
   \[\begin{split}
       \Theta_{H(M_{\infty})} = \theta_{H(M_{\infty})}\otimes\id+\id_{H(M_{\infty})}\otimes(\sum_{i=1}^d-\frac{\partial}{\partial Y_i}\otimes\frac{\dlog T_i}{t}).
   \end{split}\]
   So we have to show that for any $h\in H(M_{\infty})$ with $m\in M$, the $\Theta_M$ acts on $h$ via $\theta_{H(M_{\infty})}$. By (\ref{Equ-HiggsModule-I}), we can write $h = \prod_{i=1}^d\gamma_i^{-Y_i}(m)$ for some $m\in M$. So we have
   \[\begin{split}\Theta_M(h) = &\sum_{i=1}^d-\frac{\partial}{\partial Y_i}(\prod_{j=1}^d\gamma_j^{-Y_j})(m)\otimes\frac{\dlog T_i}{t}\\
   = & \sum_{i=1}^d\log\gamma_i((\prod_{j=1}^d\gamma_j^{-Y_j})(m))\otimes\frac{\dlog T_i}{t}=\sum_{i=1}^d\log\gamma_i(h)\otimes\frac{\dlog T_i}{t}.\end{split}\]
   The most right hand side is exactly $\theta_{H(M_{\infty})}(h)$ by (\ref{Equ-HiggsModule-II}).

   For $\Gamma(L/K)$-action: Recall the $\Gamma(L/K)$-action on $H(M_{\infty})\otimes_{R_{\widehat L}}R_{\widehat L}[Y_1,\dots,Y_d]$ is induced by that for any 
   \[h = \prod_{i=1}^d\gamma_i^{-Y_i}x\in H(M_{\infty})\text{ with }x\in M,\]
   the $\Gamma_{\geo}$ acts on $h$ trivially and that for any $g\in \Gal(L/K)$, by (\ref{Equ-G-actian on H}), $g(h) = \prod_{i=1}^d\gamma_i^{-Y_i}(g(x))$.
   So we have to show that if we regard $x\in M$ as an element in $H(M_{\infty})\otimes_{R_{\widehat L}}R_{\widehat L}[Y_1,\dots,Y_d]$, the restriction of the above $\Gamma(L/K)$-action on $x$ is compatible with the original one on $x$. 
   Recall that $\Gamma(L/K) = \Gamma_{\geo}\rtimes\Gal(L/K)$.
   As $\prod_{i=1}^d\gamma_i^{-Y_i}$ is $\Gal(L/K)$-equivariant (cf. Equation (\ref{Equ-Product of gamma is Galois equivariant})), the compatibility of $\Gal(L/K)$-actions on $x$ is clear. It remains to show the compatibility of $\Gamma_{\geo}$-actions on $x$. For any $1\leq i\leq d$, let $U_{i}\in \End_{R_{\widehat L}}(M)$ be the automorphism induced by the original $\gamma_{i}$-action on $M$. Then we have
   \[H(M_{\infty}) = \prod_{i=1}^dU_i^{-Y_i}M\]
   and in particular, we have $h = \prod_{i=1}^dU_i^{-Y_i} x$ and thus $x = \prod_{i=1}^dU_i^{Y_i} h$. Now, for any $1\leq i\leq d$, as an element in $H(M_{\infty})\otimes_{R_{\widehat L}}R_{\widehat L}[Y_1,\dots,Y_d]$, we have
   \[\gamma_i(x) = \gamma_i(\prod_{j=1}^dU_j^{Y_j} h) = \prod_{j=1}^dU_j^{\gamma_i(Y_j)} h = \prod_{j=1}^dU_j^{Y_j}(U_i h) = U_i\prod_{j=1}^dU_j^{Y_j} h = U_ix,\]
   which implies the compatibility of $\Gamma_{\geo}$-actions as desired.

   Similarly, fix an $(H,\theta_H)\in\HIG_{\Gal(L/K)}(R_{\widehat L})$ and write $\theta_H = \sum_{i=1}^d\theta_i\otimes\frac{\dlog T_i}{t}$ as in the proof of (2). Since $\theta_i$'s are nilpotent, we see that $\prod_{i=1}^d\exp(\theta_iY_i)$ is a well-defined isomorphism of $H\otimes_{R_{\widehat L}}\widehat R_{\widehat L,\infty}[Y_1,\dots,Y_d]$. By (\ref{Equ-GaloisRep-I}), we have
   \[M_{\infty}(H,\theta_H) = (\prod_{i=1}^d\exp(\theta_iY_i))(H\otimes_{R_{\widehat K_{\cyc}}}\widehat R_{\widehat K_{\cyc},\infty}),\]
   which gives rise to an isomorphism 
   \[H\otimes_{R_{\widehat L}}\widehat R_{\widehat L,\infty}[Y_1,\dots,Y_d] \cong M_{\infty}(H,\theta_H)\otimes_{\widehat R_{\widehat L,\infty}}\widehat R_{\widehat L,\infty}[Y_1,\dots,Y_d].\]
   Similarly, we can deduce from (\ref{Equ-GaloisRep-I}) and (\ref{Equ-GaloisRep-II}) that the above isomorphism is campatible with both Higgs fields and $\Gamma(L/K)$-actions. So we obtain that $(H(M_{\infty}(H,\theta_H)),\theta_{H(M_{\infty}(H,\theta_H))}) = (H,\theta_H)$.
   
   Therefore, we see functors defined in (1) and (2) induce a rank-preserving equivalence between $\Rep_{\Gamma(L/K)}(\widehat R_{\widehat L,\infty})$ and $\HIG_{\Gal(L/K)}(R_{\widehat L})$. By standard linear algebra, we see above constructions are compatible with tensor products and dualities, which completes the proof of (3).
   
   (4) The first part has been established in (3). To complete the proof, it suffices to show that 
   \[\rR\Gamma(\Gamma_{\geo},M_{\infty}) \simeq \HIG(H,\theta_H).\]
   Since Higgs complex $\HIG(C_{L,\infty},\Theta)$ induced by $(C_{L,\infty},\Theta)$ is a resolution of $\widehat R_{L,\infty}$, we get quasi-isomorphisms
   \[\rR\Gamma(\Gamma_{\geo},M_{\infty})\cong \rR\Gamma(\Gamma_{\geo},\HIG(M_{\infty}\otimes_{\widehat R_{\widehat L,\infty}}C_{L,\infty},\Theta_{M_{\infty}}))\cong \rR\Gamma(\Gamma_{\geo},\HIG(H\otimes_{R_{\widehat L}}C_{L,\infty},\Theta_{H})).\]
   By (1), for any $i\geq 1$ and any $j\geq 0$, we have 
   \[\rH^i(\Gamma_{\geo},H\otimes_{R_{\widehat L}}C_{L,\infty}\otimes_R\Omega_R^j(-j)) = 0.\]
   Using spectral sequence, we get a quasi-isomorphism
   \[\rR\Gamma(\Gamma_{\geo},\HIG(H\otimes_{R_{\widehat L}}C_{L,\infty},\Theta_{H})) \simeq \HIG(H,\theta_H),\]
   and hence that 
   \[\rR\Gamma(\Gamma_{\geo},M_{\infty})\simeq \HIG(H,\theta_H)\]
   as desired. This completes the proof of (4).
   
   (5) Since the equivalence in Proposition \ref{Prop-UnipotentEquiv} is compatible with standard \'etale localisation, this follows by noting all constructions above are compatible with standard \'etale localisation.
 \end{proof}

 \begin{rmk}\label{Rmk-SimpsonLocal}
    For any $M_{\infty}\in \Rep_{\Gamma(L/K)}(\widehat R_{\widehat L,\infty})$, let $M\in\Rep^{\rm uni}_{\Gamma(L/K)}(R_{\widehat L})$ be as in the proof of Theorem \ref{Thm-SimpsonLocal} (1). Then one can check that $(M,\theta_M:=\sum_{i=1}^d\log\gamma_i\otimes\frac{\dlog T_i}{t})$ with the restricted $\Gal(L/K)$-action gives rise to an object in $\HIG_{\Gal(L/K)}(R_{\widehat L})$. Let $(H(M_{\infty}),\theta_{H(M_{\infty})})$ be as in (\ref{Equ-HiggsModule-I}) and (\ref{Equ-HiggsModule-II}). It is easy to see that the map $\prod_{i=1}^d\gamma_i^{-Y_i}x\mapsto x$ gives rise to an isomorphism 
    \[(H(M_{\infty}),\theta_{H(M_{\infty})})\cong (M,\theta_M)\]
    of objects in $\HIG_{\Gal(L/K)}(R_{\widehat L})$. This phenomenon was discovered in \cite[Lem. 2.11]{LZ}.
 \end{rmk}
 
 \begin{cor}\label{Cor-SimpsonLocal}
   Keep notations as above.
   There exists an equivalence of categories 
   \[\Vect(X_{\proet},\OX)\simeq \HIG_{\Gal(L/K)}(X_{\widehat L})\]
   such that for any generalised representation $\calL$ on $X_{\proet}$ with associated $\Gal(L/K)$-Higgs bundle $(\calH,\theta_{\calH})$, there exists a quasi-isomorphism 
   \[\rR\Gamma(X_{\proet},\calL)\cong \rR\Gamma(\Gal(L/K),\rR\Gamma(X_{\widehat L,\et},\HIG(\calH,\theta_{\calH}))).\]
 \end{cor}
 \begin{proof}
   The desired equivalence is induced by the composition
   \[\Vect(X_{\proet},\OX)\to\Rep_{\Gamma(L/K)}(\widehat R_{\widehat L,\infty}) \to \HIG_{\Gal(L/K)}(R_{\widehat L})\xleftarrow{\Gamma(X_{\widehat L},-)}\HIG_{\Gal(L/K)}(X_{\widehat L,\et}),\]
   where the first two arrows are induced by Lemma \ref{Lem-EvaluateGRep} and Theorem \ref{Thm-SimpsonLocal} (3) while the last arrow is induced by taking global sections $\Gamma(X_{\widehat L},-)$, which is an equivalence as $X$ is affinoid. Finally, since $X$ is affinoid, we get 
   \[\rR\Gamma(X_{\widehat L,\et},\HIG(\calH,\theta_{\calH})) \simeq \HIG(H,\theta_H),\]
   where $(H,\theta_H)$ is the $\Gal(L/K)$-Higgs module induced by $(\calH,\theta_{\calH})$. So the desired quasi-isomorphism follows from Lemma \ref{Lem-EvaluateGRep} combined with Theorem \ref{Thm-SimpsonLocal} (4).
 \end{proof}
 
 \begin{rmk}
   We will see in the next subsection (cf. Corollary \ref{Cor-IndependentofChart-Simpson}) that the functor 
   \[\Vect(X_{\proet},\OX)\to \HIG_{\Gal(L/K)}(X_{\widehat L})\]
   is actually induced by $\calL\mapsto (\nu_*(\calL\otimes_{\OX}\OC),\nu_*(\Theta_{\calL}))$ as stated in Theorem \ref{Thm-Simpson}. So it does {\bf NOT} depend on the choice of toric chart on $X$. So one can glue the local constructions together to get a global equivalence.
 \end{rmk}

\subsection{Proof of Theorem \ref{Thm-Simpson}}
 In this subsection, we focus on the proof of Theorem \ref{Thm-Simpson}.
 
 We first show that $\rR^i\nu_*(\calL\otimes_{\OX}\OC)$ is a vector bundle on $X_{C,\et}$ for $i = 0$ and vanishes for $i\geq 1$. Note that it is the sheafification of the presheaf 
 \[(Y\in X_{C,\et})\mapsto \rH^i(X_{\proet}/Y,\calL\otimes_{\OX}\OC).\]
 By \cite[Cor. 2.6]{LZ}, it suffices to show that for any affinoid $Y = \Spa(R,R^+)\in X_{\et}$ admitting a toric chart, $\rH^0(X_{\proet}/Y_C,\calL\otimes_{\OX}\OC)$ is a finite projective $R_{\widehat L}$-module, and that for any standard \'etale localisation $Z = \Spa(S,S^+)\to Y$,
 \begin{equation}\label{Equ-CompatiblewithBaseChange-I}
     \rH^0(X_{\proet}/Z_C,\calL\otimes_{\OX}\OC) = \rH^0(X_{\proet}/Y_C,\calL\otimes_{\OX}\OC)\otimes_{R_C}S_C
 \end{equation}
 and for any $i\geq 1$,
 \begin{equation}\label{Equ-CompatiblewithBaseChange-II}
     \rH^i(X_{\proet}/Z_C,\calL\otimes_{\OX}\OC) = 0.
 \end{equation}
 We proceed as in \cite[\S 2.3]{LZ}.
 
 \begin{lem}\label{Lem-PerfectoidVanishing}
   For any affinoid perfectoid $U\in X_{\proet}$ and for any $i\geq 1$, we have
   \[\rH^i(X_{\proet}/U,\calL\otimes_{\OX}\OC) = 0.\]
 \end{lem}
 \begin{proof}
   Since $U$ is affinoid, the pro-\'etale site $U_{\proet}$ is coherent. In particular, taking cohomology commutes with taking direct limits by \cite[Expos\'e {\rm VI}, Theorem 5.1]{SGA4}. Since $\OC = \varinjlim_n\Sym^n_{\OX}\calE$, the result follows from \cite[Prop. 2.3]{LZ}.
 \end{proof}
 \begin{lem}\label{Lem-PasstoGroupCoho}
   For any $i\geq 0$, there exists a natural isomorphism 
   \[\rH^i(\Gamma_{\geo},(\calL\otimes_{\OX}\OC)(Y_{C,\infty}))\simeq \rH^i(X_{\proet}/Y_C,\calL\otimes_{\OX}\OC),\]
   where $Y_{C,\infty}$ is the analogue of $X_{C,\infty}$ for $Y$ replacing $X$ (cf. Notation \ref{Notation-LocalChart}).
 \end{lem}
 \begin{proof}
   Lemma \ref{Lem-PerfectoidVanishing} together with \v Cech-to-derived spectral sequence implies that the pro-\'etale cohomology $\rR\Gamma(X_{\proet}/Y_C,\calL\otimes_{\OX}\OC)$ can be computed by \v Cech--Alexander complex $(\calL\otimes_{\OX}\OC)(Y_{C,\infty}^{\bullet/Y_C})$, where $Y_{C,\infty}^{n/Y_C}$ denotes the self product of $n+1$ copies of $Y_{C,\infty}$ over $Y_C$. By a similar argument in the proof of Lemma \ref{Lem-EvaluateGRep}, we see that 
   \[(\calL\otimes_{\OX}\OC)(Y_{C,\infty}^{\bullet/Y_C})\cong C(\Gamma_{\geo}^{\bullet},(\calL\otimes_{\OX}\OC)(Y_{C,\infty})).\]
   Then the result follows.
 \end{proof}
 Now, by Theorem \ref{Thm-SimpsonLocal} (1) and (5), we conclude that $\rH^0(X_{\proet}/Y_C,\calL\otimes_{\OX}\OC)$ is a finite projective $R_C$-module such that (\ref{Equ-CompatiblewithBaseChange-I}) and (\ref{Equ-CompatiblewithBaseChange-II}) hold true. This shows that the functor
 \[\Vect(X,\OX)\to\HIG_{G_K}(X_C)\]
 sending each generalised representation $\calL$ to \[(\calH(\calL),\theta_{\calH(\calL)}):=(\nu_*(\calL\otimes_{\OX}\OC),\nu_*(\Theta_{\calL}))\]
 is a well-defined rank-preserving functor such that Theorem \ref{Thm-Simpson} (1) is true. Moreover, we also deduce that
 \begin{cor}\label{Cor-IndependentofChart-Simpson}
   Assume $X$ admits a toric chart. Then the functor $\Vect(X_{\proet},\OX)\to\HIG_{G_K}(X_C)$ introduced in Corollary \ref{Cor-SimpsonLocal} is independent of the choice of toric chart.
 \end{cor}
 \begin{proof}
   By Lemma \ref{Lem-PasstoGroupCoho}, we see that $\calH(\calL)(X_{C,\infty}) = \rH^0(\Gamma_{\geo},(\calL\otimes_{\OX}\OC)(X_{C,\infty}))$. Then the result follows from the construction of the functor in Corollary \ref{Cor-SimpsonLocal}.
 \end{proof}
 
 Thanks to above corollary, we can show that the functor $\calL\mapsto(\calH(\calL),\theta_{\calH(\calL)})$ is indeed an equivalence. For this purpose, we construct its quasi-inverse as follows:
 
 Let $\{X_i\to X\}_{i\in I}$ be an \'etale covering such that each $X_i\to X$ is standard \'etale and each $X_i$ admits a toric chart.
 For any $(\calH,\theta_{\calH})\in\HIG_{G_K}(X_{C,\et})$, let $(\calH_i,\theta_{\calH_i})$ be its restriction to $X_i$. Then we get canonical comparison isomorphisms 
 \[\iota_{ij}:(\calH_i,\theta_{\calH_i})_{\mid X_i\times_XX_j}\xrightarrow{\cong} (\calH_j,\theta_{\calH_j})_{\mid X_j\times_XX_j}.\]
 Let $\calL_i\in\Vect(X_{i,\proet},\widehat\calO_{X_i})$ be the generalised representation on $X_{i,\proet}$ corresponding to $(\calH_i,\theta_{\calH_i})$ in the sense of Corollary \ref{Cor-SimpsonLocal}. By Corollary \ref{Cor-IndependentofChart-Simpson}, $\iota_{ij}$'s induce canonical isomorphisms of 
 \[\calL_{i\mid_{X_i\times_XX_j}} \xrightarrow{\cong}\calL_{j\mid_{X_i\times_XX_j}}.\]
 So $\calL_i$'s glue and achieve an $\calL(\calH,\theta_{\calH})\in \Vect(X_{\proet},\OX)$. It is easy to check that the functor $(\calH,\theta_{\calH})\mapsto\calL(\calH,\theta_{\calH})$ is the desired quasi-inverse. By Theorem \ref{Thm-SimpsonLocal} (3), the above constructions preserve tensor products and dualities.
 
 To complete the proof of Theorem \ref{Thm-Simpson}, it remains to verify items (2), (3) and (4). 
 
 For (2), let $\calL\in\Vect(X_{\proet},\OX)$ with associated $(\calH,\theta_{\calH})\in\HIG_{G_K}(X_C)$. By adjunction formula, we get a morphism
 \[(\calH\otimes_{\calO_{X_C}}\OC_{\mid X_C},\Theta_{\calH})\to (\calL\otimes_{\OX}\OC)_{\mid X_C},\Theta_{\calL})\]
 and have to check this is an isomorphism, which reduces to Theorem \ref{Thm-SimpsonLocal} (3) as the problem is local.
 
 For (3), let $\calL\in\Vect(X_{\proet},\OX)$ and $(\calH,\theta_{\calH})\in\HIG_{G_K}(X_C)$ be as above. By noting that the universal Higgs field $(\OC,\Theta)$ induces a quasi-isomorphism $\OX\xrightarrow{\simeq} \HIG(\OC,\Theta)$, we get a quasi-isomorphism
 \[\calL\simeq \HIG(\calL\otimes_{\OX}\OC,\Theta_{\calL})\]
 and hence quasi-isomorphisms
 \begin{equation*}
     \begin{split}
         \rR\Gamma(X_{C,\proet},\calL)&\simeq \rR\Gamma(X_{C,\proet},\HIG(\calL\otimes_{\OX}\OC,\Theta_{\calL}))\\
         &\simeq\rR\Gamma(X_{C,\et},\rR\nu_*(\HIG(\calL\otimes_{\OX}\OC,\Theta_{\calL})))\\
         &\simeq\rR\Gamma(X_{C,\et},\nu_*(\HIG(\calL\otimes_{\OX}\OC,\Theta_{\calL})))\quad \text{by (1)}\\
         & \simeq\rR\Gamma(X_{C,\et},\nu_*(\HIG(\calH\otimes_{\calO_X}\OC,\Theta_{\calH})))\quad \text{by (2)}\\
         & \simeq \rR\Gamma(X_{C,\et},\HIG(\calH,\theta_{\calH})),
     \end{split}
 \end{equation*}
 where the last quasi-isomorphism is deduced by noting that $(\nu_*(\OC),\nu_*(\Theta)) \cong (\calO_X,0)$.
 
 For (4), let $\calL\in\Vect(X_{\proet},\OX)$ and $(\calH,\theta_{\calH})\in\HIG_{G_K}(X_C)$ be as above. 
 Since there exists a natural morphism $(f^*\OC_X,f^*\Theta_X)\to (\OC_{X'},\Theta_{X'})$ of Higgs fields, we get isomorphisms
 \begin{equation*}
     \begin{split}
         & (f^*\calH\otimes_{\calO_{X'_C}}\OC_{\mid X'_C},\Theta_{f^*\calH})\\
         \cong & (f^*\calH\otimes_{\calO_{X'_C}}\OC_{\mid X'_C},f^*\Theta_{\calH})\quad (\text{as~}f^*\Theta_{\calH} = \Theta_{f^*\calH}) \\
         \to & (f^*\calL\otimes_{\widehat \calO_{X'}}\OC_{\mid X'_C},f^*\Theta_{\calL})\quad (\text{by (2) via base-change along}~f^*\OC_X\to\OC_{X'})\\
         \cong & (f^*\calL\otimes_{\widehat \calO_{X'}}\OC_{\mid X'_C},\Theta_{f^*\calL})\quad (\text{as}~ f^*\Theta_{\calL} = \Theta_{f^*\calL}).
     \end{split}
 \end{equation*}
 By taking kernels of Higgs fields, we see that 
 \[f^*\calL \cong \calL(f^*\calH,f^*\theta_{\calH}).\]
 By applying $\nu_{X',*}: X'_{\proet}\to X'_{C,\et}$ to $(f^*\calH\otimes_{\calO_{X_C'}}\OC_{\mid X_C'},\Theta_{f^*\calH}) \cong (f^*\calL\otimes_{\widehat \calO_{X'}}\OC_{\mid X_C'},\Theta_{f^*\calL})$, we get 
 \[(\calH(f^*\calL),\theta_{\calH(f^*\calL)})\cong (f^*\calH,f^*\theta_{\calH}).\]
 Then Item (4) follows.
 
  Therefore, we complete the proof of Theorem \ref{Thm-Simpson}.

\section{Local description of Hodge--Tate crystals}\label{Sec-LocalConstruction}
  
  In this section, we study the category of (rational) Hodge--Tate crystals on $(\frakX)_{\Prism}$ for $\frakX = \Spf(R^+)$ small affine with a fixed chart 
  \[\Box: =\calO_K\za T_1^{\pm 1},\dots,T_d^{\pm 1}\ya \to R^+\]
  as defined in \S \ref{Intro-Notations}. Let $X = \Spa(R,R^+)$ be the rigid generic fibre of $\frakX$. For simplicity, we denote the prismatic site $(\frakX)_{\Prism}$ by $(R^+)_{\Prism}$, denote by $(R^+)_{\Prism}^{\perf}$ the sub-site of all perfect prisms, and denote by $(R^+/(\frakS,(E)))_{\Prism}$ the relative prismatic site with the base prism $(\frakS,(E))$. Let $R^+_C$ be the base-change of $R^+$ along $\calO_K\to\calO_C$.
  
  To state our main result, we make the following definition:
  \begin{dfn}\label{Dfn-EnhancedHiggsMod-Local}
  By an \emph{enhanced Higgs module} of rank $l$ over $R^+$, we mean a triple $(H,\theta_H,\phi_H)$ such that
  \begin{enumerate}
      \item $H$ is a finite projective module over $R^+$ of rank $l$ and $\theta_H:H\to H\otimes_{R^+}\widehat \Omega^1_{R^+/\calO_K}\{-1\}$ defines a nilpotent Higgs field on $H$. We denote by $\HIG(H,\theta_H)$ the induced Higgs complex.
      
      \item $\phi_H\in\End_{R^+}(H)$ such that 
      \begin{enumerate}
          \item $\lim_{n\to+\infty}\prod_{i=0}^{n-1}(\phi_H+iE'(\pi)) = 0$ with respect to the $p$-adic topology on $H$, and that
          
          \item $[\phi_H,\theta_H] = -E'(\pi)\theta_H$; that is, $\phi_H$ induces an endomorphism of $\HIG(H,\theta_H)$ as follows:
          \begin{equation}\label{Diag-EnhancedHiggsComplex-Local}
              \xymatrix@C=0.45cm{
                H\ar[r]^{\theta_H\qquad\quad}\ar[d]_{\phi_H}&H\otimes_{R^+}\widehat \Omega^1_{R^+/\calO_K}\{-1\}\ar[d]_{\phi_H+E'(\pi)\id_H} \ar[r]&\cdots\ar[r]&H\otimes_{R^+}\widehat \Omega^d_{R^+/\calO_K}\{-d\}\ar[d]_{\phi_H+dE'(\pi)\id_H}\\
                H\ar[r]^{\theta_H\qquad\quad}&H\otimes_{R^+}\widehat \Omega^1_{R^+/\calO_K}\{-1\} \ar[r]&\cdots\ar[r]&H\otimes_{R^+}\widehat \Omega^d_{R^+/\calO_K}\{-d\}.
              }
          \end{equation}
      \end{enumerate}
      We denote by $\HIG(H,\theta_H,\phi_H)$ the total complex of (\ref{Diag-EnhancedHiggsComplex-Local}). In other words, $\HIG(H,\theta_H,\phi_H)$ is the fibre of $\HIG(H,\theta_H)$:
      \[\HIG(H,\theta_H,\phi_H) = \fib(\HIG(H,\theta_H)\xrightarrow{\phi_H}\HIG(H,\theta_H).\]
  \end{enumerate}
   We denote by $\HIG^{\nil}_*(R^+)$ the category of enhanced Higgs modules over $R^+$. One can similarly define the category $\HIG^{\nil}_*(R)$ of enhanced Higgs modules over $R$ by replacing $R^+$ and $\widehat \Omega^i_{R^+/\calO_K}$ by $R$ and $\Omega^i_{R/K}$ respectively in the above sentences.
 \end{dfn}
 \begin{rmk}\label{Rmk-EnhancedHiggsMod-Local}
   In Definition \ref{Dfn-EnhancedHiggsMod-Local}, the nilpotency condition on $\theta_H$ can be deduced from the other assumptions. In fact, if we write $\theta_H = \sum_{i=1}^d\theta_i\otimes\frac{\dlog T_i}{E(u)}$ with $\theta_i$'s belonging to $\End_{R^+}(H)$, then the condition $[\phi_H,\theta_H] = -E'(\pi)\theta_H$ amounts to that for any $1\leq i\leq d$, we have $[\phi_H,\theta_i] = -E'(\pi)\theta_i$. So the Lie sub-algebra $\frakg$ of $\End_{R^+}(H)$ generated by $\phi_H$ and $\theta_i$'s is solvable with $[\frakg,\frakg]$ generated by $\theta_i$'s. Then the desired nilpotency follows from standard Lie theory (cf. \cite[\S 3.3 Cor. A]{Hum}) via embedding $R^+$ into $\Frac(R^+)$, the fractional field of $R^+$.
 \end{rmk}
 In this subsection, we want to prove the following result.
 \begin{thm}\label{Thm-HTasHiggs-Local}
   The evaluation at the Breuil--Kisin prism $(\frakS(R^+),(E))$ attached to the fixed framing $\square$ induces an equivalence of categories 
   \[\rho_{\square}:\Vect((R^+)_{\Prism},\overline \calO_{\Prism})\to\HIG_*^{\nil}(R^+)~(\text{resp.}~\rho_{\square}:\Vect((R^+)_{\Prism},\overline \calO_{\Prism}[\frac{1}{p}])\to\HIG_*^{\nil}(R)),\]
   which preserves ranks, tensor products and dualities. Moreover for any Hodge--Tate crystal (resp. rational Hodge--Tate crystal) $\bM$ with associated enhanced Higgs module over $R^+$ (resp. $R$), there exists a quasi-isomorphism
   \[\rR\Gamma((R^+)_{\Prism},\bM)\cong \HIG(H,\theta_H,\phi_H),\]
   which is functorial in $\bM$.
 \end{thm}
 \begin{proof}
   This follows from Proposition \ref{Prop-HTvsHiggs-Local} and Proposition \ref{Prop-HTvsHiggsCoho-Local} proved below.
 \end{proof}
 \begin{rmk}
   Since $(\frakS(R^+),(E))$ depends on the framing on $R^+$, the equivalence
   \[\Vect((R^+)_{\Prism},\overline \calO_{\Prism})\to\HIG_*^{\nil}(R^+)\]
   also depends on the given framing. However, we will see in Corollary \ref{Cor-IndenpendenceChart} that after inverting $p$, the equivalence $\Vect((R^+)_{\Prism},\overline \calO_{\Prism}[\frac{1}{p}])\to\HIG_*^{\nil}(R)$ is indeed independent of the choice of framings by using the $p$-adic Simpson correspondence studied in the previous section. So there is a global theory on the rational level (cf. Theorem \ref{Thm-HTasHiggs-Global}). When $\calO_K = \rW(\kappa)$ is unramified, Theorem \ref{Thm-HTasHiggs-Local} was also obtained by Bhatt--Lurie in \cite{BL22b} in a stacky way by regarding the Hodge--Tate substack of the prismatization of $\frakX$ as the classifying space of some group scheme. However, up to now, it is still a problem to achieve a global theory as their method needs a global ``Frobenius endomorphism'' of $\frakX$. See \cite[\S 9]{BL22b} for details. 
 \end{rmk}
 \begin{rmk}
     Note that there exists an obvious functor 
     \[\Vect((R^+)_{\Prism},\overline \calO_{\Prism})\to \Vect((R^+/(\frakS,(E)))_{\Prism},\overline \calO_{\Prism})\]
     induced by the forgetful functor $(R^+/(\frakS,(E)))_{\Prism}\to (R^+)_{\Prism}$. By abuse of notations, for each $\bM\in \Vect((R^+)_{\Prism},\overline \calO_{\Prism})$ (with associated enhanced Higgs module $(H,\theta_H,\phi_H)$), we also denote by $\bM$ its image under the above functor. Using Proposition \ref{Prop-Structure} below, it is easy to see that the induced topologically nilpotent Higgs module associated to $\bM\in \Vect((R^+/(\frakS,(E)))_{\Prism},\overline \calO_{\Prism})$ via the equivalences in \cite[Thm. 5.12 and Thm. 5.14]{Tia23} is exactly $(H,\theta_H)$ and there exists a quasi-isomoprhism
   \[\rR\Gamma((R^+/(\frakS,(E)))_{\Prism},\bM) = \HIG(H,\theta_H).\]
 \end{rmk}
 \begin{rmk}
      The equivalence $\Vect((R^+)_{\Prism},\overline \calO_{\Prism})\to\HIG_*^{\nil}(R^+)$ in Theorem \ref{Thm-HTasHiggs-Local} still holds true in a more general setting: Indeed, one can define a Hodge--Tate crystal of quasi-coherent modules as a sheaf $\bM$ of $\overline \calO_{\Prism}$-modules on $(R^+)_{\Prism}$ with $p$-complete evaluations satisfying Convention \ref{Convention-Crystal} (2) for $p$-complete tensor products, and define quasi-coherent enhanced Higgs modules as $p$-complete $R$-modules $H$ with Higgs fields $\theta_H$ and $R^+$-linear endomorphism $\phi_H$ satisfying Definition \ref{Dfn-EnhancedHiggsMod-Local} (2). Then the same proof of Theorem \ref{Thm-HTasHiggs-Local} also yields an equivalence between the category of Hodge--Tate crystals of quasi-coherent modules and the category of quasi-coherent enhanced Higgs modules by using ($p$-completely) faithfully flat descent for general ($p$-complete) modules (rather than finite projective ones). See \cite[Section 3]{GMW-HT} for related discussions.
 \end{rmk}
 We will prove Theorem \ref{Thm-HTasHiggs-Local} by establishing the desired equivalence and then comparing cohomologies separately in sequels.
\subsection{Preliminaries}
 \begin{construction}\label{Construction-CoverPrism}
   Let $\frakS\za\underline T^{\pm 1}\ya$ be the $(p,E)$-adic completion of $\frakS[\underline T^{\pm 1}]$ which is endowed with a $\delta$-structure over $\frakS$ such that $\delta(T_i) = 0$ for all $i$. Then $(\frakS\za\underline T^{\pm 1}\ya,(E))$ is a prism. By $p$-complete \'etaleness of $\Box$, it induces a smooth lifting $\frakS(R^+)$ of $R^+$ over $\frakS$ and a $(p,E)$-adically \'etale morphism $\frakS\za\underline T^{\pm 1}\ya \to \frakS(R^+)$. By \cite[Lem. 2.18]{BS22}, there exists a unique $\delta$-structure on $\frakS(R^+)$ which is compatible with the one on $\frakS\za\underline T^{\pm 1}\ya$ and makes $(\frakS(R^+),(E))$ a prism in $(R^+)_{\Prism}$. One can similarly define $(\Ainf\za\underline T^{\pm 1}\ya,(\xi))$ and $(\Ainf(R^+),(\xi))$ with $\delta$-structures compatible with those on $(\frakS\za \underline T^{\pm 1}\ya,(E))$ and $(\frakS(R^+),(E))$ via the morphism $\frakS\to\Ainf$ such that $\Ainf(R^+)$ is a lifting of $R^+_C$ over $\Ainf$. Let
   \[\widehat R_{C,\infty}^+: = R^+_C\widehat \otimes_{\calO_C\za T_1^{\pm 1},\dots,T_d^{\pm 1}\ya}\calO_C\za T_1^{\pm \frac{1}{p^{\infty}}},\dots,T_d^{\pm \frac{1}{p^{\infty}}}\ya.\]
   It is a perfectoid ring over $R$ with the induced perfect prism $(\Ainf(\widehat R_{C,\infty}^+),(\xi))$ in the sense of \cite[Thm. 3.10]{BS22}, which is the perfection of $(\Ainf(R^+),(\xi))$. We also adapt notations in \ref{Notation-LocalChart} and Notation \ref{Notation-LocalChart-II}.
 \end{construction}
 \begin{rmk}\label{Rmk-EtaleCompatibility}
   Note that for any \'etale localisation $\iota: R^+\to S^+$, there exists a unique morphism of prisms $(\frakS(R^+),(E))\to (\frakS(S^+),(E))$ lifting $\iota$, following from Lemma \cite[Lem. 2.18]{BS22}. Using this, we deduce from the constructions in this section that the equivalence and quasi-isomorphism in Theorem \ref{Thm-HTasHiggs-Local} are compatible with \'etale localisations. More precisely, let $\bM\in\Vect((R^+)_{\Prism},\overline \calO_{\Prism})$ be a Hodge--Tate crystal with corresponding enhanced Higgs module $(H,\theta_H,\phi_H)$. Then the enhanced Higgs module associated to the restriction of $\bM$ to $(S^+)_{\Prism}$ is exactly $(H\otimes_{R^+}S^+,\theta_H\otimes\id_{S^+},\phi_H\otimes\id_{S^+})$ and there exists a quasi-isomorphism 
   \[\rR\Gamma((S^+)_{\Prism},\bM)\cong \HIG(H,\theta_H,\phi_H)\otimes_{R^+}S^+.\]
 \end{rmk}
 \begin{lem}\label{Lem-CoverPrism}
   \begin{enumerate}
       \item Both $(\frakS(R^+),(E))$, $(\Ainf(R^+),(\xi))$ and $(\Ainf(\widehat R_{C,\infty}^+),(\xi))$ are covers of the final object of the topos $\Sh((R^+)_{\Prism})$. As a consequence, these prisms are also covers of the final object of $\Sh((R^+/(\frakS,(E)))_{\Prism})$.
       
       \item The prism $(\Ainf(\widehat R_{C,\infty}^+),(\xi))$ is a cover of the final object of the topos $\Sh((R^+)^{\perf}_{\Prism})$.
   \end{enumerate}
 \end{lem}
 \begin{proof}
   Since there exist morphisms of prisms 
   \[(\frakS(R^+),(E))\to(\Ainf(R^+),(\xi))\to(\Ainf(\widehat R_{C,\infty}^+),(\xi)),\] 
   it is enough to show $(\Ainf(\widehat R_{C,\infty}^+),(\xi))$ cover the final objects of both $\Sh((R^+)_{\Prism})$ and $\Sh((R^+)_{\Prism}^{\perf})$.
   
   For the first topos, let $(A,I)$ be any bounded prism in $(R^+)_{\Prism}$. Then $A/I\widehat \otimes_{R^+}\widehat R_{C,\infty}^+$ is a quasi-syntomic cover of $A/I$. By \cite[Prop. 7.11 (1)]{BS22}, there exists a cover $(B,IB)$ of $(A,I)$ such that $A/I\to B/IB$ factors through $A/I\to A/I\widehat \otimes_{R^+}\widehat R_{C,\infty}^+$. In particular, $B/I$ is an $\widehat R_{C,\infty}^+$-algebra. Since $\widehat R_{C,\infty}^+$ is perfectoid, by deformation theory, we get a morphism $(\Ainf(\widehat R_{C,\infty}^+),(\xi))\to(B,IB)$ as desired. 
   
   For the second topos, we conclude from a similar argument as above by using \cite[Prop. 7.11 (2)]{BS22} instead of \cite[Prop. 7.11 (1)]{BS22}.
 \end{proof}
 
 \begin{notation}\label{Notation-CechNerve}
  Let $(\frakS(R^+)^{\bullet},(E))$ (resp. $(\frakS(R^+)^{\bullet}_{\geo},(E))$, resp. $(\Ainf(\widehat R_{C,\infty}^+)^{\bullet},(E))$) be the cosimplicial prisms induced by the self-products of $(\frakS(R^+),(E))$ (resp. $\frakS(R^+)$, resp. $\Ainf(\widehat R_{C,\infty}^+)$) in $(R^+)_{\Prism}$ (resp. $(R^+/(\frakS,(E)))_{\Prism}$, resp. $(R^+)_{\Prism}^{\perf}$). Here, $E$ denotes $E(u_0)$; that is, the corresponding $E(u)$ of the first component in each degree of the cosimplicial prisms.
 \end{notation}
 \begin{lem}\label{Lem-Structure}
   We have isomorphisms of cosimplicial rings (with obvious degeneracy morphisms $p_i$ and face morphisms $\sigma_i$):
   \begin{enumerate}
       \item $\frakS(R^+)^{\bullet} = \frakS(R^+)^{\widehat \otimes(\bullet+1)}\{\frac{u_0-u_i}{E(u_0)},\frac{T_{1,0}-T_{1,i}}{E(u_0)T_{1,0}},\dots,\frac{T_{d,0}-T_{d,i}}{E(u_0)T_{d,0}}\mid 1\leq i\leq \bullet\}^{\wedge}_{\delta}$, where for any $n\geq 0$, the $\frakS(R^+)^{\widehat \otimes(n+1)}$ denotes the $(p,E)$-complete tensor product of $n+1$ copies of $\frakS(R^+)$ over $\rW(\kappa)$ and for any $0\leq i\leq n$, $u_i, T_{1,i},\dots,T_{d,i}$ denote the corresponding $u,T_1,\dots,T_d$ of the $(i+1)$-factor.
       
       \item $\frakS(R^+)_{\geo}^{\bullet} = \frakS(R^+)^{\widehat \otimes_{\frakS}(\bullet+1)}\{\frac{T_{1,0}-T_{1,i}}{E(u_0)T_{1,0}},\dots,\frac{T_{d,0}-T_{d,i}}{E(u_0)T_{d,0}}\mid 1\leq i\leq \bullet\}^{\wedge}_{\delta}$, where for any $n\geq 0$, the $\frakS(R^+)^{\widehat \otimes_{\frakS}(n+1)}$ denotes the $(p,E)$-complete tensor product of $n+1$ copies of $\frakS(R^+)$ over $\frakS$ with $T_{1,i}$'s as above.
       
       \item $\overline \calO_{\Prism}[\frac{1}{p}](\Ainf(\widehat R_{C,\infty}^+)^{\bullet}) = \rC(\Gamma(\overline K/K)^{\bullet},\widehat R_{C,\infty})$, where for any $n\geq 0$, the $\rC(\Gamma(\overline K/K)^{\bullet},\widehat R_{C,\infty})$ denotes the ring of continuous functions from $\Gamma(\overline K/K)^n$ to $\widehat R_{C,\infty}^+$ (cf. \S \ref{Intro-Notations}).
   \end{enumerate}
 \end{lem}
 \begin{proof}
   (1) For any $n\geq 0$, the $(\frakS(R^+)^n,(E))$ is the initial object of the category of prisms in $(R)_{\Prism}$ which are the targets of $n+1$ arrows from $(\frakS(R^+),(E))$. Note that for any prism $(A,I)$ in this category, there exists a unique morphism $\frakS(R^+)^{\widehat \otimes(n+1)} \to A$. Note that the reductions of the images of $u_i$'s and $T_{s,i}$'s modulo $I$ are $\pi$ and $T_s$ in $A/I$ respectively for all $0\leq i\leq n$ and $1\leq s\leq d$. In particular, $u_0-u_i,T_{1,0}-T_{1,i},\cdots,T_{d,0}-T_{d,i}$ for $1\leq i\leq n$ are all in $E(u_0)A=I$. So the morphism $\frakS(R^+)^{\widehat \otimes(n+1)}\to A$ factors uniquely through a morphism 
   \[\frakS(R^+)^{\widehat \otimes(n+1)}\{\frac{u_0-u_i}{E(u_0)},\frac{T_{1,0}-T_{1,i}}{E(u_0)T_{1,0}},\cdots,\frac{T_{d,0}-T_{d,i}}{E(u_0)T_{d,0}}\mid 1\leq i\leq n\}^{\wedge}_{\delta}\to A.\]
   Here $\frakS(R^+)^{\widehat \otimes(n+1)}\{\frac{u_0-u_i}{E(u_0)},\frac{T_{1,0}-T_{1,i}}{E(u_0)T_{1,0}},\cdots,\frac{T_{d,0}-T_{d,i}}{E(u_0)T_{d,0}}\mid 1\leq i\leq n\}^{\wedge}_{\delta}$ is the $(p,E)$-complete $\delta$-$\frakS(R)$-algebra obtained by freely adjoining $\frac{u_0-u_i}{E(u_0)},\frac{T_{1,0}-T_{1,i}}{E(u_0)T_{1,0}},\cdots,\frac{T_{d,0}-T_{d,i}}{E(u_0)T_{d,0}}$ for all $1\leq i\leq n$ to $\frakS(R^+)^{\widehat \otimes(n+1)}$. It is exactly the prismatic envelope of the $\delta$-pair $(\frakS(R^+)^{\widehat \otimes(n+1)},J)$ over $(\frakS(R^+),(E))$ as discussed in \cite[Prop. 3.13]{BS22}, where $J=(E,u_0-u_i,T_{1,0}-T_{1,i},\cdots,T_{d,0}-T_{d,i})_{1\leq i\leq n}$. Hence we have 
   \[\frakS(R^+)^{\bullet} = \frakS(R^+)^{\widehat \otimes(\bullet+1)}\{\frac{u_0-u_i}{E(u_0)},\frac{T_{1,0}-T_{1,i}}{E(u_0)T_{1,0}},\dots,\frac{T_{d,0}-T_{d,i}}{E(u_0)T_{d,0}}\mid 1\leq i\leq \bullet\}^{\wedge}_{\delta}.\]
   
   (2) This follows from a similar argument used in (1). Also see \cite[Lem. 5.2]{Tia23}.
   
   (3) By virtue of \cite[Thm. 3.10]{BS22}, for any $n\geq 0$, the $\Ainf(\widehat R_{C,\infty}^+)^n/(E)$ is the initial object of the category of perfectoid algebras over $R^+$ which are the targets of $n+1$ arrows from $\widehat R_{C,\infty}^+$. So the $(n+1)$-folds self-product of $X_{C,\infty}$ over $X$, which is an affinoid perfectoid space, is just $\Spa((\Ainf(\widehat R_{C,\infty}^+)^n/(E))[\frac{1}{p}],(\Ainf(\widehat R_{C,\infty}^+)^n/(E))[\frac{1}{p}]^+)$ where $(\Ainf(\widehat R_{C,\infty}^+)^n/(E))[\frac{1}{p}]^+$ is the $p$-completion of the integral closure of the image of $\Ainf(\widehat R_{C,\infty}^+)^n/(E)$ in $(\Ainf(\widehat R_{C,\infty}^+)^n/(E))[\frac{1}{p}]$. As shown in the proof of Lemma \ref{Lem-EvaluateGRep}, we know that $(\Ainf(\widehat R_{C,\infty}^+)^n/(E))[\frac{1}{p}]$ is just $\rC(\Gamma(\overline K/K)^n,\widehat R_{C,\infty})$.
 \end{proof}
 \begin{rmk}
   Keep notations as in the proof of \ref{Lem-Structure} (3). Then one can show that $\Ainf(\widehat R_{C,\infty}^+)^n/(E)$ is almost isomorphic to $\rC(\Gamma(\overline K/K)^n,\widehat R_{C,\infty}^+)$ with respect to $\frakm_C$ in the sense of \cite[\S 4]{Sch-IHES}. More precisely, we can show that the kernel and cokernel of the canonical morphism 
   \[\iota_n:
   \Ainf(\widehat R_{C,\infty}^+)^n/(E)\to \rC(\Gamma(\overline K/K)^n,\widehat R_{C,\infty}^+)\]
   are both killed by $\frakm_{C}$. Indeed, since $\iota_n$ becomes isomorphic after inverting $p$, we can conclude by using the structure theorem of perfectoid rings (cf. \cite[Prop. 3.2]{Bha-LectNote}).
 \end{rmk}
 \begin{prop}\label{Prop-Structure}
   Let $X_i:=\frac{u_0-u_i}{E(u_0)}$ and $Y_{s,i}=\frac{T_{s,0}-T_{s,i}}{E(u_0)T_{s,0}}$ for any $1\leq s\leq d$ and any $i\geq 1$. We identify $u_0$ and $T_{s,0}$'s in the first factor of  $\frakS(R^+)^{\bullet}$ with $u$ and $T_s$'s in $\frakS(R^+)$ respectively. Similarly, we identify $T_{s,0}$'s in the first factor of $\frakS(R)^{\bullet}_{\geo}$ with $T_s$'s in $\frakS(R^+)$.
   \begin{enumerate}
       \item We have isomorphisms of cosimplicial rings:
       \[\overline \calO_{\Prism}(\frakS(R^+)^{\bullet}_{\geo}) \cong R^+\{ Y_{1,i},\dots,Y_{d,i}\mid 1\leq i\leq \bullet\}^{\wedge}_{\pd}\]
       and
       \[\overline \calO_{\Prism}[\frac{1}{p}](\frakS(R^+)^{\bullet}_{\geo}) \cong R\{ Y_{1,i},\dots,Y_{d,i}\mid 1\leq i\leq \bullet\}^{\wedge}_{\pd},\]
       where $ R^+\{ Y_{1,i},\dots,Y_{d,i}\mid 1\leq i\leq \bullet\}^{\wedge}_{\pd}$ is the $p$-adic completion of the free divided power polynomial ring over $R^+$ generated by the reductions of $Y_{s,i}$'s modulo $E$ with $1\leq i\leq \bullet$ and 
       \[R\{ Y_{1,i},\dots,Y_{d,i}\mid 1\leq i\leq \bullet\}^{\wedge}_{\pd} =  R^+\{ Y_{1,i},\dots,Y_{d,i}\mid 1\leq i\leq \bullet\}^{\wedge}_{\pd}[\frac{1}{p}]\]
       such that the face maps are given by
       \begin{equation}\label{Equ-Face-Geo}
           p_i(Y_{s,j}) = \left\{
           \begin{array}{rcl}
                Y_{s,j+1}-Y_{s,j}, & i=0  \\ 
                Y_{s,j+1}, & i\leq j\\
                Y_{s,j}, & i > j;
           \end{array}
           \right.
       \end{equation}
       while the degeneracy maps are given by
       \begin{equation}\label{Equ-Degeneracy-Geo}
           \sigma_i(Y_{s,j}) = \left\{
            \begin{array}{rcl}
                0, & i=0,j=1 \\
                Y_{s,j-1}, & i<j, (i,j)\neq (0,1) \\
                Y_{s,j}, & i\geq j.
            \end{array}
           \right.
       \end{equation}
       
       \item We have isomorphisms of cosimplicial rings:
       \[\overline \calO_{\Prism}(\frakS(R^+)^{\bullet}) \cong R^+\{X_i, Y_{1,i},\dots,Y_{d,i}\mid 1\leq i\leq \bullet\}^{\wedge}_{\pd}\]
       and
       \[\overline \calO_{\Prism}[\frac{1}{p}](\frakS(R^+)^{\bullet}) \cong R\{X_i, Y_{1,i},\dots,Y_{d,i}\mid 1\leq i\leq \bullet\}^{\wedge}_{\pd},\]
       where $ R^+\{X_i, Y_{1,i},\dots,Y_{d,i}\mid 1\leq i\leq \bullet\}^{\wedge}_{\pd}$ is the $p$-adic completion of the free divided power polynomial ring over $R^+$ generated by the reductions of $X_i$'s and $Y_{s,i}$'s modulo $E$ with $1\leq i\leq \bullet$ and 
       \[R\{X_i, Y_{1,i},\dots,Y_{d,i}\mid 1\leq i\leq \bullet\}^{\wedge}_{\pd} =  R^+\{X_i, Y_{1,i},\dots,Y_{d,i}\mid 1\leq i\leq \bullet\}^{\wedge}_{\pd}[\frac{1}{p}]\]
       such that the face maps are given by
       \begin{equation}\label{Equ-Face-Arith}
           \begin{split}
               &p_i(X_j) = \left\{
                 \begin{array}{rcl}
                    (X_{j+1}-X_1)(1-E'(\pi)X_1)^{-1}, & i = 0 \\
                     X_{j+1}, & i\leq j\\
                     X_j, & i>j;
                 \end{array}
               \right.\\
               &p_i(Y_{s,j}) = \left\{
               \begin{array}{rcl}
                   (Y_{s,j+1}-Y_{s,1})(1-E'(\pi)X_1)^{-1}, & i = 0 \\
                   Y_{s,j+1}, & i\leq j\\
                   Y_{s,j}, & i>j;
               \end{array}
               \right.
           \end{split}
       \end{equation}
       while the degeneracy maps are given by
       \begin{equation}\label{Equ-Degeneracy-Arith}
           \begin{split}
               &\sigma_i(X_j) = \left\{
                 \begin{array}{rcl}
                    0, & i=0,j=1\\
                    X_{j-1}, & i< j,(i,j)\neq (0,1)\\
                    X_j, & i\geq j;
                 \end{array}
               \right.\\
               &\sigma_i(Y_{s,j}) = \left\{
               \begin{array}{rcl}
                   0, & i = 0,j = 1 \\
                   Y_{s,j-1}, & i< j,(i,j)\neq (0,1)\\
                   Y_{s,j}, & i\geq j.
               \end{array}
               \right.
           \end{split}
       \end{equation}
   \end{enumerate}
 \end{prop}
 \begin{proof}
   We only need to show the $\overline \calO_{\Prism}$ case while the $\overline \calO_{\Prism}[\frac{1}{p}]$ case follows from inverting $p$.
   
   (1) This is \cite[Prop. 5.7 and the paragraph below its proof]{Tia23}.
   
   (2) For any $n\geq 0$, since $\delta(u_i) = 0$ for any $0\leq i\leq n$, the desired isomorphism 
   \[\overline \calO_{\Prism}(\frakS(R^+)^{\bullet}) \cong R^+\{X_i, Y_{1,i},\dots,Y_{d,i}\mid 1\leq i\leq \bullet\}^{\wedge}_{\pd}\]
   follows from the same calculations in the proof of \cite[Lem. 5.3 and Prop. 5.7]{Tia23} together with \cite[Prop. 2.12]{GMW-HT}. It remains to check that (\ref{Equ-Face-Arith}) and (\ref{Equ-Degeneracy-Arith}) hold true. Almost all formulae follow from the definitions of $p_i$'s and $\sigma_i$'s directly (cf. Convention \ref{Convention-Strat}) while the only difficulty appears in checking (\ref{Equ-Face-Arith}) for $p_0$. We show the details as follows:
   
   Since $p_0(u_i) = u_{i+1}$ and $p_0(T_{s,i}) = T_{s,i+1}$, we see that 
   \[p_0(X_j) = \frac{u_1-u_{j+1}}{E(u_1)} = (X_{j+1}-X_1)\frac{E(u_0)}{E(u_1)}\]
   and that
   \[p_0(Y_{s,j}) = \frac{T_{s,1}-T_{s,j+1}}{E(u_1)T_{s,1}} = (Y_{s,j+1}-Y_{s,j})\frac{E(u_0)T_{s,0}}{E(u_1)T_{s,1}}.\]
   Then the desired formulae can be obtained by noting that
   \[\frac{T_{s,0}}{T_{s,1}} = (1-E(u_0)Y_{s,1}) \equiv 1\mod E\]
   and that 
   \begin{equation}\label{Equ-FractionE}
       \begin{split}
           \frac{E(u_0)}{E(u_1)} & = (1-\frac{E(u_0)-E(u_1)}{E(u_0)})^{-1}\\
           & = (1-\sum_{i\geq 1}^{\deg E}\frac{E^{(i)}(u_1)(u_0-u_1)^i}{i!E(u_0)})^{-1}\quad (\text{by~Taylar's~expansion})\\
           & \equiv (1-E'(\pi)X_1)^{-1} \mod E\quad (\text{as}~u_0-u_1 = E(u_0)X_1).
       \end{split}
   \end{equation}
 \end{proof}
 \begin{exam}\label{Exam-Structure}
   Let $R^+ = \calO_K$ and hence $d = 0$ in Proposiition \ref{Prop-Structure}. Let $\frakS^{\bullet}$ be the underlying cosimplicial ring of \v Cech nerve of $(\frakS,(E))\in (\calO_K)_{\Prism}$. Then we have 
   \[\overline \calO_{\Prism}(\frakS^{\bullet},(E)) = \calO_K\{X_1,\dots,X_{\bullet}\}^{\wedge}\]
   with face and degeneracy maps given by (\ref{Equ-Face-Arith}) and (\ref{Equ-Degeneracy-Arith}) (for $X_i$'s), respectively.
 \end{exam}
 \begin{cor}\label{Cor-Structure}
   Let $\calK$ (resp. $\calK_g$, resp. $\calK_a$) be the kernel of 
   \[\begin{split}
   &\sigma_0:\overline \calO_{\Prism}(\frakS(R^+)^1,(E))\to \overline \calO_{\Prism}(\frakS(R^+),(E))\\ 
   &(\text{ resp. }\sigma_0:\overline \calO_{\Prism}(\frakS(R^+)^1_{\geo}),(E))\to\overline \calO_{\Prism}(\frakS(R^+),(E)), \text{ resp. }\sigma_0:\overline \calO_{\Prism}(\frakS^1,(E))\to\overline \calO_{\Prism}(\frakS,(E))).\end{split}\] 
   Then the following results hold true:
   \begin{enumerate}
       \item $\calK$ is the closed pd-ideal generated by $\{X_1^{[n_0]}Y_{1,1}^{[n_1]}\cdots Y_{d,1}^{[n_d]}\mid n_0+\cdots+n_d\geq 1\}$. In particular, by letting $d = 0$, we see that $\calK_a$ is the closed pd-ideal generated by $\{X_1^{[n]}\mid n\geq 1\}$.
       
       \item $\calK_g$ is the closed pd-ideal generated by $\{Y_{1,1}^{[n_1]}\cdots Y_{d,1}^{[n_d]}\mid n_1+\cdots+n_d\geq 1\}$.
       
       \item For any $r\geq 1$, let $\calK^{[r]}$ (resp. $\calK_g^{[r]}$, resp. $\calK_a^{[r]}$) be the $r$-th closed pd-power of $\calK$ (resp. $\calK_g$, resp. $\calK_a$). Then the obvious morphisms $\frakS^{\bullet}\to\frakS(R^+)^{\bullet}\to\frakS(R^+)^{\bullet}_{\geo}$ (induce a morphism ideals $\calK_a\to\calK\to\calK_g$ and then) induce a short exact sequence $R^+$-modules:
       \begin{equation}\label{Diag-pdIdeal}
           \xymatrix@C=0.45cm{
             0\ar[r]& \calK_a/\calK_a^{[2]}\otimes_{\calO_K}R^+\ar[r]& \calK/\calK^{[2]}\ar[r]&\calK_g/\calK_g^{2}\ar[r]&0
           }
       \end{equation}
   \end{enumerate}
   Similar results hold true for replacing $\overline \calO_{\Prism}$ by $\overline \calO_{\Prism}[\frac{1}{p}]$ after inverting $p$.
 \end{cor}
 \begin{proof}
   It suffices to consider the $\overline \calO_{\Prism}$ case as the $\overline \calO_{\Prism}[\frac{1}{p}]$ case follows from the flat base-change ``$-\otimes_{\Zp}\Qp$''. 
   
   (1) By Proposition \ref{Prop-Structure} (2), especially (\ref{Equ-Degeneracy-Arith}), we see that the closed ideal $\calJ$ generated by \[\{X_1^{[n_0]}Y_{1,1}^{[n_1]}\cdots Y_{d,1}^{[n_d]}\mid n_0+\cdots+n_1\geq 1\}\]
   is contained in $\calK$. So $\sigma_0:\overline \calO_{\Prism}(\frakS(R^+)^1)\to \overline \calO_{\Prism}(\frakS(R^+)) = R^+$ factors through the quotient $\frakS(R^+)^1/\calJ$. It is enough to show that $\frakS(R^+)^1/\calJ\cong R^+$, which can be easily deduced as follows:
   
   Let $\calJ_0$ be the ideal of the pd-polynomial ring $R^+[X_1,Y_{1,1},\dots,Y_{d,1}]_{\pd}$ over $R^+$ generated by $X_1, Y_{1,1},\dots,Y_{d,1}$. By taking $p$-adic completion along the exact sequence
   \[0\to\calJ_0\to R^+[X_1,Y_{1,1},\dots,Y_{d,1}]_{\pd}\to R^+\to 0 ,\]
   we get the desired isomorphism $\overline \calO_{\Prism}(\frakS(R^+)^1)/\calJ \cong R^+$ by noting that all rings involved are $(p,E)$-torsion free and $R^+$ is itself $(p,E)$-complete.
   
   (2) This follows from the same argument used in the proof of (1).
   
   (3) By (1) and (2), we see that for any $r\geq 1$, $\calK^{[r]}$ (resp. $\calK_g^{[r]}$, resp. $\calK_a^{[r]}$) is the closed ideal generated by 
   \begin{equation*}
       \begin{split}
           &\{X_1^{[n_0]}Y_{1,1}^{[n_1]}\cdots Y_{d,1}^{[n_d]}\mid n_0+\cdots+n_1\geq r\}\\
           &(\text{resp.}~ \{Y_{1,1}^{[n_1]}\cdots Y_{d,1}^{[n_d]}\mid n_1+\cdots+n_1\geq r\},~\text{resp.}~\{X_1^{[n]}\mid n\geq r\}).
       \end{split}
   \end{equation*}
   So, we deduce that 
   \[\calK/\calK^{[2]}\cong R^+\cdot X_1\oplus\oplus_{i=1}^dR^+\cdot Y_i~(\text{resp.}~\calK_g/\calK_g^{[2]}\cong \oplus_{i=1}^dR^+\cdot Y_i,~\text{resp.}~\calK_a/\calK_a^{[2]}\cong \calO_K\cdot X_1),\]
   which implies the desired exact sequence immediately.
 \end{proof}
 \begin{rmk}\label{Rmk-Structure}
   Consider $\overline \calO_{\Prism}$ in Corollary \ref{Cor-Structure}.
   By identifying $X_1$ with $\frac{\rd u}{E(u)}$, we see that 
   \[\calK_a/\calK_a^{[2]}\cong \widehat \Omega^1_{\frakS/\rW(\kappa)}\{-1\}\otimes_{\frakS}\calO_K \cong \calO_K\cdot \frac{\rd u}{E(u)}.\]
   By the definition of $Y_{s,i}$'s, we see that 
   \[\calK_g/\calK_g^{[2]}\cong \widehat \Omega^1_{R^+/\calO_K}\{-1\}\]
   by identidying $Y_{s,i}$ with $\frac{\dlog T_i}{E(u)}$,
   where $\widehat \Omega^1_{R^+/\calO_K}$ denotes the module of continuous differentials and $\{-1\}$ denotes the Breuil--Kisin twist. In particular, the exact sequence splits (non-canonically) and induces an isomorphism 
   \begin{equation}\label{Equ-Splitting}
       \calK/\calK^{[2]}\cong\calK_a/\calK_a^{[2]}\otimes_{\calO_K}R^+\oplus\calK_g/\calK_g^{[2]}\cong R^+\cdot \frac{\rd u}{E(u)}\oplus\widehat \Omega^1_{R^+/\calO_K}\{-1\}.
   \end{equation}
   A similar remark applies to $\overline \calO_{\Prism}[\frac{1}{p}]$.
 \end{rmk}

\subsection{Hodge--Tate crystals as enhanced Higgs modules}
   Now, we establish the desired equivalence in Theorem \ref{Thm-HTasHiggs-Local}.
   \begin{convention}\label{Convention-E'PI}
     For simplicity, we put $a = E'(\pi)$.
   \end{convention}
   \begin{convention}\label{Convention-CosimplicialRingA}
     Let $A^{?,\bullet}$ be the cosimplicial ring $R^?\{X_i,Y_{1,i},\dots,Y_{d,i}\mid 1\leq i\leq \bullet\}$ for $?\in\{\emptyset,+\}$ with face and degeneracy morphisms given by (\ref{Equ-Face-Arith}) and (\ref{Equ-Degeneracy-Arith}), respectively. By Proposition \ref{Prop-Structure}, we have 
     \[A^{\bullet}\cong \overline \calO_{\Prism}[\frac{1}{p}](\frakS(R^+)^{\bullet},(E))\]
     and 
     \[A^{+,\bullet}\cong \overline \calO_{\Prism}(\frakS(R^+)^{\bullet},(E)).\]
     Let $\Strat(A^{?,\bullet})$ denote the category of stratifications with respect to $A^{?,\bullet}$ satisfying the cocycle condition (cf. Convention \ref{Convention-Strat}). Note that $A^{?,0} = R^?$.
   \end{convention}
   We start with the following lemma:
   \begin{lem}\label{Lem-HTasStrat}
     There exists a canonical equivalence
     \[\begin{split}\Vect((R^+)_{\Prism},\overline \calO_{\Prism})\xrightarrow{\simeq}\Strat(A^{?,\bullet})~(\text{resp. }\Vect((R^+)_{\Prism},\overline \calO_{\Prism}[\frac{1}{p}]))\xrightarrow{\simeq} \Strat(A^{?,\bullet}))\end{split}\]
      for $? = +$ (resp. $? = \emptyset$) by evaluating (rational) Hodge--Tate crystals along the cosimplicial prism $(\frakS(R^+)^{\bullet},(E))$.
   \end{lem}
   \begin{proof}
     By Lemma \ref{Lem-CoverPrism}, $(\frakS(R^+),(E))$ is a cover of final object of $\Sh((R^+)_{\Prism})$. So the result follows from \cite[Prop. 2.7]{BS23}.
   \end{proof}
   Therefore, in order to establish the desired equivalence in Theorem \ref{Thm-HTasHiggs-Local}, it suffices to construct an equivalence between the categories $\frakS(A^{?,\bullet})$ and $\HIG_*^{\nil}(R^?)$.
   
   Let $(H,\varepsilon)\in \Strat(A^{?,\bullet})$ be any fixed stratification satisfying the cocycle condition. By Proposition \ref{Prop-Structure}, via the embedding $H\xrightarrow{x\mapsto x\otimes 1} H\otimes_{R^?,p_0}A^{?,1}$, there exists a collection $\{\phi_{i,\underline n}\}_{i\geq 0,\underline n\in\bN^d}$ of $R^?$-linear endomorphisms of $H$ such that for any $x\in H$,
   \[
   \varepsilon(x) = \sum_{i\geq 0,\underline n\in\bN^d}\phi_{i,\underline n}(x)X_1^{[i]}\underline Y_1^{[\underline n]}.
   \]
   Note that in order to make $\varepsilon$ well-defined, we require 
   \[
   \lim_{i+|\underline n|\to+\infty}\phi_{i,\underline n} = 0
   \]
   with respect to the $p$-adic topology on $H$.
   By (\ref{Equ-Face-Arith}) and (\ref{Equ-Degeneracy-Arith}), we see that
   \begin{equation}\label{Equ-P2P0}
       \begin{split}
           & p_2^*(\varepsilon)\circ p_0^*(\varepsilon)(x)\\
           =& p_2^*(\varepsilon)(\sum_{j\geq 0,\underline n\in\bN^d}\phi_{j,\underline n}(x)(1-aX_1)^{-j-|\underline n|}(X_2-X_1)^{[j]}(\underline Y_2-\underline Y_1)^{[\underline n]})\\
           =& \sum_{i,j\geq 0,\underline l,\underline n\in\bN^d}\phi_{i,\underline l}(\phi_{j,\underline n}(x))(1-aX_1)^{-j-|\underline n|}X_1^{[i]}(X_2-X_1)^{[j]}\underline Y_1^{[\underline l]}(\underline Y_2-\underline Y_1)^{[\underline n]}\\
           =& \sum_{i,j,k\geq 0,\underline l,\underline m, \underline n\in\bN^d}\phi_{i,\underline l}(\phi_{j+k,\underline m+\underline n}(x))(1-aX_1)^{-j-k-|\underline m|-|\underline n|}(-1)^{j+|\underline m|}X_1^{[i]}X_1^{[j]}X_2^{[k]}\underline Y_1^{[\underline l]}\underline Y_1^{[\underline m]}\underline Y_2^{[\underline n]}\\
           =& \sum_{i,j,k\geq 0,\underline l,\underline m, \underline n\in\bN^d}\phi_{i,\underline l}(\phi_{j+k,\underline m+\underline n}(x))(1-aX_1)^{-j-k-|\underline m|-|\underline n|}(-1)^{j+|\underline m|}\binom{i+j}{i}\binom{\underline l+\underline m}{\underline l}X_1^{[i+j]}\underline Y_1^{[\underline l+\underline m]}X_2^{[k]}\underline Y_2^{[\underline n]},
       \end{split}
   \end{equation}
   that 
   \begin{equation}\label{Equ-P1}
       p_1^*(\varepsilon)(x) = \sum_{k\geq 0,\underline n\in\bN^d}\phi_{k,\underline n}(x)X_2^{[k]}\underline Y_2^{[\underline n]},
   \end{equation}
   and that 
   \begin{equation}
       \sigma_0^*(\varepsilon)(x) = \phi_{0,\underline 0}(x).
   \end{equation}
   Therefore, we have
   \begin{lem}\label{Lem-Stratification}
      For any stratification $(H,\varepsilon)$ with respect to $A^{?,\bullet}$, if we write 
      \[\varepsilon(x) = \sum_{i\geq 0,\underline n\in\bN^d}\phi_{i,\underline n}(x)X_1^{[i]}\underline Y_1^{[\underline n]}\]
      for any $x\in H$ with $\phi_{i,\underline n}\to 0$ as $i+|\underline n|\to+\infty$, then $(H,\varepsilon)$ satisfies the cocycle condition if and only if the following conditions are satisfied:
      \begin{enumerate}
          \item $\phi_{0,\underline 0} = \id_H$, and
          
          \item For any $k\geq 0,\underline n\in\bN^d$ and for any $x\in H$, we have 
          \[\phi_{k,\underline n}(x) = \sum_{i,j\geq 0,\underline l,\underline m, \in\bN^d}\phi_{i,\underline l}(\phi_{j+k,\underline m+\underline n}(x))(1-aX)^{-j-k-|\underline m|-|\underline n|}(-1)^{j+|\underline m|}\binom{i+j}{i}\binom{\underline l+\underline m}{\underline l}X^{[i+j]}\underline Y^{[\underline l+\underline m]}.\]
      \end{enumerate}
   \end{lem}
   \begin{notation}\label{Notation-Stratification}
     For any $(H,\varepsilon)\in\Strat(A^{?,\bullet})$ with $\phi_{i,\underline n}$'s as above. Define $\phi_H = \phi_{1,\underline 0}$. For any $1\leq i\leq d$, put $\theta_i = \phi_{0,\underline 1_i}$ and for $? = +$ (resp. $\emptyset$), define
     \[\theta_H:H\to H\otimes_{R^+}\widehat \Omega^1_{R^+/\calO_K}\{-1\} ~(\text{resp.}~\theta_H:H\to H\otimes_{R} \Omega^1_{R/K}\{-1\})\]
     by $\theta_H=\sum_{i=1}^{d}\theta_i\otimes\frac{\dlog T_i}{E(u)}$.
   \end{notation}
   The following lemma is the key ingredient to establish the desired equivalence.
   \begin{lem}\label{Lem-Technique}
     Let $\phi_{i,\underline n}$'s be endomorphisms of $H$ with $\phi_{0,\underline 0} = \id_H$. Let $\phi_H$, $\theta_H$ and $\theta_i$'s be as in Notation \ref{Notation-Stratification}. Then the following assertions are equivalent:
     \begin{enumerate}
         \item For any $k\geq 0,\underline n\in\bN^d$ and for any $x\in H$, we have 
          
          \begin{equation}\label{Equ-Technique}
               \phi_{k,\underline n}(x) = \sum_{i,j\geq 0,\underline l,\underline m, \in\bN^d}\phi_{i,\underline l}(\phi_{j+k,\underline m+\underline n}(x))(1-aX)^{-j-k-|\underline m|-|\underline n|}(-1)^{j+|\underline m|}\binom{i+j}{i}\binom{\underline l+\underline m}{\underline l}X^{[i+j]}\underline Y^{[\underline l+\underline m]}.
          \end{equation}
          
          \item For any $1\leq i,j\leq d$, we have $[\phi_H,\theta_i]=-a\theta_i$ and $[\theta_i,\theta_j] = 0$ such that for any $k\geq 0$ and $\underline n\in\bN^d$, 
          \begin{equation}\label{Equ-GeneralFormulae}
              \phi_{k,\underline n}(x) = \prod_{s=1}^d\theta_i^{n_i}\cdot\prod_{i=0}^{k-1}(\phi_H+ia)(x) = \prod_{i=0}^{k-1}(\phi_H+(i+|\underline n|)a)\cdot\prod_{i=1}^{d}\theta_i^{n_i}(x).
          \end{equation}
     \end{enumerate}
     Moreover, if one of the above conditions is satisfied, then $\theta_i$'s are nilpotent and 
     \[\lim_{k+|\underline n|\to+\infty}\phi_{k,\underline n} = 0\Leftrightarrow \lim_{k\to+\infty}\prod_{i=0}^{k-1}(\phi_H+ia) = 0.\]
   \end{lem}
   \begin{proof}
     First, we assume (1) holds true for any $k\geq 0$ and any $\underline n\in\bN^d$ and then verify (2) as follows:
     
     By letting $\underline Y = \underline 0$ in both sides of (\ref{Equ-Technique}), we get
     \[\phi_{k,\underline n}(x) = \sum_{i,j\geq 0}\phi_{i,\underline 0}(\phi_{j+k,\underline n}(x))(1-aX)^{-j-k-|\underline n|}(-1)^{j}\binom{i+j}{i}X^{[i+j]}.\]
     Compare the coefficients of $X$ in both sides of the above formula and then we get
     \[(k+|\underline n|)a\phi_{k,\underline n}(x)-\phi_{1+k,\underline n}(x)+\phi_{1,\underline 0}(\phi_{k,\underline n}(x)) = 0,\]
     which implies by iteration that for any $k\geq 0$ and $\underline n\in\bN^d$,
     \begin{equation}\label{Equ-Iteration-I}
         \phi_{k,\underline n}(x) = (\phi_{1,\underline 0}+(k-1+|\underline n|)a)\phi_{k-1,\underline n}(x) = \prod_{i=0}^{k-1}(\phi_{H}+(i+|\underline n|)a)\phi_{0,\underline n}(x).
     \end{equation}
     By letting $X = 0$ in both sides of (\ref{Equ-Technique}), we get
     \[\phi_{k,\underline n}(x) = \sum_{\underline l,\underline m, \in\bN^d}\phi_{0,\underline l}(\phi_{k,\underline m+\underline n}(x))(-1)^{|\underline m|}\binom{\underline l+\underline m}{\underline l}\underline Y^{[\underline l+\underline m]}.\]
     Compare the coefficients of $Y_i$ in both sides of the above formula and then we get
     \[\phi_{0,\underline 1_i}(\phi_{k,\underline n}(x)) - \phi_{k,\underline 1_i+\underline n}(x) = 0.\]
     In particular, by letting $k = 0$ and $\underline n = \underline 1_j$ in the above formula, we see that
     \[\theta_i\theta_j = \phi_{0,\underline 1_i+\underline 1_j} = \theta_j\theta_i.\]
     In other words, $\theta_i$'s commute with each other.
     Also, by iteration (for all $1\leq i\leq d$) that for any $k\geq 0$ and $\underline n\in\bN^d$,
     \begin{equation}\label{Equ-Iteration-II}
         \phi_{k,\underline n}(x) = \phi_{0,\underline 1_i}(\phi_{k,\underline n-\underline 1_i}(x)) = \prod_{i=1}^{d}\theta_i^{n_i}(\phi_{k,\underline 0}(x)).
     \end{equation}
     Now, (\ref{Equ-GeneralFormulae}) follows from (\ref{Equ-Iteration-I}) and (\ref{Equ-Iteration-II}). In particular, by letting $k = 1$ and $\underline n = \underline 1_i$, we see that $[\phi_H,\theta_i] = -a\theta_i$ as desired.
     
     Now assume (2) is true. By noting that the constant term of the right hand side of (\ref{Equ-Technique}) is exactly $\phi_{k,\underline n}(x)$, it suffices to show that the right hand side of (\ref{Equ-Technique}) is independent of $X$ and $\underline Y$.
   
   First, we show that the right hand side of (\ref{Equ-Technique}) is independent of $X$. In other words, if we put 
   \begin{equation*}
   \begin{split}
        (\dagger)
       =  & \partial_X(\sum_{i,j\geq 0,\underline l,\underline m, \in\bN^d}\phi_{i,\underline l}(\phi_{j+k,\underline m+\underline n}(x))(1-aX)^{-j-k-|\underline m|-|\underline n|}(-1)^{j+|\underline m|}\binom{i+j}{i}\binom{\underline l+\underline m}{\underline l}X^{[i+j]}\underline Y^{[\underline l+\underline m]})\\
       = & \sum_{i,j\geq 0,\underline l,\underline m, \in\bN^d}(j+k+|\underline m|+|\underline n|)a\phi_{i,\underline l}(\phi_{j+k,\underline m+\underline n}(x))\\
       & \cdot (1-aX)^{-j-k-1-|\underline m|-|\underline n|}(-1)^{j+|\underline m|}\binom{i+j}{i}\binom{\underline l+\underline m}{\underline l}X^{[i+j]}\underline Y^{[\underline l+\underline m]}\\
       & +\sum_{i,j\geq 0,i+j\geq 1,\underline l,\underline m, \in\bN^d}\phi_{i,\underline l}(\phi_{j+k,\underline m+\underline n}(x))(1-aX)^{-j-k-|\underline m|-|\underline n|}(-1)^{j+|\underline m|}\binom{i+j}{i}\binom{\underline l+\underline m}{\underline l}X^{[i+j-1]}\underline Y^{[\underline l+\underline m]},
       \end{split}
   \end{equation*}
   then we have to show that $(\dagger) = 0$. Note that 
   \[\{(i,j)\mid i,j\geq 0,i+j\geq 1\} = \{(i,j)\mid i,j-1\geq 0\}\cup\{(i,0)\mid i\geq 1\}.\]
   We have
   \begin{equation*}
   \begin{split}
       (\dagger) = & \sum_{i,j\geq 0,\underline l,\underline m, \in\bN^d}(j+k+|\underline m|+|\underline n|)a\phi_{i,\underline l}(\phi_{j+k,\underline m+\underline n}(x))\\
       & \cdot (1-aX)^{-j-1-k-|\underline m|-|\underline n|}(-1)^{j+|\underline m|}\binom{i+j}{i}\binom{\underline l+\underline m}{\underline l}X^{[i+j]}\underline Y^{[\underline l+\underline m]}\\
       & -\sum_{i,j\geq 0,\underline l,\underline m, \in\bN^d}\phi_{i,\underline l}(\phi_{j+1+k,\underline m+\underline n}(x))(1-aX)^{-j-1-k-|\underline m|-|\underline n|}(-1)^{j+|\underline m|}\binom{i+j+1}{i}\binom{\underline l+\underline m}{\underline l}X^{[i+j]}\underline Y^{[\underline l+\underline m]}\\
       & +\sum_{i\geq 0,\underline l,\underline m, \in\bN^d}\phi_{i+1,\underline l}(\phi_{k,\underline m+\underline n}(x))(1-aX)^{-k-|\underline m|-|\underline n|}(-1)^{|\underline m|}\binom{\underline l+\underline m}{\underline l}X^{[i]}\underline Y^{[\underline l+\underline m]}.
   \end{split}
   \end{equation*}
    By (\ref{Equ-GeneralFormulae}), we have $\phi_{j+1+k,\underline m+\underline n}(x)=(\phi_H+(j+k+|\underline m|+|\underline n|)a)\phi_{j+k,\underline m+\underline n}(x)$. Hence we can use $\phi_{j+1+k,\underline m+\underline n}(x)-\phi_H(\phi_{j+k,\underline m+\underline n}(x))$ to replace $(j+k+|\underline m|+|\underline n|)a(\phi_{j+k,\underline m+\underline n}(x))$ in above formula and conclude that
   \begin{equation*}
       \begin{split}
           (\dagger) = & \sum_{i,j\geq 0,\underline l,\underline m, \in\bN^d}-\phi_{i,\underline l}(\phi_H(\phi_{j+k,\underline m+\underline n}(x)))(1-aX)^{-j-1-k-|\underline m|-|\underline n|}(-1)^{j+|\underline m|}\binom{i+j}{i}\binom{\underline l+\underline m}{\underline l}X^{[i+j]}\underline Y^{[\underline l+\underline m]}\\
           & -\sum_{i,j\geq 0,\underline l,\underline m, \in\bN^d}\phi_{i,\underline l}(\phi_{j+1+k,\underline m+\underline n}(x))(1-aX)^{-j-1-k-|\underline m|-|\underline n|}(-1)^{j+|\underline m|}\binom{i+j}{i-1}\binom{\underline l+\underline m}{\underline l}X^{[i+j]}\underline Y^{[\underline l+\underline m]}\\
           & +\sum_{i\geq 0,\underline l,\underline m, \in\bN^d}\phi_{i+1,\underline l}(\phi_{k,\underline m+\underline n}(x))(1-aX)^{-k-|\underline m|-|\underline n|}(-1)^{|\underline m|}\binom{\underline l+\underline m}{\underline l}X^{[i]}\underline Y^{[\underline l+\underline m]}.
       \end{split}
   \end{equation*}
  By (\ref{Equ-GeneralFormulae}) again, we have $\phi_{i+1,\underline l}(\phi_{j+k,\underline m+\underline n}(x))=(\phi_H+(i+|\underline j|)a)\phi_{i,\underline l}(\phi_{j+k,\underline m+\underline n}(x))$. Then by using $[\phi_H,\theta_i]=-a\theta_i$, $[\theta_i,\theta_j] = 0$ and (\ref{Equ-GeneralFormulae}), we have $\phi_{i+1,\underline l}(\phi_{j+k,\underline m+\underline n}(x))=ia\phi_{i,\underline l}(\phi_{j+k,\underline m+\underline n}(x))+\phi_{i,\underline l}(\phi_H(\phi_{j+k,\underline m+\underline n}(x)))$. Now the above formula becomes
   \begin{equation*}
       \begin{split}
           (\dagger) = & \sum_{i,j\geq 0,\underline l,\underline m, \in\bN^d}ia\phi_{i,\underline l}(\phi_{j+k,\underline m+\underline n}(x))(1-aX)^{-j-1-k-|\underline m|-|\underline n|}(-1)^{j+|\underline m|}\binom{i+j}{i}\binom{\underline l+\underline m}{\underline l}X^{[i+j]}\underline Y^{[\underline l+\underline m]}\\
           & -\sum_{i,j\geq 0,\underline l,\underline m, \in\bN^d}\phi_{i+1,\underline l}(\phi_{j+k,\underline m+\underline n}(x))(1-aX)^{-j-1-k-|\underline m|-|\underline n|}(-1)^{j+|\underline m|}\binom{i+j}{i}\binom{\underline l+\underline m}{\underline l}X^{[i+j]}\underline Y^{[\underline l+\underline m]}\\
           & -\sum_{i,j\geq 0,\underline l,\underline m, \in\bN^d}\phi_{i,\underline l}(\phi_{j+1+k,\underline m+\underline n}(x))(1-aX)^{-j-1-k-|\underline m|-|\underline n|}(-1)^{j+|\underline m|}\binom{i+j}{i-1}\binom{\underline l+\underline m}{\underline l}X^{[i+j]}\underline Y^{[\underline l+\underline m]}\\
           & +\sum_{i\geq 0,\underline l,\underline m, \in\bN^d}\phi_{i+1,\underline l}(\phi_{k,\underline m+\underline n}(x))(1-aX)^{-k-|\underline m|-|\underline n|}(-1)^{|\underline m|}\binom{\underline l+\underline m}{\underline l}X^{[i]}\underline Y^{[\underline l+\underline m]}\\
           = & \sum_{i,j\geq 0,\underline l,\underline m, \in\bN^d}aX\phi_{i,\underline l}(\phi_{j+k,\underline m+\underline n}(x))(1-aX)^{-j-1-k-|\underline m|-|\underline n|}(-1)^{j+|\underline m|}\binom{i-1+j}{i-1}\binom{\underline l+\underline m}{\underline l}X^{[i-1+j]}\underline Y^{[\underline l+\underline m]}\\
           & -\sum_{i,j\geq 0,\underline l,\underline m, \in\bN^d}\phi_{i+1,\underline l}(\phi_{j+k,\underline m+\underline n}(x))(1-aX)^{-j-1-k-|\underline m|-|\underline n|}(-1)^{j+|\underline m|}\binom{i+j}{i}\binom{\underline l+\underline m}{\underline l}X^{[i+j]}\underline Y^{[\underline l+\underline m]}\\
           & -\sum_{i,j\geq 0,\underline l,\underline m, \in\bN^d}\phi_{i,\underline l}(\phi_{j+1+k,\underline m+\underline n}(x))(1-aX)^{-j-1-k-|\underline m|-|\underline n|}(-1)^{j+|\underline m|}\binom{i+j}{i-1}\binom{\underline l+\underline m}{\underline l}X^{[i+j]}\underline Y^{[\underline l+\underline m]}\\
           & +\sum_{i\geq 0,\underline l,\underline m, \in\bN^d}\phi_{i+1,\underline l}(\phi_{k,\underline m+\underline n}(x))(1-aX)^{-k-|\underline m|-|\underline n|}(-1)^{|\underline m|}\binom{\underline l+\underline m}{\underline l}X^{[i]}\underline Y^{[\underline l+\underline m]},
       \end{split}
   \end{equation*}
   where we use $i\binom{i+j}{i}X^{[i+j]} = X\binom{i-1+j}{i-1}X^{[i-1+j]}$ to get the second equality. In the most right hand side of the above formula, replacing $i-1$ by $i$ in the first and third summands and using $\binom{j-1}{-1} = \binom{j}{-1} = 0$, we see that
   \begin{equation*}
       \begin{split}
           (\dagger) = & \sum_{i,j\geq 0,\underline l,\underline m, \in\bN^d}aX\phi_{i+1,\underline l}(\phi_{j+k,\underline m+\underline n}(x))(1-aX)^{-j-1-k-|\underline m|-|\underline n|}(-1)^{j+|\underline m|}\binom{i+j}{i}\binom{\underline l+\underline m}{\underline l}X^{[i+j]}\underline Y^{[\underline l+\underline m]}\\
           & -\sum_{i,j\geq 0,\underline l,\underline m, \in\bN^d}\phi_{i+1,\underline l}(\phi_{j+k,\underline m+\underline n}(x))(1-aX)^{-j-1-k-|\underline m|-|\underline n|}(-1)^{j+|\underline m|}\binom{i+j}{i}\binom{\underline l+\underline m}{\underline l}X^{[i+j]}\underline Y^{[\underline l+\underline m]}\\
           & -\sum_{i,j\geq 0,\underline l,\underline m, \in\bN^d}\phi_{i+1,\underline l}(\phi_{j+1+k,\underline m+\underline n}(x))(1-aX)^{-j-1-k-|\underline m|-|\underline n|}(-1)^{j+|\underline m|}\binom{i+1+j}{i}\binom{\underline l+\underline m}{\underline l}X^{[i+1+j]}\underline Y^{[\underline l+\underline m]}\\
           & +\sum_{i\geq 0,\underline l,\underline m, \in\bN^d}\phi_{i+1,\underline l}(\phi_{k,\underline m+\underline n}(x))(1-aX)^{-k-|\underline m|-|\underline n|}(-1)^{|\underline m|}\binom{\underline l+\underline m}{\underline l}X^{[i]}\underline Y^{[\underline l+\underline m]}\\
           = & -\sum_{i,j\geq 0,\underline l,\underline m, \in\bN^d}\phi_{i+1,\underline l}(\phi_{j+k,\underline m+\underline n}(x))(1-aX)^{-j-k-|\underline m|-|\underline n|}(-1)^{j+|\underline m|}\binom{i+j}{i}\binom{\underline l+\underline m}{\underline l}X^{[i+j]}\underline Y^{[\underline l+\underline m]}\\
           & -\sum_{i,j\geq 0,\underline l,\underline m, \in\bN^d}\phi_{i+1,\underline l}(\phi_{j+1+k,\underline m+\underline n}(x))(1-aX)^{-j-1-k-|\underline m|-|\underline n|}(-1)^{j+|\underline m|}\binom{i+1+j}{i}\binom{\underline l+\underline m}{\underline l}X^{[i+1+j]}\underline Y^{[\underline l+\underline m]}\\
           & +\sum_{i\geq 0,\underline l,\underline m, \in\bN^d}\phi_{i+1,\underline l}(\phi_{k,\underline m+\underline n}(x))(1-aX)^{-k-|\underline m|-|\underline n|}(-1)^{|\underline m|}\binom{\underline l+\underline m}{\underline l}X^{[i]}\underline Y^{[\underline l+\underline m]}.
       \end{split}
   \end{equation*}
   Finally, replacing $j+1$ by $j$ in the second term of the most right hand side of above formula, we see that
   \begin{equation*}
       \begin{split}
           (\dagger) = & -\sum_{i,j\geq 0,\underline l,\underline m, \in\bN^d}\phi_{i+1,\underline l}(\phi_{j+k,\underline m+\underline n}(x))(1-aX)^{-j-k-|\underline m|-|\underline n|}(-1)^{j+|\underline m|}\binom{i+j}{i}\binom{\underline l+\underline m}{\underline l}X^{[i+j]}\underline Y^{[\underline l+\underline m]}\\
           &+ \sum_{i\geq 0,j\geq 1,\underline l,\underline m, \in\bN^d}\phi_{i+1,\underline l}(\phi_{j+k,\underline m+\underline n}(x))(1-aX)^{-j-k-|\underline m|-|\underline n|}(-1)^{j+|\underline m|}\binom{i+j}{i}\binom{\underline l+\underline m}{\underline l}X^{[i+j]}\underline Y^{[\underline l+\underline m]}\\
           & +\sum_{i\geq 0,\underline l,\underline m, \in\bN^d}\phi_{i+1,\underline l}(\phi_{k,\underline m+\underline n}(x))(1-aX)^{-k-|\underline m|-|\underline n|}(-1)^{|\underline m|}\binom{\underline l+\underline m}{\underline l}X^{[i]}\underline Y^{[\underline l+\underline m]}\\
           = & 0
       \end{split}
   \end{equation*}
   as desired, which shows that the right hand side of (\ref{Equ-Technique}) is independent of $X$. In particular by plugging $X=0$, to deduce (\ref{Equ-Technique}), it is enough to show that for any $k\geq 0$ and any $\underline n\in \bN^d$, we have
   \begin{equation}\label{Equ-Technique-I}
       \phi_{k,\underline n}(x) = \sum_{\underline l,\underline m, \in\bN^d}\phi_{0,\underline l}(\phi_{k,\underline m+\underline n}(x))(-1)^{|\underline m|}\binom{\underline l+\underline m}{\underline l}\underline Y^{[\underline l+\underline m]}.
   \end{equation}
   
     We consider a special case for $k = 0$ in (\ref{Equ-Technique-I}). In other words, we want to show that for any $\underline n\in\bN^d$,
     \begin{equation}\label{Equ-Technique-II}
         \phi_{0,\underline n}(x) = \sum_{\underline l,\underline m, \in\bN^d}\phi_{0,\underline l}(\phi_{0,\underline m+\underline n}(x))(-1)^{|\underline m|}\binom{\underline l+\underline m}{\underline l}\underline Y^{[\underline l+\underline m]}.
     \end{equation}
     To do so, it suffices to show that for free variables $Z_1,\dots,Z_d$,
     \[\sum_{\underline n\in\bN^d}\phi_{0,\underline n}(x)\underline Z^{[\underline n]} = \sum_{\underline l,\underline m, \underline n\in\bN^d}\phi_{0,\underline l}(\phi_{0,\underline m+\underline n}(x))(-1)^{|\underline m|}\binom{\underline l+\underline m}{\underline l}\underline Y^{[\underline l+\underline m]}\underline Z^{[\underline n]}.\]
     Since $\theta_i$'s commute with each other and $\phi_{0,\underline n} = \prod_{i=1}^d\theta_i^{n_i}$ by (\ref{Equ-GeneralFormulae}), we see that
     \[\sum_{\underline n\in\bN^d}\phi_{0,\underline n}(x)\underline Z^{[\underline n]} = \exp(\sum_{i=1}^d\theta_iZ_i)(x)\]
     and that 
     \begin{equation*}
         \begin{split}
             &\sum_{\underline l,\underline m, \underline n\in\bN^d}\phi_{0,\underline l}(\phi_{0,\underline m+\underline n}(x))(-1)^{|\underline m|}\binom{\underline l+\underline m}{\underline l}\underline Y^{[\underline l+\underline m]}\underline Z^{[\underline n]}\\
             =&\sum_{\underline l,\underline m, \underline n\in\bN^d}\prod_{i=1}^d\theta_i^{l_i+m_i+n_i}(x)(-1)^{|\underline m|}\underline Y^{[\underline l]}\underline Y^{[\underline m]}\underline Z^{[\underline n]}\\
             =&\sum_{\underline l\in\bN^d}\prod_{i=1}^d\theta_i^{l_i}\underline Y^{[\underline l]}\sum_{\underline m\in\bN^d}\prod_{i=1}^d\theta_i^{m_i}(-\underline Y)^{[\underline m]}\sum_{\underline n\in\bN^d}\prod_{i=1}^d\theta_i^{n_i}(x)\underline Z^{[\underline n]}\\
             =&\exp(\sum_{i=1}^d\theta_iY_i)(x)\exp(-\sum_{i=1}^d\theta_iY_i)\exp(\sum_{i=1}^d\theta_iZ_i)(x)\\
             =&\exp(\sum_{i=1}^d\theta_iZ_i)(x).
         \end{split}
     \end{equation*}
     Therefore, the special case (\ref{Equ-Technique-II}) holds true. In particular, replacing $x$ by $\phi_{k,\underline 0}(x)$ in (\ref{Equ-Technique-II}), we get
     \begin{equation}\label{Equ-Technique-III}
         \phi_{0,\underline n}(\phi_{k,\underline 0}(x)) = \sum_{\underline l,\underline m, \in\bN^d}\phi_{0,\underline l}(\phi_{0,\underline m+\underline n}(\phi_{k,\underline 0}(x)))(-1)^{|\underline m|}\binom{\underline l+\underline m}{\underline l}\underline Y^{[\underline l+\underline m]}.
     \end{equation}
     By using (\ref{Equ-GeneralFormulae}), especially $\phi_{0,\underline n}(\phi_{k,\underline 0}(x)) = \phi_{k,\underline n}(x)$, we see that (\ref{Equ-Technique-III}) is exactly (\ref{Equ-Technique-I}) as desired. So we conclude that (2) implies (1).
     
     To complete the proof, it remains to show the ``moreover'' part. The nilpotency of $\theta_i$'s follows from the same argument use in Remark \ref{Rmk-EnhancedHiggsMod-Local}. Since $\phi_{k,\underline n} = \phi_{0,\underline n}\circ\phi_{k,0}$ for any $(k,\underline n)\in\bN\times\bN^d$, we conclude from the nilpotency of $\theta_i$'s that 
     \[\lim_{k+|\underline n|\to+\infty}\phi_{k,\underline n} = 0\Leftrightarrow \lim_{k\to+\infty}\prod_{i=0}^{k-1}(\phi_H+ia) = 0.\]
 \end{proof}
 \begin{prop}\label{Prop-HTvsHiggs-Local}
   The evaluation at $(\frakS(R^+),(E))$ induces equivalences 
   \[\begin{split}
       &\Vect((R^+)_{\Prism},\overline \calO_{\Prism})\to \Strat(A^{+,\bullet})\to\HIG^{\nil}_*(R^+)\\(\text{resp.}&\Vect((R^+)_{\Prism},\overline \calO_{\Prism}[\frac{1}{p}])\to\Strat(A^{\bullet})\to\HIG^{\nil}_*(R))
   \end{split}\]
   of categories, which preserves ranks, tensor products and dualities. Moreover, for any $\bM\in \Vect((R^+)_{\Prism},\overline \calO_{\Prism})$ (resp. $\Vect((R^+)_{\Prism},\overline \calO_{\Prism}[\frac{1}{p}])$) with the corresponding stratification $(H,\varepsilon)\in \Strat(A^{?,\bullet})$ for $? = +$ (resp. $\emptyset$), let $\phi_H,\theta_H$ and $\theta_i$'s be as in Notation \ref{Notation-Stratification} and then we have
   \begin{enumerate}
      \item the enhanced Higgs module induced by $\bM$ is $(H,\theta_H,\phi_H)$, and 
      
      \item the stratification $\varepsilon$ on $H$ is determined such that for any $x\in H$,
      \begin{equation}\label{Equ-Stratification}
          \begin{split}
              \varepsilon(x)=& \exp(\sum_{i=1}^d\theta_iY_1)(1-aX_1)^{-\frac{\phi_H}{a}}(x)\\
               &: = \sum_{k\geq 0,\underline n\in\bN^d}\prod_{i=1}^d\theta_i^{n_i}( \prod_{j=0}^{k-1}(\phi_H+jE'(\pi))(x))X_1^{[k]}\underline Y_1^{[\underline n]}\\
              =& (1-aX_1)^{-\frac{\phi_H}{a}}\prod_{i=1}^d\exp((1-aX_1)^{-1}\theta_iY_1)(x)\\
               &: = \sum_{k\geq 0,\underline n\in\bN^d}\prod_{j=0}^{k-1}(\phi_H+(j+|\underline n|)E'(\pi))(\prod_{i=1}^d\theta_i^{n_i}(x))X_1^{[k]}\underline Y_1^{[\underline n]}
          \end{split}
      \end{equation}
   \end{enumerate}
 \end{prop}
 \begin{proof}
   The ``moreover'' part of Lemma \ref{Lem-Technique} shows that $(H,\theta_H,\phi_H)$ is an enhanced Higgs module over $R^?$. Then
   the proposition follows from Lemma \ref{Lem-HTasStrat} and Lemma \ref{Lem-Technique} immediately as all constructions involved preserve ranks, tensor products and dualities.
 \end{proof}
 \begin{rmk}\label{Rmk-Intrinsic}
   We give an intrinsic construction of $(H,\theta_H,\phi_H)$ from a stratification $(H,\varepsilon)\in\Strat(A^{?,\bullet})$ for $? = +$ (resp. $\emptyset$) as follows:
   
   Let $\calK,\calK_g$ and $\calK_a$ be as in Corollary \ref{Cor-Structure} (and Remark \ref{Rmk-Structure}).
   Noticing that $\sigma_0^*(\varepsilon) = \id_H$, we deduce that for any $x\in H$,
   \[\varepsilon(x) - \id_H(x) \in H\otimes_{R^?,p_1}A^1\cdot\calK.\]
   Modulo $\calK^{[2]}$, we get an $R^?$-linear morphism 
   \[\overline{(\varepsilon-\id_H)}:H\to H\otimes_{R^?}\calK/\calK^{[2]}.\]
   Let $\theta:H\to H\otimes_{R^+}\widehat \Omega^1_{R^+/\calO_K}\{-1\}$ (resp, $\theta:H\to H\otimes_{R}\Omega^1_{R/K}\{-1\}$) be the $R^?$-linear morphism induced by the projection $\calK/\calK^{[2]}\to \calK_g/\calK_g^{[2]}$ in Corollary \ref{Cor-Structure} (3). Using the (non-canonical) decomposition (\ref{Equ-Splitting}),
   we see that $\overline{(\varepsilon-\id_H)}$ is of the form 
   \[(\phi,\theta):H\to H\oplus H\otimes_{R^+}\widehat \Omega^1_{R^+/\calO_K}\{-1\}~(\text{resp.}~(\phi,\theta):H\to H\oplus H\otimes_{R} \Omega^1_{R/K}\{-1\}).\]
   By (\ref{Equ-Stratification}), we see that $\phi = \phi_H$ and $\theta = \theta_H$.
 \end{rmk}
 \begin{exam}[Breuil--Kisin Twist]\label{Exam-BK twist}
   Let $\overline \calO_{\Prism}\{n\}:=\calI_{\Prism}^n/\calI_{\Prism}^{n+1}$ for any $n\in\bZ$, which is known as the $n$-th \emph{Breuil--Kisin twist} of $\calO_{\Prism}$. Then we see that 
   \[\overline \calO_{\Prism}\{n\}(\frakS(R^+),(E)) = R^+\cdot E(u)^n.\]

  Note that by Taylor’s expansion, we have 
  \[\begin{split}
      E(u_0-E(u_0)X_1)&=E(u_0)+E^{(1)}(u_0)(-E(u_0)X_1)+E^{(2)}(u_0)\frac{(-E(u_0)X_1)^2}{2!}+\cdots\\
      &\equiv E(u_0)(1-E'(u_0)X_1)\mod E(u_0)^2.
  \end{split}\]
  Write $E(u_0-E(u_0)X_1) = E(u_0)\big((1-E'(u_0)X_1)+E(u_0)v\big)$ for some $v\in \frakS(R)^1$, and then we can get the following congruence
   \[\begin{split}
       p_0(E(u_0)^n) &= E(u_1)^n\\
       &= E(u_0-E(u_0)X_1)^n \\
       &= E(u_0)^n\big((1-E'(u_0)X_1)+E(u_0)v\big)^n\\
       &\equiv (1-E'(u_0)X_1)^nE(u_0)^n \mod E(u_0)^{n+1}.
   \end{split}\]
   By chasing constructions above, we see that the enhanced Higgs bundle associated to $\overline \calO_{\Prism}\{n\}$ is $(R^+,0,-nE'(\pi)\id_H)$. 
   In general, for any (rational) Hodge--Tate crystal $\bM$ on $(R^+)_{\Prism}$ with the associated enhanced Higgs module $(H,\theta_H,\phi_H)$, the enhanced Higgs module induced by $\bM\{n\}: = \bM\otimes_{\overline \calO_{\Prism}}\overline \calO_{\Prism}\{n\}$ is $(H,\theta_H,\phi_H-nE'(\pi)\id_H)$. 
 \end{exam}

 \subsection{Prismatic cohomology vs Higgs cohomology: Local case}
   Now, we want to show the following result:
   \begin{prop}\label{Prop-HTvsHiggsCoho-Local}
      For any (rational) Hodge--Tate crystal $\bM$ on $(R^+)_{\Prism}$ with associated enhanced Higgs module $(H,\theta_H,\phi_H)$, there exists a quasi-isomorphism 
      \[\rR\Gamma((R^+)_{\Prism},\bM)\simeq \HIG(H,\theta_H,\phi_H).\]
   \end{prop}
  \begin{rmk}
  Keep notations as in Proposition \ref{Prop-HTvsHiggsCoho-Local}.
Write $Rf_*\bM$ for the derived direct image of $\bM$ along $f: \Spf(R^+)\to \Spf(\calO_K)$. On the one hand, for any Hodge--Tate crystal $\bL$ over $(\calO_K)_{\Prism}$ with the induced ``Sen module'' $(L,\phi_L)$ where $L$ is $\bL(\frakS,(E))$, the cohomology $\rR\Gamma((\calO_K)_{\Prism},\bL)$ can be computed by the complex $[L\xrightarrow{\phi_L}L]$ (cf. \cite[Thm. 4.5]{GMW-HT}). On the other hand, we see that $(\rR f_*\bM)(\frakS,(E))$ is just $\rR\Gamma((R^+/(\frakS,(E)))_{\Prism},\bM)$, which can be computed via $\HIG(H,\theta_H)$ (cf. \cite[Thm. 5.12]{Tia23}). So it is reasonable to believe there exists a quasi-isomorphism between $\rR\Gamma((R^+)_{\Prism},\bM)$ and $\HIG(H,\theta_H,\phi_H)$.
 Indeed, we will show Proposition \ref{Prop-HTvsHiggsCoho-Local} by combining calculations in \cite{Tia23} with \cite{GMW-HT} together.
  \end{rmk}
  
  For the sake of simplicity, we still assume $a = E'(\pi)$ as in Convention \ref{Convention-E'PI}. We first show that one can compute $\rR\Gamma((R^+)_{\Prism},\bM)$ via \v Cech--Alexander complex.
  \begin{lem}\label{Lem-Vanishing}
     Let $\frakX$ be any bounded $p$-adic formal scheme and let $\calC$ be either $(\frakX)_{\Prism}$ or $(\frakX)_{\Prism}^{\perf}$.
     For any (rational) Hodge--Tate crystal $\bM$ on $\calC$ and for any $\calU = (A,I)\in \calC$, we have $\rH^i(\calC/\calU,\bM) = 0$ for any $i\geq 1$.
  \end{lem}
  \begin{proof}
    For any $\calV = (B,IB)\in (R)_{\Prism}$ which is a cover of $\calU$ in $\calC$, we denote by $(B^{\bullet},IB^{\bullet})$ the \v Cech nerve induced by $\calV$. Since $\bM$ is a (rational) Hodge--Tate crystal on $\calC$, we have a canonical isomorphism of comsimplicial $B^{\bullet}$-modules
    \[\bM(B^{\bullet},IB^{\bullet})\cong \bM(A,I)\otimes_{A/I}B^{\bullet}/IB^{\bullet}.\]
    Since $\calV$ is a cover of $\calU$, by $p$-adically faithfully flat descent, we see that $\rH^i(\overline\calO_{\Prism}(B^{\bullet},IB^{\bullet})) = 0$ and hence $\rH^i(\bM(B^{\bullet},IB^{\bullet})) = 0$ for any $i\geq 1$. So the lemma follows from \cite[\href{https://stacks.math.columbia.edu/tag/03F9}{Lemma 03F9}]{Sta}.
  \end{proof}
  Now, we still keep the assumption on the smallness of $R^+$.
  \begin{lem}\label{Lem-CechAlexanderMethod}
    For any (rational) Hodge--Tate crystal $\bM$ on $(R^+)_{\Prism}$, there exists a quasi-isomorphism 
    \[\rR\Gamma((R^+)_{\Prism},\bM)\cong \bM(\frakS(R^+)^{\bullet},(E)).\]
  \end{lem}
  \begin{proof}
    This follows from Lemma \ref{Lem-CoverPrism} together with Lemma \ref{Lem-Vanishing} by using \v Cech-to-derived spectral sequence.
  \end{proof}
  \begin{construction}\label{Construction-Bicomsimplicial}
    Let $\bM$ be a (rational) Hodge--Tate crystal on $(R^+)_{\Prism}$ with induced enhanced Higgs complex $(H,\theta_H,\phi_H)$. Then there exists a canonical isomorphism 
    \[H\otimes_{R^?,q_0}A^{?,\bullet}\cong \bM(\frakS(R^+)^{\bullet},(E)).\]
    By virtues of Proposition \ref{Prop-Structure} (2), we may write
    \[H\otimes_{R^?,q_0}A^{?,\bullet} = H\{X_i,Y_{1,i},\dots,Y_{d,i}\mid 1\leq i\leq \bullet\}^{\wedge}_{\pd}\]
    and
    \[B^{\bullet} = R^+\{Y_{1,i},\dots,Y_{d,i}\mid 1\leq i\leq \bullet\}^{\wedge}_{\pd}.\]
    We regard $A^{?,m}$ as a $B^m$-algebra via identifying $Y_{s,i}$'s\footnote{We warn that the induced morphism $B^{\bullet}\to A^{?,\bullet}$ does not preserve the cosimplicial structures.}. 
    Let $\widehat \Omega^q_{B^m}$ be the module of continuous $q$-forms of $B^{m}$ over $\calO_K$. By Proposition \ref{Prop-Structure}, for any $n\geq 0$, we have
    \[R^+\{X_i,Y_{s,j}\mid 1\leq i\leq n,1\leq j\leq m\}^{\wedge}_{\pd}\widehat \otimes_{B^m}\widehat \Omega^q_{B^m}\cong \widehat \Omega^q_{R^+\{X_i,Y_{s,j}\mid 1\leq i\leq n,1\leq j\leq m\}^{\wedge}_{\pd}/R^+\{X_i\mid 1\leq i\leq n\}^{\wedge}_{\pd}}.\]
    
    For any $m,n,q\geq 0$, define 
    \[C^{n,m}\{-q\}:=H\{X_i,Y_{1,j},\dots,Y_{d,j}\mid 1\leq i\leq n,1\leq j\leq m\}^{\wedge}_{\pd}\widehat \otimes_{B^{m}}\widehat \Omega^q_{B^{m}}\{-q\}.\]
    We then define a structure of ``Higgs complex of bicosimplicial objects'' on $C^{\bullet,\bullet}\{-\bullet\}$ as follows:
    
    \begin{enumerate}
        \item For any $0\leq i\leq n+1$, define $p_i^1:C^{n,\bullet}\{-\bullet\}\to C^{n+1,\bullet}\{-\bullet\}$ such that it acts via $(1-aX_1)^{-1}$ on $E(u)$ (where we view $E(u)$ as a basis of the Breuil--Kisin twist $R^?\{1\}$ mentioned in Example \ref{Exam-BK twist}) for $i=0$ and the identity for $i\geq 1$, acts via (\ref{Equ-Face-Arith}) on $X_k$'s, acts on $Y_{s,j}$'s via the scalar $(1-aX_1)^{-1}$ for $i=0$ and as the identity for $i\geq 1$, and acts on $x\in H$ via $(1-aX_1)^{-\frac{\phi_H}{a}}$ for $i=0$ and via $\id_H$ for $i\geq 1$. 
        
        \item For any $0\leq i\leq n$, define $\sigma_i^1: C^{n+1,\bullet}\{-\bullet\}\to C^{n,\bullet}\{-\bullet\}$ such that it acts trivially on $E(u)$, acts via (\ref{Equ-Degeneracy-Arith}) on $X_i$'s and acts on $Y_{s,j}$'s and $x\in H$ as the identity.
        
        \item For any $0\leq j\leq m+1$, define $p_j^2:C^{\bullet,m}\{-\bullet\}\to C^{\bullet,m+1}\{-\bullet\}$ such that it acts trivially on $E(u)$, acts via the identity on $X_i$'s, via (\ref{Equ-Face-Geo}) on $Y_{s,k}$'s, and acts on $x\in H$ via $\exp(\sum_{i=1}^d\theta_iY_{i,1})$ for $i = 0$ and via the identity for $i\geq 1$.
        
        \item For any $0\leq j\leq m$, define $\sigma_j^2:C^{\bullet,m+1}\{-\bullet\}\to C^{\bullet,m}\{-\bullet\}$ such that it acts trivially on $E(u)$, $X_i$'s and $x\in H$, and acts via (\ref{Equ-Degeneracy-Geo}) on $Y_{s,k}$'s.
        
        \item For any $q\geq 0$, define $d_q^{\bullet,\bullet}:C^{\bullet,\bullet}\{-q\}\to C^{\bullet,\bullet}\{-q-1\}$ such that for any $x\in H\{X_i,Y_{s,j}\mid 1\leq i\leq \bullet,1\leq j\leq \bullet\}^{\wedge}_{\pd}$ and $\omega\in \widehat \Omega^q_{B^{\bullet}}\{-q\}$, it carries $x\otimes\omega$ to $\theta_H(x)\wedge\omega+x\otimes \rd\omega$, where $\theta_H$ extends to $H\{X_i,Y_{s,j}\mid 1\leq i\leq \bullet,1\leq j\leq \bullet\}^{\wedge}_{\pd}$ by linear extension and $\rd:\widehat \Omega^q_{B^{\bullet}}\{-q\}\to \widehat \Omega^{q+1}_{B^{\bullet}}\{-q-1\}$ is defined in the paragraph after \cite[Lem. 5.15]{Tia23} (which is denoted by $\rd_R$ in loc.cit.).
    \end{enumerate}
  \end{construction}
  Then we have the following lemma:
  \begin{lem}\label{Lem-Bicosimplicial}
     The $(C^{\bullet,\bullet}\{-\bullet\},p_i^1,\sigma_i^1,p_i^2,\sigma_i^2, d_q^{\bullet,\bullet})$ defines a complex of bicosimplicial $R^?$-modules. 
     In other words, for any $q\geq 0$, $(C^{\bullet,\bullet}\{-q\},p_i^1,\sigma_i^1,p_i^2,\sigma_i^2,)$ is a bicomsimplicial $R^?$-module and $d_q^{\bullet,\bullet}:C^{\bullet,\bullet}\{-q\}\to C^{\bullet,\bullet}\{-q-1\}$ preserves bicomsimplicial structures such that $d_{q+1}^{\bullet,\bullet}\circ d_q^{\bullet,\bullet} = 0$.
  \end{lem}
  \begin{proof}
    By \cite[Lem. 5.16]{Tia23}, we see that for any $n,q\geq 0$, $(C^{n,\bullet}\{-q\},p_i^2,\sigma_i^2)$ defines a cosimplicial $R^?$-module such that $\{d_q^{n,\bullet}:C^{n,\bullet}\{-q\}\to C^{n,\bullet}\{-q-1\}\}_{q\geq 0}$ is a complex of cosimplicial $R^?$-modules.
    
    Next, we show that for any $m,q\geq 0$, $(C^{\bullet,m}\{-q\},p_i^1,\sigma_i^1)$ also defines a cosimplicial $R^?$-module such that $\{d_q^{\bullet,m}:C^{\bullet,m}\{-q\}\to C^{\bullet,m}\{-q-1\}\}_{q\geq 0}$ is a complex of cosimplicial $R^?$-modules. It is easy to check the first criterion. Using Example \ref{Exam-BK twist}, we only need to check that $(C^{\bullet,m}\{-q\},p_i^1,\sigma_i^1)$ defines a cosimplicial $R^?$-module for $q = 0$. But in this case, $(C^{\bullet,m}\{-0\},p_i^1,\sigma_i^1)$ is the comsimplicial $R^?$-module induced by evaluating $\bM$ at the fibre product $(\frakS(R^+)^{m}_{\geo},(E))\times_{(\frakS,(E))}(\frakS^n,(E))$ of $(\frakS(R^+)^m_{\geo},(E))$ and $(\frakS^n,(E))$ over the first factor of the latter, where $(\frakS^n,(E))$ is defined in Example \ref{Exam-Structure}. This fiber product exists as its underlying $\delta$-ring is $\frakS(R^+)^m_{\geo}\widehat \otimes_{\frakS}\frakS^n$. To show the second criterion, we have to check that $d^{\bullet,m}_q$ preserves cosimplicial structures. We only show that it commutes with $p_0^1$ as the rest is much easier and can be deduced in a similar way. Now, for any $x\in H$ and $\omega\in \widehat \Omega^q_{R^+\{X_i,Y_{s,j}\mid 1\leq i\leq n,1\leq j\leq m\}^{\wedge}_{\pd}/R^+\{X_i\mid 1\leq i\leq n\}^{\wedge}_{\pd}}\{-q\}$, we see that 
    \begin{equation*}
    \begin{split}
        p_0^1(d_{q}^{n,m}(x\otimes\omega)) &= p_0^1(\sum_{i=1}^d\theta_i(x)\otimes\frac{\dlog T_i}{E(u)}\wedge\omega+x\otimes\rd \omega)\\
        & = \sum_{i=1}^d(1-aX_1)^{-\frac{\phi_H}{a}-1}\theta_i(x)\otimes\frac{\dlog T_i}{E(u)}\wedge p_0^1(\omega)+(1-aX_1)^{-\frac{\phi_H}{a}}(x)\otimes p_0^1(\rd \omega),
    \end{split}
    \end{equation*}
    where the second equality follows as $p_0^1(E(u)) = (1-aX_1)E(u)$ (cf. Example \ref{Exam-BK twist}).
    On the other hand, we have
    \begin{equation*}
    \begin{split}
        d_{q}^{n,m}(p_0^1(x\otimes\omega)) &= d_{q}^{n,m}((1-aX_1)^{-\frac{\phi_H}{a}}(x)\otimes p_0^1(\omega))\\
        & = \sum_{i=1}^d\theta_i((1-aX_1)^{-\frac{\phi_H}{a}}(x))\otimes\frac{\dlog T_i}{E(u)}\wedge p_0^1(\omega)+(1-aX_1)^{-\frac{\phi_H}{a}}(x)\otimes \rd p_0^1( \omega).
    \end{split}
    \end{equation*}
    By noting that $\rd$ commutes with $p_0^1$, we reduce the case to showing that for any $1\leq i\leq d$, 
    \[(1-aX_1)^{-\frac{\phi_H}{a}-1}\theta_i(x) = \theta_i((1-aX_1)^{-\frac{\phi_H}{a}}(x)).\]
    But this follows from (\ref{Equ-GeneralFormulae}) directly.
    
    It remains to show for each $q\geq 0$, the $(C^{\bullet,\bullet}\{-q\},p_i^1,\sigma_i^1,p_i^2,\sigma_i^2)$ is a bicosimplicial $R^?$-module. In other words, we need to show $p_i^1$'s and $\sigma_i^1$'s commute with $p_i^2$'s and $\sigma_i^2$'s. We only show $p_0^1$ commutes with $p_0^2$'s as the rest can be checked similarly. For any $x\in H$, we have 
    \begin{equation*}
        \begin{split}
            p_0^1(p_0^2(x)) & = p_0^1(\exp(\sum_{i=1}^d\theta_iY_{i,1})x)\\
             & = (1-aX_1)^{-\frac{\phi_H}{a}}(\exp(\sum_{i=1}^d(1-aX_1)^{-1}\theta_iY_{i,1})x)\\
             & = \exp(\sum_{i=1}^d\theta_iY_{i,1})((1-aX_1)^{-\frac{\phi_H}{a}}(x)) \quad (\text{by}~(\ref{Equ-Stratification}))\\
             & = p_0^2(p_0^1(x)).
        \end{split}
    \end{equation*}
    For any $k\geq 1$, we have that
    \begin{equation*}
        \begin{split}
            p_0^1(p_0^2(X_k)) = p_0^1(X_k) = (X_{k+1}-X_1)(1-aX_1)^{-1} = p_0^2(p_0^1(X_k)),
        \end{split}
    \end{equation*}
    and that for any $1\leq s\leq d$,
    \begin{equation*}
        \begin{split}
            p_0^1(p_0^2(Y_{s,k})) = p_0^1(Y_{s,k+1}-Y_{s,1}) = (Y_{s,k+1}-Y_{s,1})(1-aX_1)^{-1} = p_0^2(Y_{s,k}(1-aX_1)^{-1}) = p_0^2(p_0^1(Y_{s,k})),
        \end{split}
    \end{equation*}
    as desired. The proof is complete.
  \end{proof}
  The $(C^{\bullet,\bullet}\{-\bullet\},p_i^1,\sigma_i^1,p_i^2,\sigma_i^2)$ is related with $\bM(\frakS(R)^{\bullet},(E))$ in the following sense:
  \begin{lem}\label{Lem-Diagonal}
    Let $(C^{\bullet,\bullet},p_i^1,\sigma_i^1,p_i^2,\sigma_i^2)$ denote the bi-cosimplicial $R^?$-module of the restriction of the complex $(C^{\bullet,\bullet}\{-\bullet\},p_i^1,\sigma_i^1,p_i^2,\sigma_i^2, d_q^{\bullet,\bullet})$ at $q = 0$. Let $(D^{\bullet},p_i^D,\sigma_i^D)$ be the diagonal of this bi-cosimplicial object (cf. \cite[\S 8.5]{Wei}). Then we have an isomorphism of cosimplicial $R^?$-modules $D^{\bullet} = \bM(\frakS(R^+)^{\bullet},(E))$.
  \end{lem}
  \begin{proof}
    We identify $\bM(\frakS(R^+),(E))$ with $H\{X_i,Y_{1,i},\dots,Y_{d,i}\mid 1\leq i\leq \bullet\}_{\pd}^{\wedge}$ as in Construction \ref{Construction-Bicomsimplicial} and denote the induced face and degeneracy morphisms by $p_i^H$'s and $\sigma_i^H$'s. Then by definition of $C^{\bullet,\bullet}$, we see that for any $n\geq 0$, 
    \[D^n = C^{n,n} = H\{X_i,Y_{1,i},\dots,Y_{d,i}\mid 1\leq i\leq n\}_{\pd}^{\wedge}.\]
    For any $f(X_i,Y_{1,i},\dots,Y_{d,i}\mid 1\leq i\leq n)x\in D^n$, we have that
    \begin{equation*}
        \begin{split}
           & p_0^D(f(X_i,Y_{1,i},\dots,Y_{d,i}\mid 1\leq i\leq n)x)\\
           =& p_0^1(p_0^2(f(X_i,Y_{1,i},\dots,Y_{d,i}\mid 1\leq i\leq n)x))\\
           =&p_0^1(f(X_i,Y_{1,i+1}-Y_{1,1},\dots,Y_{d,i+1}-Y_{d,1}\mid 1\leq i\leq n)\exp(\sum_{k=1}^d\theta_kY_k)x)\\
           =& f((X_{i+1}-X_1)(1-aX_1)^{-1},(Y_{1,i+1}-Y_{1,1})(1-aX_1)^{-1},\dots,(Y_{d,i+1}-Y_{d,1})(1-aX_1)^{-1}\mid 1\leq i\leq n)\\
           & \cdot (1-aX_1)^{-\frac{\phi_H}{a}}\exp(\sum_{k=1}^d(1-aX_1)^{-1}\theta_kY_k)x\\
           = & p_0^H(f(X_i,Y_{1,i},\dots,Y_{d,i}\mid 1\leq i\leq n)x)\quad(\text{by}~(\ref{Equ-Stratification},\ref{Equ-Face-Arith}))
        \end{split}
    \end{equation*}
    and that for any $1\leq k\leq n+1$,
    \begin{equation*}
        \begin{split}
            & p_k^D(f(X_i,Y_{1,i},\dots,Y_{d,i}\mid 1\leq i\leq n)x)\\
            =& p_k^1(p_k^2(f(X_i,Y_{1,i},\dots,Y_{d,i}\mid 1\leq i\leq n)x))\\
            =& p_k^1(f(X_i,Y_{1,j},\dots,Y_{d,j}\mid 1\leq i\leq n, 1\leq j\leq n+1,j\neq k)\exp(\sum_{l=1}^d\theta_lY_{l,1})x)\\
            =& f(X_i,Y_{1,i},\dots,Y_{d,i}\mid 1\leq i\leq n+1,i\neq k)\exp(\sum_{l=1}^d\theta_lY_{l,1})x\\
            =&p_k^H(f(X_i,Y_{1,i},\dots,Y_{d,i}\mid 1\leq i\leq n)x).
        \end{split}
    \end{equation*}
    Similarly, for any $f(X_i,Y_{1,i},\dots,Y_{d,i}\mid 1\leq i\leq n+1)x\in D^n$, we have that
    \begin{equation*}
        \begin{split}
            &\sigma_0^D(f(X_i,Y_{1,i},\dots,Y_{d,i}\mid 1\leq i\leq n+1)x) \\
            =& \sigma_0^1(\sigma_0^2(f(X_i,Y_{1,i},\dots,Y_{d,i}\mid 1\leq i\leq n+1)x))\\
            =&\sigma_0^1(f(X_i,0,Y_{1,i},\dots,Y_{d,i}\mid 1\leq i\leq n)x)\\
            =&f(0,X_i,0,Y_{1,i},\dots,Y_{d,i}\mid 1\leq i\leq n)x\\
            =&\sigma_0^H(f(X_i,Y_{1,i},\dots,Y_{d,i}\mid 1\leq i\leq n+1)x),
        \end{split}
    \end{equation*}
    and that for any $1\leq k\leq n$, after letting $X_{i}'$ (resp. $Y_{s,i}'$) be $X_{i}$ (resp. $Y_{s,i}$) for $1\leq i\leq k$ and be $X_{i-1}$ (resp. $Y_{s,i-1}$) for $k+1\leq i\leq n+1$, 
    \begin{equation*}
        \begin{split}
            &\sigma_k^D(f(X_i,Y_{1,i},\dots,Y_{d,i}\mid 1\leq i\leq n+1)x)\\
            =&\sigma_k^1(\sigma_k^2(f(X_i,Y_{1,i},\dots,Y_{d,i}\mid 1\leq i\leq n+1)x))\\
            =&  \sigma_0^1(f(X_i,Y'_{1,i},\dots,Y'_{d,i}\mid 1\leq i\leq n+1)x)\\
            =& f(X_i',Y_{1,i}',\dots,Y_{d,i}'\mid 1\leq i\leq n+1)x\\
            =& \sigma_k^H(f(X_i,Y_{1,i},\dots,Y_{d,i}\mid 1\leq i\leq n+1)x).
        \end{split}
    \end{equation*}
    These imply the desired isomorphism of cosimplicial $R^?$-modules and hence complete the proof.
  \end{proof}
  \begin{cor}\label{Cor-Eilenberg–Zilber}
    Let $\Tot(C^{\bullet})$ denote be the total complex induced by the bi-cosimplicial $R^?$-module $C^{\bullet,\bullet}$. Then we have a quasi-isomorphism 
    \[\bM(\frakS(R^+)^{\bullet},(E))\simeq \Tot(C^{\bullet,\bullet}).\]
  \end{cor}
  \begin{proof}
     This is well-known as the cosimplicial Eilenberg–Zilber theorem and the shuffle coproduct formula (cf. \cite[Appendix]{GM}).
  \end{proof}
  \begin{lem}\label{Lem-ArithGeo-to-ArithGeoHiggs}
    Let $\Tot(C^{\bullet,\bullet}\{-\bullet\})$ (resp. $\Tot(C^{\bullet,0}\{-\bullet\})$) be the total complex of \[(C^{\bullet,\bullet}\{-\bullet\},p_i^1,\sigma_i^1,p_i^2,\sigma_i^2, d_q^{\bullet,\bullet})~(\text{resp.}~(C^{\bullet,0}\{-\bullet\},p_i^1,\sigma_i^1,d_q^{\bullet,0})).\] 
    Then the natural projections $C^{\bullet,\bullet}\leftarrow C^{\bullet,\bullet}\{-\bullet\} \to C^{\bullet,0}\{-\bullet\}$ induce quasi-isomorphisms
    \[\Tot(C^{\bullet,\bullet})\simeq \Tot( C^{\bullet,\bullet}\{-\bullet\} )\simeq \Tot(C^{\bullet,0}\{-\bullet\}).\]
  \end{lem}
  \begin{proof}
    Let $\frakS^{\bullet}$ be as in Example \ref{Exam-Structure}. Note that \cite[Lem. 5.15]{Tia23} gives rise to a quasi-isomorphism \[C^{0,0}\{-\bullet\}\simeq \Tot(C^{0,\bullet}\{-\bullet\}).\]
    By Construction \ref{Construction-Bicomsimplicial}, for any $n\geq 0$, we get an isomorphism of complexes of simplicial $R^?$-morphisms
    \[C^{0,\bullet}\{-\bullet\}\widehat \otimes_{\calO_K}\frakS^n/(E)\cong C^{n,\bullet}\{-\bullet\}.\]
    Define $T^{n,\bullet}:=\Tot(C^{n,\bullet}\{-\bullet\})$ for all $n$. Then using faithful flatness of $\calO_K\to\frakS^n/(E)$, we get quasi-isomorphisms
    \[C^{n,0}\{-\bullet\}\simeq \Tot(C^{n,\bullet}\{-\bullet\}) = T^{n,\bullet} \]
    for all $n$. By Lemma \ref{Lem-Bicosimplicial}, we see that $\{C^{n,0}\{-\bullet\}\simeq T^{n,\bullet}\}_{n\geq 0}$ is compatible with $p_i^1$'s and $\sigma_i^1$'s. So by using the spectral sequence associated to double complexes, we deduce that 
    \[\Tot(C^{\bullet,0}\{-\bullet\})\simeq \Tot(T^{\bullet,\bullet}) = \Tot(C^{\bullet,\bullet}\{-\bullet\})\]
    as desired (alternatively, we have a weak equivalence  (i.e. degree-wise weak equivalence) of cosimplicial cochain complexes, and by taking the homotopy limits, i.e. the totalization, we get the first quasi-isomorphism).

    On the other hand, by \cite[Lem. 2.17]{BdJ}, for any $q\geq 1$, $C^{0,\bullet}\{-q\}$ is homotopic to zero and hence so is $C^{n,\bullet}\{-q\}$ for each $n\geq 0$. So we get a quasi-isomorphism
    \[C^{n,\bullet}\simeq T^{n,\bullet}\]
    for each $n$, which is compatible with $p_i^1$'s and $\sigma_i^1$'s, again following from Lemma \ref{Lem-Bicosimplicial}. Therefore, we obtain a quasi-isomorphism 
    \[\Tot(C^{\bullet,\bullet})\simeq \Tot(T^{\bullet,\bullet}) = \Tot(C^{\bullet,\bullet}\{-\bullet\})\]
    by using the spectral sequence associated to double complexes. The proof is complete.
  \end{proof}
  \begin{construction}\label{Construction-Rho}
    Let $F(X,Y):= \frac{(1-aX)^{-\frac{Y}{a}}-1}{Y} = X+\sum_{n\geq 1}\prod_{i=1}^n(Y+ia)X^{[n+1]}$. Then for any $q\geq 0$, as $\lim_{n\to+\infty}\prod_{i=0}^{n-1}(\phi_H+ia) = 0$, we see that $F(X,\phi_H+qa)$ is well-defined and induces a morphism
    \[F(X,\phi_H+qa): H\otimes_{R^+}\widehat \Omega^q_{R^+/\calO_K}\{-q\}\to H\{X_1\}\otimes_{R^+}\widehat \Omega^q_{R^+/\calO_K}\{-q\} = C^{1,0}\{-q\} .\]
    We claim it makes the following diagram
    \begin{equation*}
        \xymatrix@C=0.45cm{
          H\otimes_{R^+}\widehat \Omega^q_{R^+/\calO_K}\{-q\}\ar@{=}[d]\ar[rr]^{\phi_H+qa}&&H\otimes_{R^+}\widehat \Omega^q_{R^+/\calO_K}\{-q\}\ar[d]^{F(X_1,\phi_H+qa)}\\
          C^{0,0}\{-q\}\ar[rr]^{p_0^1-p_1^1}&&C^{1,0}\{-q\}
        }
    \end{equation*}
    commute. Indeed, by Construction \ref{Construction-Bicomsimplicial}(1), we see that on $C^{0,0}\{-q\}$, the face map $p_0^1$ acts via $(1-aX_1)^{-\frac{\phi_H}{a}-q} = (1-aX_1)^{-\frac{\phi_H+qa}{a}}$ while $p_1^1$ acts via the identity map. So we see that
    \[p_0^1-p_1^1 = (1-aX_1)^{-\frac{\phi_H+qa}{a}} - 1 = F(X_1,\phi_H+qa)\circ(\phi_H+qa)\]
    as desired, where the second equality follows from the definition of $F(X,Y)$. Using the same argument in the second part of the proof of \cite[Proposition 3.9(1)]{GMW-HT}, we get a morphism of complexes
    \[\rho_q:[H\otimes_{R^+}\widehat \Omega^1_{R^+/\calO_K}\{-q\}\to H\otimes_{R^+}\widehat \Omega^1_{R^+/\calO_K}\{-q\}]\to C^{\bullet,0}\{-q\}.\]
    Since $[\phi_H,\theta_H] = -a\theta_H$, we see that $\rho_{\bullet}$ induces a morphism of complexes
    \[\rho: \HIG(H,\theta_H,\phi_H)\to \Tot(C^{\bullet,0}\{-\bullet\}).\]
  \end{construction}
  \begin{lem}\label{Lem-Rho}
    The morphism $\rho$ in Construction \ref{Construction-Rho} is a quasi-isomorphism.
  \end{lem}
  \begin{proof}
      Using the spectral sequence associated to double complexes, it suffices to show that each $\rho_q$ is a quasi-isomorphism. But this follows from \cite[Proposition 3.9]{GMW-HT} (simply by replacing ``$\calO_K$-linear'' there by ``$R^?$-linear'').
  \end{proof}
   Note that the geometric part does not intervene in Lemma \ref{Lem-Rho}. Now, we are prepared to prove Proposition \ref{Prop-HTvsHiggsCoho-Local}:
  \begin{proof}[Proof of Proposition \ref{Prop-HTvsHiggsCoho-Local}:]
    This follows from quasi-isomorphisms
    \begin{equation*}
        \begin{split}
            \rR\Gamma((R^+)_{\Prism},\bM)& \simeq \bM(\frakS(R^+)^{\bullet},(E))\quad (\text{by~Lemma~\ref{Lem-CechAlexanderMethod}})\\
            & \simeq \Tot(C^{\bullet,\bullet})\quad (\text{by~Corollary~\ref{Cor-Eilenberg–Zilber}})\\
            & \simeq \Tot(C^{\bullet,0}\{-\bullet\})\quad (\text{by~Lemma~\ref{Lem-ArithGeo-to-ArithGeoHiggs}})\\
            & \simeq \HIG(H,\theta_H,\phi_H)\quad (\text{by~Lemma \ref{Lem-Rho}}).
        \end{split}
    \end{equation*}
    Then we can conclude by noting that all constructions above are functorial in $\bM$.
  \end{proof}

\section{Hodge--Tate crystals as generalised representations}\label{Sec-EtaleComparison}
 In this section, we work with smooth $p$-adic formal schemes over $\calO_K$ with rigid analytic generic fiber $X$ to prove the following result.
 \begin{thm}\label{Thm-EtaleRealisation}
   There exists a natural equivalence of categories 
   \[\rL:\Vect((\frakX)_{\Prism}^{\perf},\overline \calO_{\Prism}[\frac{1}{p}])\to \Vect(X_{\proet},\OX),\]
   which preserves ranks, tensor products and dualities, such that for any rational Hodge--Tate crystal $\bM$ on $(\frakX)_{\Prism}^{\perf}$, we have a quasi-isomorphism
   \[\rR\Gamma((\frakX)_{\Prism}^{\perf},\bM)\cong \rR\Gamma(X_{\proet},\rL(\bM)),\]
   which is functorial in $\bM$.
 \end{thm}
 \begin{rmk}\label{Rmk-EtaleRealisation}
     By \cite[Thm 3.5.8 and Cor. 3.5.9]{KL16}, one may replace $X_{\proet}$ by $X_v$ in Theorem \ref{Thm-EtaleRealisation}. 
 \end{rmk}
 
   Now, we focus on the proof of Theorem \ref{Thm-EtaleRealisation}. We first construct the desired equivalence. Let $X_{v,\aff,\perf}$ be the site of affinoid perfectoid spaces over $X$ with the $v$-topology. As in Example \ref{Exam-GRep}, we see from \cite[Thm. 3.5.8]{KL16} that there is an equivalence of categories
   \[\Vect(X_{\proet},\OX)\simeq \Vect(X_{v}^{\aff,\perf},\calO_X),\]
   where $X_v^{\aff,\perf}$ denotes the sub-site of $X_v$ consisting of affinoid perfectoids.
   We then construct the desired functor $\rL$ as follows.
   \begin{construction}\label{Construction-EtaleRealisation-I}
      For any $ U = \Spa(S,S^+)\in X_{v,\aff,\perf}$, we can assign it to a unique perfect prism $\calA_U:=(\Ainf(S^+),\Ker(\theta:\Ainf(S^+)\to S^+)) \in (\frakX)_{\Prism}^{\perf}$ by using \cite[Thm. 3.10]{BS22}. Via this functor, we have 
      \[\OX(U) = S = \overline \calO_{\Prism}[\frac{1}{p}](\calA_U),\]
      and then get a natural functor 
      \[\rL: \Vect((\frakX)_{\Prism}^{\perf},\overline \calO_{\Prism}[\frac{1}{p}]) \to \Vect(X_{v,\aff,\perf},\calO_X)\]
      such that for any $\bM\in\Vect((\frakX)_{\Prism}^{\perf},\overline \calO_{\Prism}[\frac{1}{p}])$ and for any affinoid perfectoid space $U\in X_{\proet}$, we have 
      \[\rL(\bM)(U) = \bM(\calA_U).\]
      Note that by \cite[Thm 3.5.8]{KL16}, the presheaf $\rL(\bM)$ is indeed a sheaf.
   \end{construction}
   
 Now let us fix an \'etale covering $\{\frakX_i = \Spf(R_i^+)\to \frakX\}_{i\in I}$ of small affine $\frakX_i$'s. Let $\calA_i:= (\Ainf(\widehat R_{i,C,\infty}^+),(\xi))$ be as in Construction \ref{Construction-CoverPrism} with $X_{i,C,\infty} = \Spa(\widehat R_{i,C,\infty},\widehat R_{i,C,\infty}^+)$. Then $\{X_{i,C,\infty}\to X\}_{i\in I}$ is a pro-\'etale covering of $X$. We claim that $\{\calA_i\}_{i\in I}$ forms a cover of the final object of $\Sh((\frakX)_{\Prism}^{\perf})$.
   
   Indeed, for any $(A,I)\in (\frakX)_{\Prism}^{\perf}$, using the separatedness of $\frakX$, we see that $\Spf(A/I)\times_{\frakX}\frakX_i$ is affine and thus denote the corresponding ring of regular functions by $\overline A_i$. Then $\Spf(\overline A_i)\to \Spf(A/I)$ is an ($p$-complete) \'etale morphism. In particular, by Elkik's algebraization theorem and \cite[Corollay 2.1.6]{CS}, $\overline A_i$ is a perfectoid ring. By deformation theory, $\overline A_i$ admits a unique lifting $A_i$ over $A$. Then \cite[Lem. 2.18]{BS22} shows that there exists a unique $\delta$-structure on $A_i$ which is compatible with the one on $A$ such that $(A_i,IA_i)$ is a perfect prism in $(\frakX_i)_{\Prism}^{\perf}$. As there exist finitely many $i$'s such that $\{(A,I)\to (A_i,IA_i)\}$ forms a cover of $(A,I)$, the claim now follows from Lemma \ref{Lem-CoverPrism} (2).

     To see $\rL$ is an equivalence, we have to construct the quasi-inverse $\rL^{-1}$ of $\rL$. 
   \begin{construction}\label{Construction-EtaleRealisation-II}
      For any $ \calA = (A,I)\in (\frakX)_{\Prism}^{\perf}$ such that $I\neq (p)$ (in fact, it suffices to consider the perfect prisms appearing in the \v Cech nerve corresponding to the cover above), we can assign it to an affinoid perfectoid space $U_{\calA}$ over $X$ by setting  \[U_{\calA}:=\Spa((A/I)[\frac{1}{p}],(A/I)[\frac{1}{p}]^{+})\]
      by using \cite[Thm. 3.10]{BS22} and \cite[Lem. 3.21]{BMS18}, where $(A/I)[\frac{1}{p}]^{+}$ denotes the $p$-adic completion of the integral closure of $A/I$ in $(A/I)[\frac{1}{p}]$. Via this functor, we have 
      \[\calO_X(U_{\calA}) = (A/I)[\frac{1}{p}] = \overline \calO_{\Prism}[\frac{1}{p}](\calA).\]

      So we can get a natural functor 
      \[\rL^{-1}: \Vect(X_{v,\aff,\perf},\calO_X) \to \Vect((\frakX)_{\Prism}^{\perf},\overline \calO_{\Prism}[\frac{1}{p}]).\]
      More precisely, given $\calL\in\Vect(X_{v,\aff,\perf},\calO_X)$, for any perfect prism $\calA=(A,I\neq (p))\in (\frakX)_{\Prism}^{\perf}$, we define 
      \[\rL^{-1}(\calL)(\calA) = \calL(U_{\calA}),\]
      while for any perfect crystalline prism $\calA = (A,(p))\in (\frakX)_{\Prism}^{\perf}$, we define 
      \[\rL^{-1}(\calL)(\calA) = 0.\]
   \end{construction}
   It is easy to see that $\rL^{-1}$ is the quasi-inverse of $\rL$ as desired.
 
   \begin{rmk}
     A similar construction of $\rL$ also appeared in the work of Morrow--Tsuji \cite{MT}. In loc.cit., they worked with smooth formal schemes $\frakX$ over $\calO_C$ and established an equivalence between the category of prismatic crystals and the category of relative Breuil--Kisin--Fargues modules on $X_{\proet}$.
   \end{rmk}
   
   In order to complete the proof of Theorem \ref{Thm-EtaleRealisation}, it remains to show that
   \[\rR\Gamma((\frakX)_{\Prism}^{\perf},\bM)\simeq \rR\Gamma(X_{\proet},\rL(\bM)).\]

   Note that $\calA_i= (\Ainf(\widehat R_{i,C,\infty}^+),(\xi))$ corresponds to $X_{i,C,\infty}$ via the functors in Construction \ref{Construction-EtaleRealisation-I} and Construction \ref{Construction-EtaleRealisation-II}. 
   \begin{lem}\label{Lem-Fibreproduct}
      For any finite subset $J = \{j_1,\dots,j_r\}\subset I$, the fibre product $\calA_J = \calA_{j_1}\times\cdots\times\calA_{j_r}$ exists in $(\frakX)_{\Prism}^{\perf}$ such that 
      \[X_{j_1,C,\infty}\times_X\cdots\times_XX_{j_r,C,\infty} = U_{\calA_J}.\]
   \end{lem}
   \begin{proof}
      Note that $\calA_J$ should be the initial object of the category of prisms $\calB = (B,J)\in (\frakX)_{\Prism}^{\perf}$ together with morphisms $\calA_j\to\calB$ for all $j\in J$. For any such a prism $\calB$, the structure morphism $\Spf(B/J)\to\frakX$ then factors through $\frakX_J \to \frakX$, where $\frakX_J = \Spf(R_J^+) = \frakX_{j_1}\times_{\frakX}\cdots\times_{\frakX}\frakX_{j_r}$ (as $\frakX$ is separated). Put $\widehat R_{j,C,\infty,J}^+:=\widehat R_{j,C,\infty}^+\widehat \otimes_{R_j^+}R_J^+$. Then it is perfectoid (which plays the role of ``$\widehat R_{C,\infty}^+$'' for $R_J^+$ instead of $R_j^+$ with respect to the induced framing from $R_i^+$) and hence induces a perfect prism $\calA_{j,J}:=(A_{j,J},IA_{j,J})\in (R_J^+)_{\Prism}^{\perf}$ by \cite[Thm. 3.10]{BS22}. By construction, the given morphism $\calA_i\to\calB$ has to factor over $\calA_{j}\to\calA_{j,J}$ for each $j\in J$. To conclude, it is enough to show the fibre product $\calA_{j_1,J}\times\cdots\times\calA_{j_r,J}$ exists in $(R_J^+)_{\Prism}^{\perf}$ as it is exactly $\calA_J$ as desired. Using \cite[Thm. 3.10]{BS22}, it is enough to show that the category of perfectoid rings over 
      \[S^+:=\widehat R_{j_1,C,\infty,J}^+\widehat \otimes_{R_J^+}\cdots\widehat \otimes_{R_J^+}\widehat R_{j_r,C,\infty,J}^+\]
      has an initial object $S_J^+$, which follows from \cite[Cor. 7.3]{BS22} as $S^+$ is quasi-regular semi-perfectoid. So we conclude the existence of $\calA_J$. Moreover, by constructions above, $U_{\calA_J} = \Spa(S_J[\frac{1}{p}],S_J[\frac{1}{p}]^+)$ is the initial object of the category of perfectoid spaces $Y$ over $X$ which admits a fixed morphism $Y\to X_{j,C,\infty}$ for all $j\in J$. This shows that $U_{\calA_J} = X_{j_1,C,\infty}\times_X\cdots\times_XX_{j_r,C,\infty}$ as desired.
   \end{proof}
    
  \begin{proof}[Proof of Theorem \ref{Thm-EtaleRealisation}]
  Now we have proved the first part of Theorem \ref{Thm-EtaleRealisation}. Thanks to Lemma \ref{Lem-Fibreproduct} above, for any $\bM\in \Vect((\frakX)_{\Prism}^{\perf},\overline \calO_{\Prism}[\frac{1}{p}])$, by Construction \ref{Construction-EtaleRealisation-I} and Construction \ref{Construction-EtaleRealisation-II}, we have the same \v Cech complex
    \[\check \rC^{\bullet}(\{\calA_i\}_{i\in I},\bM) = \check \rC^{\bullet}(\{X_{i,C,\infty}\}_{i\in I},\rL(\bM)).\]
    By Lemma \ref{Lem-Vanishing} and \cite[Prop. 2.3]{LZ}, using \v Cech-to-derived spectral sequence, we get a quasi-isomorphism 
    \[\rR\Gamma((\frakX)_{\Prism}^{\perf},\bM)\simeq \rR\Gamma(X_{\proet},\rL(\bM))\]
    as desired, which completes the proof of Theorem \ref{Thm-EtaleRealisation}.
    \end{proof}
   
   \begin{exam}\label{Exam-PerfHTasGeneRep}
       Assume $\frakX = \Spf(R^+)$ is small with respect to the framing $\Box$ (cf. \S \ref{Intro-Notations}). Then one can give a direct proof of Theorem \ref{Thm-EtaleRealisation} as follows: 
       
       Keep notations in Construction \ref{Construction-CoverPrism} and let $X_{C,\infty} = \Spa(\widehat R_{C,\infty},\widehat R_{C,\infty}^+)$. Then $X_{C,\infty}\in X_{\proet}$ is affinoid perfectoid and is a Galois cover of $X$ with Galois group $\Gamma(\overline K/K)$ (cf. Notation \ref{Notation-LocalChart-II}). Note that $\calA_{X_{C,\infty}} = (\Ainf(\widehat R_{C,\infty}^+),(\xi))$, which is a cover of the final object of $\Sh((\frakX)_{\Prism}^{\perf})$, following from Lemma \ref{Lem-CoverPrism} (2). Let $X_{C,\infty}^{\bullet+1/X}$ (resp. $\calA_{X_{C,\infty}}^{\bullet+1}$) be the induced \v Cech nerve of $X_{C,\infty}$ (resp. $\calA_{X_{C,\infty}}$). Then the proof of Lemma \ref{Lem-EvaluateGRep} combined with Lemma \ref{Lem-Structure} (3) shows that 
      \[\OX(X_{C,\infty}^{\bullet+1/X}) = \rC(\Gamma(\overline K/K)^{\bullet},\widehat R_{C,\infty}) = \overline \calO_{\Prism}(\calA_{X_{C,\infty}}^{\bullet+1}).\]
       Note that \cite[Prop. 2.7]{BS23} also applies to $(\frakX)_{\Prism}^{\perf}$. We obtain a natural equivalence
       \[\Vect((\frakX)_{\Prism}^{\perf},\overline \calO_{\Prism}[\frac{1}{p}]) \simeq \Strat(\rC(\Gamma(\overline K/K)^\bullet,\widehat R_{C,\infty}))\simeq \Rep_{\Gamma(\overline K/K)}(\widehat R_{C,\infty}).\]
       Then we show that $\rL$ is an equivalence by applying Lemma \ref{Lem-EvaluateGRep}. Clearly, the above construction preserves ranks, tensor products and dualities. 
       
       By Lemma \ref{Lem-Vanishing} and \v Cech-to-derived spectral sequence, for any rational Hodge--Tate crystal $\bM$ on $(\frakX)_{\Prism}^{\perf}$ with induced $\widehat R_{C,\infty}^+$-representation $M$ of $\Gamma(\overline K/K)$, we have quasi-isomorphism
       \[\iota_{\Box}: \rR\Gamma((\frakX)_{\Prism}^{\perf},\bM)\simeq \rR\Gamma(\Gamma(\overline K/K),M)\simeq \rR\Gamma(X_{\proet},\rL(\bM)),\]
       where the second quasi-isomorphism follows from Lemma \ref{Lem-EvaluateGRep} again.
       We complete the proof.
       
       At the end of this example, we show that the quasi-isomorphism $\iota_{\Box}$ is actually independent of the choice of the framing $\Box$ on $\frakX$. Indeed, let $\Box_1$ and $\Box_2$ be two framings on $\frakX$ with corresponding $X_{C,\infty}^{\Box_1}$ and $X_{C,\infty}^{\Box_2}$, respectively. Let $U_{i,j}$ be the fibre products of $(i+1)$ copies of $X_{C,\infty}^{\Box_1}$ and $(j+1)$ copies of $X_{C,\infty}^{\Box_2}$ over $X$. Then each $U_{i,j}$ is affinoid perfectoid over $X$ and then we get an isomorphism of bicomsimplicial $R$-modules
       \[\bM(\calA_{U_{\bullet,\bullet}})\cong \rL(\bM)(U_{\bullet,\bullet}).\]
       By noting that $U_{\bullet,0}$ (resp. $U_{0,\bullet}$) is the \v Cech nerve associated to $X_{C,\infty}^{\Box_1}$ (resp. $X_{C,\infty}^{\Box_2}$), we have a natural commutative diagram
       \begin{equation*}
          \xymatrix@C=0.45cm{ \bM(\calA^{\bullet+1/X}_{X_{C,\infty}^{\Box_1}})\ar[d] \ar[r]& \bM(\calA_{U_{\bullet,\bullet}})\ar[d]&\ar[l] \bM(\calA^{\bullet+1/X}_{X_{C,\infty}^{\Box_2}})\ar[d]\\
           \rL(\bM)(U_{\bullet,0})\ar[r]&\rL(\bM)(U_{\bullet,\bullet})&\ar[l]\rL(\bM)(U_{0,\bullet}).}
       \end{equation*}
       Thanks to \cite[Prop. 2.3]{LZ} and Lemma \ref{Lem-Vanishing}, all horizontal arrows induce quasi-isomorphisms of corresponding total complexes and hence so is the middle vertical arrow (as the left and the right arrows induce $\iota_{\Box_1}$ and $\iota_{\Box_2}$ on totalizations). Therefore, we get a natural comparison between $\iota_{\Box_1}$ and $\iota_{\Box_2}$.
   \end{exam}

\section{Inverse Simpson functor for enhanced Higgs bundles}\label{inverse Simpson}\label{Sec-Global}
 We still work with smooth $p$-adic formal schemes over $\calO_K$ in this section.
 \begin{notation}\label{Notation-lambda}
 Note that both $\xi$ and $E([\pi^{\flat}])$ are generators of the principal ideal $\Ker(\theta:\Ainf\to \calO_C)$. Let $\lambda$ be the image of $\frac{\xi}{E([\pi^{\flat}])}$ under the surjection $\theta:\Ainf\to\calO_C$. Clearly, $\lambda$ belongs to $\calO_{\widehat K_{\cyc,\infty}}$ and is indeed a unit. For any $g\in G_K$, we see that 
   \begin{equation}\label{Equ-Lambda}
       \begin{split}
           g(\lambda) = & \theta(g(\frac{[\epsilon]-1}{E([\pi^{\flat}])([\epsilon^{\frac{1}{p}}]-1)})) \\ 
           =& \theta(\frac{[\epsilon]^{\chi(g)}-1}{[\epsilon]-1}\frac{[\epsilon^{\frac{1}{p}}]-1}{[\epsilon^{\frac{1}{p}}]^{\chi(g)}-1}\frac{\xi}{E([\pi]^{\flat})}\frac{E([\pi]^{\flat})}{E([\epsilon]^{c(g)}[\pi]^{\flat})})\\
           =& \chi(g)\frac{\zeta_p-1}{\zeta_p^{\chi(g)}-1}\lambda\theta(\frac{E([\pi^{\flat}])}{E([\pi]^{\flat})+E'([\pi^{\flat}])[\pi^{\flat}]([\epsilon]^{c(g)}-1)})\\
            = & \chi(g)\frac{\zeta_p-1}{\zeta_p^{\chi(g)}-1}\lambda(1+\pi E'(\pi)\theta(\frac{[\epsilon]^{c(g)}-1}{[\epsilon]-1}\frac{[\epsilon]-1}{E([\pi^{\flat}])}))^{-1}\\
            = & \chi(g)\frac{\zeta_p-1}{\zeta_p^{\chi(g)}-1}\lambda(1+\pi E'(\pi)(\zeta_p-1)\lambda c(g)))^{-1},
       \end{split}
   \end{equation}
   where $c:G_K\to \Zp$ is defined by $g(\pi^{\flat}) = \epsilon^{c(g)}\pi^{\flat}$ as in \S \ref{Intro-Notations}. For the third equality above, we used Taylor's expansion \[E([\epsilon]^{c(g)}[\pi]^{\flat})=E([\pi]^{\flat})+E'([\pi^{\flat}])[\pi^{\flat}]([\epsilon]^{c(g)}-1)+E^{''}([\pi^{\flat}])[\pi^{\flat}]^2([\epsilon]^{c(g)}-1)^2/2!+\cdots\] and the facts that $([\epsilon]-1)|([\epsilon]^{c(g)}-1)$ and $\theta([\epsilon]-1)=0$.
 \end{notation}
 \begin{dfn}\label{Dfn-EnhancedHiggsBundle}
   By an \emph{enhanced Higgs bundle} of rank $l$ on $\frakX_{\et}$, we mean a triple $(\calH,\theta_{\calH},\phi_{\calH})$ consisting of 
   \begin{enumerate}
       \item a locally finite free $\calO_{\frakX}$-module $\calH$ of rank $l$ on $\frakX$ together with a nilpotent Higgs field $\theta_{\calH}$ on $\calH$; denote by $\HIG(\calH,\theta_{\calH})$ the induced Higgs complex of $(\calH,\theta_{\calH})$;
       
       \item an $\calO_{\frakX}[\frac{1}{p}]$-linear endomorphism $\phi_{\calH}$ of $\phi_{\calH}$ satisfying 
       \begin{enumerate}
           \item $\lim_{n\to+\infty}\prod_{i=0}^{n-1}(\phi_{\calH}+iE'(\pi)) = 0$ for the $p$-adic topology,
           
           \item $[\phi_{\calH},\theta_{\calH}] = -E'(\pi)\theta_{\calH}$; that is, $\phi_{\calH}$ makes the following diagram
           \begin{equation}\label{Diag-EnhancedHiggsComplex-Global}
              \xymatrix@C=0.45cm{
                \calH\ar[r]^{\theta_{\calH}\qquad\quad}\ar[d]_{\phi_{\calH}}&H\otimes_{\calO_{\frakX}}\widehat \Omega^1_{\frakX}\{-1\}\ar[d]_{\phi_{\calH}+E'(\pi)\id_{\calH}} \ar[r]&\cdots\ar[r]&\calH\otimes_{\calO_{\frakX}}\widehat \Omega^d_{\frakX}\{-d\}\ar[d]_{\phi_{\calH}+dE'(\pi)\id_{\calH}}\\
                \calH\ar[r]^{\theta_{\calH}\qquad\quad}&\calH\otimes_{\calO_{\frakX}}\widehat \Omega^1_{\calO_{\frakX}}\{-1\} \ar[r]&\cdots\ar[r]&\calH\otimes_{\calO_{\frakX}}\widehat \Omega^d_{\frakX}\{-d\}
              }
          \end{equation}
          commute. We denote by $\HIG(\calH,\theta_{\calH},\phi_{\calH})$ the total complex of the bicomplex (\ref{Diag-EnhancedHiggsComplex-Global}) and hence
          \[\HIG(\calH,\theta_{\calH},\phi_{\calH}) = \fib(\HIG(\calH,\theta_{\calH})\xrightarrow{\phi_{\calH}}\HIG(\calH,\theta_{\calH})).\]
       \end{enumerate}
   \end{enumerate}
   We denote by $\HIG^{\nil}_*(\frakX)$ the category of enhanced Higgs bundles on $\frakX_{\et}$. One can similarly define the category $\HIG^{\nil}_*(X)$ of enhanced Higgs bundles on $X_{\et}$.
 \end{dfn}
 \begin{rmk}\label{Rmk-EnhancedHiggsBundle}
  The nilpotency of $\theta_{\calH}$ in Definition \ref{Dfn-EnhancedHiggsBundle} can be deduced from the other assumptions (cf. Remark \ref{Rmk-EnhancedHiggsMod-Local}).
 \end{rmk}
 This section is devoted to proving the following theorem.
 \begin{thm}\label{Thm-HTasHiggs-Global}
   Let $\frakX$ be a smooth $p$-adic formal scheme over $\calO_K$ with rigid analytic generic fiber $X$.
   Then there exists an equivalence of categories
   \[\rho:\Vect((\frakX)_{\Prism},\overline \calO_{\Prism}[\frac{1}{p}])\simeq \HIG^{\nil}_*(X),\]
   which preserves ranks, tensor products and dualities. Moreover, this equivalence fits into the following commutative diagram of functors:
   \begin{equation}\label{Diag-CommutativeDiagram}
       \xymatrix@C=0.45cm{
         \Vect((\frakX)_{\Prism},\overline \calO_{\Prism}[\frac{1}{p}])\ar[rrrr]^{\rho}\ar[d]^{\Res}&&&&\HIG^{\nil}_*(X)\ar[d]^{\rF}\ar[lld]_{\rF_{S}}\\
         \Vect((\frakX)_{\Prism}^{\perf},\overline \calO_{\Prism}[\frac{1}{p}])\ar[rr]^{\simeq}
         &&\Vect(X_{\proet},\OX)\ar[rr]^{\simeq}&&\HIG_{G_K}(X_C),
       }
   \end{equation}
   where $\Res$ is induced by inclusion of sites $(\frakX)_{\Prism}^{\perf}\subset(\frakX)_{\Prism}$, $\rF$ will be defined in Construction \ref{Construction-F}. All arrows in Diagram \ref{Diag-CommutativeDiagram} are fully faithful functors while the bottows arrows are given in Theorem \ref{Thm-EtaleRealisation} and Theorem \ref{Thm-Simpson}. In particular, we obtain the fully faithful inverse Simpson functor 
   \[\rF_S:\HIG^{\nil}_*(X)\to \Vect(X_{\proet},\OX)\] 
   defined as the composite of the functor $\rF$ and the equivalence $\HIG_{G_K}(X_C)\xrightarrow{\simeq}\Vect(X_{\proet},\OX)$.
 \end{thm}

    The equivalence $\rho:\Vect((\frakX)_{\Prism},\overline \calO_{\Prism}[\frac{1}{p}])\simeq \HIG^{\nil}_*(X)$ in Theorem \ref{Thm-HTasHiggs-Global} has been established when $\frakX = \Spf(R^+)$ is small affine and upgrades to the integral level in this case (cf. Theorem \ref{Thm-HTasHiggs-Local}). However, the constructions there depend on the choice of framing on $R^+$ and look difficult to glue together. The problem can be solved by showing that both $\rF$ and $\Res$ are fully faithful making Diagram (\ref{Diag-CommutativeDiagram}) commute locally. Then by some standard argument, the local equivalence $\rho_{\square}:\Vect((R^+)_{\Prism},\overline \calO_{\Prism}[\frac{1}{p}])\simeq \HIG^{\nil}_*(R)$ in Theorem \ref{Thm-HTasHiggs-Local} is independent of the choice of framing and hence glue to a global equivalence. We will follow this strategy in sequels.
 
 \subsection{Local version of Theorem \ref{Thm-HTasHiggs-Global} and the construction of $\rF$}
  In this subsection, we assume $\frakX = \Spf(R^+)$ is small affine and keep notations in Notation \ref{Notation-LocalChart-II}. Then by Example \ref{Exam-PerfHTasGeneRep} and Lemma \ref{Lem-EvaluateGRep}, we have equivalences of categories
  \[\Vect((\frakX)_{\Prism}^{\perf},\overline \calO_{\Prism}[\frac{1}{p}])\simeq \Rep_{\Gamma(\overline K/K)}(\widehat R_{C,\infty})\simeq \Vect(X_{\proet},\OX).\]
  Recall we also have the equivalence 
  \[\rho_{\square}:\Vect((R^+)_{\Prism},\overline \calO_{\Prism}[\frac{1}{p}])\simeq \HIG_*^{\nil}(R).\]
  The purpose in this subsection is to give an explicit description of $\Res$ as a functor 
  \[\rV: \HIG_*^{\nil}(R)\to \Rep_{\Gamma(\overline K/K)}(\widehat R_{C,\infty}),\]
  and then describe $\rF$ by using the local Simpson correspondence i.e. Theorem \ref{Thm-SimpsonLocal}.
  \begin{lem}\label{Lem-Restriction}
    For any $(H,\theta_H,\phi_H)\in\HIG^{\nil}_*(R)$ with induced $\widehat R_{C,\infty}$-representation $V:=\rV(H,\theta_H,\phi_H)$ of $\Gamma(\overline K/K)$, we have $V = H\otimes_R\widehat R_{C,\infty}$ such that for any $1\leq i\leq d$, any $g\in G_K$ and any $x\in H$, 
    \[\gamma_i(x) = \exp(-(\zeta_p-1)\lambda\theta_i)(x)~\text{and}~g(x) = (1+\pi E'(\pi)(\zeta_p-1)\lambda c(g))^{-\frac{\phi_H}{E'(\pi)}}(x).\]
  \end{lem}
  \begin{proof}
    Let $A^{\bullet}$ be the cosimplicial ring in Convention \ref{Convention-CosimplicialRingA}. By Lemma \ref{Lem-HTasStrat} and Example \ref{Exam-PerfHTasGeneRep}, it suffices to compare stratifications induced by $(H,\theta_H,\phi_H)$ and $V$ with respect to $A^{\bullet}$ and $\rC(\Gamma(\overline K/K)^{\bullet},\widehat R_{C,\infty})$, respectively. Note that the natural map
    \[\Strat(A^{\bullet})\to\Strat(\rC(\Gamma(\overline K/K)^{\bullet},\widehat R_{C,\infty}))\]
    is induced by the morphism of cosimplicial rings $\frakS(R)^{\bullet}\to \Ainf(\widehat R_{C,\infty}^+)^{\bullet}$. By Lemma \ref{Lem-Structure} (3) and Proposition \ref{Prop-Structure}, we have to determine the functions in $\rC(\Gamma(\overline K/K),\widehat R_{C,\infty})$ induced by $X_1,Y_{1,1},\dots,Y_{d,1}$.
    \begin{lem}\label{Lem-Function}
       For any $\underline n = (n_1,\dots,n_d)\in\bN^d$, any $1\leq s\leq d$, and any $g\in G_K$, as functions in $\rC(\Gamma(\overline K/K),\widehat R_{C,\infty})$, we have 
       \[X_1(\gamma_1^{n_1}\cdots\gamma_d^{n_d}g) = -\pi(\zeta_p-1)\lambda c(g)\]
       and  
       \[Y_{s,1}(\gamma_1^{n_1}\cdots\gamma_d^{n_d}g) = -n_s(\zeta_p-1)\lambda.\]
    \end{lem}
    \begin{proof}
      Recall $X_1 = \frac{u_0-u_1}{E(u_0)}$ and $Y_{s,1} = \frac{T_{s,1}-T_{s,0}}{E(u_0)T_{s,0}}$. Regard $(\rC(\Gamma(\overline K/K),\Ainf(\widehat R_{C,\infty}^+)),(\xi))$ as a perfect prism in $(R^+)_{\Prism}$ where $\xi$ is the constant function. Then we have two morphisms from $(\Ainf(\widehat R_{C,\infty}^+),(\xi))$ to $(\rC(\Gamma(\overline K/K),\Ainf(\widehat R_{C,\infty}^+)),(\xi))$: one sends  $x\in \Ainf(\widehat R_{C,\infty}^+)$ to the constant function corresponding to $x$, the other sends $x\in \Ainf(\widehat R_{C,\infty}^+)$ to the function sending $g\in \Gamma(\overline K/K)$ to $g(x)$. Then we get a morphism from $\Ainf(\widehat R_{C,\infty}^+)^1,(\xi))$ to $(\rC(\Gamma(\overline K/K),\Ainf(\widehat R_{C,\infty}^+)),(\xi))$, which induces the isomorphism $\Ainf(\widehat R_{C,\infty}^+)^1/\xi[1/p]\cong \rC(\Gamma(\overline K/K),\widehat R_{C,\infty})$ in Lemma \ref{Lem-Structure} (3).
      
      In particular, the image of  $u_0$ (resp. $T_{s,0}$) in $\rC(\Gamma(\overline K/K),\Ainf(\widehat R_{C,\infty}^+))$ is the constant function with the value $[{\pi^{\flat}}]$ (resp. $[T_s^{\flat}]$), where $T_s^{\flat} = (T_s,T_s^{\frac{1}{p}},\dots)\in \widehat R_{C,\infty}^{+\flat}$). And the image of
      $u_1$ (resp. $T_{s,1}$) in $\rC(\Gamma(\overline K/K),\Ainf(\widehat R_{C,\infty}^+))$ is the evaluation function on $\Gamma(\overline K/K)$ at $[\pi^{\flat}]$ (resp. $[T_s^{\flat}]$). Therefore, we conclude that
      \begin{equation*}
          \begin{split}
              X_1(\gamma_1^{n_1}\cdots\gamma_d^{n_d}g) = \theta(\frac{[\pi^{\flat}](1-[\epsilon]^{c(g)})}{E([\pi^{\flat}])}) = -\pi(\zeta_p-1)\lambda c(g)
          \end{split}
      \end{equation*}
      and that
      \begin{equation*}
          \begin{split}
              Y_{s,1}(\gamma_1^{n_1}\cdots\gamma_d^{n_d}g) = \theta(\frac{[T_s^{\flat}](1-[\epsilon]^{n_s})}{E([\pi^{\flat}])[T_s^{\flat}]}) = -n_s(\zeta_p-1)\lambda.
          \end{split}
      \end{equation*}
      as desired. 
    \end{proof}
    Now, we continue the proof of Lemma \ref{Lem-Restriction}. Let $\varepsilon_H: H\otimes_{p_0}A^1\cong H\otimes_{p_1}A^1$ be the stratification corresponding to $(H,\theta_H,\varphi_H)$. The base change $\varepsilon_H\otimes\rC(\Gamma(\overline K/K),\widehat R_{C,\infty})$ of $\varepsilon_H$ to $\Ainf(\widehat R_{C,\infty}^+)^1/\xi[1/p]\cong\rC(\Gamma(\overline K/K),\widehat R_{C,\infty})$ is the stratification corresponding to the representation $V$. For any $\gamma_1^{n_1}\cdots\gamma_d^{n_d}g\in \Gamma(\overline K/K)$, the base change of $\varepsilon_H\otimes\rC(\Gamma(\overline K/K),\widehat R_{C,\infty})$ along the evaluation map $\rC(\Gamma(\overline K/K),\widehat R_{C,\infty})\to \widehat R_{C,\infty}$ at $\gamma_1^{n_1}\cdots\gamma_d^{n_d}g$ is then the (linearization) of the $\gamma_1^{n_1}\cdots\gamma_d^{n_d}g$-action on $V$. Now by Proposition \ref{Prop-HTvsHiggs-Local}, especially (\ref{Equ-Stratification}), we deduce from Lemma \ref{Lem-Function} (i.e. evaluate $\varepsilon_H$ at $\gamma_1^{n_1}\cdots\gamma_d^{n_d}g$) that for any $x\in H$,
    \begin{equation}\label{Equ-GammaAction}
        \gamma_1^{n_1}\cdots\gamma_d^{n_d}g(x) = \exp(\sum_{i=1}^d-n_i(\zeta_p-1)\lambda\theta_i)(1+\pi E'(\pi)(\zeta_p-1)\lambda c(g))^{-\frac{\phi_H}{E'(\pi)}}(x).
    \end{equation}
    This completes the proof.
  \end{proof}
  \begin{rmk}\label{Rmk-Restriction}
    Since $\theta_H$ and $\phi_H$ are both defined over $R$, we see that $H_C:=H\otimes_RR_C$ is itself stable under the action of $\Gamma(\overline K/K)$. Noting that $\gamma_i$ acts on $H_C$ unipotently, we see that $H_C$ is the $R_C$-representation in $\Rep_{\Gamma(\overline K/K)}^{\rm uni}(R_C)$ associated to $\rV(H,\theta_H,\phi_H)$ under the equivalence in Proposition \ref{Prop-UnipotentEquiv} (and Remark \ref{Rmk-UnipotentEquiv}).
  \end{rmk}
  Now we can describe the functor $\rF$ in the case for $\frakX = \Spf(R^+)$.
  \begin{construction}\label{Construction-F-Local}
     Define $\rF: \HIG^{\nil}_*(R)\to \HIG_{G_K}(R_C)$ as the composition 
     \[\HIG^{\nil}_*(R)\xrightarrow{\rV} \Rep_{\Gamma(\overline K/K)}(\widehat R_{C,\infty})\xrightarrow{\text{Theorem \ref{Thm-SimpsonLocal}}}\HIG_{G_K}(R_C).\]
     In particular, it makes the Diagram (\ref{Diag-CommutativeDiagram}) commute.
  \end{construction}
  \begin{lem}\label{Lem-F-Local}
    For any $(H,\theta_H,\phi_H)\in \HIG^{\nil}_*(R)$, we have 
    \[\rF(H,\theta_H,\phi_H) = (H\otimes_RR_C,-(\zeta_p-1)\lambda\theta_H)\]
    with the $G_K$-action such that for any $x\in H$ and any $g\in G_K$,
    \[g(x) = (1+\pi E'(\pi)(\zeta_p-1)\lambda c(g))^{-\frac{\phi_H}{E'(\pi)}}(x).\]
  \end{lem}
  \begin{proof}
    This follows from Remark \ref{Rmk-SimpsonLocal} and Remark \ref{Rmk-Restriction} immediately.
  \end{proof}
  \begin{rmk}\label{Rmk-F-Local}
    Note that $\lambda \in \calO_{\widehat K_{\cyc,\infty}}$. Thanks to Remark \ref{Rmk-UnipotentEquiv} and Remark \ref{Rmk-Simpson}, one can use $\widehat L$ and $\Gal(L/K)$ instead of $C$ and $G_K$ in Lemma \ref{Lem-Restriction}, Remark \ref{Rmk-Restriction} and Lemma \ref{Lem-F-Local} for any Galois extension $L/K$ in $\overline K$ containing $K_{\cyc,\infty}$.
  \end{rmk}
  
  \begin{prop}\label{Prop-F-FullyFaithful}
    The functors $\rF$, $\rV$ and $\Res$ are both fully faithful.
  \end{prop}
  \begin{proof}
    We only need to show $\rV$ is fully faithful. Granting this, we get the full faithfulness of $\rF$ by Theorem \ref{Thm-SimpsonLocal} and then the full faithfulness of $\Res$ as Diagram (\ref{Diag-CommutativeDiagram}) commutes in this case. Moreover, by \'etale descent, we can further assume $H$ is finite free over $R$. Thanks to Remark \ref{Rmk-F-Local}, we may work with $\widehat K_{\cyc,\infty}$ and $\Gamma(K_{\cyc,\infty}/K)$ instead of $C$ and $\Gamma(\overline K/K)$ in the following arguments. For simplicity, we put $L=K_{\cyc,\infty}$.
    
    Let $H_{\widehat L}\in\Rep_{\Gamma(L/K)}(R_{\widehat L})$ be as in Remark \ref{Rmk-Restriction}. By Proposition \ref{Prop-UnipotentEquiv}, if we put $V = \rV(H,\theta_H,\phi_H)$, then we have a quasi-isomorphism
    \[\rR\Gamma(\Gamma(L/K),H_{\widehat L})\simeq \rR\Gamma(\Gamma(L/K),V).\]
    
    We claim that $\Gamma(L/K)\cong \Gamma_{\geo}\rtimes\Gal(L/K)$ satisfies Axiom (1)-(4) of Tate--Sen theory for $R_{\widehat L}$ formulated in \cite[\S 5.1]{Por22}. Indeed, since $\Gamma_{\geo}$ acts on $R_{\widehat L}$ trivially, we only need to check $\Gal(L/K)$ satisfies desired axioms. Note that $R^+$ is a topologically free $\calO_K$-module by lifting $\kappa$-basis of $R^+/\pi R^+$ over $\kappa$. By equipping $R_{\widehat L}$ with the supreme norm induced by the corresponding $\calO_K$-basis of $R^+$, it suffices to check that the desired axioms are satisfied for $\widehat L$, which reduces to the example below \cite[Cor. 5.3]{Por22}. In particular, by \cite[Cor. 5.3]{Por22}, we get quasi-isomorphisms
    \[\rR\Gamma(\Gamma(L/K),H_{\widehat L})\simeq \rR\Gamma(\Gamma(L/K),H_{\widehat L}^{\rm la})\simeq \rR\Gamma(\Lie(\Gamma(L/K)),H_{\widehat L}^{\rm la})^{\Gamma(L/K)}, \]
    where $H_{\widehat L}^{\rm la}$ denotes the locally analytic vectors (cf. Definition \ref{defLAV}) in $H_{\widehat L}$ with respect to the action of $\Gamma(L/K)$, which turns out to be $H\otimes_{R}R_{\widehat L}^{\rm la}$ as $\Gamma(L/K)$ acts on $H$ analytically (cf. (\ref{Equ-GammaAction})).
    
    For any $(H_1,\theta_{H_1},\varphi_{H_1}),(H_2,\theta_{H_2},\varphi_{H_2})\in \HIG^{\nil}_*(R)$, we want to show \[\Hom((H_1,\theta_{H_1},\varphi_{H_1}),(H_2,\theta_{H_2},\varphi_{H_2}))=\Hom(\rV(H_1,\theta_{H_1},\varphi_{H_1}),\rV(H_2,\theta_{H_2},\varphi_{H_2})).\] If we write $H=H_1^{\vee}\otimes H_2$, then it suffices to show that 
    \begin{equation}\label{Equ-F-Fullyfaithful}
        H^{\phi_H = 0,\theta_1=\dots=\theta_d = 0} = \rH^0(\Gamma(L/K),H_{\widehat L})=\rH^0(\Lie(\Gamma(L/K)),H\otimes_RR^{\rm la}_{\widehat L})^{\Gamma(L/K)}.
    \end{equation}
    Recall $\Gamma(L/K)\cong (\Zp\gamma_1\oplus\dots\oplus\Zp\gamma_d\oplus\Zp\tau)\rtimes\Gal(K_{\cyc}/K)$ and let $\gamma$ be a topological generator of $\Gal(K_{\cyc}/K)$. Let $\nabla_i=\log\gamma_i$, $\nabla_{\tau}=\log\tau$ and $\nabla_{\gamma}=\frac{\log\gamma}{\log\chi(\gamma)}$. Then we have
    \[\Lie(\Gamma(L/K)) \cong \oplus_{i=1}^d\Qp\nabla_i\oplus\Qp\nabla_{\tau}\oplus\Qp\nabla_{\gamma}\]
    such that $\nabla_i$'s and $\nabla_{\tau}$ commute with each other and $[\nabla_{\gamma},\nabla_{*}] = \nabla_*$ for $*\in\{1,\dots,d,\tau\}$. Note that $\lambda\in \widehat L^{\rm la}\subset R_{\widehat L}^{\rm la}$ by (\ref{Equ-Lambda}). Using (\ref{Equ-Lambda}) and the facts that $c(\tau)=1$, $c(\gamma)=0$ and $\chi(\tau)=1$, it is easy to see that 
    \[\nabla_{\tau}(\lambda) =\lim_{m\to+\infty}\frac{\tau^{p^m}-1}{p^m}(\lambda)= -\pi E'(\pi)(\zeta_p-1)\lambda^2~\text{and}~\nabla_{\gamma}(\lambda)=\frac{1}{\log\chi(\gamma)}\lim_{m\to+\infty}\frac{\gamma^{p^m}-1}{p^m}(\lambda) = \lambda.\]
    
    We claim that 
    \[H^{\phi_H = 0,\theta_1=\dots=\theta_d = 0} \subset \rH^0(\Lie(\Gamma(L/K)),H\otimes_RR^{\rm la}_{\widehat L})^{\Gamma(L/K)}.\]
    Indeed, by (\ref{Equ-GammaAction}), the claim is an immediate consequence of that for any $x\in H$,
     \begin{equation*}
       \begin{split}
           \lim_{n\to+\infty}\frac{\tau^{p^n}-1}{p^n}(x) &= \lim_{n\to+\infty}\frac{1}{p^n}((1+\pi E'(\pi)(\zeta_p-1)\lambda p^n)^{-\frac{\phi_H}{E'(\pi)}}-1)(x)\\
           &= \lim_{n\to+\infty}\sum_{k\geq 1}\prod_{i=0}^{k-1}(\phi_H+iE'(\pi))\frac{(-\pi(\zeta_p-1)\lambda)^kp^{(k-1)n}}{k!}(x)\\
           & = -\pi(\zeta_p-1)\lambda\phi_H(x),
       \end{split}
   \end{equation*}
   and
    \begin{equation*}
       \begin{split}
           \lim_{n\to+\infty}\frac{\gamma_i^{p^n}-1}{p^n} (x)&= \lim_{n\to+\infty}\frac{1}{p^n}(\exp(-p^n(\zeta_p-1)\lambda\theta_i)-1)(x)\\
           &= \lim_{n\to+\infty}\sum_{k\geq 1}\frac{(-(\zeta_p-1)\lambda \theta_i)^kp^{(k-1)n}}{k!}(x)\\
           & = -(\zeta_p-1)\lambda\theta_i(x).
       \end{split}
   \end{equation*}
   
    Moreover, the above argument shows that for $x\in H$, we have 
    \begin{equation*}
        \begin{split}
            &\nabla_i(x) = -\lambda(\zeta_p-1)\theta_i(x),\\
            &\nabla_{\tau}(x) = -\pi\lambda(\zeta_p-1)\phi_H(x);\\
            &\nabla_{\gamma}(x) = 0.
        \end{split}
    \end{equation*}
    
    It remains to prove that
    \[ ((H\otimes_RR^{\rm la}_{\widehat L})^{\nabla_1=\dots=\nabla_{d}=0,\nabla_{\tau}=\nabla_{\gamma}=0})^{\Gamma(L/K)}\subset H^{\phi_H = 0,\theta_1=\dots=\theta_d = 0} .\]
    For this, we have to apply some results in \S\ref{Sect-Sen operator} (whose proofs certainly do not rely on this proposition).
    
    Let $\nabla_{\rm Sen}$ be the operator in the proof of Theorem \ref{Kummer Sen}. Then we have 
    \[
    ((H\otimes_RR^{\rm la}_{\widehat L})^{\nabla_1=\dots=\nabla_{d}=0,\nabla_{\tau}=\nabla_{\gamma}=0})^{\Gamma(L/K)}\subset ((H\otimes_RR^{\rm la}_{\widehat L})^{\nabla_1=\dots=\nabla_{d}=0,\nabla_{\rm Sen}=0})^{\Gamma(L/K)}.
    \]
    The right hand side is equal to  \begin{equation*}
        \begin{split}
        ((H\otimes_RR^{\rm la}_{\widehat L})^{\theta_1=\dots=\theta_{d}=0,\phi_{H}=0})^{\Gal(L/K)}&=(((H\otimes_RR^{\rm la}_{\widehat L})^{\theta_1=\dots=\theta_{d}=0,\phi_{H}=0})^{\Gal(L/K_{\infty})})^{\Gal(L/K_{\cyc})}\\
        &=((H\otimes_RR_{K_{\infty}})^{\theta_1=\dots=\theta_d=0,\phi_H=0})^{\Gal(L/K_{\cyc})}\\
        &=(H^{\theta_1=\dots=\theta_d=0,\phi_H=0}\otimes_RR_{K_{\infty}})^{\Gal(L/K_{\cyc})}
        \end{split}
      \end{equation*}
    where the second equality is due to $R_{\widehat L}^{{\rm la},\Gal(L/K_{\infty})=1}=R_{K_{\infty}}$ by Proposition \ref{R-la} and the third equality is due to that $R\to R_{K_{\infty}}$ is faithfully flat.

    Write $(H^{\theta_1=\dots=\theta_d=0,\phi_H=0}\otimes_RR_{K_{\infty}})^{\Gal(L/K_{\cyc})}$ as 
    \[(H^{\theta_1=\dots=\theta_d=0,\phi_H=0}\otimes_RR_{K_{\infty}})^{\Gal(L/K_{\cyc})} = \cup_m (H^{\theta_1=\dots=\theta_d=0,\phi_H=0}\otimes_RR_{K(\pi_m)})^{\Gal(K(\pi_m,\zeta_{p^m})/K(\zeta_{p^m}))}.\]
    For each $m$, we see that \[(H^{\theta_1=\dots=\theta_d=0,\phi_H=0}\otimes_RR_{K(\pi_m,\zeta_{p^m})})^{\Gal(K(\pi_m,\zeta_{p^m})/K(\zeta_{p^m}))}=H^{\theta_1=\dots=\theta_d=0,\phi_H=0}\otimes_RR_{K(\zeta_{p^m})}\] by Galois descent. Then
    \[\begin{split}
    &(H^{\theta_1=\dots=\theta_d=0,\phi_H=0}\otimes_RR_{K(\pi_m)})^{\Gal(K(\pi_m,\zeta_{p^m})/K(\zeta_{p^m}))}\\
    =&H^{\theta_1=\dots=\theta_d=0,\phi_H=0}\otimes_RR_{K(\pi_m)}\cap H^{\theta_1=\dots=\theta_d=0,\phi_H=0}\otimes_RR_{K(\zeta_{p^m})}
    \end{split}\]
    which is exactly $H^{\theta_1=\dots=\theta_d=0,\phi_H=0}$. Hence we obtain that 
    \[((H\otimes_RR^{\rm la}_{\widehat L})^{\nabla_1=\dots=\nabla_{d}=0,\nabla_{\rm Sen}=0})^{\Gamma(L/K)}=H^{\theta_1=\dots=\theta_d=0,\phi_H=0}.\]
    So we are done.
  \end{proof}
  \begin{rmk}
    The idea of using   Lie algebra cohomology to prove the above proposition is due to Hui Gao.
\end{rmk}
 \begin{construction}\label{Construction-F} 
   Assume $\frakX$ is a smooth $p$-adic formal scheme over $\calO_K$ with generic fiber $X$.
   For any enhanced Higgs bundle $(\calH,\theta_{\calH},\phi_{\calH})\in\HIG^{\nil}_*(X)$, define $\rF(\calH,\theta_{\calH},\phi_{\calH}) \in \HIG_{G_K}(X_C)$ as follows: 
   \begin{enumerate}
       \item The underlying Higgs bundle on $X_{C,\et}$ is $(\calH\otimes_{\calO_{\frakX}}\calO_{X_C},-(\zeta_p-1)\lambda\theta_{\calH})$.
       
       \item For any local section $x$ of $\calH$ and for any $g\in G_K$, we have 
       \begin{equation}\label{Equ-GaloisAction-F}
           g(x) = (1+\pi E'(\pi)(\zeta_p-1)\lambda c(g))^{-\frac{\phi_H}{E'(\pi)}}(x). 
       \end{equation}
   \end{enumerate}
 \end{construction}
 \begin{prop}\label{Prop-F-welldefined}
   The functor $\rF: \HIG^{\nil}_*(X)\to \HIG_{G_K}(X_C)$ is a well-defined fully faithful functor, which preserves ranks, tensor products and dualities.
 \end{prop}
 \begin{proof}
   We first show that $\rF$ is well-defined. More precisely, we have to check that (\ref{Equ-GaloisAction-F}) induces a $G_K$-action on $\calH\otimes_{\calO_{\frakX}}\calO_{X_C}$ and the Higgs field $-(\zeta_p-1)\lambda\theta_H$ is $G_K$-equivariant.
   
   For any $g_1,g_2\in G_K$ and any local section $x$, we have 
   \begin{equation*}
       \begin{split}
           g_1(g_2(x)) = & g_1((1+\pi E'(\pi)(\zeta_p-1)\lambda c(g_2))^{-\frac{\phi_{\calH}}{E'(\pi)}}(x))\\
           =& (1+\pi E'(\pi)(\zeta_p^{\chi(g_1)}-1)g_1(\lambda) c(g_2))^{-\frac{\phi_{\calH}}{E'(\pi)}}(g_1(x))\quad (\text{as~}\phi_{\calH}~\text{is~defined~over}~\calH)\\
           =&(1+\chi(g_1)c(g_2)\frac{\pi E'(\pi)(\zeta_p-1)\lambda}{1+\pi E'(\pi)(\zeta_p-1)\lambda c(g_1)} )^{-\frac{\phi_{\calH}}{E'(\pi)}}(g_1(x))\quad (\text{by~(\ref{Equ-Lambda})})\\
           = &((1+\chi(g_1)c(g_2)\frac{\pi E'(\pi)(\zeta_p-1)\lambda}{1+\pi E'(\pi)(\zeta_p-1)\lambda c(g_1)} )(1+\pi E'(\pi)(\zeta_p-1)\lambda c(g_1)))^{-\frac{\phi_{\calH}}{E'(\pi)}}(x)\\
           =& (1+\pi E'(\pi)(\zeta_p-1)(\lambda c(g_1)+ \lambda\chi(g_1)c(g_2)))^{-\frac{\phi_{\calH}}{E'(\pi)}}(x)\\
           =& (1+\pi E'(\pi)(\zeta_p-1)\lambda c(g_1g_2))^{-\frac{\phi_{\calH}}{E'(\pi)}}(x)\\
           =&(g_1g_2)(x).
       \end{split}
   \end{equation*}
   So (\ref{Equ-GaloisAction-F}) is a well-defined $G_K$-action on $\rF(\calH,\theta_{\calH},\phi_{\calH})$.
   
   It remains to check that $-(\zeta_p-1)\lambda\theta_{\calH}$ is $G_K$-equivariant. For any $g\in G_K$ and any local section $x$, if we write $\theta_{\calH}(x) = \sum_{i=1}^d\theta_i(x)\otimes\frac{\dlog T_i}{t}$, then we have
   \begin{equation*}
       \begin{split}
           &g(-(\zeta_p-1)\lambda\theta_{\calH}(x))\\
           =&-(\zeta_p^{\chi(g)}-1)g(\lambda)g(\sum_{i=1}^d\theta_i(x)\otimes\frac{\dlog T_i}{t})\\
           =& -\chi(g)^{-1}(\zeta_p^{\chi(g)}-1)g(\lambda)\sum_{i=1}^dg(\theta_i(x))\otimes\frac{\dlog T_i}{t}\\
           =& -\chi(g)^{-1}(\zeta_p^{\chi(g)}-1)g(\lambda)\sum_{i=1}^d(1+\pi E'(\pi)(\zeta_p-1)\lambda c(g))^{-\frac{\phi_{\calH}}{E'(\pi)}}(\theta_i(x))\otimes\frac{\dlog T_i}{t}\\
           =& -(\zeta_p-1)\lambda(1+\pi E'(\pi)(\zeta_p-1)\lambda c(g))^{-1}\sum_{i=1}^d(1+\pi E'(\pi)(\zeta_p-1)\lambda c(g))^{-\frac{\phi_{\calH}}{E'(\pi)}}(\theta_i(x))\otimes\frac{\dlog T_i}{t}\quad(\text{by~(\ref{Equ-Lambda})})\\
           =& -(\zeta_p-1)\lambda\sum_{i=1}^d(1+\pi E'(\pi)(\zeta_p-1)\lambda c(g))^{-\frac{\phi_{\calH}}{E'(\pi)}-1}(\theta_i(x))\otimes\frac{\dlog T_i}{t}
       \end{split}
   \end{equation*}
   and that
   \begin{equation*}
       \begin{split}
           & -(\zeta_p-1)\lambda\theta_{\calH}(g(x))\\
           =& -(\zeta_p-1)\lambda\theta_{\calH}((1+\pi E'(\pi)(\zeta_p-1)\lambda c(g))^{-\frac{\phi_{\calH}}{E'(\pi)}}(x))\\
           = &-(\zeta_p-1)\lambda\sum_{i=1}^d\theta_i((1+\pi E'(\pi)(\zeta_p-1)\lambda c(g))^{-\frac{\phi_{\calH}}{E'(\pi)}}(x))\otimes\frac{\dlog T_i}{t}.
       \end{split}
   \end{equation*}
   Since $[\phi_{\calH},\theta_{\calH}] = -E'(\pi)\theta_{\calH}$ (and hence $[\phi_{\calH},\theta_i] = -E'(\pi)\theta_i$ for all $i$), by (\ref{Equ-GeneralFormulae}), we see that 
   \[(1+\pi E'(\pi)(\zeta_p-1)\lambda c(g))^{-\frac{\phi_{\calH}}{E'(\pi)}-1}(\theta_i(x)) = \theta_i((1+\pi E'(\pi)(\zeta_p-1)\lambda c(g))^{-\frac{\phi_{\calH}}{E'(\pi)}}(x)),\]
   which shows that $-(\zeta_p-1)\lambda\theta_{\calH}$ is compatible with $G_K$-action. 
   So we conclude that $\rF$ is well-defined. By its construction, we see that $\rF$ preserves ranks, tensor products and dualities.

   It remains show that $\rF$ is fully faithful. It suffices to show that there is an isomorphism
   \[\rH^0(X, \HIG(\calH,\theta_{\calH},\phi_{\calH}))\cong \rH^0(X_C,\HIG(\rF(\calH,\theta_{\calH},\phi_{\calH})))^{G_K},\]
   where $\HIG(\rF(\calH,\theta_{\calH},\phi_{\calH})) = \HIG(\calH\otimes_{\calO_{\frakX}}\calO_{X_C},-(\zeta_p-1)\lambda\theta_H)$ denotes the Higgs complex induced by $\rF(\calH,\theta_{\calH},\phi_{\calH})$.
   
   Let $\{\frakX_i\to\frakX\}_{i\in I}$ be a covering of $\frakX$ by small affines $\frakX_i = \Spf(R_i^+)$ and for any $i,j\in I$, let $\frakX_{ij} = \frakX_i\times_{\frakX}\frakX_j = \Spf(R_{ij}^+)$. Let $X_i$ be the generic fibers of $\frakX_i$ and $X_{i,C}$ be its base-change to $C$. Similarly define $X_{ij}$ and $X_{ij,C}$. By noting that 
   \[\rH^0(X, \HIG(\calH,\theta_{\calH},\phi_{\calH})) = \Ker(\prod_{i\in I}\rH^0(X_i, \HIG(\calH,\theta_{\calH},\phi_{\calH})_{\mid  X_i})\to\prod_{i,j\in I}\rH^0(X_{ij}, \HIG(\calH,\theta_{\calH},\phi_{\calH}))_{\mid X_{ij}})\]
   and that 
   \[\begin{split}
       &\rH^0(X_C,\HIG(\rF(\calH,\theta_{\calH},\phi_{\calH})))^{G_K}\\
       =&\Ker(\prod_{i\in I}\rH^0(X_{i,C},\HIG(\rF(\calH,\theta_{\calH},\phi_{\calH}))_{\mid X_{i,C}})^{G_K})\to \prod_{i,j\in I}\rH^0(X_{ij,C},\HIG(\rF(\calH,\theta_{\calH},\phi_{\calH}))_{\mid X_{ij,C}})^{G_K},
   \end{split}\]
    it suffices to show that for $* = i$ or $ij$, 
   \[\rH^0(X_*, \HIG(\calH,\theta_{\calH},\phi_{\calH})_{\mid  X_*})\cong\rH^0(X_{*,C},\HIG(\rF(\calH,\theta_{\calH},\phi_{\calH}))_{\mid X_{*,C}})^{G_K}.\]
   Then Proposition \ref{Prop-F-FullyFaithful} applies (as each $\frakX_*$ is small affine). We are done.
 \end{proof}
 \begin{cor}\label{Cor-IndenpendenceChart}
   Assume $\frakX = \Spf(R^+)$ is small affine. Then the equivalence \[\Vect((R^+)_{\Prism},\overline \calO_{\Prism}[\frac{1}{p}])\simeq \HIG^{\nil}_*(R)\]
   constructed in Theorem \ref{Thm-HTasHiggs-Local} is independent of the framing on $\frakX$.
 \end{cor}
 \begin{proof}
   Let $\rho_1$ and $\rho_2$ be equivalences from $\Vect((R^+)_{\Prism},\overline \calO_{\Prism}[\frac{1}{p}])$ to $\HIG^{\nil}_*(R)$ corresponding to the framings $\Box_1$ and $\Box_2$ on $R^+$ in the sense of Theorem \ref{Thm-HTasHiggs-Local}. Note that in this case, Theorem \ref{Thm-HTasHiggs-Global} holds true and in particular, the diagram (\ref{Diag-CommutativeDiagram}) commutes. So we see that 
   \[\rF\circ\rho_1=\rF\circ\rho_2,\]
   which coincides with the composition
   \[\Vect((R^+)_{\Prism},\overline \calO_{\Prism}[\frac{1}{p}])\xrightarrow{\Res}\Vect((R^+)_{\Prism}^{\perf},\overline \calO_{\Prism}[\frac{1}{p}])\xrightarrow{\simeq}\Vect(X_{\proet},\OX)\xrightarrow{\simeq}\HIG_{G_K}(X_C),\]
   where $X = \Spa(R,R^+)$ denotes the generic fiber of $\frakX$.
   We denote this composition by $\rT$. It is clearly independent of the choice of framings on $R^+$. 
   
   By Proposition \ref{Prop-F-FullyFaithful}, both $\rF$ and $\rT$ are fully faithful.
   Therefore, for any $\bM\in \Vect((R^+)_{\Prism},\overline \calO_{\Prism}[\frac{1}{p}])$, we have
   \[\Hom(\bM,\bM) = \Hom(\rT(\bM),\rT(\bM)) = \Hom(\rF(\rho_1(\bM)),\rF(\rho_2(\bM))) = \Hom(\rho_1(\bM),\rho_2(\bM)).\]
   So the identity $\id_{\bM}$ provides a canonical isomorphism between $\rho_1(\bM)$ and $\rho_2(\bM)$. We win.
 \end{proof}

\subsection{Proof of Theorem \ref{Thm-HTasHiggs-Global}}
 
 Now, we are going to prove Theorem \ref{Thm-HTasHiggs-Global} by noting that both $\Vect((\frakX)_{\Prism},\overline \calO_{\Prism}[\frac{1}{p}])$ and $\HIG^{\nil}_*(X)$ are indeed \'etale stacks.
 
 We fix a covering $\{\frakX_i\to\frakX\}_{i\in I}$ of $\frakX$ by small affine $\frakX_i = \Spf(R_i^+)$ and let $\frakX_{ij} = \frakX_i\times_{\frakX}\frakX_j = \Spf(R_{ij}^+)$ for all $i,j\in I$. For any $\bM\in \Vect((\frakX)_{\Prism},\overline \calO_{\Prism}[\frac{1}{p}])$, denote by $\bM_i$ its restriction to $(\frakX_i)_{\Prism}$ for any $i\in I$. Then we get canonical isomorphisms 
 $\iota_{ij}: \bM_{i\mid_{(\frakX_{ij})_{\Prism}}}\xrightarrow{\cong}\bM_{j\mid_{(\frakX_{ij})_{\Prism}}}$
 satisfying the cocycle condition.
 
 Now applying Theorem \ref{Thm-HTasHiggs-Local}, each $\bM_i$ induces an enhanced Higgs bundle $(\calH_i,\theta_{\calH_i},\phi_{\calH_i})$ in $\HIG^{\nil}_*(X_i)$. By Corollary \ref{Cor-IndenpendenceChart}, the isomorphisms $\iota_{ij}$'s induce isomorphisms 
 \[\rho_{ij}: (\calH_i,\theta_{\calH_i},\phi_{\calH_i})_{\mid_{X_{ij}}}\xrightarrow{\cong}(\calH_j,\theta_{\calH_j},\phi_{\calH_j})_{\mid_{X_{ij}}}\]
 satisfying the cocycle condition. Therefore, we get an enhanced Higgs bundle $(\calH,\theta_{\calH},\phi_{\calH})\in \HIG^{\nil}_*(X)$ whose restriction to $X_i$ coincides with $(\calH_i,\theta_{\calH_i},\phi_{\calH_i})$. This induces a functor 
 \[\rho:\Vect((\frakX)_{\Prism},\overline \calO_{\Prism}[\frac{1}{p}])\to\HIG^{\nil}_*(X).\]

 As $\rho$ is now globally defined and is \'etale locally an equivalence by Theorem \ref{Thm-HTasHiggs-Local}, we see that $\rho$ is an equivalence. Now the functors in Diagram \ref{Diag-CommutativeDiagram} are all globally defined and it is commutative \'etale locally. So we must have Diagram \ref{Diag-CommutativeDiagram} commutes globally.
\begin{rmk}
    If we start with a smooth $p$-adic formal scheme $\frakX$ over $\calO_C$, under some deformation condition on $\frakX$, Theorem \ref{Thm-HTasHiggs-Global} also holds true for ``small Hodge--Tate crystals'' and ``small Higgs bundles'', by using the overconvergent period sheaf $\calO\bC^{\dagger}$ constructed in \cite{Wan23} to replace $\OC$. Here, ``small Hodge--Tate'' crystals should be understood as rational Hodge--Tate crystals admitting $\overline \calO_{\Prism}$-lattices which are ``close'' to $\overline \calO^l_{\Prism}$ for some $l\geq 0$. Similar remark applies to ``small Higgs bundles''. A similar but much more stronger result is also obtained by Tsuji and we thank him for telling us this. 
\end{rmk}

\newpage

\section{The Sen operator}\label{Sect-Sen operator}
Let $\frakX=\Spf(R^+)$ be a small smooth $p$-adic formal scheme over $\calO_K$. Given a rational Hodge--Tate crystal $\bM\in \Vect((R^+)_{\Prism},\overline\calO_{\Prism}[\frac{1}{p}])$, there is an associated enhanced Higgs bundle $(H,\theta_H,\phi_H)$ by Proposition \ref{Prop-HTvsHiggs-Local}. The goal of this section is to prove that the linear operator $\phi_H$ is essentially the classical Sen operator in the cyclotomic case (after extending scalar). This result has been used in Section \ref{inverse Simpson} to prove a global version of Proposition \ref{Prop-HTvsHiggs-Local} in the $p$-inverted case and construct a global inverse Simpson functor at the same time. 

When $R^+=\calO_K$, the coincidence of the operators is \cite[Theorem 8.2]{GMW-HT}. We will follow the strategy in loc.cit and use the theory of locally analytic vectors. For the basics of the theory of locally analytic vectors, we refer to \cite[Section 2]{BC16} and \cite[Section 6]{GMW-HT}.
 
 \subsection{Sen theory in the cyclotomic case}\label{SSec-cyclotomic Sen}
Recall that in \cite{Sen}, Sen established equivalences of the following categories
 \[\Rep_{\Gal(K_{\cyc}/K)}(K_{\cyc}) \to \Rep_{\Gal(K_{\cyc}/K)}(\widehat K_{\cyc})(\simeq \Rep_{G_K}(C))\]
 which preserves ranks, tensor products and dualities.
 Using this, he constructed a faithful functor
 from $\Rep_{G_K}(C)$ to $\Mod_*(C)$, the category of pairs $(V,\phi_V)$ consisting of a finite dimensional $C$-space $V$ and an endomorphism $\phi_V\in \End_{C}(V)$ such that for any $V\in \Rep_{G_K}(C)$ with the corresponding pair $(W,\phi_W)$, there exists a $G_K$-equivariant isomorphism 
 \[\rR\Gamma(G_K,V)\otimes_KC\simeq [W\xrightarrow{\phi_W}W].\]
 
 We briefly review the construction of Sen. For any $V\in \Rep_{G_K}(C)$, one can regard it as a representation in $\Rep_{\Gal(K_{\cyc}/K)}(\widehat K_{\cyc})$ and denote by $V_0$ the corresponding $K_{\cyc}$-representation of $\Gal(K_{\cyc}/K)$. Then $\phi_V$ is the unique endomorphism of $V_0$ such that for any $v\in V_0$, there exists an open subgroup $H_{v_0}$ of $\Gal(K_{\cyc}/K)$ such that for any $\gamma\in H_{v_0}$, 
 \begin{equation}\label{Equ-ActionSen}
     \gamma(v_0) = \exp(\phi_V\log\chi(\gamma))(v_0).
 \end{equation}
  
 Note that $\Gal(K_{\cyc}/K)$ is a $p$-adic Lie group. It is easy to see that $V_0 = V^{\rm la}$, the subspace of $V$ consisting of locally analytic vectors with respect to the action of $\Gal(K_{\cyc}/K)$. Then $\phi_V$ is the generator of the Lie algebra of $\Gal(K_{\cyc}/K)$. See \cite{BC16} for more discussions.
 
 The result of Sen can be generalised to the geometric case. Assume $X = \Spa(R,R^+)$ is smooth over $K$ and admits a toric chart. In \cite{Shi}, Shimizu showed that for any $V\in \Rep^{\rm free}_{\Gal(K_{\cyc}/K)}(R_{\widehat K_{\cyc}})$ with associated $V_0\in \varinjlim_n\Rep^{\rm free}_{\Gal(K_{\cyc}/K)}(R_{K(\zeta_{p^n})})$ (cf. Theorem \ref{Thm-ArithDecompletion} and Remark \ref{Rmk-Decompletion}), there exists a unique endomorphism $\phi_V$ of $V_0$ such that (\ref{Equ-ActionSen}) is still true for any $v_0\in V_0$ and some open subgroup $H_{v_0}\subset \Gal(K_{\cyc}/K)$.
 His result was generalised to any smooth quasi-compact rigid analytic spaces $X$ over $K$ by Petrov in \cite{Pet}. Let us fix some notations.
 
 \begin{notation}\label{Notation-RingedSpace}
For any smooth quasi-compact rigid analytic space $X$ over $K$, let $\calX$ be the ringed space $(X,\calO_X\otimes_KK_{\cyc})$. Note that there exists an obvious way to assign a vector bundle on $X_{\widehat K_{\cyc}}$ (with a continuous $\Gal(K_{\cyc}/K)$-action) to a vector bundle on $\calX$ (with a continuous $\Gal(K_{\cyc}/K)$-action.). Let $\Vect_{\Gal(K_{\cyc}/K)}(\calX)$ denote the category of vector bundles on $\calX$ with continuous $\Gal(K_{\cyc}/K)$-actions.
 \end{notation}
 \begin{prop}[\emph{\cite[Prop. 3.2]{Pet}}]\label{Prop-Petrov}
   There exists an equivalence of categories
   \[\Vect_{\Gal(K_{\cyc}/K)}(\calX)\simeq\Vect_{\Gal(K_{\cyc}/K)}(X_{\widehat K_{\cyc}})\]
   which preserves ranks, tensor products and dualities. Here $\Vect_{\Gal(K_{\cyc}/K)}(X_{\widehat K_{\cyc}})$ denotes the category of vector bundles on $X_{\widehat K_{\cyc}}$ with continuous $\Gal(K_{\cyc}/K)$-actions. Moreover, for any $\calE\in \Vect_{\Gal(K_{\cyc}/K)}(\calX)$, there exists a unique endomorphism $\phi_{\calE}$ of $\calE$ such that for any affinoid $U = \Spa(R,R^+)\subset X$ and for any section $x\in\calE(U)$, there exists an open subgroup $H_x\subset \Gal(K_{\cyc}/K)$ such that for any $\gamma\in H_x$, 
   \begin{equation}\label{Equ-Petrov}
       \gamma(x) = \exp(\phi_{\calE}\log\chi(\gamma))(x).
   \end{equation}
 \end{prop}

 Now let's assume $X=\Spa(R,R^+)$ admits a toric chart over $\Spa(K,\calO_K)$. Following Sen's strategy, Shimizu uses finite vectors (i.e. vectors whose $\Gal(K_{\cyc}/K)$-orbit is a finite set) to construct the decompletion of representations in $\Rep_{\Gal(K_{\cyc}/K)}^{\rm free}(R_{\widehat K_{\cyc}})$. We are now going to show finite vectors coincide with locally analytic vectors in this case.

 Let's first briefly recall the definition of locally analytic vectors. We refer to \cite[\S 2.1]{BC16} for more details.

\begin{dfn}\label{defLAV}
\begin{enumerate}
\item 
Let $G$ be a $p$-adic Lie group, and let $(W, \|\cdot \|)$ be a $\Qp$-Banach representation of $G$.
Let $H$ be an open subgroup of $G$ such that there exist coordinates $c_1,\hdots,c_d : H \to \Zp$ giving rise to an analytic bijection $\cbf : H \to \Zp^d$.
 We say that an element $w \in W$ is an $H$-analytic vector if there exists a sequence $\{w_\kbf\}_{\kbf \in \mathbb{N}^d}$ with $w_\kbf \to 0$ in $W$, such that $$g(w) = \sum_{\kbf \in \mathbb{N}^d} \cbf(g)^\kbf w_\kbf, \quad \forall g \in H.$$
Let $W^{H\dan}$ denote the space of $H$-analytic vectors.

\item $W^{H\dan}$ injects into $\mathcal{C}^{\an}(H, W)$ (the space of analytic functions on $H$ valued in $W$), and we endow it with the induced norm, which we denote as $\|\cdot\|_H$. We have $\|w\|_H=\sup_{\kbf \in \mathbb{N}^d}\|w_{\kbf}\|$, and $W^{H\dan}$ is a Banach space.

\item 
We say that a vector $w \in W$ is \emph{locally analytic} if there exists an open subgroup $H$ as above such that $w \in W^{H\dan}$. Let $W^{G\dla}$ denote the space of such vectors. We have $W^{G\dla} = \cup_{H} W^{H\dan}$ where $H$ runs through  open subgroups of $G$. When $G$ is clear from the context, we write $W^{\la}$ for short.
\end{enumerate}
\end{dfn}
 
 \begin{thm}
   For any finite free $R_{\widehat K_{\cyc}}$-representation $W$ of $\Gal(K_{\cyc}/K)$, we have $W^{\rm fin}=W^{\rm la}$, where $W^{\rm fin}$ denotes the subset of $W$ consisting of elements whose $\Gal(K_{\cyc}/K)$-orbit is finite while $W^{\rm la}$ denotes the subset consisting of locally analytic vectors in the sense of \cite{BC16}.
 \end{thm}
 
 \begin{proof}
   The proof is similar to that of \cite[Thm. 3.2]{BC16}. We first deal with the trivial case, i.e. $W=\widehat K_{\cyc}$. It is easy to see $R_{K_{\cyc}}\subset R_{\widehat K_{\cyc}}^{\rm la}$. For the other direction, let $R_n: R_{\widehat K_{\cyc}}\to R_{K(\zeta_{p^n})}$ be the normalised Tate trace such that for any $x\in R_{\widehat K_{\cyc}}$, $\lim_{n\to +\infty}R_n(x)=x$. Suppose $x\in R_{\widehat K_{\cyc}}^{\Gal(K_{\cyc}/K)\text{-}\rm la}$ is $\Gal(K_{\cyc}/K(\zeta_{p^m}))$-analytic for some $m$. Then for any $k\geq 1$, we have that $R_{m+k}(x)$ is also $\Gal(K_{\cyc}/K(\zeta_{p^m}))$-analytic and fixed by $\Gal(K_{\cyc}/K(\zeta_{p^{m+k}}))$. Therefore $R_{m+k}(x)$ is also fixed by $\Gal(K_{\cyc}/K(\zeta_{p^m}))$. This shows $x=\lim_{k\to +\infty} R_{m+k}(x)$ is in $R_{K(\zeta_{p^m})}$.
   
   Now we come to the general case. By the description of the Galois action in Proposition \ref{Prop-Petrov}, it is easy to see that $W^{\rm fin}\subset W^{\rm la}$. For the converse direction, let's choose a basis $\{e_1,\cdots,e_d\}$ of $W^{\rm fin}$. Then $W^{\rm la}=\oplus_{i=1}^dR_{\widehat K_{\cyc}}^{\rm la}\cdot e_i=\oplus_{i=1}^d R_{K_{\cyc}}\cdot e_i=W^{\rm fin}$ by \cite[Prop. 2.3]{BC16}.
 \end{proof}

 Under the equivalence $\Rep_{\Gal(K_{\cyc}/K)}^{\rm free}(R_{\widehat K_{\cyc}})\simeq \Rep_{\Gal(K_{\cyc,\infty}/K)}^{\rm free}(R_{\widehat K_{\cyc,\infty} })$, we interpret the decompletion result in \cite{Shi}, \cite{Pet} in terms of locally analytic vectors.
 
 \begin{thm}[\cite{Shi}, \cite{Pet}]\label{Sen-cyclotomic}
   Let $W\in \Rep^{\rm free}_{R_{\widehat K_{\cyc,\infty}}}(\Gal(K_{\cyc,\infty}/K))$. Define 
   \[D_{{\rm Sen},K_{\cyc}}(W):=(W^{\Gal(K_{\cyc,\infty}/K_{\cyc})})^{\rm fin}=(W^{\Gal(K_{\cyc,\infty}/K_{\cyc})})^{\gamma\text{-}{\rm la}},\]
   where by $\gamma\text{-}\rm la$, we simply mean the $\Gal(K_{\cyc}/K)$-locally analytic vectors. Then $D_{{\rm Sen},K_{\cyc}}(W)$ is a finite free $R_{K_{\cyc}}$-module such that $D_{{\rm Sen},K_{\cyc}}(W)\otimes_{R_{K_{\cyc}}}R_{\widehat K_{\cyc,\infty}}=W$. Moreover there is a linear operator
   \[
   \nabla_{\gamma}:D_{{\rm Sen},K_{\cyc}}(W)\to D_{{\rm Sen},K_{\cyc}}(W)
   \]
   which is called the Sen operator and can be defined as $\frac{\log(g)}{\log(\chi(g))}$ for some $g\in \Gal(K_{\cyc,\infty}/K_{\infty})\cong \Gal(K_{\cyc/K})$ close enough to $1$.
 \end{thm}
 
 \subsection{Sen theory in the Kummer case}\label{SSec-Kummer Sen}
 
Following the strategy of \cite{GMW-HT}, we now develop a Sen theory with respect to the Kummer tower, which turns out to be closely related to the prismatic theory. We would like to mention that the results in this subsection are due to Hui Gao, and thank him for allowing us to include them here.
\begin{notation}
  \begin{enumerate}
      \item Let $K_m:=K(\pi^{\frac{1}{p^m}},\zeta_{p^m})$ and $\hat G_m:=\Gal(K_{\cyc,\infty}/K_m)$ for any $m\geq 0$. In particular, put $\hat G = \Gal(K_{\cyc,\infty}/K)$.
      
      \item Given any $\bQ_p$-Banach representation $W$ of $\hat G$, let
      \[W^{\tau\text{-}{\rm la},\gamma^{p^m}=1}:=W^{\Gal(K_{\cyc,\infty}/K_{\cyc})\text{-}\rm la}\cap W^{\Gal(K_{\cyc,\infty}/K_{\infty}(\zeta_{p^m}))=1}\]
      for any $m\geq 0$, where $\tau\in\Gal(K_{\cyc,\infty}/K_{\cyc})$ is described as in Convention \ref{Convention-LinearIndependence}.
  \end{enumerate}
\end{notation}

\begin{construction}[\emph{\cite[\S 4.4]{BC16}}]\label{Construction-Alpha}
\ \
  \begin{enumerate}
      \item Since $c:G_K\to \Zp$ represents a cocycle in $\rH^1(G_K,C(1)) = \{0\}$, there exists $\alpha\in C$ such that $c(g)=g(\alpha)\chi(g)-\alpha$. This shows $g(\alpha)=\frac{\alpha}{\chi(g)}+\frac{c(g)}{\chi(g)}$ and in particular $\alpha\in \widehat K_{\cyc,\infty}^{\hat G-\rm la}$.
      
      \item Similarly as in the beginning of \cite[\S 4.2]{BC16}, for any $n\geq 0$, let $\alpha_n\in K_{\cyc,\infty}$ such that $\|\alpha-\alpha_n\|\leq p^{-n}$. Then there exists some $r(n)\gg 0$ such that if $m\geq r(n)$, then $\|\alpha-\alpha_n\|_{\hat G}=\|\alpha-\alpha_n\|\leq p^{-n}$ and $\alpha-\alpha_n\in \widehat K_{\cyc,\infty}^{\hat G_m\text{-}\rm la}$. We may suppose $\{r(n)\}_{n\geq 0}$ is an increasing sequence.
  \end{enumerate}
\end{construction}
 
 \begin{dfn}[\emph{\cite[\S 4.4]{BC16}, \cite[Def. 7.9]{GMW-HT}}]
   Let $(H,\| \cdot \|)$ be a $\bQ_p$-Banach algebra such that $\| \cdot \|$ is sub-multiplicative and let $W\subset H$ is a $\bQ_p$-subalgebra. For any $n\geq 0$, let $W\{\{T\}\}_n$ denote the vector space consisting of $\sum_{k\geq 0} a_kT^k$ with $a_k\in W$ and $p^{nk}a_k\to 0$ when $k\to +\infty$.
 \end{dfn}

 \begin{prop}\label{R-la}
 \begin{enumerate}
     \item $R_{\widehat K_{\cyc,\infty}}^{\hat G\text{-}\rm la}=\cup_n R_{K_{r(n)}}\{\{\alpha-\alpha_n\}\}_n$. 
     
     \item $R_{\widehat K_{\cyc,\infty}}^{\hat G\text{-}{\rm la},\nabla_{\gamma}=0}=R\otimes_K {K_{\cyc,\infty}}=R_{K_{\cyc,\infty}}$.
     
     \item $R_{\widehat K_{\cyc,\infty}}^{\tau\text{-}{\rm la},\gamma^{p^m}=1} = R\otimes_K K_{\infty}(\zeta_{p^m})=R_{K_{\infty}(\zeta_{p^m})}$.
 \end{enumerate}
\end{prop}

\begin{proof}
  (1) The proof is essentially the same as \cite[Prop. 4.12]{BC16}. We now recall the proof therein. 

  Suppose $x\in (R_{\widehat K_{\cyc,\infty}})^{G_n\text{-}{\rm an}}$. For $i\geq 0$, we let
  \[
    y_i=\sum_{k\geq 0}(-1)^k(\alpha-\alpha_m)^k\frac{\nabla_{\tau}^{k+i}(x)}{(k+i)!}\binom{k+i}{k},
  \]
  where $\nabla_{\tau} = \frac{\log(\tau^{p^k})}{p^k}$ for $k\gg 0$.
  Then by similar arguments in the proof of \cite[Thm. 4.2]{BC16}, there exists $m\geq n$ such that $y_i\in R_{\widehat K_{\cyc,\infty}}^{G_m\text{-}\rm an}$ for all $i$ and $x=\sum_{i\geq 0}y_i(\alpha-\alpha_m)^i$ in $(R_{\widehat K_{\cyc,\infty}})^{G_m\text{-}\rm an}$. By Construction \ref{Construction-Alpha}, we see that $\tau^{p^t}(\alpha)=\alpha+p^t$ for any $t\geq 0$. As $\alpha_m\in K_{\cyc,\infty}$, we then have $\nabla_{\tau}(\alpha_m)=0$ and $\nabla_{\tau}(\alpha-\alpha_m) = \nabla_{\tau}(\alpha)=\lim_{t\to+\infty}\frac{\tau^{p^t}-1}{p^t}(\alpha)=1$. From this, we deduce that $\nabla_{\tau}(y_i)=0$ and hence that $\tau^{p^m}(y_i)=y_i$. So $y_i\in ( R_{\widehat{K_m(\zeta_{p^{\infty}})}})^{\Gal(K_m(\zeta_{p^{\infty}})/K_m)\text{-}\rm an}=R_{K_m}$. We are done.

  (2) As $\alpha_m\in K_{\cyc,\infty}$, we have $\nabla_{\gamma}(\alpha_m)=0$. Then we obtain from Construction \ref{Construction-Alpha} that $\nabla_{\gamma}(\alpha-\alpha_m) = \nabla_{\gamma}(\alpha)=\lim_{t\to +\infty}\frac{\gamma^{p^t}-1}{\chi(\gamma^{p^t})-1}(\alpha)=-1$. Therefore, for any $n\geq 0$, we have $(R_{K_{r(n)}}\{\{\alpha-\alpha_n\}\})^{\nabla_{\gamma} = 0} = R_{K_{r(n)}}$. Then Item (2) follows from Item (1) immediately.

  (3) This follows from Item (2) by taking $\gamma^{p^m}$-invariants.
\end{proof}
 
We are now ready to study the Sen theory in the Kummer case.
 \begin{thm}
   Let $W\in \Rep^{\rm free}_{R_{\widehat K_{\cyc,\infty}}}(\hat G)$ of rank $d$ and let
   \[
   D_{{\rm Sen},K_{\infty}}(W):=W^{\tau\text{-}{\rm la},\gamma=1}.
   \]
If $W$ admits a basis consisting of $\hat G$-locally analytic vectors, then $D_{{\rm Sen},K_{\infty}}(W)$ is a finite projective $R_{K_{\infty}}$-module of rank $d$. Moreover, there are identifications
   \[
  D_{{\rm Sen},K_{\infty}}(W)\otimes_{R_{K_{\infty}}}R_{\widehat K_{\cyc,\infty}}^{\hat G\text{-}\rm la}=W^{\hat G\text{-}\rm la}=D_{{\rm Sen},K_{\cyc}}(W)\otimes_{R_{K_{\cyc}}}R_{\widehat K_{\cyc,\infty}}^{\hat G\text{-}\rm la}.
   \]
 \end{thm}
 \begin{proof}
   By \cite[Lem. 8.5]{GMW}, we see that $W^{\hat G\text{-}{\rm la}}$ is a finite free $(R_{\widehat K_{\cyc,\infty}})^{\hat G\text{-}\rm la}$-module with
   \[W^{\hat G\text{-}{\rm la}}\otimes_{R_{\widehat K_{\cyc,\infty}}^{\hat G\text{-}\rm la}}R_{\widehat K_{\cyc,\infty}}\cong W.\] 
  Furthermore, by \cite[Prop. 8.9]{GMW}, we see that $W^{\hat G\text{-}{\rm la},\nabla_{\gamma}=0}$ is a finite free $R_{ K_{\cyc,\infty}}$-module with 
   \[W^{\hat G\text{-}{\rm la},\nabla_{\gamma}=0}\otimes_{R_{K_{\cyc,\infty}}}R_{\widehat K_{\cyc,\infty}}^{\hat G\text{-}\rm la}\cong W^{\hat G\text{-}{\rm la}}.\]
   
   Note that $W^{\hat G\text{-}{\rm la},\nabla_{\gamma}=0}$ admits a $\Gal(K_{\cyc,\infty}/K_{\infty})$-action. As $\Gal(K_{\cyc,\infty}/K_{\infty})$ is topologically finitely generated, we can find a finite free $R_{K_m}$-module $W_m$ of rank $d$ with a $\Gal(K_m/K(\pi^{\frac{1}{p^m}}))$-action for some $m\gg 0$, such that $W_m\otimes_{R_{K_m}}R_{K_{\cyc,\infty}}\cong W^{\hat G\text{-}{\rm la},\nabla_{\gamma}=0}$ is a $\Gal(K_{\cyc,\infty}/K_{\infty})$-equivariant isomorphism. Now by Galois descent, we get a finite projective $R_{K(\pi^{\frac{1}{p^m}})}$-module $W_m^{\gamma=1}$ such that $W_m^{\gamma=1}\otimes_{R_{K(\pi^{\frac{1}{p^m}})}}R_{K_m}\cong W_m$. 
   
   We claim that the finite projective $R_{K_{\infty}}$-module $W_m^{\gamma=1}\otimes_{R_{K(\pi^{\frac{1}{p^m}})}}R_{K_{\infty}}$ is exactly $D_{{\rm Sen},K_{\infty}}(W)$. To see this, we first show that $W^{\tau\text{-}\rm la}=W_m^{\gamma=1}\otimes_{R_{K(\pi^{\frac{1}{p^m}})}}R_{\widehat K_{\cyc,\infty}}^{\tau\text{-}\rm la}$. Note that $W^{\hat G\text{-}{\rm la},\nabla_{\gamma}=0}$ is contained in $W^{\tau\text{-}\rm la}$, which means $W$ admits a basis consisting of $\Gal(K_{\cyc,\infty}/K_{\cyc})$-locally analytic vectors. Moreover as $W^{\tau\text{-}\la}=\bigcup_nW^{\tau^{p^n}\text{-}\an}$ and $W^{\tau^{p^n}\text{-}\an}\subset W^{\tau^{p^{n+1}}\text{-}\an}$, where $\tau^{p^n}\text{-}\an$ means the analytic vectors corresponding to the subgroup of $\Gal(K_{\cyc,\infty}/K_{\cyc})$ generated by $\tau^{p^n}$, we know that $W$ contains a basis contained in $W^{\tau^{p^s}\text{-}\an}$ for some $s>0$. Also, we have $W^{\tau\text{-}\la}=W^{\tau^{p^s}\text{-}\la}$. Then by \cite[Prop. 2.3]{BC16}, we have 
   \[W^{\tau\text{-}\rm la}=W^{\hat G\text{-}{\rm la},\nabla_{\gamma}=0}\otimes_{R_{K_{\cyc,\infty}}}R_{\widehat K_{\cyc,\infty}}^{\tau\text{-}\rm la}=W_m^{\gamma=1}\otimes_{R_{K(\pi^{\frac{1}{p^m}})}}R_{\widehat K_{\cyc,\infty}}^{\tau\text{-}\rm la}.\]
   Now, the claim follows from Proposition \ref{R-la} (3) after taking $\gamma$-invariants. We win.
 \end{proof}
 
 Similar to the cyclotomic case, there is also a linear operator in the Kummer case. Before we move on, we recall some constructions in \cite{GMW-HT}.
 
 \begin{construction}\label{two operators}
   \begin{enumerate}
       \item For any $\hat G$-locally analytic representation $W$, there are two Lie algebra differential operators:
       \begin{enumerate}
           \item $\nabla_{\gamma}:=\frac{\log(g)}{\log(\chi(g))}$ for $g\in \Gal(K_{\cyc,\infty}/K_{\infty})$ close enough to 1. (See also Theorem \ref{Sen-cyclotomic}.)
           \item $\nabla_{\tau}:=\frac{\log(\tau^{p^n})}{p^n}$ for $n\gg 0$. (See also the proof of Proposition \ref{R-la} (1).)
       \end{enumerate}
       
       \item Let $\tilde \lambda:=\prod_{n\geq 0}\varphi^n(\frac{E(u)}{E(0)})\in \calO_{[0,1)}\subset B_{\rm cris}^+$ (cf. \cite{Kisin})\footnote{Note that $\tilde \lambda$ is denoted by $\lambda$ in \cite[Def. 6.10]{GMW-HT}. To distinguish it from the $\lambda:=\theta(\frac{\xi}{E([\pi^{\flat}])})$ defined in this paper, we denote it by $\tilde \lambda$.}, where $\calO_{[0,1)}=\Gamma(D_{[0,1)},\calO_{D_{[0,1)}})$ and $D_{[0,1)}$ is the open unit disc over $W(\kappa)[1/p]$ with coordinate $u$. Then we can define $\mathfrak t:=\frac{t}{p\tilde\lambda}$ which turns out to be in $A_{\inf}$ (cf. \cite[paragraph above Thm. 3.2.2]{Liu07}). Let ${\tilde\lambda}^{\prime}$ be the $u$-derivative of $\tilde\lambda$ in $\calO_{[0,1)}$.
       
       \item
       Define $N_{\nabla}:R_{\widehat K_{\cyc,\infty}}^{\hat G\text{-}\rm la}\to R_{\widehat K_{\cyc,\infty}}^{\hat G\text{-}\rm la}$ by setting $N_{\nabla}:=\frac{1}{p\theta(\mathfrak t)}\nabla_{\tau}$. Recall we always have $K_{\infty}\cap K_{\cyc} = K$ (cf. Convention \ref{Convention-LinearIndependence}). By \cite[Lem. 7.4]{GMW-HT}, $\theta(\mathfrak t)\in \widehat K_{\cyc,\infty}^{\hat G\text{-}{\rm la}}$ is non-zero and hence the operator $N_{\nabla}$ is well-defined.
   \end{enumerate}
 \end{construction}

 Let $W\in \Rep_{R_{\widehat K_{\cyc,\infty}}}^{\rm free}(\hat G)$. Then $D_{{\rm Sen},K_{\infty}}(W)$ consists of locally analytic vectors of $W$ and we have the following operator
 \[
 N_{\nabla}: D_{{\rm Sen},K_{\infty}}(W)\to W^{\hat G\text{-}\rm la}.
 \]

 \begin{thm}\label{Kummer Sen}
   Let $W\in \Rep^{\rm free}_{R_{\hat L}}(\hat G)$. Then there is an $R_{K_{\infty}}$-linear operator
   \[\frac{1}{\theta(u{\tilde\lambda}^{\prime})}N_{\nabla}:D_{{\rm Sen},K_{\infty}}(W)\to D_{{\rm Sen},K_{\infty}}(W).
   \]
    After extending $R_{\widehat K_{\cyc,\infty}}$-linearly, the operator $\frac{1}{\theta(u{\tilde\lambda}^{\prime})}N_{\nabla}: W\to W$ is exactly the Sen operator in Theorem \ref{Sen-cyclotomic}.
 \end{thm}
 \begin{proof}
   Note that for any $g\in \Gal(K_{\cyc,\infty}/K_{\infty})$, $g({\mathfrak t}) = \chi(g){\mathfrak t}$ and $g\tau g^{-1} = \tau^{\chi(g)}$. We see that $N_{\nabla}$ commutes with the action of $\Gal(K_{\cyc,\infty}/K_{\infty})$. So $N_{\nabla}:D_{{\rm Sen},K_{\infty}}(W)\to D_{{\rm Sen},K_{\infty}}(W)$ is well-defined. 
   
   By \cite[Cor. 7.7]{GMW-HT} and Proposition \ref{R-la}, we see that the operator $\nabla_{\rm Sen}:=\frac{1}{\theta(u{\tilde\lambda}^{\prime})}N_{\nabla}+\nabla_{\gamma}$ acts $R_{\widehat K_{\cyc,\infty}}^{\hat G\text{-}\rm la}$-linearly on $W^{\hat G\text{-}\rm la}$. Since $\nabla_{\gamma}$ kills $D_{{\rm Sen},K_{\infty}}(W)$, the operator $\nabla_{\rm Sen}$ is equal to $\frac{1}{\theta(u{\tilde\lambda}^{\prime})}N_{\nabla}$ on $W^{\hat G\text{-}\rm la}=D_{{\rm Sen},K_{\infty}}(W)\otimes_{R_{K_{\infty}}}R_{\widehat K_{\cyc,\infty}}^{\hat G\text{-}\rm la}$. Similarly, as $N_{\nabla}$ kills $D_{{\rm Sen},K_{\cyc}}(W)$, the operator $\nabla_{\rm Sen}$ is equal to $\nabla_{\gamma}$ on $W^{\hat G\text{-}\rm la}=D_{{\rm Sen},K_{\cyc}}(W)\otimes_{R_{K_{\cyc}}}R_{\widehat K_{\cyc,\infty}}^{\hat G\text{-}\rm la}$. So we are done.
 \end{proof}
 
 Now we come back to the prismatic theory and show how it is related to the Sen operator. Let $\frakX=\Spf(R^+)$ be a small smooth $p$-adic formal scheme over $\calO_K$ with the adic generic fiber $X=\Spa(R,R^+)$.
 We will apply the above results to $R_{\widehat K_{\cyc,\infty}}$-representations of $\hat G$ coming from enhanced Higgs modules in $\HIG^{\nil}_*(R)$ in the sense of Construction \ref{Construction-F-Local} and Remark \ref{Rmk-F-Local}.
 
 \begin{thm}[Hui Gao]\label{Prismatic-Sen}
  Let $\rF:\HIG^{\nil}_*(R)\to \HIG_{\hat G}(R_{\widehat K_{\cyc,\infty}})$ be the functor defined in Construction \ref{Construction-F-Local} (together with Remark \ref{Rmk-F-Local}). For any $(H,\theta_H,\phi_H)\in\HIG^{\nil}_*(R)$, let $W:=\rF(H,\theta_H,\phi_H)$ be the underlying $R_{\widehat K_{\cyc,\infty}}$-representations of $\hat G$. Assume furthermore that $H$ is finite free over $R$. Then 
  \begin{enumerate}
      \item $D_{{\rm Sen},K_{\infty}}(W)=H\otimes_{R}R_{K_{\infty}}$.
      \item The operator $\frac{-\phi_H}{E'(\pi)}$ is exactly the operator $\frac{1}{\theta(u{\tilde\lambda}^{\prime})}N_{\nabla}$ on $D_{{\rm Sen},K_{\infty}}(W)$ after extending linearly to $R_{K_{\infty}}$. Moreover, $\frac{-\phi_H}{E'(\pi)}$ is the Sen operator of $W$.
  \end{enumerate}
 \end{thm}
 \begin{proof}
   (1) By the Galois action described in Lemma \ref{Lem-F-Local}, we see that $H\subset W^{\tau\text{-}\an,\gamma=1}\subset D_{{\rm Sen},K_{\infty}}(W)$. In particular, $H$ contains a basis consisting of $\tau$-analytic vectors. Then by \cite[Prop. 2.3]{BC16}, we see that $W^{\tau\text{-}\rm la} = H\otimes_{R}R_{\widehat K_{\cyc,\infty}}^{\tau\text{-}\rm la}$. Since $W^{\gamma=1} = H\otimes_RR_{\widehat K_{\cyc,\infty}}^{\gamma=1}$, we get that 
   \[D_{{\rm Sen},K_{\infty}}(W) = (H\otimes_{R}R_{\widehat K_{\cyc,\infty}}^{\tau\text{-}\rm la})\cap(H\otimes_RR_{\widehat K_{\cyc,\infty}}^{\gamma=1}) = H\otimes_R(R_{\widehat K_{\cyc,\infty}}^{\tau\text{-}{\rm la},\gamma=1}) = H\otimes_RR_{K_{\infty}},\]
   where the last equality follows from Proposition \ref{R-la} (3).
   
   (2) By Lemma \ref{Lem-F-Local}, we see that $\nabla_{\tau}$ acts on $D_{{\rm Sen},K_{\infty}}(W)$ via
   \begin{equation*}
       \begin{split}
           \lim_{n\to+\infty}\frac{\tau^{p^n}-1}{p^n} &= \lim_{n\to+\infty}\frac{1}{p^n}((1+\pi E'(\pi)(\zeta_p-1)\lambda p^n)^{-\frac{\phi_H}{E'(\pi)}}-1)\\
           &= \lim_{n\to+\infty}\sum_{k\geq 1}\prod_{i=0}^{k-1}(\phi_H+iE'(\pi))\frac{(-\pi(\zeta_p-1)\lambda)^kp^{(k-1)n}}{k!}\\
           & = -\pi(\zeta_p-1)\lambda\phi_H.
       \end{split}
   \end{equation*}
   Therefore, we deduce that 
   \[\frac{1}{\theta(u{\tilde\lambda}^{\prime})}N_{\nabla} = \frac{1}{p\theta(u{\tilde\lambda}^{\prime}\mathfrak t)}\nabla_{\tau} = \frac{\pi E'(\pi)(\zeta_p-1)\lambda}{p\theta(u{\tilde\lambda}^{\prime}\mathfrak t)}(-\frac{\phi_H}{E'(\pi)}).\]
   To conclude Item (2), it suffices to show that
   \[\frac{\pi E'(\pi)(\zeta_p-1)\lambda}{p\theta(u{\tilde\lambda}^{\prime}\mathfrak t)} = 1,\]
   which follows exactly from the calculations in the proof of \cite[Thm. 8.2(2)]{GMW-HT}.
 \end{proof}

 \begin{rmk}
   As we have mentioned in the beginning of this section, when $R=K$, Theorem \ref{Prismatic-Sen} proved in \cite{GMW-HT}.
   
   When $R=W(k)[\frac{1}{p}]$ i.e. in the unramified case, Bhatt and Lurie have proved a similar result in \cite{BL22a}. More precisely, by using the prism $(W(k)[[\tilde p]],(\tilde p))$ (corresponding to the cyclotomic case), which is the $\bF_p^{\times}$-invariants of the $q$-de Rham prism, instead of the Breuil--Kisin prism $(W(k)[[u]],(u-p))$ (corresponding to the Kummer case), they proved that the linear operator associated with a Hodge--Tate crystal is exactly the classical Sen operator in the cyclotomic case. Unlike the $q$-de Rham prism, its $\bF_p^{\times}$-invariants behaves quite similarly to the Breuil--Kisin prism. It will be very interesting to figure out whether such a prism exists in the ramified and geometric case. 
 \end{rmk}

\subsection{Arithmetic Higgs bundles}\label{Arithmetic-Higgs-bundles}
 In this subsection, let's assume $\frakX$ is a smooth quasi-compact $p$-adic formal scheme over $\calO_K$ with adic generic fiber $X$. Let $\calX$ be the ringed space introduced in Notation \ref{Notation-RingedSpace}.

 \begin{dfn}\label{Dfn-ArithmeticHiggsBundle}
   By an \emph{arithmetic Higgs bundle} of rank $l$ on $X$, we mean a triple $(\calE,\theta_{\calE},\phi_{\calE})$ satisfying that
   \begin{enumerate}
       \item $(\calE,\theta_{\calE})$ is a Higgs bundle bundle on $\calX$ (see Notation \ref{Notation-RingedSpace}); in other words, $\calE$ is a vector bundle on $\calX$ and $\theta_{\calE}:\calE\to\calE\otimes_{\calO_X}\Omega^1_{X/K}(-1)$ is a Higgs field on $\calE$ (i.e. $\theta_{\calE}$ is $\calO_{\calX}$-linear such that $\theta_{\calE}\wedge\theta_{\calE} = 0$);
       
       \item $\phi_{\calE}$ is an $\calO_{\calX}$-linear endormophism of $\calE$ satisfying $[\phi_{\calE},\theta_{\calE}] = -\theta_{\calE}$; that is, the following diagram
       \begin{equation*}
           \xymatrix@C=0.45cm{
                \calE\ar[r]^{\theta_{\calE}\qquad\quad}\ar[d]_{\phi_{\calE}}&\calE\otimes_{\calO_{X}}\Omega^1_{X/K}(-1)\ar[d]_{\phi_{\calE}-\id_{\calE}} \ar[r]&\cdots\ar[r]&\calE\otimes_{\calO_X} \Omega^d_{X/K}(-d)\ar[d]_{\phi_{\calE}-d\id_{\calE}}\\
                \calE\ar[r]^{\theta_{\calE}\qquad\quad}&\calE\otimes_{\calO_X} \Omega^1_{X/K}(-1) \ar[r]&\cdots\ar[r]&\calE\otimes_{\calO_X} \Omega^d_{X/K}(-d)
              }
       \end{equation*}
       is commutative. We denote its total complex by $\HIG(\calE,\theta_{\calE},\phi_{\calE})$.
   \end{enumerate}
   Let $\HIG^{\rm arith}(X,\calO_{\calX})$ be the category of arithmetic Higgs bundles on $X$.
 \end{dfn}
 Similarly, we have the following definition.
 \begin{dfn}\label{Dfn-ArithmeticHiggsBundle-II}
   Let $L/K$ be any Galois extension of $K$ in $\overline K$ containing $K_{\cyc}$.
   By an \emph{arithmetic Higgs bundle} of rank $l$ on $X_{\widehat L}$, we mean a triple $(\calH,\theta_{\calH},\phi_{\calH})$ satisfying that
   \begin{enumerate}
       \item $(\calH,\theta_{\calH})$ is a Higgs bundle on $X_{\widehat L}$;
       
       \item $\phi_{\calH}$ is an $\calO_{X_{\widehat L}}$-linear endormophism of $\calH$ satisfying $[\phi_{\calH},\theta_{\calH}] = -\theta_{\calH}$; that is, the following diagram
       \begin{equation*}
           \xymatrix@C=0.45cm{
                \calH\ar[r]^{\theta_{\calH}\qquad\quad}\ar[d]_{\phi_{\calH}}&\calH\otimes_{\calO_{X}}\Omega^1_{X/K}(-1)\ar[d]_{\phi_{\calH}-\id_{\calH}} \ar[r]&\cdots\ar[r]&\calH\otimes_{\calO_X} \Omega^d_{X/K}(-d)\ar[d]_{\phi_{\calH}-d\id_{\calH}}\\
                \calH\ar[r]^{\theta_{\calH}\qquad\quad}&\calH\otimes_{\calO_X} \Omega^1_{X/K}(-1) \ar[r]&\cdots\ar[r]&\calH\otimes_{\calO_X} \Omega^d_{X/K}(-d)
              }
       \end{equation*}
       is commutative. We denote its total complex by $\HIG(\calH,\theta_{\calH},\phi_{\calH})$.
   \end{enumerate}

  Let $\HIG^{\rm arith}(X_{\widehat L})$ be the category of arithmetic Higgs bundles on $X_{\widehat L}$.
 \end{dfn}
 The following result follows from Proposition \ref{Prop-Petrov} directly.
 \begin{cor}\label{Cor-Petrov}
   Assume $X$ is a quasi-compact smooth rigid analytic space over $K$. 
   \begin{enumerate}
       \item There exists an equivalence of categories
       \[\HIG_{\Gal(K_{\cyc}/K)}(X,\calO_{\calX})\simeq \HIG_{\Gal(K_{\cyc}/K)}(X_{\widehat K_{\cyc}}),\]
       which preserves ranks, tensor products and dualities. Here, $\HIG_{\Gal(K_{\cyc}/K)}(X,\calO_{\calO_X})$ denotes the category of Higgs bundles on $\calX$ equipped with continuous $\Gal(K_{\cyc}/K)$-actions with are compatible with Higgs fields.
       
       \item There is a faithful functor 
       \[\HIG_{\Gal(K_{\cyc}/K)}(X,\calO_{\calX}) \to \HIG^{\rm arith}(X,\calO_{\calX}).\]
       
       \item For any extension $L$ of $K$ containing $K_{\cyc}$, there is a faithful functor
        \[\HIG^{\rm arith}(X,\calO_{\calX})\to \HIG^{\rm arith}(X_{\widehat L}).\]
   \end{enumerate}
 \end{cor}
 \begin{proof}
       (1) This follows from Proposition \ref{Prop-Petrov} directly.
       
       (2) For any $(\calH,\theta_{\calH})\in \HIG_{\Gal(K_{\cyc}/K)}(X,\calO_{\calX})$, let $(\calE,\theta_{\calE},\phi_{\calE})\in  \HIG^{\rm arith}(X,\calO_{\calX})$ such that $(\calE,\theta_{\calE}) = (\calH,\theta_{\calH})$ and that $\phi_{\calE}$ is the unique operator determined by (\ref{Equ-Petrov}). Then the functor sending $(\calH,\theta_{\calH})$ to $(\calE,\theta_{\calE},\phi_{\calE})$ is well-defined. We claim it is faithful. 
       For this purpose, we may assume $X = \Spa(R,R^+)$ is affinoid and are reduced to showing that 
       \[(E^{\theta_E = 0})^{\Gal(K_{\cyc}/K)}\subset E^{\theta_E = 0,\phi_E=0},\]
       where $(E,\theta_E,\phi_E)$ denotes the global section of $(\calE,\theta_{\calE},\phi_E)$ on $X$, which also inherits a $\Gal(K_{\cyc}/K)$-action from $(\calH,\theta_{\calH})$. However, for any $x\in (E^{\theta_E = 0})^{\Gal(K_{\cyc}/K)}$, by (\ref{Equ-Petrov}), we have 
       \[\phi_E(x) = \lim_{\gamma\to 1}\frac{\gamma(x)-x}{\log\chi(\gamma)} = 0,\]
       which is exactly what we want.
       
       (3) We do have the desired functors by considering corresponding scalar extensions. To see the faithfulness, we may assume $X = \Spa(R,R^+)$ admits a toric chart such that $(H,\theta_H,\phi_H)$, the global section of $(\calH,\theta_{\calH},\phi_{\calH})\in \HIG^{\rm arith}(X,\calO_{\calX})$ on $X$, has $H$ finite free over $R\otimes_KK_{\cyc}$. So we need to show that
       \[H^{\theta_{H}=0,\phi_{H}=0}\subset (H\otimes_{R}R_{\widehat L})^{\theta_{H}=0,\phi_{H}=0}.\]
       This follows from the faithful flatness of $R\to R_{\widehat L}$.
 \end{proof}
 
 Now, we give a relative version of classical Sen theory.
 \begin{cor}\label{Cor-GReptoArith}
   There exists a faithful functor 
   \[\rD:\Vect(X_{\proet},\OX)\to\HIG^{\rm arith}(X_C)\]
   from the category of generalised representations on $X_{\proet}$ to the category of arithmetic Higgs bundles on $X_C$, which preserves ranks, tensor products and dualities.
 \end{cor}
 \begin{proof}
   Just combine Theorem \ref{Thm-HTasHiggs-Global} (for replacing $C$ and $G_K$ by $\widehat K_{\cyc}$ and $\Gal(K_{\cyc}/K)$) with Corollary \ref{Cor-Petrov}.
 \end{proof}
 \begin{rmk}
   When $X = \Spa(K,\calO_K)$, we see that $\Vect(X_{\proet},\OX) = \Rep_{G_K}(C)$ is the category of $C$-representations of $G_K$ and that $\HIG^{\rm arith}(X_C) = \Mod_{*}(C)$ is the category of pairs $(V,\phi_V)$ consisting of finite dimensional $C$-spaces together with endomorphisms $\phi_V$ of $V$. Then the functor in Corollary \ref{Cor-GReptoArith} coincides with the classical functor 
   \[\rD:\Rep_{G_K}(C)\to\Mod_*(C)\]
   introduced in \cite{Sen}. Note that even in this case, the above functor is not fully faithful.
 \end{rmk}
 
 Recall that for any $V\in \Rep_{G_K}(C)$, the underlying $C$-space of $\rD(V)$ is $V$ and the induced Sen operator $\phi_V$ is indeed defined over $K$ (cf. \cite[Thm. 5]{Sen}) by using matrix theory and hence $\rD$ upgrades to a functor from $\Rep_{G_K}(C)$ to $\Mod_*(K)$, the category of pairs $(V,\phi_V)$ consisting of finite dimensional $K$-spaces together with endomorphisms $\phi_V$ of $V$. So the following question appears naturally.
 \begin{question}\label{Ques-ArithOverK}
   For a generalised representation $\calL\in\Vect(X_{\proet},\OX)$ with the induced arithmetic Higgs bundle $\rD(\calL)$ over $X_C$, is $\rD(\calL)$ always defined over $X$ itself?
 \end{question}
 On the other hand, not any pairs $(V,\phi_V)$ in $\Mod_*(C)$ with $\phi_V$ defined over $K$ comes from a $C$-representation $V$ of $G_K$ (cf. \cite[Thm. 7]{Sen}). So one may also ask 
 \begin{question}\label{Ques-ArithComeFromGrep}
   Is any arithmetic Higgs bundle $(\calH,\theta_{\calH},\phi_{\calH})$ over $X_C$ of the form $\rD(\calL)$ for some generalised representation $\calL\in\Vect(X_{\proet},\OX)$?
 \end{question}
 \begin{rmk}
     When $\frakX = \Spf(\calO_K)$, Question \ref{Ques-ArithComeFromGrep} was solved by Fontaine \cite{Fon} after classifying all $C$-representations of $G_K$, using the representation theory of certain algebraic group over $K$. His method looks difficult to generalised to the relative case. For example, even for small affine $\frakX$, we do not have a classification of generalised representations on $X_{\proet}$. We even do not know whether the Hodge--Tate weights of a generalised representation at each classical point is constant, unless the generalised representations comes from $\Qp$-local systems (cf. \cite{Shi}).
 \end{rmk}
 
 We are going to give partial answers of the above two questions in sequels. 
 
 \begin{construction}\label{Construction-TwoFunctor}
   For any enhanced Higgs bundle $(\calH,\theta_{\calH},\phi_{\calH})\in\HIG_*^{\nil}(X)$, define 
   \begin{enumerate}
       \item the Kummer case: $\rF_{\infty}(\calH,\theta_{\calH},\phi_{\calH})= (\calH\otimes_{\calO_{\frakX}}\calO_{X_C},-(\zeta_p-1)\lambda\theta_{\calH},-\frac{\phi_{\calH}}{E'(\pi)})\in \HIG^{\rm arith}(X_C)$, and
       
       \item the cyclotomic case: $\rF_{\cyc}(\calH,\theta_{\calH},\phi_{\calH})\in \HIG^{\rm arith}(X_C)$ as the image under the composition of
       \[\HIG_*^{\nil}(X)\xrightarrow{\rF_S} \Vect(X_{\proet},\OX)\xrightarrow{\rD}\HIG^{\rm arith}(X_C),\]
       where $\rF_S$ is the inverse Simpson functor and $\rD$ is defined in Corollary \ref{Cor-GReptoArith}.
   \end{enumerate}
   Then we get two functors 
   \[\rF_{\infty},\rF_{\cyc}: \HIG_*^{\nil}(X)\to \HIG^{\rm arith}(X_C).\]
   Let $\rF$ be as in Construction \ref{Construction-F}. Then we see that $\rF_{\cyc}$ coincides with the composition 
   \[\HIG_*^{\nil}(X)\xrightarrow{\rF}\HIG_{G_K}(X_C)\xrightarrow{\text{Corollary~ \ref{Cor-Petrov}}}\HIG^{\rm arith}(X_C).\]
 \end{construction}

 \begin{thm}\label{Thm-SenOperator}
 There is an equivalence of functors $\rF_{\infty}\simeq \rF_{\cyc}$.
 \end{thm}
 \begin{proof}
   Fix an $(\calH,\theta_{\calH},\phi_{\calH})\in\HIG_*^{\nil}(X)$.
   By Construction \ref{Construction-F}, we see that the underlying Higgs bundles of $\rF_{\infty}(\calH,\theta_{\calH},\phi_{\calH})\in\HIG_*^{\nil}(X)$ and $\rF_{\cyc}(\calH,\theta_{\calH},\phi_{\calH})\in\HIG_*^{\nil}(X)$ are same. So it suffices to compare the corresponding Sen operators. In other words, we have to show the Sen operator of the underlying $G_K$-vector bundle of $\rF(\calH,\theta_{\calH},\phi_{\calH})\in\HIG_*^{\nil}(X)$ is exactly $-\frac{\phi_{\calH}}{E'(\pi)}$. By the uniqueness criterion of Sen operator in Proposition \ref{Prop-Petrov}, we may assume $\frakX = \Spf(R^+)$ is small affine such that $(\calH,\theta_{\calH},\phi_{\calH})$ is induced by an enhanced Higgs module $(H,\theta_H,\phi_H)$ over $R$ with $H$ finite free. Then the result follows from Theorem \ref{Prismatic-Sen}.
 \end{proof}
 \begin{cor}\label{Cor-AnwserQuestion}
   Let $\frakX$ be a quasi-compact formal scheme over $\calO_K$ with rigid generic fiber $X$.
   \begin{enumerate}
       \item Let $\calL$ be a generalised representation with the corresponding arithmetic Higgs bundle $\rD(\calL)\in\HIG^{\rm arith}(X_C)$. If $\calL$ belongs to the essential image of the functor 
       \[\Vect((\frakX)_{\Prism},\overline \calO_{\Prism}[\frac{1}{p}])\xrightarrow{\Res}\Vect((\frakX)^{\perf}_{\Prism},\overline \calO_{\Prism}[\frac{1}{p}])\xrightarrow{\simeq}\Vect(X_{\proet},\OX),\]
       then $\rD(\calL)$ is defined over $X$.
       
       \item Let $(\calE,\theta_{\calE},\phi_{\calE})$ be an arithmetic Higgs bundle over $X_C$. If it is enhanced in the sense that it lies in the essential image of $\rF_{\infty}$, then there is a generalised representation $\calL\in\Vect(X_{\proet},\OX)$ such that $\rD(\calL) \cong (\calE,\theta_{\calE},\phi_{\calE})$.
   \end{enumerate}
 \end{cor}
 \begin{proof}
   (1) Assume $\calL$ is associated to $\Res(\bM)$ for some $\bM\in\Vect((\frakX)_{\Prism},\overline \calO_{\Prism}[\frac{1}{p}])$ via the equivalence in Theorem \ref{Thm-EtaleRealisation}. By Theorem \ref{Thm-SenOperator}, we see that $\rD(\calL) \cong \rF_{\cyc}(\rho(\bM))$ is defined over $X$, where $\rho$ is the equivalence defined in Theorem \ref{Thm-HTasHiggs-Global}.
   
   (2) Assume $(\calE,\theta_{\calE},\phi_{\calE})=\rF_{\infty}(\calH,\theta_{\calH},\phi_{\calH})$ for some enhance Higgs bundle $(\calH,\theta_{\calH},\phi_{\calH})\in\HIG_*^{\nil}(X)$. Let $\bM\in\Vect((\frakX)_{\Prism},\overline \calO_{\Prism}[\frac{1}{p}])$ such that $\rho(\bM) = (\calH,\theta_{\calH},\phi_{\calH})$. Let $\calL$ be the generalised representation on $X_{\proet}$ corresponding to $\Res(\bM)$ via the equivalence in Theorem \ref{Thm-EtaleRealisation}. Then by Theorem \ref{Thm-SenOperator}, we see that $(\calE,\theta_{\calE},\phi_{\calE})\cong\rD(\calL)$ as desired.
 \end{proof}
 
 Note that Corollary \ref{Cor-AnwserQuestion} answers Question \ref{Ques-ArithOverK} and Question \ref{Ques-ArithComeFromGrep} partially.
 
 \begin{rmk}\label{Rmk-MW-Question}
   Let $\frakX$ be a quasi-compact formal scheme over $\calO_K$ with rigid generic fiber $X$ as before.
   
   Let $\calL$ be a generalised representation on $X_{\proet}$
   Note that for any classical point $x\in X$, $\calL_x$ is a $\widehat{\overline{k(x)}}$-representation of $\Gal(\overline{k(x)}/k(x))$, where $k(x)$ denotes the residue field of $x$. Let $\phi_x$ be the Sen operator of $\calL_x$.
   
   Assume $\calL$ is induced from some rational Hodge--Tate crystal $\bM\in\Vect((\frakX)_{\Prism},\overline \calO_{\Prism}[\frac{1}{p}])$. Then by Theorem \ref{Thm-SenOperator}, we see from Definition \ref{Dfn-EnhancedHiggsBundle} (2) that $\phi_x$ is topologically nilpotent in the sense that
   \[\lim_{n\to+\infty}E'(\pi)^n\prod_{i=0}^{n-1}(\phi_V-i) = 0.\]
   Therefore, $\calL_x$ is a \emph{nearly Hodge--Tate representation} of $\Gal(\overline{k(x)}/k(x))$ over $\widehat{\overline{k(x)}}$ in the sense of \cite{GMW-HT}. 
   
   However, assume that for any classical point $x\in X$, the representation $\calL_x$ of $\Gal(\overline{k(x)}/k(x))$ over $\widehat{\overline{k(x)}}$ is nearly Hodge--Tate (or even Hodge--Tate). We do not know whether $\calL$ is induced from a rational Hodge--Tate crystal on $(\frakX)_{\Prism}$.
 \end{rmk}

 \begin{rmk}\label{rmk:essential image}
    It is worth mentioning that Theorem \ref{Thm-HTasHiggs-Global} was recently generalised by Ansch\"utz--Heuer--Le Bras to the derived case in 
    \cite{AHLB-b,AHLB23c} by using the Hodge--Tate stack developed by Bhatt--Lurie in \cite{BL22a,BL22b}. In particular, they can give a positive anwser to the question in Remark \ref{Rmk-MW-Question}: A generalised representation $\calL$ on $X_{\proet}$ comes from a rational Hodge--Tate crystal if and only if it is nearly Hodge--Tate at every classical point of $X$ (cf. \cite[Cor. 5.21]{AHLB23c}).
 \end{rmk}


\end{document}